\documentclass[reqno]{amsart}
\usepackage{amsmath}
\usepackage{amsfonts}
\usepackage{amssymb}
\usepackage{amstext}
\usepackage{amsbsy}
\usepackage{amsopn}
\usepackage{amsthm}
\usepackage{amsxtra}
\usepackage{upref}
\usepackage{graphicx}
\usepackage{epstopdf}
\DeclareFontFamily{OML}{rsfs}{\skewchar\font'177}
\DeclareFontShape{OML}{rsfs}{m}{n}{ <5> <6> rsfs5 <7> <8> <9> rsfs7
  <10> <10.95> <12> <14.4> <17.28> <20.74> <24.88> rsfs10 }{}
\DeclareMathAlphabet{\mathfs}{OML}{rsfs}{m}{n}

\newcommand{\x}{\ensuremath{\underline{x}}}
\newtheorem{thm}{Theorem}[section]
\newtheorem{lem}[thm]{Lemma}
\newtheorem{prop}[thm]{Proposition}
\newtheorem{defi}[thm]{Definition}

\newtheorem*{theorem*}{Theorem}

\newtheorem{proposition}{Proposition}[subsection]

\newtheorem*{example*}{Example}

\numberwithin{equation}{section}

\newcommand{\del}{\partial}

\renewcommand{\epsilon}{\varepsilon}
\def\wt{\widetilde}

\def\text#1{\textrm{#1}}

\def\emptyset{\varnothing}

\def\e{\epsilon}

\def\vf{\varphi}
\def\a{\alpha}
\def\b{\beta}

\def\d{\delta}

\def\g{\gamma}

\def\l{\lambda}
\def\L{\Lambda}

\def\s{\sigma}

\def\x{\times}
\def \R{\mathbb R}
\def \N{{\mathbb N}}

\def \Z{\mathbb Z}

\def\ov{\overline}
\def\un{\underline}

\def\wh{\widehat}
\def\({\biggl(}
\def\){\biggr)}

\def\<{\bold\langle}
\def\>{\bold\rangle}

\DeclareMathOperator\tr{tr}

\DeclareMathOperator\diam{diam}
\DeclareMathOperator\dist{dist}

\DeclareMathOperator\GL{GL}

\DeclareMathOperator\Lip{Lip}
\DeclareMathOperator\Span{span}
\DeclareMathOperator\NUH{NUH}
\DeclareMathOperator\graph{graph}
\DeclareMathOperator\Hol{H\ddot{o}l}
\DeclareMathOperator\id{Id}
\title[Symbolic dynamics for  surface diffeomorphisms]{Symbolic dynamics for surface diffeomorphisms with positive entropy}
\author{Omri M. Sarig}\thanks{This work was partially supported by NSF grant DMS--0400687 and by ERC award  ERC-2009-StG n$^\circ$ 239885.}
\date{January 17, 2011}
\address{Faculty of Mathematics and Computer Science\\ The Weizmann Institute of Science\\ POB 26, Rehovot, Israel}
\email{omsarig@gmail.com}
\begin{document}
\maketitle
\begin{abstract}
Let $f$ be a  $C^r$ diffeomorphism ($r>1$)   on a compact orientable smooth surface. Suppose the topological entropy $h_{top}(f)$ is positive. Given $0<\chi<h_{top}(f)$, we construct a countable Markov partition for the restriction of $f$ to an invariant set  which is  ``large" in the sense that it has full measure with respect to every  ergodic invariant probability measure with entropy greater than $\chi$. The following results follow: (1) $f$
 has at most countably many ergodic measures of maximal entropy (a conjecture of J. Buzzi), and (2) if $f$
 is $C^\infty$,  then  $\limsup\limits_{n\to\infty}e^{-n h_{top}(f)}\#\{x:f^n(x)=x\}>0$  (a conjecture of A. Katok).
\end{abstract}
\tableofcontents

\addtocounter{part}{-1}
\part{Introduction and statement of results}
\addtocounter{section}{1}
\subsection{Results}
Let $M$ be a compact orientable $C^\infty$ Riemannian manifold of dimension two, and let $f:M\to M$ be a $C^{1+\b}$ diffeomorphism, where $0<\b<1$.
We assume throughout that the topological entropy of $f$ is positive.

Let $P_n(f):=|\{x\in M: f^n(x)=x\}|$. Anatole Katok showed in \cite{KatokIHES} and \cite{KatokICM} that
$
\limsup\limits_{n\to\infty}\frac{1}{n}\log P_n(f)\geq h_{top}(f),
$
 and  conjectured in  \cite{KatokJMD} that if $f$ is $C^\infty$ then
 $
 \limsup\limits_{n\to\infty} e^{-nh_{top}(f)}P_n(f)>0.
 $
\begin{thm}\label{Theorem_Main_PP}
Suppose $f$ is a $C^{1+\b}$ diffeomorphism of a compact orientable smooth surface, and assume $h_{top}(f)>0$. If $f$ has a measure of maximal entropy, then $\exists p\in\N$ s.t.
$
\liminf\limits_{n\to\infty, p|n} e^{-nh_{top}(f)}P_n(f)>0.
$
\end{thm}
\noindent
This proves Katok's conjecture, because  $C^\infty$ diffeomorphisms on compact manifolds have measures of maximal entropy (Newhouse \cite{N}).
Theorem \ref{Theorem_Main_PP} was conjectured to hold as stated above by J\'er\^{o}me Buzzi  \cite{BuzziAffine}.

It was also conjectured in \cite{BuzziAffine} that $f$ admits at most countably many different ergodic measures of maximal entropy. This turns out to be correct:

\begin{thm}\label{Theorem_Main_Intrinsic_Ergodicity}
Suppose $f$ is a $C^{1+\b}$ diffeomorphism of a compact orientable smooth surface. If $h_{top}(f)>0$  then $f$ possesses at most countably many ergodic invariant probability measures with maximal entropy.
\end{thm}
\noindent
Buzzi  conjectured that if $f$ is $C^\infty$, then the number of different ergodic invariant measures of maximal entropy is finite. This conjecture remains open.

Katok's conjecture and Buzzi's conjectures were previously known to hold  in the following cases: Hyperbolic automorphisms of the torus \cite{AW}, Anosov diffeomorphisms \cite{Sinai, SinaiGibbs},  \cite{M}, Axiom A diffeomorphisms \cite{B}, \cite{PP}, continuous piecewise affine  homeomorphisms of affine surfaces \cite{BuzziAffine}. There are also results on non--invertible maps, see \cite{Hof,Hof2} and  \cite{BuzziIntrinsic,BuzziPuzzles}.

\subsection{Symbolic dynamics}\label{SectionSymbDyn} The proof of Theorems  \ref{Theorem_Main_PP} and \ref{Theorem_Main_Intrinsic_Ergodicity} is based on   a change of coordinates which simplifies the iteration of $f$.
The idea, which goes back to the work of Hadamard and Artin on geodesic flows, is to semi-conjugate $f$ on a large set to the left shift  on a {\em topological Markov shift}. We recall the definition.

Let $\mathfs G$ be a directed graph with a countable collection of vertices $\mathfs V$ s.t. every vertex has at least one  edge coming in, and at least one edge coming out. The  {\em topological Markov shift} associated to $\mathfs G$ is the set
$$
\Sigma=\Sigma(\mathfs G):=\{(v_i)_{i\in\Z}\in\mathfs V^{\mathbb Z}:v_i\to v_{i+1}\textrm{ for all }i\}.
$$
We equip $\Sigma$ with the {\em natural metric}: $d(\un{u},\un{v}):=\exp[-\min\{|i|:u_i\neq v_i\}]$, thus turning it into a complete separable metric space. $\Sigma$ is compact iff $\mathfs G$ is finite. $\Sigma$  is locally compact iff every vertex of $\mathfs G$ has finite degree.

The {\em left shift map} $\s:\Sigma\to \Sigma$ is defined by $\s[(v_i)_{i\in\Z}]=(v_{i+1})_{i\in\Z}$.

Let
$
\Sigma^\#:=\{(v_i)_{i\in\Z}\in\Sigma:\exists u,v\in\mathfs V\exists n_k,m_k\uparrow\infty\textrm{ s.t. } v_{-m_k}=u, v_{n_k}=v\}.
$
$\Sigma^\#$ contains all the periodic points of $\s$, and   by the Poincar\'e Recurrence Theorem, every $\s$--invariant probability measure gives $\Sigma^\#$ full measure.

We say that a set $\Omega\subset M$ is {\em $\chi$--large}, if $\mu(\Omega)=1$ for every ergodic invariant probability measure $\mu$ whose entropy is greater than $\chi$. We prove:
\begin{thm}\label{Theorem_Main_Extension}
For every $0<\chi<h_{top}(f)$ there exists a locally compact topological Markov shift $\Sigma_\chi$ and a H\"older continuous map $\pi_\chi:\Sigma_\chi\to M$ s.t.  $\pi_\chi\circ \s=f\circ \pi_\chi$;
 $\pi_\chi[\Sigma_\chi^\#]$ is $\chi$--large; and s.t.
every point in $\pi_\chi[\Sigma_\chi^\#]$ has finitely many pre-images.
\end{thm}
\begin{thm}\label{Theorem_Main_Finite_To_One}
Denote the set of states of $\Sigma_\chi$ by $\mathfs V_\chi$. There exists a function $\vf_\chi:\mathfs V_\chi\x \mathfs V_\chi\to\N$ s.t. if $x=\pi_\chi[(v_i)_{i\in\Z}]$ and $v_i=u$ for infinitely many negative $i$, and $v_i=v$ for infinitely many positive $i$, then $|\pi_\chi^{-1}(x)|\leq \vf_\chi(u,v)$.
\end{thm}
\begin{thm}\label{Theorem_Main_Lift}
 Every ergodic $f$--invariant probability measure $\mu$ on $M$ such that $h_\mu(f)>\chi$ equals
$\wh{\mu}\circ\pi_\chi^{-1}$ for some   ergodic $\s$--invariant probability measure $\wh{\mu}$ on $\Sigma_\chi$ with the same entropy.
\end{thm}
\noindent
The other direction is trivial: If $\wh{\mu}$ is an ergodic $\s$--invariant probability  measure on $\Sigma_\chi$, then $\mu:=\wh{\mu}\circ\pi_\chi^{-1}$ is an ergodic $f$--invariant probability measure on $M$, and $\mu$ has the same entropy as $\wh{\mu}$ because $\pi_\chi$ is finite-to-one.

\medskip
We explain how to use these results to prove
Theorems \ref{Theorem_Main_PP} and \ref{Theorem_Main_Intrinsic_Ergodicity}. This reduction  was already known to Katok and Buzzi \cite{KatokJMD},\cite{BuzziAffine}.

Write $\Sigma_\chi=\Sigma(\mathfs G)$.
By Theorem \ref{Theorem_Main_Lift}, every ergodic measure of maximal entropy $\mu$ for $f$ lifts to  an ergodic measure of maximal entropy $\wh{\mu}$ for $\s$. By ergodicity, $\wh{\mu}$ is  carried by a set $\Sigma(\mathfs G')$ where (1) $\mathfs G'$ is a subgraph of $\mathfs G$, and (2) $\mathfs G'$ is {\em irreducible}: for any two vertices $v_0,v_1$ there exists a path in $\mathfs G'$ from $v_0$ to $v_1$. Since $\wh{\mu}$ is a  measure of maximal entropy for $\s:\Sigma(\mathfs G)\to\Sigma(\mathfs G)$, it is also a measure of maximal entropy for  $\s:\Sigma(\mathfs G')\to\Sigma(\mathfs G')$.

The irreducibility of $\mathfs G'$ means that $\s:\Sigma(\mathfs G')\to\Sigma(\mathfs G')$ is topologically transitive.
Gurevich proved in \cite{Gu1,Gu2} that a topologically transitive topological Markov shift $\Sigma(\mathfs G')$ admits at most one measure of maximal entropy, and that such a measure exists iff $\exists p\in\N$ s.t. for every vertex $v_0$ in $\mathfs G'$,
$$
|\{\un{v}\in\Sigma(\mathfs G'):\s^n(\un{v})=\un{v}, v_0=v\}|\asymp \exp[n h_{\max}(\Sigma(\mathfs G'))]\textrm{ as $n\to\infty$ in  $p\N$},
$$
where  $h_{\max}(\Sigma(\mathfs G'))=\sup\{h_\mu(\s):\mu\textrm{ a $\s$--invariant Borel prob. measure on $\Sigma(\mathfs G')$}\}$, and $h_\mu(\s)$ denotes the metric entropy of $\mu$ w.r.t. $\s$. Here and throughout, $\asymp$ means equality up to bounded multiplicative error.

Since $\pi_\chi\circ\s=f\circ\pi_\chi$, the collection $\{\un{v}\in\Sigma(\mathfs G'):\s^n(\un{v})=\un{v}, v_0=v\}$ is mapped by $\pi_\chi$ to a collection of points $x\in M$ s.t. $f^n(x)=x$. By Theorem \ref{Theorem_Main_Finite_To_One}, the mapping is bounded-to-one, with the number of pre-images bounded by $\vf_\chi(v_0,v_0)$. Thus
$
\liminf_{n\to\infty, p|n}e^{-nh_{\max}(\Sigma(\mathfs G'))}P_n(f)>0.
$
By construction, $h_{\max}(\Sigma(\mathfs G'))=h_{\wh{\mu}}(\s)=h_\mu(f)=\max\{h_{\nu}(f):\nu \textrm{ $f$--inv.}\}$. The last quantity is equal to $h_{top}(f)$ by  the variational principle \cite{G}.
Theorem \ref{Theorem_Main_PP} follows.

This argument also shows that the cardinality of the collection of measures of maximal entropy for $f$ is bounded by the cardinality of the collection of subgraphs $\mathfs G'\subset\mathfs G$ s.t.
(1) $\mathfs G'$ is irreducible, (2) $\Sigma(\mathfs G')$ has a measure of maximal entropy,  and (3) $h_{\max}(\Sigma(\mathfs G'))=h_{\max}(\Sigma(\mathfs G))$.

By a theorem of Salama \cite{Sal} (see also Ruette \cite{Rut}), if $\Sigma(\mathfs G')$ carries a measure of maximal entropy, then  every addition of a vertex or an edge to $\mathfs G'$ increases  $h_{\max}(\Sigma(\mathfs G'))$.  This implies that the subgraphs  $\mathfs G'\subset\mathfs G$ which satisfy (1), (2), and (3) have disjoint sets of vertices. Since $\mathfs G$ is countable, there can be at most countably many such subgraphs, and  Theorem \ref{Theorem_Main_Intrinsic_Ergodicity} follows.

\subsection{Markov partitions}
As in \cite{AW, Sinai, BowenMP},  the symbolic description of $f$ relies on the existence of  a {\em countable Markov partition}. This is a pairwise disjoint collection $\mathfs R$ of Borel sets with the following properties:
\begin{enumerate}
\item {\bf Covering property\/:} The union of $\mathfs R$  is $\chi$--large.
\item {\bf Product structure\/:} There are  $W^s(x,R), W^u(x,R)\subset R$ $(x\in R\in\mathfs R)$ s.t.
\begin{enumerate}
\item $W^u(x,R)\cap W^s(x,R)=\{x\}$.
\item $\forall x,y\in R$,  $\exists z\in R$  s.t.  $W^u(x,R)\cap W^s(y,R)=\{z\}$.
\item $\forall x,y\in R$, $W^s(x,R)$ and $W^s(y,R)$ are equal, or they are disjoint. Similarly for  $W^u(x,R),W^u(y,R)$.
\end{enumerate}
\item {\bf Hyperbolicity:} If $y,z\in W^s(x,R)$, then $d(f^n(y),f^n(z))\xrightarrow[n\to\infty]{}0$. If
$y,z\in W^u(x,R)$, then $d(f^{-n}(y),f^{-n}(z))\xrightarrow[n\to\infty]{}0$.
\item {\bf Markov property\/:} Suppose $R_1,R_2\in\mathfs R$ and $x\in R_1, f(x)\in R_2$, then
$f[W^s(x,R_1)]\subseteq W^s(f(x),R_2)$ and $f^{-1}[W^u(f(x),R_2)]\subseteq W^u(x,R_1)$.
\end{enumerate}
We do not ask for the sets $R$ to be the closure of their interiors.

\subsection{Comparison to other results in the literature}\

{\em Markov partitions for diffeomorphisms\/}. These were previously constructed in the following cases: Hyperbolic toral automorphisms \cite{Berg},\cite{AW}, Anosov diffeomorphisms \cite{Sinai}, pseudo--Anosov diffeomorphisms \cite{FS}, and  Axiom A diffeomorphisms \cite{BowenMP,BowenTAMS}.
This  paper treats   the general case, in dimension two.

{\em Katok horseshoes \cite{KatokIHES,KatokICM},\cite{KM}.} Katok showed that if a $C^{1+\b}$ surface diffeomorphism $f$ has positive entropy, then for every $\e>0$ there is a compact invariant subset $\L_\e$ s.t. $f:\Lambda_\e\to\Lambda_\e$ has a finite Markov partition, and $$h_{top}(f|_{\Lambda_\e})>h_{top}(f)-\e.$$
Typically,  $\L_\e$ will have zero measure w.r.t. any ergodic invariant measure with large entropy. This paper constructs a ``horseshoe" $\pi_\chi(\Sigma_\chi)$ with full measure for all ergodic invariant measures with  large entropy. But  (a) our horseshoe is not compact, (b) its Markov partition is infinite, and (c) the semi-conjugacy $\pi_\chi$ is  not one-to-one as in \cite{KM}. (a) and (b) seem to be unavoidable.

{\em Tower extensions \cite{Ta},\cite{Hof},\cite{Y}\/:} These are representations of certain maps  as  infinite-to-one factors of  other maps (``towers")  which possess   obvious infinite Markov partitions.
Such extensions have been used in the study of  one--dimensional systems with great success, see e.g.  \cite{Hof2},\cite{BuzziIntrinsic}, \cite{Br},\cite{Keller}, \cite{PSZ},\cite{Z}. For  higher dimension, see  \cite{BuzziAffine,BuzziMulti,BuzziPuzzles}, \cite{BT}, \cite{BY}, \cite{Y}.

Unlike tower extensions, our coding is  finite-to-one. This ensures that any ergodic invariant measure with high entropy can be lifted to the symbolic space (Theorem \ref{Theorem_Main_Lift}, see also (\ref{Mu Tilde})). For tower extensions proving the existence of a lift is highly non-trivial, and there are very few results in dimension higher than one,  see  \cite{Ke}, \cite{BuzziAffine}, \cite{BT}, \cite{PSZ} and references therein.

{\em Symbolic extensions \cite{BD},\cite{DN},\cite{BFF}\/.}  These are representations of a diffeomorphism as a  topological factor of $\s:\L\to\L$ where $\L\subset \{1,\ldots,N\}^{\mathbb Z}$ is closed and shift invariant and  $\s$ is the left shift (``subshift").  Recently, Burguet has shown that every  $C^2$ surface diffeomorphism has a symbolic extension \cite{Bur}.

Unlike  symbolic extensions, our coding is by a  non--compact shift space. On the positive side, our space has Markov structure. This gives us access to many results which are not true for  general subshifts, e.g. Gurevich's theory mentioned in  the end of \S\ref{SectionSymbDyn}.

\subsection{Overview of the construction of a Markov partition}\label{SectionOverview}
It is useful first to recall Bowen's construction in the case of an Anosov  diffeomorphisms  \cite{B}.

Bowen's idea was to use  {\em $\e$--pseudo--orbits}. These are sequences of points $\un{x}=\{x_i\}_{i\in\Z}$ such that $d(x_{i+1},f(x_i))<\e$ for all $i$. A pseudo--orbit $\un{x}$ is said to {\em $\d$--shadow} a real orbit $\{f^i(x)\}_{i\in\Z}$ if $d(x_i,f^i(x))<\d$ for all $i\in\Z$.
 Anosov showed that for every $\d$ small enough, there exists an $\e>0$ s.t.
\begin{enumerate}
\item[(A1)] Every $\e$--pseudo--orbit $\un{x}$ $\d$--shadows the real orbit of some unique point $\pi(\un{x})$.
\item[(A2)] ``Finite alphabet suffices": There exists a finite set of points $A$ such that
$\{\pi(\un{x}):\un{x}\in A^\Z\textrm{ is an $\e$--pseudo-orbit}\}$ is the entire manifold.
\item[(A3)] ``Inverse problem": If two pseudo--orbits $\un{x},\un{y}$ $\d$--shadow the same  orbit, then their corresponding coordinates  are close,  $d(x_i,y_i)<2\d$ for all $i\in\Z$.
\end{enumerate}

Since pseudo--orbits are defined in terms of nearest neighbor constraints, one can view the collection of pseudo--orbits in $A^\Z$ as the collection of infinite paths on the graph with set of vertices $A$, and edges $a\to b$ when $d(f(a),b)<\e$.   (A1) and (A2) say that $f$ is a factor of the topological Markov shift
$$
\Sigma:=\{\un{x}\in A^\Z: d(x_{i+1},f(x_i))<\e\textrm{ for all }i\in\Z\}.
$$
The factor map is $\pi$. It is an  infinite--to--one map.

The sets $_0[a]:=\{\un{x}\in\Sigma: x_0=a\}$ form a natural Markov partition for the left shift on $\Sigma$.\footnote{The product structure is given by $W^u(\un{x},{_0[a]}):=\{\un{y}\in\Sigma:y_i=x_i\ (i\leq 0)\}$, $W^s(\un{x},{_0[a]}):=\{\un{y}\in\Sigma:y_i=x_i\ (i\geq 0)\}$.} Their projections
$
Z(a)=\{\pi(\un{x}):\un{x}\in\Sigma\ , x_0=a\}$ $(a\in A)$
would have been natural candidates for a Markov partition, had they not overlapped.  Sinai came up with a set--theoretic procedure for refining
$$
\mathfs Z:=\{Z(a):a\in A\}
$$ into a partition without destroying the product structure. This  partition is a Markov partition \cite{B}.

\medskip
Our proof follows a similar strategy.  But since Anosov's theory of pseudo--orbits relies on uniform hyperbolicity and our setting is only non-uniformly hyperbolic, we have to use  a different device to generate orbits from  symbolic sequences. This problem was previously considered by Kr\"uger \& Troubetzkoy \cite{KT}, but their construction does not work in our setting.

\medskip
In part 1, we introduce  {\em $\e$--chains} as a replacement to  $\e$--pseudo--orbits in the non--uniformly hyperbolic setup. Much like a pseudo--orbit, a chain is a sequence of symbols which satisfies nearest neighbor conditions. Each symbol contains partial information on the location of the point and the position and size of its local stable and unstable manifolds. The nearest neighbor conditions are  tailored in such a way that the following analogues of parts (A1) and (A2) of Anosov's theorem hold for a suitable choice of $\e$:
\begin{enumerate}
\item[(A1')] Every $\e$--chain $\un{v}$ corresponds to a unique real orbit $\pi(\un{u})$;
\item[(A2')] There is a countable set $A$ of symbols s.t.
$
\{\pi(\un{u}): \un{u}\in A^\Z\textrm{ is an $\e$--chain}\}
$
is $\chi$--large. $A$ and $\e$ depend on $\chi$.
\end{enumerate}
As a result, we obtain a representation of $f$ (restricted to a large invariant set) as a factor of a topological Markov shift.

The next step is to construct $\mathfs Z$ as before and try to apply Sinai's method to obtain a countable refining partition. Here we run into a serious problem: whereas Sinai dealt with a finite cover, our cover is infinite, and a  general countable cover need not have a countable refining partition. To avoid such pathologies one needs to ensure that $\mathfs Z$ is {\em locally finite}: Every $Z\in\mathfs Z$  intersects at most finitely many other $Z'\in \mathfs Z$.
This difficulty turns out to be  the heart of the matter.

We deal with this issue in  part 2. Here we obtain the following analogue of part (A3) of Anosov's theorem:
\begin{itemize}
\item[(A3')] If two $\e$--chains $\un{v}, \un{u}$ are ``regular" and $\pi(\un{u})=\pi(\un{v})$, then $u_i$ and $v_i$ are ``close" for every $i\in\Z$ (see \S\ref{Section_Strategy} for the precise statement).
\end{itemize}
Unlike (A3), this is not a trivial statement, because the symbols $u_i,v_i$ contain much more information than mere location. The fact that $\e$--chains satisfy (A3') is the main point of this work.

The alphabet $A$ from part 1 can be chosen s.t. (a) for every $u\in A$, the number of  $v\in A$  ``close" to $u$ is finite, and (b) $\{\pi(\un{u}):\un{u}\in A^\Z,\text{ $\un{u}$ is a regular $\e$--chain}\}$ has full measure  w.r.t. any ergodic invariant probability measure with entropy more than $\chi$. As a result, the  sets $Z(v):=\{\pi(\un{v}):\un{v}\in A^\Z\textrm{ is a regular $\e$-chain}\}$ form a   locally finite cover $\mathfs Z$ of a large set.

Sinai's refinement procedure can now be safely applied to $\mathfs Z$. In part 3, we check that the elements of $\mathfs Z$ have the ``product structure" and ``symbolic Markov properties" needed to push through Bowen's proof that Sinai's  refinement is a Markov partition. We also explain how to deduce Theorems \ref{Theorem_Main_Extension}, \ref{Theorem_Main_Finite_To_One}, and \ref{Theorem_Main_Lift}.
The proofs are modeled on  \cite{B,Bsimple}.

Some of the lemmas we need to develop the theory of $\e$--chains are routine modification  of well--known results in Pesin Theory. Part 4 collects their proofs.

\subsection{Notational conventions} In what follows, $M$ is a compact orientable $C^\infty$ Riemannian manifold of dimension two, and
$f:M\to M$ is  a $C^{1+\b}$ diffeomorphism where $0<\b<1$. We assume that  the topological entropy of $f$ is positive, and we fix once and for all a constant $0<\chi<h_{top}(f)$.

Suppose $P$ is a property. The statement  {\em ``for all $\e$ small enough $P$ holds"} means {\em ``$\exists\e_0>0$ which only depends on $f,M,\b$ and $\chi$ s.t. for all $0<\e<\e_0$ $P$ holds"}.

$T_x M$ is the  tangent space to $M$ at $x$.  The exponential map is denoted by $\exp_x: T_x M\to M$.
The Riemannian norm and inner product on $T_x M$ are denoted by $\|\cdot\|_x$ and $\<\cdot,\cdot\>_x$. Sometimes, we drop the subscript $x$. Given two non-zero vectors $\un{u},\un{v}\in T_x M$, the angle from $\un{u}$ to $\un{v}$ is denoted by $\measuredangle(\un{u},\un{v})$. This is a signed quantity.

Let $V$ be a vector space. The zero element in $V$ is denoted by $\un{0}$.  We identify the tangent space to $V$ at $\un{v}\in V$ with $V$. Let $A:V\to W$ be a linear map between two linear vector space $V,W$. We identify $(dA)_{\un{v}}: T_{\un{v}}V\to T_{A\un{v}}W$ with $A:V\to W$.

Suppose $a,b,c\in\R$.  We write $a=b\pm c$ if $b-c\leq a\leq b+c$, and
 $a=e^{\pm c} b$ if $e^{-c} b\leq a\leq e^c b$.
Let $a_n,b_n>0$, then $a_n\sim b_n$ means that $\frac{a_n}{b_n}\xrightarrow[n\to\infty]{}1$, and   $a_n\asymp b_n$ means that $\exists N,c$ s.t. $\forall n>N$ $(e^{-c}b_n\leq a_n\leq e^c b_n)$.
 Finally, $a\wedge b:=\min\{a,b\}$.

Some abbreviations: {\em s.t.} is ``{such that}", {\em w.r.t} is ``{with respect to}", {\em i.o.} is ``infinitely often",  {\em resp.} is ``{respectively}", and {\em w.l.o.g} is ``without loss of generality".

\part{Chains as pseudo--orbits}
\section{Pesin charts}
\subsection{Non-uniform hyperbolicity}\label{Section_NUH}
 By the variational principle, $f$ admits ergodic invariant probability measures of  entropy larger than $\chi$ (see \cite{G}).
 Quite a lot is known about the properties  of these measures. We will use the following fact, which follows from Ruelle's Entropy Inequality \cite{Ruelle1} and the Oseledets Multiplicative Ergodic Theorem \cite{Os} (see \cite{BP}):
 \begin{thm}[Oseledets, Ruelle]
 Any ergodic invariant probability measure $\mu$ for $f$ s.t. $h_{\mu}(f)>\chi$ gives full probability to the (invariant) set $\NUH_\chi(f)$ of  points $x\in M$ for which there is a decomposition $T_x M=E^s(x)\oplus E^u(x)$ so that
 \begin{enumerate}
 \item $E^s(x)=\Span\{\un{e}^s(x)\}$, $\|\un{e}^s(x)\|_x=1$,  $\lim\limits_{n\to\pm\infty}\frac{1}{n}\log\|(df^n)_x \un{e}^s(x)\|_{f^n(x)}<-\chi$;
 \item $E^u(x)=\Span\{\un{e}^u(x)\}$, $\|\un{e}^u(x)\|_x=1$,  $\lim\limits_{n\to\pm\infty}\frac{1}{n}\log\|(df^n)_x \un{e}^u(x)\|_{f^n(x)}>\chi$;
 \item $\lim\limits_{n\to \pm\infty}\frac{1}{n}\log|\sin\a(f^n(x))|=0$, where $\a(x):=\measuredangle(\un{e}^s(x),\un{e}^u(x))$;
 \item $df_x[E^s(x)]=E^s(f(x))$ and $df_x[E^u(x)]=E^u(f(x))$.
 \end{enumerate}
 \end{thm}

 The splitting $T_x M=E^s(x)\oplus E^u(x)$ is unique, but the vectors $\un{e}^s(x), \un{e}^u(x)$ are only determined up to a sign.  Fix a measurable family of positively oriented bases $(\un{e}_x^1,\un{e}_x^2)$ of $T_x M$ $(x\in M)$.  Choose the signs of  $\un{e}^{s/u}(x)$ in such a way that
 $\measuredangle(\un{e}^1_x,\un{e}^s(x))\in [0,\pi)$ and  $(\un{e}^s(x), \un{e}^u(x))$ has positive orientation.

The set $\NUH(f):=\bigcup_{\chi>0}\NUH_\chi(f)$ is called the {\em non-uniformly hyperbolic set} of $f$, and is $f$--invariant. This set has full probability w.r.t. any ergodic invariant probability measure with positive entropy.

The linear spaces $E^s(x), E^u(x)$ are called, respectively, the {\em stable} and {\em unstable} spaces of $df$. The numbers
 \begin{align*}
 \begin{aligned}
 \log\l(x)&:=\lim\limits_{n\to\pm\infty}\frac{1}{n}\log\|(df^n)_x \un{e}^s(x)\|_{f^n(x)}\\
 \log\mu(x)&:=\lim\limits_{n\to\pm\infty}\frac{1}{n}\log\|(df^n)_x \un{e}^u(x)\|_{f^n(x)}
 \end{aligned}\hspace{2cm} (x\in \NUH(f))
 \end{align*}
  are called  the {\em Lyapunov exponents} of $x$. They are  $f$--invariant, whence constant a.e. w.r.t. any ergodic invariant measure. The value  depends on the measure. On $\NUH_\chi(f)$, $\log\l(x)<-\chi$ and $\log\mu(x)>\chi$.

\subsection{Lyapunov change of coordinates}\label{Section_LCC}
The splitting $T_x M=E^s(x)\oplus E^u(x)$ can be used to diagonalize the action of $df$  on $\{T_x M: x\in\NUH(f)\}$ (``Oseledets--Pesin Reduction").

We describe a  change of coordinates which achieves this.
The construction depends on $\chi$.
Given $x\in\NUH_\chi(f)$, let
 \begin{align*}
 s_\chi(x)&:=\sqrt{2}\left(\sum_{k=0}^\infty e^{2k\chi}\|(df^k)_x \un{e}^s(x)\|^2_{f^k(x)}\right)^{1/2};\\
  u_\chi(x)&:=\sqrt{2}\left(\sum_{k=0}^\infty e^{2k\chi}\|(df^{-k})_x \un{e}^u(x)\|^2_{f^{-k}(x)}\right)^{1/2}.
 \end{align*}
(The factor $\sqrt{2}$ is needed for Lemma \ref{Lemma_C_contracts} below.)
 \begin{defi}
The {\em Lyapunov change of coordinates} (with parameter $\chi$) is the linear map
$
C_\chi(x):\R^2\to T_x M\ \ \ (x\in \NUH_\chi(f))
$
s.t. $C_\chi(x)\un{e}_1=s_\chi(x)^{-1}\un{e}^s(x)$, and
$C_\chi(x)\un{e}_2=u_\chi(x)^{-1}\un{e}^u(x)$, where $\un{e}_1={1\choose 0}$ and $\un{e}_2={0\choose 1}$.
\end{defi}
\noindent
 Notice that $C_\chi(x)$ preserves orientation.
\begin{thm}[Oseledets--Pesin Reduction Theorem]\label{TheoremOPR}
There exists a constant $C_f$ which only depends on $f$ s.t. for every $x\in\NUH_\chi(f)$,
$$
C_\chi(f(x))^{-1}\circ df_x\circ C_\chi(x)=\left(
\begin{array}{cc}
\l_\chi(x) & 0\\
0 & \mu_\chi(x)
\end{array}
\right)
$$
where $C_f^{-1}<|\l_\chi(x)|< e^{-\chi}$ and $e^{\chi}<|\mu_\chi(x)|<C_f$.
\end{thm}

Pesin's original construction in \cite{Pesin}  is slightly different. He defined $s_\chi(x)$ and $u_\chi(x)$ with $e^{-2k\e}\l(x)^{-2k}$ or $e^{-2k\e}\mu(x)^{2k}$ replacing  $e^{2k\chi}$. His method   gives  better bounds on $\l_\chi(x)$ and $\mu_\chi(x)$, and makes sense on all of $\NUH(f)$. Our method can only be guaranteed to work on $\NUH_\chi(f)$, but  it has the advantage that  $C_\chi(x)$ is not sensitive to the values of  $\l(x), \mu(x)$. This is important, because we want to capture the dynamics of all orbits with exponents bounded away from $\chi$, therefore we have to work with points with different Lyapunov exponents.

We need the following definition from linear algebra: suppose $L:V\to W$ is an invertible linear map between two finite dimensional vector spaces equipped with inner products, then the {\em operator norm} of $L$ is $\|L\|:=\max\{\|L\un{v}\|_W:\|\un{v}\|_V=1\}$, and the {\em Frobenius norm} of $L$ is $\|L\|_{Fr}:=\sqrt{\tr(\Theta^t L^t L\Theta)}$, where $\Theta$ is some (any) isometry $\Theta: W\to V$. $\|L\|_{Fr}$ is well defined,\footnote{Proof:
$\tr(\Theta_2^t L^t L \Theta_2)=\tr[\Theta_2^t\Theta_1(\Theta_1^t L^t L\Theta_1)(\Theta_2^t\Theta_1)^t]=\tr(\Theta_1^t L^t L\Theta_1)$.
} and  $\|L\|\leq \|L\|_{Fr}\leq \sqrt{2}\|L\|$.\footnote{Proof: Let $s_1(L)\geq s_2(L)$ denote the singular values of $L$ (equal by definition to the eigenvalues of $\sqrt{L^\ast L}$), then $\|L\|=s_1(L)$, and $\|L\|_{Fr}=\sqrt{s_1(L)^2+s_2(L)^2}$.} One of the advantages of the Frobenius norm is that it has an explicit formula: If $L$ is represented by the matrix $(a_{ij})$ w.r.t. to some (any) orthonormal bases for $V,W$, then \label{FrobeniusPage} $\|L\|_{Fr}=\left(\sum_{ij} a_{ij}^2\right)^{1/2}$.\footnote{Proof: Let $\Theta:W\to V$ be the isometry which maps the base we chose for $W$ to the base we chose for $V$, then $L\Theta:W\to W$ is represented w.r.t. the base we chose for $W$ by the matrix $(a_{ij})$. A calculation shows that $\tr(\Theta^t L^t L \Theta)=\sum a_{ij}^2$.}

Some more information on $C_\chi(x)$ (see the appendix for proofs):
\begin{lem}\label{Lemma_C_norm}
$\|C_\chi(x)^{-1}\|_{Fr}=\sqrt{s_\chi(x)^2+u_\chi(x)^2}/|\sin\a(x)|$.
\end{lem}
\begin{lem}\label{Lemma_C_contracts}
 $C_\chi(x)$ is a contraction: $\|C_\chi(x){\xi\choose\eta}\|_x\leq \|{\xi\choose\eta}\|$ for all $\xi,\eta\in\R$.
\end{lem}
\begin{lem}\label{Lemma_C_tempered}
There is an $\chi$--large invariant set $\NUH^\ast_\chi(f)\subset\NUH_\chi(f)$ s.t.  for every  $x\in\NUH^\ast_\chi(f)$,
 \begin{enumerate}
 \item $\lim\limits_{k\to\pm\infty}\frac{1}{k}\log\|C_\chi(f^k(x))^{-1}\|=0$;
 \item $\lim\limits_{k\to\pm\infty}\frac{1}{k}\log\|C_\chi(f^k(x))\un{e}_i\|_{f^k(x)}=0$, where $\un{e}_1={1\choose 0}$ and  $\un{e}_2={0\choose 1}$;
 \item $\lim\limits_{k\to\pm\infty}\frac{1}{k}\log|\det C_\chi(f^k(x))|=0$.
 \end{enumerate}
\end{lem}

\subsection{Pesin Charts}\label{SectionPC}
Having diagonalized the action of the differential of $f$, we turn to  the action of $f$ itself. The basic result (due to Pesin \cite{Pesin}) is that  $\NUH_\chi(f)$ has an atlas of charts with respect to which  $f$ is   close to a linear hyperbolic  map.

Some notation. Let $\exp_x: T_x M\to M$ denote the exponential map. We denote the zero vector (in $T_x M$ or $\R^2$) by $\un{0}$. Balls and boxes are denoted as follows:
$$
\begin{array}{ll}
B_\eta(x):=\{y\in M: d(x,y)<\eta\} &
B_\eta(\un{0}):=\{\un{v}\in\R^2:\un{v}={v_1\choose v_2},
\sqrt{v_1^2+v_2^2}<\eta\}\\
B_\eta^x(\un{0})=\{\un{v}\in T_x M: \|\un{v}\|_x<\eta\} &
R_\eta(\un{0}):=\{\un{v}\in\R^2:\un{v}={v_1\choose v_2}, |v_1|,|v_2|<\eta\}
\end{array}
$$

Since $M$ is compact, there exist  $r(M), \rho(M)>0$ s.t. for every $x\in M$
\begin{equation}
\label{r_M}
\textrm{$\exp_x$ maps $B_{2r(M)}^x(\un{0})$ diffeomorphically onto a neighborhood of $B_{\rho(M)}(x)$.}
\end{equation}
We take  $\rho(M)$ so small that
$(x,y)\mapsto \exp_{x}^{-1}(y)$ is well defined and $2$--Lipschitz on $B_{\rho(M)}(z)\x B_{\rho(M)}(z)$ for all $z\in M$, and so small that  $\|(d\exp_x^{-1})_{y}\|\leq 2$ for all $y\in B_{\rho(M)}(x)$ (see e.g. \cite[chapter 9]{Spivak}). Since $C_\chi$ is a contraction,
\begin{equation}\label{OPC}
\Psi_x:=\exp_x\circ C_\chi(x)
\end{equation}
maps $R_{r(M)}(\un{0})$ diffeomorphically into $M$. Since $C_\chi(x)$ preserves orientation, $\Psi_x$ preserves orientation.

Let $f_x:=\Psi_{f(x)}^{-1}\circ f\circ \Psi_x$, then the linearization of $f_x$ at $\un{0}$ is the linear hyperbolic map $\left(\begin{array}{cc}
\l_\chi(x) & 0 \\
0 & \mu_\chi(x)
\end{array}
\right)$.
The question is how large is the neighborhood of $\un{0}$ where $f_x$ can be approximated by its linearization. The size of the neighborhood is known. For reasons that will become clear later, we prefer to define it as a quantity taking values in $I_\e:=\{e^{-\frac{1}{3}\ell\e}:\ell\in\N\}$, where $\e$ will be determined later. Set
\begin{equation}\label{Qdef}
\begin{aligned}
Q_\e(x)&:=\max\{q\in I_\e: q\leq \wt{Q}_\chi(x)\}\textrm{ where }\\ \wt{Q}_\chi(x)&:=\e^{3/\b}\bigr(\|C_\chi(x)^{-1}\|_{Fr}
\bigl)^{-12/\b}
\end{aligned}
\end{equation}
\begin{thm}[Pesin]\label{Theorem_OP_charts}
For all $\e$ small enough, and for every $x\in\NUH_\chi(f)$,
\begin{enumerate}
\item $\Psi_x\!:\! R_{10 Q_\e(x)}(\un{0})\!\to\! M$ is a diffeomorphism,  $\Psi_x(\un{0})=x$, and $\|(d\Psi_x)_{\un{u}}\|\leq 2$ on $R_{10 Q_\e(x)}(\un{0})$;
\item $f_x:=\Psi_{f(x)}^{-1}\circ f\circ\Psi_x$ is well defined and injective on $R_{10Q_\e(x)}(\un{0})$ and
\begin{enumerate}
\item $f_x(\un{0})=\un{0}$ and $(df_x)_{\un{0}}=\left(\begin{array}{cc}
A(x) & 0\\
0 & B(x)
\end{array}
\right)$ where  $C_f^{-1}<|A(x)|<e^{-\chi}$ and $e^{\chi}<|B(x)|<C_f$ (cf. Theorem. \ref{TheoremOPR});
\item The $C^{1+\frac{\b}{2}}$-distance between $f_x$ \!and $(\!df_x\!)_{\un{0}}$ on $R_{10 Q_\e(x)}\!(\un{0})$ is less than $\e$.
\end{enumerate}
\item The symmetric statement holds for $f_x^{-1}=\Psi_x^{-1}\circ f^{-1}\circ \Psi_{f(x)}$.
\end{enumerate}
\end{thm}
\noindent
This is a version of \cite[Theorem 5.6.1]{BP}. See the appendix for the proof.

\begin{defi}
Suppose $x\in \NUH_\chi(f)$ and $0<\eta\leq Q_\e(x)$. The {\em Pesin chart} $\Psi_x^\eta$ is the map
$\Psi_x: R_\eta(\un{0})\to M$.
\end{defi}

\noindent
Some  additional information on $Q_\e(x)$ (see the appendix for proofs):

\begin{lem}\label{Lemma_Q_tempered}
The following holds for all $\e$ small enough:
\begin{enumerate}
\item $Q_\e(x)<\e^{3/\b}$ on $\NUH_\chi(f)$;
\item $\|C_\chi(f^i (x))^{-1}\|^{12}<\e^{2/\b}/Q_\e(x)$ for $i=-1,0,1$;
\item $\{Q_\e(x):Q_\e(x)>t,x\in\NUH_\chi(f)\}$ is finite for all $t>0$;
\item $\frac{1}{n}\log Q_\e(f^n(x))\xrightarrow[n\to\infty]{}0$ on $\NUH^\ast_\chi(f)$ (cf. Lemma \ref{Lemma_C_tempered});
\item $F^{-1}\leq Q_\e\circ f/Q_\e\leq F$ on $\NUH_\chi(f)$, where $F$ is independent of $\e$;
\item there exists a function $q_\e:\NUH^\ast_\chi(f)\to (0,1)$ so that $q_\e(x)<\e Q_\e(x)$ and $e^{-\e/3}\leq q_\e\circ f/q_\e\leq e^{\e/3}$ on $\NUH^\ast_\chi(f)$.
\end{enumerate}
\end{lem}
\subsection{$\NUH_\chi^\#(f)$}\label{SectionSharp}
The set $\NUH_\chi^\ast(f)$ constructed in Lemma \ref{Lemma_C_tempered} is $\chi$--large. By the Poincar\'e Recurrence Theorem,  the  set
\begin{equation}\label{NUH_sharp}
\NUH_\chi^\#(f):=\{x\in\NUH_\chi^\ast(f):\limsup_{n\to\infty} q_\e(f^n(x)), \limsup_{n\to\infty}q_\e(f^{-n}(x))\neq 0\}
\end{equation}
is $\chi$--large. This is the set that we will attempt to cover by  a Markov partition.

\section{Overlapping charts}
We would like to replace
$
\mathfs C:=\{\Psi_x^\eta: x\in \NUH^\ast_\chi(f),0<\eta\leq Q_\e(x)\}
$ by a countable collection $\mathfs A$ in such a way that every  element of $\mathfs C$  ``overlaps" some element of $\mathfs A$ ``well". Later, we will    use $\mathfs A$ to construct  the set of vertices of a directed graph  related the dynamics of  $f$.

\subsection{The overlap condition}\label{SectionOC}
We need to compare the maps $C_\chi(x): \R^2\to T_x M$ for different $x\in M$, even though they take values in different spaces.
 We circumvent the problem as follows.
Every $x\in M$ has an open neighborhood $D$ of diameter less than $\rho(M)$ and a smooth map $\Theta_D:TD\to \R^2$ s.t.
\begin{enumerate}
\item $\Theta_D: T_x M\to \R^2$ is a linear isometry for every $x\in D$;
\item let $\vartheta_x:=(\Theta_D|_{T_x M})^{-1}: \R^2\to T_x M$, then  $(x,\un{u})\mapsto(\exp_x\circ \vartheta_x)(\un{u})$ is smooth and Lipschitz on $D\x B_2(\un{0})$ with respect to the metric $d(x,x')+\|\un{u}-\un{u}'\|$;
\item $x\mapsto \vartheta_x^{-1}\circ\exp_x^{-1}$ is a Lipschitz map from $D$ into  $C^{2}(D,\R^2)$, the space of $C^2$ maps  from $D$ to $\R^2$.
\end{enumerate}
Let $\mathfs D$ be an finite cover of $M$ by such neighborhoods. Let $\e(\mathfs D$) be a Lebesgue number for $\mathfs D$. If $d(x,y)<\e(\mathfs D)$, then $x,y$ fall in some element $D$. Instead of comparing $C_\chi(x)$ to $C_\chi(y)$, we will compare $\Theta_D\circ C_\chi(x)$ to $\Theta_D\circ C_\chi(y)$ (two linear maps from $\R^2$ to $\R^2$).

\begin{defi}
Two Pesin charts $\Psi_{x_1}^{\eta_1}, \Psi_{x_2}^{\eta_2}$   {\em $\e$--overlap} if $e^{-\e}<\frac{\eta_1}{\eta_2}<e^{\e}$, and for some $D\in\mathfs D$,
$x_1, x_2\in D$ and
 $
 d(x_1,x_2)+\|\Theta_D\circ C_\chi(x_1)-\Theta_D\circ C_\chi(x_2)\|<\eta_1^4\eta_2^4.
 $
\end{defi}
The overlap condition is symmetric. It is also monotone: if  $\Psi_{x_i}^{\eta_i}$ $\e$--overlap, then $\Psi_{x_i}^{\xi_i}$
 $\e$--overlap for all  $\eta_i\leq\xi_i\leq Q_\e(x_i)$ s.t. $e^{-\e}<\xi_1/\xi_2<e^{\e}$.
 Notice that the overlap requirement is stronger at areas of $\NUH_\chi(f)$ where $s_\chi(x)$ or $u_\chi(x)$ are large or where $\un{e}^s(x)$ and $\un{e}^u(x)$ are nearly parallel. This is because by construction
$$
\eta_i\leq Q_\e(x_i)\ll \|C_\chi(x_i)^{-1}\|^{-1}_{Fr}=\frac{\sqrt{s_\chi(x)^2+u_\chi(x)^2}}{|\sin\a(x)|}.
$$
The following proposition explains what the overlap condition means.

\begin{prop}\label{Prop_Overlap_Meaning}
The following holds for all $\e$ small.
If $\Psi_{x_1}:R_{\eta_1}(\un{0})\to M$ and $\Psi_{x_2}:R_{\eta_2}(\un{0})\to M$ $\e$--overlap, then
\begin{enumerate}
\item $\Psi_{x_1}[R_{e^{-2\e}\eta_1}(\un{0})]\subset\Psi_{x_2}[R_{\eta_2}(\un{0})]$ and $\Psi_{x_2}[R_{e^{-2\e}\eta_2}(\un{0})]\subset\Psi_{x_1}[R_{\eta_1}(\un{0})]$;
\item $\dist_{C^{1+\frac{\b}{2}}}(\Psi_{x_i}^{-1}\circ\Psi_{x_j},\id)<\e\eta_i^2\eta_j^2$  $(\{i,j\}=\{1,2\})$, where the $C^{1+\frac{\b}{2}}$--distance is calculated on $R_{e^{-\e}r(M)}(\un{0})$ and $r_M$ is defined in  $(\ref{r_M})$.
\end{enumerate}
\end{prop}
\begin{proof}
Suppose $\Psi_{x_i}^{\eta_i}$ $\e$--overlap, and fix some $D\in\mathfs D$ which contains $x_1$ and $x_2$
such that $d(x_1,x_2)+\|\Theta_D\circ C_\chi(x_1)
-\Theta_D\circ C_\chi(x_2)\|<\eta_1^4\eta_2^4$. Write $C_i:=\Theta_D\circ C_\chi(x_i)$, then
 $
 \Psi_{x_i}=\exp_{x_i}\circ\vartheta_{x_i}\circ C_i.
 $

By the  definition of Pesin charts, $\eta_i\leq Q_\e(x_i)$, where $Q_\e(x_i)$ is given by
(\ref{Qdef}). Lemma \ref{Lemma_C_contracts} and  the general inequality   $\|\cdot\|_{Fr}\geq \|\cdot\|$ (see page \pageref{FrobeniusPage}) guarantee that
\begin{equation}\label{settle}
\eta_i\leq \e^{3/\b}\|C_\chi(x_i)^{-1}\|^{-12/\b}.
\end{equation}
In particular, $\eta_i<\e^{3/\b}$.

Our first constraint on $\e$ is that it be so small that
\begin{equation}\label{irene}
\e^{3/\b}<\frac{\min\{1,r(M),\rho(M)\}}{5(L_1+L_2+L_3+L_4)^3},
\end{equation}
 where $r(M)$ and $\rho(M)$ are given by (\ref{r_M}), and
\begin{enumerate}
\item $L_1$ is a common Lipschitz constant for the maps $(x,\un{v})\mapsto(\exp_x\circ\vartheta_x)(\un{v})$ on $D\x B_{r(M)}(\un{0})$ $(D\in\mathfs D)$;\label{L_Page}
\item $L_2$ is a common Lipschitz constant for the maps $x\mapsto \vartheta_x^{-1}\circ\exp_x^{-1}$ from $D$ into $C^{2}(D,\R^2)$ ($D\in\mathfs D$);
\item $L_3$ is a  common Lipschitz constant for $\exp_x^{-1}: B_{\rho(M)}(x)\to T_x M$ $(x\in M)$;
\item $L_4$ is a common Lipschitz constant for $\exp_x: B_{r(M)}^x(\un{0})\to M$ $(x\in M)$.
\end{enumerate}
We assume w.l.o.g. that these constants are all larger than one.

\medskip
\noindent
{\em Part 1.\/} $\Psi_{x_1}[R_{e^{-2\e}\eta_1}(\un{0})]\subset\Psi_{x_2}[R_{\eta_2}(\un{0})]$.

\medskip
\noindent
{\em Proof.\/}
 Suppose $\un{v}\in R_{e^{-2\e}\eta_1}(\un{0})$. Lemma \ref{Lemma_C_contracts} says that
 $C_\chi(x_1)$ is  a contraction, therefore
 $
\|C_1\un{v}\|=\|C_\chi(x_1)\un{v}\|\leq \|\un{v}\|$, and $(x_1,{C_1}\un{v}), (x_2,{C_1}\un{v})\in D\x B_{r(M)}(\un{0})$.  Since $d(x_1,x_2)<\eta_1^4\eta_2^4$,
 $$
 d\left(\exp_{x_2}\circ\vartheta_{x_2}[C_1\un{v}],\exp_{x_1}\circ\vartheta_{x_1}[C_1\un{v}]\right)<L_1\eta_1^4\eta_2^4.
 $$
It follows that $\Psi_{x_1}(\un{v})\in B_{L_1\eta_1^4\eta_2^4}(\exp_{x_2}\circ\vartheta_{x_2}(C_1\un{v}))$. Call this ball $B$.

The radius of $B$ is less than $\rho(M)$ because of our assumptions on $\e$. Therefore $\exp_{x_2}^{-1}$ is well defined and Lipschitz on $B$, and its Lipschitz constant  is at most $L_3$. Writing $B=\exp_{x_2}[\exp_{x_2}^{-1}(B)]$, we deduce that
$$
\Psi_{x_1}(\un{v})\in B\subset \exp_{x_2}[B_{L_3L_1\eta_1^4\eta_2^4}^{x_2}(\vartheta_{x_2}(C_1\un{v}))]=:\Psi_{x_2}[E],
$$
where ${E}:=C_\chi(x_2)^{-1}[B_{L_3L_1\eta_1^4\eta_2^4}^{x_2}(\vartheta_{x_2}(C_1\un{v}))]$.

We claim that $E\subset R_{\eta_2}(\un{0})$.  First note that
$
E\subset B_{\|C_\chi(x_2)^{-1}\| L_3 L_1 \eta_1^4\eta_2^4}(C_2^{-2}C_1\un{v}),
$
 therefore if $\un{w}\in E$, then
\begin{align*}
\|\un{w}\|_\infty&\leq \|C_2^{-1}C_1\un{v}\|_\infty+\|C_\chi(x_2)^{-1}\| L_3 L_1 \eta_1^4\eta_2^4\\
&\leq \|(C_2^{-1}C_1-\id)\un{v}\|_\infty+ \|\un{v}\|_\infty+\|C_\chi(x_2)^{-1}\| L_3 L_1 \eta_1^4\eta_2^4\\
&\leq \|\un{v}\|_\infty +\sqrt{2}\|C_2^{-1}\|\|C_1-C_2\|\|\un{v}\|_\infty+\|C_\chi(x_2)^{-1}\| L_3 L_1 \eta_1^4\eta_2^4\\
&\leq e^{-2\e}\eta_1+\|C_\chi(x_2)^{-1}\|(\eta_1^4\eta_2^4\sqrt{2} e^{-2\e}\eta_1+L_3 L_1 \eta_1^4\eta_2^4)\ \ (\because \|C_1-C_2\|<\eta_1^4\eta_2^4)\\
&\leq e^{-2\e}\eta_1+\|C_\chi(x_2)^{-1}\|\eta_2^4\cdot [(e^{-2\e}\sqrt{2}\eta_1+L_3 L_1)\eta_1^3]\cdot\eta_1\\
&< e^{-2\e}\eta_1+\e^2 \eta_1,\ \ \textrm{ because of (\ref{settle}) and (\ref{irene})}\\
&<e^\e(e^{-2\e}+\e^2)\eta_2<\eta_2, \textrm{ because $\eta_1<e^{\e}\eta_2$ and $0<\e<\tfrac{1}{5}$ by (\ref{irene})}.
\end{align*}
It follows that $E\subset R_{\eta_2}(\un{0})$. Thus  $\Psi_{x_1}(\un{v})\in \Psi_{x_2}[R_{\eta_2}(\un{0})]$. Part 1 follows.

\medskip
\noindent
{\em Part 2.\/} The $C^{1+\b/2}$--distance between $\Psi_{x_1}^{-1}\circ\Psi_{x_2}$  on
$R_{e^{-\e}r(M)}(\un{0})$ is less than $\e\eta_1$.

\medskip
\noindent
{\em Proof.\/} One can show exactly as in the proof of part 1 that $\Psi_{x_1}[R_{e^{-\e}r(M)}(\un{0})]\subset \Psi_{x_2}[R_{r(M)}(\un{0})]$, therefore $\Psi_{x_1}^{-1}\circ\Psi_{x_2}$  is well defined on
$R_{e^{-\e}r(M)}(\un{0})$.
We calculate the distance of this map from the identity:
\begin{align*}
\Psi_{x_1}^{-1}\circ \Psi_{x_2}&=C_1^{-1}\circ\vartheta_{x_1}^{-1}\circ \exp_{x_1}^{-1}\circ\exp_{x_2}\circ\vartheta_{x_2}\circ C_2\\
&=C_1^{-1}\circ[\vartheta_{x_1}^{-1}\circ \exp_{x_1}^{-1}+\vartheta_{x_2}^{-1}\circ\exp_{x_2}^{-1}-\vartheta_{x_2}^{-1}\circ\exp_{x_2}^{-1}]\circ\exp_{x_2}\circ\vartheta_{x_2}\circ C_2\\
&=C_1^{-1}C_2+C_1^{-1}\circ[\vartheta_{x_1}^{-1}\circ \exp_{x_1}^{-1}-\vartheta_{x_2}^{-1}\circ\exp_{x_2}^{-1}]\circ\Psi_{x_2}\\
&=\id+C_1^{-1}(C_2-C_1)+C_1^{-1}\circ[\vartheta_{x_1}^{-1}\circ \exp_{x_1}^{-1}-\vartheta_{x_2}^{-1}\circ\exp_{x_2}^{-1}]\circ\Psi_{x_2}.
\end{align*}
The $C^{1+\b/2}$--norm of the  second summand is less than $\|C_1^{-1}\|\eta_1^4\eta_2^4$. The $C^{1+\b/2}$--norm of the third summand is less than
$$
\|C_1^{-1}\|\cdot L_2d(x_1,x_2)\cdot L_4^{1+\frac{\b}{2}}.
$$
This is less than $\|C_1^{-1}\|L_2 L_4^2 \eta_1^4\eta_2^4$.

It follows that $\dist_{C^{1+\b/2}}(\Psi_{x_1}^{-1}\circ\Psi_{x_2},\id)<\|C_1^{-1}\|(1+L_2 L_4^2)\eta_1^4\eta_2^4$. This is (much) smaller than $\e\eta_1^2\eta_2^2$, because of (\ref{settle}) and (\ref{irene}).
\end{proof}
We record the following fact for future reference:
\begin{lem}\label{Lemma_QQ}
Suppose $\Psi_{x_1}^{\eta_1} ,\Psi_{x_2}^{\eta_2}$ $\e$--overlap, then
$$
\frac{s_\chi(x_1)}{s_\chi(x_2)},\frac{u_\chi(x_1)}{u_\chi(x_2)}\in [ e^{-Q_\e(x_1)Q_\e(x_2)},e^{Q_\e(x_1)Q_\e(x_2)}].
$$
\end{lem}
\begin{proof} We use the notation of the previous proof.
 $\Psi_{x_2}^{-1}\circ\Psi_{x_1}$ maps $R_{e^{-\e}\eta_1}(\un{0})$ into $\R^2$. Its derivative at the origin is
\begin{align*}
 A&:=
C_\chi(x_2)^{-1} d(\exp_{x_2}^{-1})_{x_1}C_\chi(x_1)=C_2^{-1}d[\vartheta_{x_2}^{-1}\exp_{x_2}^{-1}]_{x_1}\vartheta_{x_1} C_1\\
&=C_2^{-1} C_1+C_2^{-1}[d[\vartheta_{x_2}^{-1}\exp_{x_2}^{-1}]_{x_1}-\vartheta_{x_1}^{-1}] \vartheta_{x_1}C_1\\
&\equiv C_2^{-1} C_1+C_2^{-1}\left(d[\vartheta_{x_2}^{-1}\exp_{x_2}^{-1}]_{x_1}-d[\vartheta_{x_1}^{-1}\exp_{x_1}^{-1}]_{x_1}\right)\vartheta_{x_1} C_1.
\end{align*}
Since  $
\|d[\vartheta_{x_2}^{-1}\exp_{x_2}^{-1}]_{x_1}-d[\vartheta_{x_1}^{-1}\exp_{x_1}^{-1}]_{x_1}\|<L_2 d(x_1,x_2)<L_2 \eta_1^4\eta_2^4<\e\eta_1^2\eta_2^2$, and
$
\|A-\id\|<\dist_{C^1}(\Psi_{x_2}^{-1}\circ\Psi_{x_1},\id)<\e\eta_1^2\eta_2^2,
$
$$
\|C_2^{-1}C_1-\id\|<2\e\eta_1^2\eta_2^2.
$$
It follows that $\|C_2-C_1\|<2\e\|C_2^{-1}\|\eta_1^2\eta_2^2$.

Recall that $s_\chi(x_i)^{-1}=\|C_\chi(x_i)\un{e}_1\|$ and $s_\chi(x_i)=\|C_\chi(x_i)^{-1}\un{e}^s(x_i)\|$, so
\begin{align*}
\left|\frac{s_\chi(x_1)}{s_\chi(x_2)}-1\right|&=\left|\frac{s_\chi(x_2)^{-1}-s_\chi(x_1)^{-1}}{s_\chi(x_1)^{-1}}\right|\\
&\leq \|C_\chi(x_1)^{-1}\|\cdot\bigl| \|C_\chi(x_1)\un{e}_1\|-\|C_\chi(x_2)\un{e}_1\|\bigr|\\
&=\|C_1^{-1}\|\cdot\bigl| \|C_1\un{e}_1\|-\|C_2\un{e}_1\|\bigr|\\
&\leq \|C_1^{-1}\|\cdot\|C_1-C_2\|<2\e\|C_1^{-1}\|\|C_2^{-1}\|\eta_1^2\eta_2^2< \e \eta_1\eta_2.
\end{align*}
Similarly
$
\left|\frac{u_\chi(x_1)}{u_\chi(x_2)}-1\right|<\e\eta_1\eta_2.
$
Since $\eta_i<Q_\e(x_i)$, the lemma follows.
\end{proof}

\subsection{The form of $f$ in overlapping charts} Theorem \ref{Theorem_OP_charts} says that $\Psi_{f(x)}^{-1}\circ f\circ\Psi_x$ is close to a linear hyperbolic map. This remains the case if we replace  $\Psi_{f(x)}$ by some overlapping chart $\Psi_y$:
\begin{prop}\label{Prop_f_xy}
The following holds for all $\e$ small enough. Suppose $x,y\in\NUH_\chi(f)$ and $\Psi_{f(x)}^\eta$ $\e$--overlaps $\Psi_y^{\eta'}$, then
$
f_{xy}:=\Psi_y^{-1}\circ f\circ \Psi_x
$
is a  well defined injective map from  $R_{10Q_\e(x)}(\un{0})$ to $\R^2$, and $f_{xy}$  can be put in the form
\begin{equation}\label{offenbach}
f_{xy}(u,v)=(Au+h_1(u,v), Bv+h_2(u,v)),
\end{equation}
where $C_f^{-1}<|A|<e^{-\chi}$, $e^\chi<|B|<C_f$ (cf. Theorem \ref{TheoremOPR}),  $|h_i(\un{0})|<\e\eta$, $\|\nabla h_i(\un{0})\|<\e\eta^{\b/3}$, and $\|\nabla h_i(\un{u})-\nabla h_i(\un{v})\|\leq \e\|\un{u}-\un{v}\|^{\b/3}$ on $R_{10 Q_\e(x)}(\un{0})$.

\medskip
\noindent
A similar  statement holds for $f_{xy}^{-1}$, assuming that
 $\Psi_{f^{-1}(y)}^{\eta'}$ $\e$--overlaps $\Psi_x^{\eta}$.
\end{prop}
\begin{proof}
We write $f_{xy}=(\Psi_y^{-1}\circ\Psi_{f(x)})\circ f_x$ where $f_x=\Psi_{f(x)}^{-1}\circ f\circ\Psi_x$, and treat $f_{xy}$ as a perturbation of $f_x$.

By Theorem \ref{Theorem_OP_charts},  if $\e$ is small enough, then $f_x$ has the following properties:
\begin{enumerate}
\item It is is well--defined, differentiable, and injective on $R_{10Q_\e(x)}(\un{0})$.
\item $f_x(\un{0})=\un{0}$ and $(df_x)_{\un{0}}=\left(\begin{array}{cc}
A & 0 \\
0 & B
\end{array}\right)$\! where $C_f^{-1}<|A|<e^{-\chi}$, $e^\chi<|B|<C_f$.
\item For all $\un{u},\un{v}\in R_{10 Q_\e(x)}(\un{0})$, $\|(df_x)_{\un{u}}-(df_x)_{\un{v}}\|\leq \e\|\un{u}-\un{v}\|^{\b/2}$ (because the $C^{1+\frac{\b}{2}}$ distance between $f_x$ and $(df_x)_{\un{0}}$ on $R_{10 Q_\e(x)}(\un{0})$ is less than $\e$).
\item For every $0<\eta<10Q_\e(x)$ and $\un{u}\in R_\eta(\un{0})$, $\|(df_x)_{\un{u}}\|<3C_f$, provided $\e$ is small enough (because $\|(df_x)_{\un{u}}\|\leq \|(df_x)_{\un{0}}\|+\e\eta^{\b/2}<2C_f+\e$).
\end{enumerate}

(2) and (4) imply that  $f_x[R_{10 Q_\e(x)}(\un{0})]\subset B_{30 Q_\e(x)C_f}(\un{0})$. Since $Q_\e(x)<\e^{3/\b}$,
$
f_x[R_{10 Q_\e(x)}(\un{0})]\subset B_{30C_f\e^{3/\b}}(\un{0}).
$
If $\e$ so small that $30C_f\e^{3/\b}<e^{-\e}r(M)$, then
$
f_x[R_{10 Q_\e(x)}(\un{0})]\subset R_{e^{-\e}r(M)}(\un{0}).
$
 $R_{e^{-\e}r(M)}(\un{0})$ is in the domain of $\Psi_{y}^{-1}\circ\Psi_{f(x)}$ (Proposition \ref{Prop_Overlap_Meaning}, part 2), therefore  $f_{xy}$ is well defined, differentiable, and injective on $R_{10 Q_\e(x)}(\un{0})$.

\medskip
Equation  (\ref{offenbach}) can be used to define the functions $h_i(u,v)$. We check that the resulting functions satisfy the properties proclaimed by the proposition.

We have $(h_1(\un{0}),h_2(\un{0}))=f_{xy}(\un{0})=\Psi_y^{-1}(f(x))=(\Psi_y^{-1}\circ\Psi_{f(x)})(\un{0})$, therefore $\|(h_1(\un{0}),h_2(\un{0}))\|\leq \dist_{C^0}(\Psi_y^{-1}\circ\Psi_{f(x)}, \id)<\e\eta^2(\eta')^2<\e\eta$.

\medskip
We differentiate the identity $f_{xy}=(\Psi_y^{-1}\circ\Psi_{f(x)})\circ f_x$ at an arbitrary $\un{u}\in R_\eta(\un{0})$. The result, after some rearrangement is
\begin{equation}\label{julietta}
(df_{xy})_{\un{u}}=[d(\Psi_y^{-1}\circ\Psi_{f(x)})_{f_x(\un{u})}-\id](df_x)_{\un{u}}+[(df_x)_{\un{u}}-(df_x)_{\un{0}}]+(df_x)_{\un{0}}.
\end{equation}
The norm of the first summand is less than $3C_f\dist_{C^1}(\Psi_y^{-1}\circ\Psi_{f(x)},\id)$, which  by Proposition \ref{Prop_Overlap_Meaning}   is less than $3C_f\e\eta^2(\eta')^2<3C_f\e\eta^2$. The norm of the second summand is less than $\e\|\un{u}\|^{\b/2}<2\e\eta^{\b/2}$. The third term is $\left(\begin{array}{cc}
A & 0 \\ 0 & B
\end{array}\right)$. Thus
\begin{align*}
\left\|\frac{\partial(h_1,h_2)}{\partial(u,v)}\right\|&=\left\|(df_{xy})_{\un{u}}-\left(\begin{array}{cc}
A & 0 \\ 0 & B
\end{array}\right)\right\|<\e[3C_f+2]\eta^{\b/2}\\
&<\e\eta^{\b/3}\cdot [3C_f+2]\eta^{\b/6}<\e\eta^{\b/3}\cdot [3C_f+2]\sqrt{\e}\ \ \textrm{by (\ref{settle})}.
\end{align*}
If  $\e$ is so small that $[3C_f+2]\sqrt{\e}<1$, then  $\|\nabla h_i\|<\e\eta^{\b/3}$ on $R_\eta(\un{0})$. In particular, $\|\nabla h_i(\un{0})\|<\e\eta^{\b/3}$.

Equation (\ref{julietta}) also shows that for every $\un{u},\un{v}\in R_{10 Q_\e(x)}(\un{0})$,
\begin{align*}
\|(df_{xy})_{\un{u}}-(df_{xy})_{\un{v}}\|&\leq \|d(\Psi_y^{-1}\circ\Psi_{f(x)})_{f_x(\un{u})}-d(\Psi_y^{-1}\circ\Psi_{f(x)})_{f_x(\un{v})}\|\cdot\|(df_x)_{\un{u}}\|\\
&\hspace{1cm}+\|(df_x)_{\un{u}}-(df_x)_{\un{v}}\|\cdot\left(\|d(\Psi_y^{-1}\circ\Psi_{f(x)})_{f_x(\un{v})}\|+1\right).
\end{align*}
By Proposition \ref{Prop_Overlap_Meaning}, $\dist_{C^{1+\b/2}}(\Psi_y^{-1}\circ\Psi_{f(x)},\id)<\e\eta^2(\eta')^2$, therefore
\begin{align*}
\|(df_{xy})_{\un{u}}-(df_{xy})_{\un{v}}\|&\leq  \e\eta^2(\eta')^2\cdot\|f_x(\un{u})-f_x(\un{v})\|^{\frac{\b}{2}}\cdot 3C_f+\e\|\un{u}-\un{v}\|^{\frac{\b}{2}}\left(\e\eta^2(\eta')^2+2
\right)\\
&\leq\e\eta^2\cdot\sup_{\un{w}\in R_{10 Q_\e(x)}(\un{0})}\|(df_x)_{\un{w}}\|^{\frac{\b}{2}}\cdot \|\un{u}-\un{v}\|^{\frac{\b}{2}}\cdot 3C_f+3\e\|\un{u}-\un{v}\|^{\frac{\b}{2}}\\
&\leq \e((3C_f)^{1+{\frac{\b}{2}}}\eta^2+3)\|\un{u}-\un{v}\|^{\frac{\b}{2}}
\leq\e((3C_f)^{1+{\frac{\b}{2}}}\e^{6/\b}+3)\|\un{u}-\un{v}\|^{\frac{\b}{2}}\\
&\leq 4\e\|\un{u}-\un{v}\|^{\frac{\b}{2}},\textrm{ provided $\e$ is small enough}\\
&\leq 3\e(30 Q_\e(x))^{\b/6}\|\un{u}-\un{v}\|^{\b/3}<6\e^{3/2}\|\un{u}-\un{v}\|^{\b/3}\ \ (\because Q_\e<\e^{3/\b})\\
&\leq \frac{1}{3}\e\|\un{u}-\un{v}\|^{\b/3},\textrm{ provided $\e$ is small enough.}
\end{align*}
It follows that $\|\frac{\partial(h_1,h_2)}{\partial(u,v)}(\un{u})-\frac{\partial(h_1,h_2)}{\partial(u,v)}(\un{v})\|
<\frac{1}{3}\e\|\un{u}-\un{v}\|^{\b/3}$ for all $\un{u},\un{v}\in R_{10 Q_\e(x)}(\un{0})$, whence  $\|\nabla h_i(\un{u})-\nabla h_i(\un{v})\|\leq \frac{1}{3}\e\|\un{u}-\un{v}\|^{\b/3}$ $(i=1,2)$ for all
$\un{u},\un{v}\in R_{10 Q_\e(x)}(\un{0})$.
\end{proof}

\subsection{Coarse graining} \label{SectionA}
We replace $\mathfs C:=\{\Psi_x^\eta:x\in\NUH^\ast_\chi(f), 0<\eta\leq Q_\e(x)\}$
by a ``sufficient" countable subset $\mathfs A$. We remind the reader that  $\NUH_\chi^\ast$ is defined in Lemma \ref{Lemma_C_tempered}, and that $I_\e=\{e^{-\frac{1}{3}k\e}:k\in\N\}$.

\begin{prop}\label{Prop_A}
The following holds for all $\e$ small.
There exists a countable collection $\mathfs A$ of Pesin charts  with the following properties:
\begin{enumerate}
\item {\em Sufficiency:} For every $x\in\NUH_\chi^\ast(f)$ and for every sequence of positive numbers $0<\eta_n\leq e^{-\e/3}Q_\e(f^n(x))$ in $I_\e$ s.t. $e^{-\e}\leq \eta_n/\eta_{n+1}\leq e^\e$, there exists a sequence $\{\Psi_{x_n}^{\eta_n}\}_{n\in\Z}$ of elements of $\mathfs A$   s.t. for every $n$,
    \begin{enumerate}
    \item $\Psi_{x_n}^{\eta_n}$ $\e$--overlaps $\Psi_{f^n(x)}^{\eta_n}$ and $e^{-\e/3}\leq Q_\e(f^n(x))/Q_\e(x_n)\leq e^{\e/3}$;
    \item $\Psi_{f(x_n)}^{\eta_{n+1}}$ $\e$--overlaps $\Psi_{x_{n+1}}^{\eta_{n+1}}$;
    \item $\Psi_{f^{-1}(x_n)}^{\eta_{n-1}}$ $\e$--overlaps $\Psi_{x_{n-1}}^{\eta_{n-1}}$;
        \item $\Psi_{x_n}^{\eta_n'}\in\mathfs A$ for all $\eta_n'\in I_\e$ s.t.  $\eta_n\leq\eta_n'\leq \min\{Q_\e(x_n), e^\e\eta_n\}$.
    \end{enumerate}
\item {\em Discreteness:} $\{\Psi_x^\eta\in \mathfs A:\eta>t\}$ is finite for every $t>0$.
\end{enumerate}
\end{prop}

\begin{proof}
The proposition would have been easy had $C_\chi(x)$ been a continuous function of $x$. In general it is not, and as a result there is no clear connection between conditions (a), (b), and (c). We must treat the three conditions  separately, and simultaneously.

The following construction will help us to do this. Let
$$
X:=\bigcup_{D_0, D_1, D_{-1}\in\mathfs D} D_0\x D_1\x D_{-1}\x \R^3\x\mathrm{GL}(2,\R)^3.
$$
Here $\mathfs D$ is the finite open cover of $M$ which we constructed in   \S\ref{SectionOC}. $X$ is a subset of $M^3\x\R^3\x\mathrm{GL}(2,\R)^3$. We equip it with the relative product topology.

Let $Y\subset X$ denote the collection of all $(\un{x},\un{Q},\un{C})\in X$ where
\begin{itemize}
\item $\un{x}=(x,f(x),f^{-1}(x))$, $x\in\NUH_\chi^\ast(f)$;
\item $\un{Q}=(Q_\e(x), Q_\e(f(x)), Q_\e(f^{-1}(x)))$ (cf. (\ref{Qdef}));
\item $\un{C}=(\Theta_{D_0}\circ C_\chi(x), \Theta_{D_1}\circ C_\chi(f(x)), \Theta_{D_{-1}}\circ C_\chi(f^{-1}(x)))$, where
$D_0, D_1, D_{-1}\in\mathfs D$ satisfy
$(x,f(x),f^{-1}(x))\in D_0\x D_1\x D_{-1}$.
\end{itemize}

Let $Y_k:=\{(\un{x},\un{Q},\un{C})\in Y:x\in \NUH_\chi^\ast(f), e^{-(k+1)}\leq Q_\e(x)\leq e^{-(k-1)}\}$ $(k\in\N)$. $Y_k$ is a pre-compact subset of $X$. To see this, pick some $(\un{x},\un{Q},\un{C})\in Y_k$. The vector $\un{x}$ belongs to the compact set $M^3$. $\un{Q}$ belongs to a compact subset of $\R^3$ because by Lemma \ref{Lemma_Q_tempered} for each $i=-1,0,1$,
$$
F^{-1}e^{-(k+1)}\leq Q_\e(f^i(x))<Fe^{-(k-1)}.
$$
$\un{C}$ belongs to a compact subset of $\GL(2,\R)$, because (a) $\Theta_{D_i}$ are isometries; (b) $\|C_\chi(f^i(x))\|<1$ (Lemma \ref{Lemma_C_contracts});  and  (c)
$
\bigr\|C_\chi(f^{i}(x))^{-1}\bigl\|\leq \left(\e^{3/\b}F e^{k+1}\right)^{\b/12}
$
by  (\ref{Qdef}).\footnote{Here we use the obvious observation that  $\{A\in\GL(2,\R):\|A\|,\|A^{-1}\|\leq C\}$ is a compact subset of $\GL(2,\R)$ for every $C>0$.}
It follows that $Y_k$ is a subset of a compact subset of $M^3\x\R^3\x\mathrm{GL}(2,\R)^3$.

\medskip
Since $Y_k$ is pre-compact, it contains a finite set $Y_{k,m}$ s.t. for  every $(\un{x},\un{Q},\un{C})\in Y_k$   there exists some $(\un{y},\un{Q}',\un{C}')\in Y_{k,m}$ such that for every $|i|\leq 1$,
\begin{enumerate}
\item $d(f^i(x),f^i(y))<\frac{1}{2}\e(\mathfs D)$ where $\e(\mathfs D)$ is a Lebesgue number of $\mathfs D$.
\item $d(f^i(x), f^i(y))+\|\Theta_D\circ C_\chi(f^i(x))-\Theta_D\circ C_\chi(f^i(y))\|<e^{-8(m+2)}$ for every $D\in\mathfs D$ which contains  $f^i(x)$ and $f^i(y)$.
\item $e^{-\e/3}<Q_\e(f^i(x))/Q_\e(f^i(y))<e^{\e/3}$.
\end{enumerate}

\medskip
Define $\mathfs A$ to be the collection of all Pesin charts $\Psi_x^\eta$ such that for some $k,m\in\N$, $x$ is the first coordinate of some element $(\un{x},\un{Q},\un{C})\in Y_{k,m}$, and
$$
0<\eta\leq Q_\e(x),
 e^{-(m+2)}\leq \eta<e^{-(m-2)}, \textrm{ and }\eta\in I_\e=\{e^{-\ell\e/3}:\ell=0,1,2,\ldots\}.
 $$

\medskip
\noindent
{\em Part 1.\/} Discreteness.

\medskip
\noindent
{\em Proof.\/}
Suppose $\Psi_x^\eta\in\mathfs A$. Choose $k,m\in\N$ s.t.  $x$ is the first coordinate of some $(\un{x},\un{Q},\un{C})\in Y_{k,m}$, $0<\eta\leq Q_\e(x)$, and  $\eta\in [e^{-m-2}, e^{-m+2}]$.
Since $Y_{k,m}\subset Y_k$, $Q_\e(x)\leq e^{-k+1}$, so $k\leq |\log Q_\e(x)|+1$. It follows that
$k,m\leq |\log\eta|+2$, and so
$$
|\{\Psi_x^\eta\in\mathfs A: \eta>t\}|\leq \sum_{k,m<|\log t|+2}|Y_{k,m}|\x|\{\eta\in I_\e: \eta>t\}|.
$$
The last quantity is finite, because $Y_{k,m}$ are finite.

\medskip
\noindent
{\em Part 2.\/} Sufficiency.

\medskip
\noindent
{\em Proof.\/}
Suppose $x\in \NUH_\chi^\ast(f)$, and  $\eta_n\in I_\e$ satisfy $0<\eta_n\leq e^{-\e/3}Q_\e(f^n(x))$ and $e^{-\e}\leq \eta_n/\eta_{n+1}\leq e^\e$ for all $n\in\Z$.

Choose $ m_n,k_n\in\N$ s.t. $\eta_n\in [e^{-m_n-1}, e^{-m_n+1}]$ and $Q_\e(f^n(x))\in [e^{-k_n-1}, e^{-k_n+1}]$. Find some element of $Y_{k_n}$ whose first coordinate is $f^n(x)$,  and
approximate it  by some element of $Y_{k_n, m_n}$ with first coordinate $x_n$ so that for $i=-1,0,1$,
\begin{enumerate}
\item[(${\mathrm A}_n$)] $d(f^i(f^n(x)),f^i(x_n))<\frac{1}{2}\e(\mathfs D)$;
\item[(${\mathrm B}_n$)] $d(f^i(f^n(x)), f^i(x_n))+\|\Theta_{D}\circ C_\chi(f^i(f^n(x)))-\Theta_{D}\circ C_\chi(f^i(x_n))\|<e^{-8(m_n+2)}$
for every $D\in\mathfs D$ which contains $f^i(f^n(x)), f^i(x_n)$;
\item[(${\mathrm C}_n$)] $e^{-\e/3}<Q_\e(f^i(f^n(x)))/Q_\e(f^i(x_n))<e^{\e/3}$.
\end{enumerate}

\medskip
\noindent
{\em Claim 1.\/} $\Psi_{x_n}^{\eta_n}\in\mathfs A$ and $\Psi_{x_n}^{\eta_n'}\in\mathfs A$ for all $\eta_n'\in I_\e$ s.t. $\eta_n\leq \eta_n'\leq \min\{e^\e \eta_n, Q_\e(x_n)\}$.

\medskip
\noindent
{\em Proof.\/} By construction $x_n$ is the first coordinate of an element of $Y_{k_n,m_n}$, and   $\eta_n\in [e^{-m_n-1}, e^{m_n+1}]$. Since $\eta_n\leq \eta_n'\leq e^{\e}\eta_n$,  $ \eta_n'\in [e^{-m_n-2}, e^{m_n+2}]$. It remains to check that $\eta_n, \eta_n'\leq Q_\e(x_n)$. In case of $\eta_n'$ there is nothing to check. In case of $\eta_n$, (${\mathrm C}_n$) with  $i=0$ says that
$
Q_\e(x_n)> e^{-\e/3}Q_\e(f^n(x))\geq {\eta_n}.
$

\medskip
\noindent
{\em Claim 2.\/} $\Psi_{x_n}^{\eta_n}$ and $\Psi_{f^n(x)}^{\eta_n}$ $\e$--overlap.

\medskip
\noindent
{\em Proof.\/} $(\mathrm A_n)$ with $i=0$ says that  $d(f^n(x), x_n)$  is smaller than the Lebesgue number of $\mathfs D$, so there exists  $D\in\mathfs D$ s.t.  $f^n(x),x_n\in D$. (${\mathrm B}_n$) with $i=0$ says that
$$
d(f^n(x), x_n)+\|\Theta_D\circ C_\chi(f^n(x))-\Theta_D\circ C_\chi(x_n)\|<e^{-8(m_n+2)}.
$$
Since  $\eta_n\in [e^{-(m_n+1)}, e^{-(m_n-1)}]$,
$e^{-8(m_n+2)}<\eta_n^4 \eta_{n+1}^4$. Since  $e^{-\e}\leq \eta_{n+1}/\eta_n\leq e^\e$,
 $\Psi_{x_n}^{\eta_n}, \Psi_{f^n(x)}^{\eta_n}$ $\e$--overlap.

\medskip
\noindent
{\em Claim 3.\/} $\Psi_{f^i(x_n)}^{\eta_n}$ $\e$--overlaps $\Psi_{x_{n+i}}^{\eta_{n+i}}$ for $i=\pm 1$.

\medskip
\noindent
{\em Proof.\/} We do the case $i=1$ and leave the case $i=-1$ to the reader.

Setting $i=1$ in  (${\mathrm A}_n$), we see that
 $d(f(x_n), f(f^n(x)))<\frac{1}{2}\e(\mathfs D)$. Setting $i=0$ in (${\mathrm A}_{n+1}$), we see that  $d(f^{n+1}(x), x_{n+1})<\frac{1}{2}\e(\mathfs D)$. It follows that  there exists some $D\in\mathfs D$ s.t.
$
f(x_n), x_{n+1}, f^{n+1}(x)\in D.
$

By   (${\mathrm B}_n$) with $i=1$ and (${\mathrm B}_{n+1}$) with $i=0$,
$$
\hspace{-3cm}d(f(x_n), x_{n+1})+\|\Theta_D\circ C_\chi(f(x_n))-\Theta_D\circ C_\chi(x_{n+1})\|\leq
$$
\begin{align*}
\hspace{1cm}&\leq \bigl(d(f(x_n), f(f^n(x)))+\|\Theta_D\circ C_\chi(f(x_n))-\Theta_D\circ C_\chi(f(f^n(x)))\|\bigr)+\\
&\hspace{1cm}+\bigl(d(f^{n+1}(x), x_{n+1})+\|\Theta_D\circ C_\chi(f^{n+1}(x))-\Theta_D\circ C_\chi(x_{n+1})\|\bigr)\\
&\leq e^{-8(m_n+2)}+e^{-8(m_{n+1}+2)}\\
&<e^{-8}(\eta_n^8+\eta_{n+1}^8)<2e^{-8}(1+e^{8\e})\eta_{n+1}^4\eta_{n+1}^4<\eta_{n+1}^4\eta_{n+1}^4.
\end{align*}
It follows that  $\Psi_{f(x_n)}^{\eta_{n+1}}$ $\e$--overlaps $\Psi_{x_{n+1}}^{\eta_{n+1}}$.
\end{proof}

\section{$\e$--chains and an infinite-to-one Markov extension of $f$}\label{SectionGraph}
\subsection{Double charts and $\e$--chains} Recall that $\Psi_x^\eta$  $(0<\eta\leq Q_\e(x))$ stands for the Pesin chart $\Psi_x:R_\eta(\un{0})\to M$. An {\em $\e$--double Pesin chart} (or just {\em ``double chart"}) is a pair $\Psi_x^{p^u,p^s}:=(\Psi_x^{p^s},\Psi_x^{p^u})$, where $0<p^u,p^s\leq Q_\e(x)$.

\begin{defi}
$\Psi_x^{p^u,p^s}\to \Psi_y^{q^u,q^s}$ means
\begin{itemize}
\item $\Psi_y^{q^u\wedge q^s}$ and $\Psi_{f(x)}^{q^u\wedge q^s}$ $\e$--overlap (recall that $a\wedge b:=\min\{a,b\}$);
\item  $\Psi_x^{p^u\wedge p^s}$ and $\Psi_{f^{-1}(y)}^{p^u\wedge p^s}$ $\e$--overlap;
\item
 $q^u=\min\{e^\e p^u, Q_\e(y)\}$ and $p^s=\min\{e^{\e}q^s, Q_\e(x)\}$.
\end{itemize}
\end{defi}
\begin{defi}
 $\{\Psi_{x_i}^{p^u_i,p^s_i}\}_{i\in\Z}$ (resp. $\{\Psi_{x_i}^{p^u_i,p^s_i}\}_{i\geq 0}$, $\{\Psi_{x_i}^{p^u_i,p^s_i}\}_{i\leq 0}$) is called an {\em $\e$--chain} (resp. {\em positive $\e$--chain, negative $\e$--chain}), if $\Psi_{x_i}^{p^u_i, p^s_i}\to \Psi_{x_{i+1}}^{p^u_{i+1},p^s_{i+1}}$ for all $i$. We abuse terminology and drop the $\e$ in ``$\e$--chains".
\end{defi}

\medskip
Let $\mathfs A$ denote the countable set of Pesin charts which we have constructed in  \S\ref{SectionA}, and recall that  $I_\e=\{e^{-k\e/3}:k\in\N\}$.
\begin{defi} $\mathfs G$ is the directed graph with vertices $\mathfs V$ and edges $\mathfs E$ where
\begin{itemize}
\item $\mathfs V:=\{\Psi_x^{p^u, p^s}:\Psi_x^{p^u\wedge p^s}\in\mathfs A, p^u, p^s\in I_\e, p^u, p^s\leq Q_\e(x)\}$;
\item $\mathfs E:=\{(\Psi_x^{p^u, p^s},\Psi_y^{q^u, q^s})\in\mathfs V\x\mathfs V: \Psi_x^{p^u, p^s}\to\Psi_y^{q^u, q^s}\}$.
\end{itemize}
\end{defi}
\noindent
This is a countable directed graph. Every vertex has finite degree, because of the following lemma, and Proposition \ref{Prop_A}(2):
\begin{lem}\label{Lemma_Subordinated_Tempered}
If $\Psi_x^{p^u,p^s}\to\Psi_y^{q^u,q^s}$, then $e^{-\e}\leq (q^u\wedge q^s)/(p^u\wedge p^s)\leq e^\e$. Therefore for every $\Psi_x^{p^u,p^s}\in\mathfs V$ there are only finitely many $\Psi_y^{q^u,q^s}\in\mathfs V$ s.t. $\Psi_x^{p^u,p^s}\to\Psi_y^{q^u,q^s}$ or $\Psi_y^{q^u,q^s}\to\Psi_x^{p^u,p^s}$.
\end{lem}
\begin{proof}
The proof is a manipulation of the following relations:
\begin{equation}
\begin{aligned}
q^u=\min\{e^\e p^u, Q_\e(y)\} & , & p^u\leq Q_\e(x);\\
p^s=\min\{e^\e q^s, Q_\e(x)\} & , & q^s\leq Q_\e(y).\\
\end{aligned}
\end{equation}
Let $p:=p^u\wedge p^s$ and $q:=q^u\wedge q^s$. We show that $e^{-\e}\leq p/q\leq e^\e$ by considering each of the following cases separately:
\begin{enumerate}
\item $p=p^u$, $q=q^u$,
\item $p=p^s$, $q=q^s$,
\item $p=p^u$, $q=q^s$,
\item $p=p^s$, $q=q^u$.
\end{enumerate}

\noindent
{\em Case 1.\/}
If $e^\e p^u\leq Q_\e(y)$, then
$q^u=\min\{e^\e p^u, Q_\e(y)\}=e^\e p^u$, and $\frac{q}{p}=\frac{q^u}{p^u}=e^\e$.
If $e^\e p^u>Q_\e(y)$, then $p\leq p^s\leq e^\e q^s\leq e^\e Q_\e(y)=e^\e\min\{e^\e p^u, Q_\e(y)\}=e^\e q^u=e^\e q$, so $\frac{p}{q}\leq e^\e$. Also, $q=q^u=\min\{Q_\e(y), e^\e p^u\}\leq e^\e p^u=e^\e p$, so $\frac{q}{p}\leq e^\e$.

\medskip
\noindent
{\em Case 2.\/} This is the same as case 1.

\medskip
\noindent
{\em Case 3.\/} In this case $p^u\leq p^s$, so  $p=p^u\leq p^s\leq e^\e q^s=e^\e q$, whence $p/q\leq e^\e$. Also, $q^s\leq q^u$, so $q=q^s\leq q^u\leq e^\e p^u=e^\e p$, whence $q/p\leq e^\e$.

\medskip
\noindent
{\em Case 4.\/} In this case $p^s\leq p^u$ and $q^u\leq q^s$. Since $p^s=\min\{e^\e q^s, Q_\e(x)\}$, either $p^s=e^\e q^s$, or $p^s=Q_\e(x)$.

Suppose $p^s=e^\e q^s$, then $q^u\leq q^s=e^{-\e}p^s\leq e^{-\e}p^u<e^\e p^u$. The inequality $q^u<e^\e p^u$ and the identity $q^u=\min\{e^\e p^u, Q_\e(y)\}$ force $q^u=Q_\e(y)$. But $q^u\leq q^s\leq Q_\e(y)$, so $q^u=q^s=Q_\e(y)$. It follows that
$p=p^s=e^\e q^s=e^\e q^u=e^\e q$, and we are done.

Next suppose that $p^s=Q_\e(x)$. Since $p^s\leq p^u\leq Q_\e(x)$, we must have $$p=p^s=p^u=Q_\e(x).$$
At the same time, $q^u=\min\{e^\e p^u, Q_\e(y)\}\leq e^\e p^u= e^\e p$. If there is an equality, then we are done. Otherwise $q^u=Q_\e(y)$, and since $q^u\leq q^s\leq Q_\e(y)$,
$$
q=q^s=q^u=Q_\e(y).
$$
Since $\min\{e^\e p^u, Q_\e(y)\}=q^u=Q_\e(y)$, $e^\e p^u\geq Q_\e(y)$. Thus $e^\e Q_\e(x)\geq Q_\e(y)$. Similarly,  $\min\{e^\e q^s, Q_\e(x)\}=p^s=Q_\e(x)$ implies that  $e^\e q^s\geq Q_\e(x)$, whence $e^\e Q_\e(y)\geq Q_\e(x)$. It follows that $p/q=Q_\e(x)/Q_\e(y)\in [e^{-\e},e^\e]$.
\end{proof}

We claim that the collection of infinite admissible paths on $\mathfs G$ is as rich as the set of orbits of $f$ in  $\NUH_\chi^\#(f)$.
Recall that $\NUH_\chi^\#(f)$ has full measure w.r.t. every $f$--ergodic invariant probability measure with  entropy greater than $\chi$.

\begin{prop}\label{Prop_Chains_Exist}
For every $x\in \NUH_\chi^\#(f)$, there is a chain $\{\Psi_{x_k}^{p^u_k, p^s_k}\}_{k\in\Z}\subset\Sigma(\mathfs G)$ s.t. $\Psi_{x_k}^{p^u_k\wedge p^s_k}$ $\e$--overlaps $\Psi_{f^k(x)}^{p^u_k\wedge p^s_k}$ for all $k\in\Z$.
\end{prop}
\noindent
The proof relies on two simple properties of  chains, which we now describe.

Some terminology: Let $(Q_k)_{k\in\Z}$ be a sequence  in $I_\e=\{e^{-\ell\e/3}:\ell\in\N\}$. A sequence of pairs $\{(p_k^u, p_k^s)\}_{k\in\Z}$ is called {\em $\e$--subordinated} to $(Q_k)_{k\in\Z}$ if for every $k\in\Z$,  $0<p_k^u, p_k^s\leq Q_k$, $p_k^u, p_k^s\in I_\e$, and
$$
p_{k+1}^u=\min\{e^\e p_k^u, Q_{k+1}\}\textrm{ and }p_{k-1}^s=\min\{e^\e p_k^s, Q_{k-1}\}.
$$
For example, if $\{\Psi_{x_k}^{p^u_k, p^s_k}\}_{k\in\Z}$ is a chain, then $\{(p^u_k,p^s_k)\}_{k\in\Z}$ is $\e$--subordinated to $\{Q_\e(x_k)\}_{k\in\Z}$.

\begin{lem}\label{Lemma_Subordinated_Exist}
Let $(Q_k)_{k\in\Z}$ be a sequence in $I_\e$, and suppose  $q_k\in I_\e$ satisfy $0<q_k\leq Q_k$ and $e^{-\e}\leq q_k/q_{k+1}\leq e^\e$ for all $k\in\Z$. There exists a sequence $\{(p^u_k,p^s_k)\}_{k\in\Z}$ which is  $\e$--subordinated to $\{Q_k\}_{k\in\Z}$, and so that  $p^u_k\wedge p^s_k\geq q_k$ for all $k$.
\end{lem}
\begin{proof} The following short proof was shown to me by F. Ledrappier. By the assumptions on $q_k$,  $Q_\e(x_{k-n}), Q_\e(x_{k+n})\geq e^{-\e n}q_k$ for all $n\geq 0$, therefore the following definitions make sense:
\begin{align*}
p^u_k&:=\max\{t\in I_\e: e^{-\e n}t\leq Q_\e(x_{k-n})\textrm{ for all }n\geq 0\};\\
p^s_k&:=\max\{t\in I_\e: e^{-\e n}t\leq Q_\e(x_{k+n})\textrm{ for all }n\geq 0\}.
\end{align*}
The sequence $\{(p^u_k,p^s_k)\}_{k\in\Z}$ is $\e$--subordinated to $\{Q_\e(x_k)\}_{k\in\Z}$.
\end{proof}

\begin{lem}\label{Lemma_Subordinated_Sharp}
Suppose $\{(p_n^u, p_n^s)\}_{n\in\Z}$ is $\e$--subordinated to a sequence $\{Q_n\}_{n\in\Z}\subset I_\e$. If $\limsup\limits_{n\to\pm\infty}(p^u_n\wedge p^s_n)>0$, then $p_n^{u}$ (resp. $p_n^s$) is equal to $Q_n$ for infinitely many  $n>0$, and for infinitely many $n<0$.
\end{lem}
\begin{proof}
We prove the statement for $p^u_n$, and leave the statement for $p^s_n$ to the reader.

$M:=\sup Q_n$ is finite, because $Q_n\in I_\e$ for all $n$.
Let $p_n:=p^u_n\wedge p^s_n$, and define  $m:=\frac{1}{2}\min\{\limsup\limits_{n\to \infty}p_{-n}, \limsup\limits_{n\to \infty}p_{n}\}$ and
$
N:=\lceil \e^{-1}\log(M/m)\rceil$.

There exists infinitely many positive (resp. negative) $n$ s.t. $p_n>m$. We claim that for every such $n$, there must exists some $k\in [n, n+N]$ s.t. $p^u_k=Q_k$. Otherwise, by $\e$--subordination,
$$
p^u_{n+N}=\min\{Q_{n+N}, e^{\e}p^u_{n+N-1}\}=e^\e p^u_{n+N-1}=\cdots =e^{N\e}p^u_n\geq e^{N\e}p_n>e^{N\e}m>M,
$$
which is false.
\end{proof}

We can now prove Proposition \ref{Prop_Chains_Exist}:
Suppose $x\in\NUH_\chi^\#(f)$, and recall the definition of $q_\e(\cdot)$ from Lemma \ref{Lemma_Q_tempered}. Choose  $q_n\in I_\e\cap[e^{-\e/3}q_\e(f^n(x)), e^{\e/3}q_\e(f^n(x))]$. The sequence $\{q_n\}_{n\in\Z}$ satisfies the assumptions of  Lemma \ref{Lemma_Subordinated_Exist}, therefore there exists a sequence $\{(q^u_n, q^s_n)\}_{n\in\Z}$ that is $\e$--subordinated to $\{e^{-\e/3}Q_\e(f^n(x))\}_{n\in\Z}$ and that satisfies $q^u_k\wedge q^s_k\geq q_k$.

Let $\eta_n:=q^u_n\wedge q^s_n$. By Lemma \ref{Lemma_Subordinated_Tempered}, $e^{-\e}\leq \eta_{n+1}/\eta_n\leq e^\e$, so we are free to use Proposition \ref{Prop_A} to construct   an  infinite sequence $\Psi_{x_n}^{\eta_n}\in\mathfs A$ such that
\begin{enumerate}
\item[(a)] $\Psi_{x_n}^{\eta_n}$ $\e$--overlaps $\Psi_{f^n(x)}^{\eta_n}$ and $e^{-\e/3}\leq Q_\e(f^n(x))/Q_\e(x_n)\leq e^{\e/3}$;
    \item[(b)] $\Psi_{f(x_n)}^{\eta_{n+1}}$ $\e$--overlaps $\Psi_{x_{n+1}}^{\eta_{n+1}}$;
    \item[(c)] $\Psi_{f^{-1}(x_n)}^{\eta_{n-1}}$ $\e$--overlaps $\Psi_{x_{n-1}}^{\eta_{n-1}}$;
        \item[(d)] $\Psi_{x_n}^{\eta_n'}\in\mathfs A$ for all $\eta_n'\in I_\e$ s.t.  $\eta_n\leq\eta_n'\leq \min\{Q_\e(x_n), e^\e\eta_n\}$.
\end{enumerate}

Construct a sequence $\{(p^u_n, p^s_n)\}_{n\in\Z}$ which is $\e$--subordinated to $\{Q_\e(x_n)\}_{n\in\Z}$ and which satisfies $p^u_n\wedge p^s_n\geq \eta_n$.

\medskip
\noindent
{\em Claim 1.\/}
$\Psi_{x_n}^{p^u_n, p^s_n}\in\mathfs V$ for all $n$.

\medskip
\noindent
{\em Proof.\/}
It is sufficient to show that
$
1\leq \frac{p^u_n\wedge p^s_n}{q^u_n\wedge q^s_n}\leq e^\e\ (n\in\Z)
$,
because  property (d)  with $\eta_n':=p^u_n\wedge p^s_n$ says that  in this case $\Psi_{x_n}^{p^u_n\wedge p^s_n}\in\mathfs A$, whence $\Psi_{x_n}^{p^u_n, p^s_n}\in\mathfs V$.

We start by showing that  there are infinitely many  $n<0$ such that $p^u_n\leq e^\e q^u_n$. Since $x\in\NUH_\chi^\#(f)$, $\limsup\limits_{n\to\infty}q_n,\limsup\limits_{n\to-\infty}q_n>0$.  Therefore  by Lemma \ref{Lemma_Subordinated_Sharp}, there are infinitely many $n<0$ for which
$
q^u_n=e^{-\e/3}Q_\e(f^n(x))$. Property (a) guarantees that for such $n$,  $q^u_n>e^{-\e}Q_\e(x_n)\geq e^{-\e}p^u_n$, whence $p^u_n<e^\e q^u_n$.

If $p^u_n\leq e^\e q^u_n$, then $p^u_{n+1}\leq e^\e q^u_{n+1}$, because
\begin{align*}
p^u_{n+1}&=\min\{e^\e p^u_n, Q_\e(x_{n+1})\}=e^{\e}\min\{p^u_n, e^{-\e}Q_\e(x_{n+1})\}\\
&\leq e^{\e}\min\{e^\e q^u_n, e^{-\e/3}Q_\e(f^{n+1}(x))\}\equiv e^\e q^u_{n+1}.
\end{align*}
It follows that $p^u_n\leq e^\e q^u_n$ for all $n\in\Z$.

Working with positive $n$, one can show in the same manner that $p^s_n\leq e^\e q^s_n$ for all $n\in\Z$. Combining the two results we see that
$p^u_n\wedge p^s_n\leq (e^\e q^u_n)\wedge (e^\e q^s_n)=e^\e(q^u_n\wedge q^s_n)$ for all $n\in\Z$. Since by construction $p^u_n\wedge p^s_n\geq \eta_n=q^u_n\wedge q^s_n$, we obtain
$1\leq \frac{q^u_n\wedge q^s_n}{p^u_n\wedge p^s_n}\leq e^\e$ as needed.

\medskip
\noindent
{\em Claim 2.\/} For every $n\in\Z$,  $\Psi_{x_n}^{p^u_n, p^s_n}\to\Psi_{x_{n+1}}^{p^u_{n+1}, p^s_{n+1}}$, and  $\Psi_{x_n}^{p_n^u\wedge p^s_n}$ $\e$--overlaps $\Psi_{f^n(x)}^{p_n^u\wedge p^s_n}$.

\medskip
\noindent
{\em Proof.\/} This follows from properties (a), (b), and (c) above, the inequality $p^u_n\wedge p^s_n\geq \eta_n$, and the  monotonicity property of the overlap condition.
\hfill$\Box$

\subsection{Admissible manifolds and the graph transform}
Suppose $x\in\NUH_\chi(f)$. A {\em $u$--manifold in $\Psi_x$} is a manifold $V^u\subset M$ of the form
$$
V^u=\Psi_x\{(F^u(t), t):|t|\leq q\},
$$
where $0<q\leq Q_\e(x)$, and $F^u$ is a $C^{1+\b/3}$--function s.t. $\|F^u\|_\infty\leq Q_\e(x)$.

An {\em $s$--manifold in $\Psi_x$} is a manifold $V^s\subset M$ of the form
$$
V^s=\Psi_x\{(t,F^s(t)):|t|\leq q\},
$$
where $0<q\leq Q_\e(x)$, and $F^s$ is a $C^{1+\b/3}$--function s.t. $\|F^s\|_\infty\leq Q_\e(x)$.

We will use the superscript ``$u/s$" in statements which apply both to the $s$ case and to the $u$ case.

The function $F=F^{u/s}$ is called the {\em representing function} of $V^{u/s}$ at $\Psi_x$.
The {\em parameters} of a $u/s$ manifold in $\Psi_x$ are
\begin{itemize}
\item {\em $\s$--parameter\/}: $\s(V^{u/s}):=\|F'\|_{\b/3}:=\|F'\|_\infty+\sup\left\{\frac{|F'(t_1)-F'(t_2)|}{|t_1-t_2|^{\b/3}}\right\}$;
\item {\em $\g$--parameter\/}: $\g(V^{u/s}):=|F'(0)|$;
\item {\em $\vf$--parameter\/}: $\vf(V^{u/s}):=|F(0)|$;
\item {\em $q$--parameter\/}: $q(V^{u/s}):=q$.
\end{itemize}
A {\em $(u/s, \s,\g,\vf,q)$--manifold in $\Psi_x$} is a $u/s$--manifold $V^{u/s}$ in $\Psi_x$ whose parameters satisfy
$\s(V^{u/s})\leq \s$, $\g(V^{u/s})\leq \g$, $\vf(V^{u/s})\leq \vf$, and $q^{u/s}(V^u)=q$.

\begin{defi}\label{Def_Admissible}
Suppose $\Psi_x^{p^u,p^s}$ is a double chart.  A {\em $u/s$--admissible manifold in $\Psi_x^{p^u,p^s}$} is a $(u/s,\s,\g,\vf,q)$--manifold in $\Psi_x$ s.t.
$$
\s\leq \frac{1}{2},\ \g\leq \frac{1}{2}(p^u\wedge p^s)^{\b/3},\ \vf\leq 10^{-3}(p^u\wedge p^s),\textrm{ and } q=\begin{cases}
p^u & \textrm{$u$--manifolds}\\
p^s & \textrm{$s$--manifolds}.
\end{cases}
$$
\end{defi}

This is similar, but stronger, than the admissibility condition in  Katok \& Mendoza \cite[Definition S.3.4]{KM} or Katok \cite{KatokIHES}. We needed to strengthen the condition to get Proposition \ref{Prop_Intersection} (4) below.

Let $F$ be the representing function of a $u/s$--admissible manifold in $\Psi_x^{p^u, p^s}$.
If $\e<1$ (as we always assume), then the conditions $\s\leq \frac{1}{2}$, $\vf<10^{-3}(p^u\wedge p^s)$ and $p^u, p^s<Q_\e(x)$  force
\begin{equation}\label{Lip(F)}
\textrm{Lip}(F)<\e,
\end{equation}
because for every $t$ in the domain of $F$, $|t|\leq p^{u/s}\leq Q_\e(x)<\e^{3/\b}$ and
\begin{equation}\label{deriv}
|F'(t)|\leq |F'(0)|+\textrm{H\"ol}(F')|t|^{\frac{\b}{3}}\leq \frac{1}{2}(p^{u}\wedge p^s)^{\frac{\b}{3}}+\frac{1}{2}(p^{u/s})^{\frac{\b}{3}}<(p^{u/s})^{\frac{\b}{3}}<\e.
\end{equation}
Another important fact is that if $\e$ is small enough then
$\|F\|_{\infty}<10^{-2}Q_\e(x)$, because
$\|F\|_\infty\leq |F(0)|+\max|F'|\cdot p^{u/s}<\vf+\e p^{u/s}\leq (10^{-3}+\e)p^{u/s}<10^{-2} p^{u/s}$.

\begin{defi}\label{Def_Dist}
Let $V_1, V_2$ be two  $u$--manifolds (resp.  $s$--manifolds) in $\Psi_x$ s.t. $q(V_1)=q(V_2)$, then $\dist(V_1,V_2):=\max|F_1-F_2|$
 where $F_1$ and $F_2$ are the representing functions of $V_1$ and $V_2$ in $\Psi_x$.
\end{defi}
\noindent
Occasionally we will also need the $C^1$--distance defined by
$$
\dist_{C^{1}}(V_1,V_2):=\max|F_1-F_2|+\max|F_1'-F_2'|.
$$

Notice that $\dist$ and $\dist_{C^1}$ are defined using the Pesin charts, not its ``natural" charts. Distances using natural charts are bounded by a constant times distances w.r.t. Pesin charts, because
Pesin charts take the form $\Psi_x=\exp_x\circ C_\chi(x)$ where $C_\chi(x): \R^2\to M$ is a contraction.

\begin{defi}
Let $V^s,V^u$ be a $u$--manifold and an $s$--manifold in $\Psi_x$, with representing functions $F_s, F_u$. Suppose $V^s,V^u$ intersect at a unique point $P=\Psi_x(u,v)$, then  $\measuredangle(V^s,V^u):=\measuredangle((d\Psi_x)_{(u,v)}{1\choose F_s'(u)},(d\Psi_x)_{(u,v)}{F_u'(v)\choose 1})$.
\end{defi}
\noindent
{\em Remark\/:} Pesin charts preserve orientation, therefore there are only two possible choices to the pair of directions of $V^s,V^u$ at $P$. Both lead to the same angle, and this angle is in $(0,\pi)$. Thus the angle of intersection is independent of the chart.

\begin{prop}\label{Prop_Intersection}
The following holds for all $\e$ small enough.  Let $V^u$ be a $u$--admissible manifold in $\Psi_x^{p^u, p^s}$, and $V^s$ be an $s$--admissible manifold in $\Psi_x^{p^u, p^s}$, then
\begin{enumerate}
\item $V^u$ intersects $V^s$ at a unique point $P$;
\item $P=\Psi_x(v,w)$ with $|v|,|w|\leq 10^{-2}(p^u\wedge p^s)$;
\item $P$ is a Lipschitz function of $(V^u, V^s)$, with Lipschitz constant less than $3$;
\item Suppose $\eta:=p^u\wedge p^s$, then the angle of intersection at $P$ satisfies
\begin{eqnarray*}
&e^{-\eta^{\b/4}}\leq \frac{\sin\measuredangle(V^u, V^s)}{\sin\measuredangle(E^s(x), E^u(x))}\leq e^{\eta^{\b/4}}&\\
&|\cos\measuredangle(V^u, V^s)-\cos\measuredangle(E^s(x), E^u(x))|<2\eta^{\b/4}.&
\end{eqnarray*}
\end{enumerate}
\end{prop}
\noindent
Parts (1),(2), and (3) follow from \cite[Corollary S.3.8]{KH}. Part (3) is because of the assumptions on $\g$ and $\s$, and is the reason why we require more than Katok \& Mendoza did in \cite{KM}. See the appendix for proofs.

The following result describes the action of $f$ on admissible manifolds. Results of this type (often called ``graph transform" lemmas) are used to prove Pesin's stable manifold theorem \cite[chapter 7]{BP}, \cite{Pesin}. The proof is in the appendix.

\begin{prop}[Graph Transform]\label{Prop_Graph_Transform}
The following holds for all $\e$ small enough. Suppose $\Psi_x^{p^u,p^s}\to\Psi_y^{q^u,q^s}$, and $V^u$ is a $u$--admissible manifold in $\Psi_x^{p^u,p^s}$, then
 \begin{enumerate}
 \item $f(V^u)$ contains a $u$--manifold $\wh{V}^u$ in  $\Psi_y^{q^u,q^s}$ with parameters
\begin{equation}\label{GraphTransform}
\begin{aligned}
\s(\wh{V}^u)&\leq e^{\sqrt{\e}}e^{-2\chi}[\s(V^u)+\sqrt{\e}]\\
\g(\wh{V}^u)&\leq e^{\sqrt{\e}}e^{-2\chi}[\g(V^u)+\e^{\b/3}(q^u\wedge q^s)^{\b/3}]\\
\vf(\wh{V}^u)&\leq e^{\sqrt{\e}}e^{-\chi} [\vf+\sqrt{\e}(q^u\wedge q^s)]\\
q(\wh{V}^u)&\geq \min\{e^{-\sqrt{\e}}e^\chi q(V^u),Q_\e(y)\}
\end{aligned}
\end{equation}
\item $f({V}^u)$ intersects any $s$--admissible manifold in  $\Psi_y^{q^u,q^s}$ at a unique point.
\item $\wh{V}^u$ restricts to a $u$--admissible manifold in $\Psi_y^{q^u,q^s}$. This is the unique $u$--admissible manifold in $\Psi_y^{q^u,q^s}$ inside $f(V^u)$. We call it $\mathcal F_u[V^u]$.
\item Suppose $V^u$  is represented by the function $F$. If $p:=\Psi_x(F(0),0)$, then   $f(p)\in\mathcal F_u[V^u]$.
\end{enumerate}
Similar statements hold for the $f^{-1}$--image of an $s$--admissible manifold in $\Psi_y^{q^u,q^s}$.
\end{prop}

\begin{defi}
Suppose $\Psi_x^{p^u,p^s}\to\Psi_y^{q^u,q^s}$. The {\em graph transforms} are the maps
\begin{itemize}
\item $\mathcal F_u$ which maps a $u$--admissible manifold $V^u$ in $\Psi_x^{p^u,p^s}$ to the unique $u$--admissible manifold  in $\Psi_y^{q^u,q^s}$ contained in $f(V^u)$;
\item $\mathcal F_s$ which maps an $s$--admissible manifold $V^s$ in $\Psi_y^{q^u,q^s}$ to the unique $s$--admissible manifold  in $\Psi_x^{p^u,p^s}$ contained in $f^{-1}(V^s)$.
\end{itemize}
\end{defi}
\noindent
(The operators  $\mathcal F_s, \mathcal F_u$ depend on the edge $\Psi_x^{p^u,p^s}\to\Psi_y^{q^u,q^s}$.)

\begin{prop}\label{Prop_Graph_Contracts}
If  $\e$ is small enough then the following holds. Let $t=s,u$,  then for any $t$--admissible manifolds $V_1^{t}, V_2^{t}$ in $\Psi_x^{p^u,p^s}$,
\begin{align}
\dist(\mathcal F_{t}(V_1^{t}),\mathcal F_{t}(V_2^{t}))&\leq e^{-\chi/2}\dist(V_1^{t},V_2^{t}\bigr);\label{magnifico}\\
\dist_{C^1}(\mathcal F_{t}(V_1^{t}),\mathcal F_{t}(V_2^{t}))&\leq e^{-\chi/2}\bigl[\dist_{C^1}(V_1^{t},V_2^{t})+\bigl(\dist(V_1^{t},V_2^{t})\bigr)^{\b/3}
\bigr]. \label{dandini}
\end{align}
\end{prop}
\noindent
See \cite[chapter 7]{BP}, \cite{KM}, and the appendix.

\subsection{A Markov extension} Let $\Sigma:=\Sigma(\mathfs G)$ denote the topological Markov shift of two sided infinite paths on the graph $G(\mathfs V, \mathfs E)$:
$$
\Sigma:=\{(v_i)_{i\in\Z}:v_i\in\mathfs V, v_i\to v_{i+1}\textrm{ for all }i\}.
$$
We equip $\Sigma$ with the metric  $d(\un{v},\un{w})=\exp[-\min\{k: v_k\neq w_k\}]$, and the action of the left shift map
 $\s:\Sigma\to\Sigma$, $\s:(v_i)_{i\in\Z}\mapsto (v_{i+1})_{i\in\Z}$.

\medskip
Our aim is to construct a map $\pi:\Sigma\to M$ with a $\chi$--large image s.t. $\pi\circ\s=f\circ\pi$. In fact, the map we construct will be well-defined for all chains.

\medskip
We begin with some comments on general chains of double charts. Suppose $(v_i)_{i\in\Z}$, $v_i=\Psi_{x_i}^{p^u_i,p^s_i}$ is  a chain, and let  $V^u_{-n}$ be a $u$--admissible manifold in $v_{-n}$. The graph transform relative to $v_{-n}\to v_{-n+1}$ maps $V^u_{-n}$ to a $u$--admissible manifold in $v_{-n+1}$,  $\mathcal F_u[V_{-n}]$. Another application of the graph transform, this time relative to $v_{-n+1}\to v_{-n+2}$, maps $\mathcal F_u[V_{-n}]$ to a $u$--admissible manifold in $v_{-n+2}$, which we denote by $\mathcal F_u^2[V^u_{-n}]$. Continuing this way, we eventually reach a $u$--admissible manifold in $v_0$ which we denote by $\mathcal F_u^n[V^u_{-n}]$.
Similarly, any $s$--admissible manifold in $v_n$ is mapped by $n$ applications of $\mathcal F_s$ to an $s$--admissible manifold in $v_0$. The manifolds $\mathcal F_u^n[V^u_{-n}]$ and $\mathcal F_s^n[V^u_n]$ depend on $(v_{-n},\ldots,v_n)$.

Let $V_n$ denote a sequence of $u/s$--manifolds in a chart $\Psi_x$. We say that $V_n$ converges to a $u/s$--manifold $V$, if the representing functions of $V_n$ converge uniformly to the representing function of $V$. Compare with definition \ref{Def_Dist}.
\begin{prop}\label{Prop_V}
Suppose $(v_i)_{i\in\Z}$ is a chain of double charts, and  choose arbitrary $u$--admissible manifolds $V^u_{-n}$ in $v_{-n}$, and $s$--admissible manifolds $V^s_n$ in $v_n$.
\begin{enumerate}
\item The limits
$
V^u[(v_i)_{i\leq 0}]:=\lim\limits_{n\to\infty}\mathcal F_u^n[V^u_{-n}],\text{ and }
V^s[(v_i)_{i\geq 0}]:=\lim\limits_{n\to\infty}\mathcal F_s^n[V^s_{n}]
$
exist, and are independent of the choice of $V^u_{-n}$ and $V^s_n$.
\item $V^u[(v_i)_{i\leq 0}]$ is a $u$--admissible manifold in $v_0$, and $V^s[(v_i)_{i\geq 0}]$ is an $s$--admissible manifold in $v_0$.
\item $f(V^s[(v_i)_{i\geq 0}])\subset V^s[(v_{i+1})_{i\geq 0}]$ and $f^{-1}(V^u[(v_i)_{i\leq 0}])\subset V^u[(v_{i-1})_{i\leq 0}]$;
\item Write $v_i=\Psi_{x_i}^{p^u_i, p^s_i}$, then
\begin{align*}
\hspace{1.5cm}V^s[(v_i)_{i\geq 0}]&=\{p\in \Psi_{x_0}[R_{p^s_0}(\un{0})]:\forall k\geq 0,\ f^{k}(p)\in \Psi_{x_{k}}[R_{10 Q_\e(x_{k})}(\un{0})]\},\\
V^u[(v_i)_{i\leq 0}]&=\{p\in \Psi_{x_0}[R_{p^u_0}(\un{0})]:\forall k\geq 0,\ f^{-k}(p)\in \Psi_{x_{-k}}[R_{10 Q_\e(x_{-k})}(\un{0})]\}.
\end{align*}
\item The maps $(u_i)_{i\in\Z}\mapsto V^{u}[(u_i)_{i\leq 0}], V^s[(u_i)_{i\geq 0}]$ are H\"older continuous: there exist  constants $K>0$ and $0<\theta<1$ s.t. for every $n\geq 0$ and  any two chains $\un{u}, \un{v}$,  if $u_i=v_i$ for all $|i|\leq n$, then
\begin{align*}
\dist_{C^1}(V^{u}[(u_i)_{i\leq 0}],V^{u}[(v_i)_{i\leq 0}])&<K\theta^n;\\
\dist_{C^1}(V^{s}[(u_i)_{i\geq 0}],V^{s}[(v_i)_{i\geq 0}])&<K\theta^n.
\end{align*}
\end{enumerate}
\end{prop}
\noindent
Parts (1)--(4) should be compared with Pesin's Stable Manifold Theorem \cite{Pesin}. Part (5) should be compared to Brin's Theorem on the H\"older continuity of the Oseledets distribution on Pesin sets \cite{Brin}.
\begin{proof}
We give the proof in the case of $u$--manifolds. The case of $s$--manifolds is symmetric.
Before we begin, we mention the following obvious fact: for any double chart $\Psi_x^{p^u,p^s}$ and any two $u$--manifolds $V^u_1,V^u_2$ in $\Psi_x^{p^u,p^s}$,
$$
\dist(V^u_1,V^u_2)\leq 2Q_\e(x)<1.
$$

\medskip
\noindent
{\em Part 1.\/} Existence of the limit.

\medskip
By Proposition \ref{Prop_Graph_Transform},  $\mathcal F_u^n[V^u_{-n}]$ is a $u$--admissible manifold in $v_0$. By Proposition \ref{Prop_Graph_Contracts}, for any other choice $u$--admissible manifolds $W^u_{-n}$ in $v_{-n}$,
$$
\dist(\mathcal F_u^n[V^u_{-n}], \mathcal F_u^n[W^u_{-n}])<\exp[-\tfrac{1}{2}\chi n]\dist(V^u_{-n},W^u_{-n})<\exp[-\tfrac{1}{2}\chi n].
$$
Thus, if the limit exists then it is independent of $V^u_{-n}$.

For every $m>n$,  $W^u_{-n}:=\mathcal F_u^{m-n}[V^u_{-m}]$ is a $u$--admissible manifold in $v_{-n}$. It follows that for every $m>n$,
$
\dist(\mathcal F_u^n[V^u_{-n}], \mathcal F_u^m[V^u_{-m}])<\exp[-\tfrac{1}{2}\chi n].
$
It follows that $\lim \mathcal F_u^n[V^u_{-n}]$ exists.

\medskip
\noindent
{\em Part 2.\/} Admissibility of the limit.

\medskip
Write $v_0=\Psi_x^{p^u,p^s}$, and let $F_{n}$ denote the functions which represent $\mathcal F_u^n[V^u_{-n}]$ in $v_0$. Since $\mathcal F_u^n[V^u_{-n}]$ are $u$--admissible in $v_0$,  for every $n$,
\begin{itemize}
\item $\|F_n'\|_{\b/3}\leq \frac{1}{2}$;
\item $\|F_n'(0)\|\leq \frac{1}{2}(p^u\wedge p^s)^{\b/3}$;
\item $|F_n(0)|\leq 10^{-3}(p^u\wedge p^s)$.
\end{itemize}
Since $\mathcal F_u^n[V^u_{-n}]\xrightarrow[n\to\infty]{}V^u[(v_i)_{i\leq 0}]$, $F_n\xrightarrow[n\to\infty]{} F$ uniformly on $[-p^u,p^u]$, where $F$ represents $V^u[(v_i)_{i\leq 0}]$.

By the Arzela--Ascoli Theorem,  $\exists n_k\uparrow\infty$ s.t. $F_{n_k}'\xrightarrow[k\to\infty]{}G$ uniformly,  where $\|G\|_{\b/3}\leq \frac{1}{2}$. Thus
$
F_{n_k}(t)=F_{n_k}(-p^u)+\int_{-p^u}^t F_{n_k}'(t)dt\xrightarrow[k\to\infty]{}F(-p^u)+\int_{-p^u}^t G(t)dt
$, whence $F$ is differentiable, and  $F'=G$. We also see that $\{F_{n}'\}$ can only have one limit point. Consequently,
$
F_n'\xrightarrow[n\to\infty]{}F'\textrm{ uniformly}$.

It  follows that $\|F'\|_{\b/3}\leq \frac{1}{2}$, $|F'(0)|\leq \frac{1}{2}(p^u\wedge p^s)^{\b/3}$, and $|F(0)|\leq 10^{-3}(p^u\wedge p^s)$, whence the $u$--admissibility of $V^u[(v_i)_{i\in\Z}]$.

\medskip
\noindent
{\em Part 3.\/} Invariance properties of the limit.

\medskip
 Let $V^u:=V^u[(v_i)_{i\leq 0}]=\lim\mathcal F_u^n[V^u_{-n}]$, and $W^u:=V^u[(v_{i-1})_{i\leq 0}]=\lim\mathcal F_u^n[V^u_{-n-1}]$.

  \medskip
  \noindent
  $ \dist(V^u, \mathcal F_u(W^u))\leq \dist(V^u,\mathcal F_u^n(V^u_{-n}))+\dist(\mathcal F_u^n(V^u_{-n}), \mathcal F_u^{n+1}(V^u_{-n-1}))$
  \begin{align*}
 &\hspace{7cm} +\dist(\mathcal F_u^{n+1}(V^u_{-n-1}), \mathcal F_u(W^u))\\
 &\leq \dist(V^u,\mathcal F_u^n(V^u_{-n}))+e^{-\frac{1}{2}n\chi}\dist(V^u_{-n},\mathcal F_u(V^u_{-n-1}))
 +e^{-\frac{1}{2}\chi}\dist(\mathcal F_u^n(V^u_{-n-1}),W^u).
 \end{align*}
 The first and third  summands tend to zero, by the definition of $V^u$ and $W^u$. The second summand tends to zero, because $\dist(V^u_{-n},\mathcal F_u(V^u_{n-1}))<2Q_\e(x)<1$.
 It follows that $V^u=\mathcal F_u(W^u)\subset f(W^u)$.

\medskip
\noindent
{\em Part 4.\/} Suppose $v_i=\Psi_{x_i}^{p^u_i,p^s_i}$, then
$$
V^u=\{p\in \Psi_{x_0}[R_{p^u_0}(\un{0})]: \forall k\geq 0,\ f^{-k}(p)\in\Psi_{x_{-k}}[R_{10Q_\e(x_{-k})}(\un{0})]\}. $$

\medskip
The inclusion $\subseteq$ is simple: Every $u$--admissible manifold $W^u_i$ in  $\Psi_{x_i}^{p^u_i,p^s_i}$ is contained in $\Psi_{x_i}[R_{p^u_i}(\un{0})]$, because if $W^u_i$ is represented by the function $F$ then any $p=\Psi_{x_i}(v,w)$ in $W^u_i$  satisfies $|w|\leq p^u_i$, and
$$
|v|=|F(w)|\leq |F(0)|+\max|F'|\cdot |w|\leq \vf+\e|w|\leq (10^{-3}+\e)p^u_i<p^u_i.
$$
Applying this to $V^u:=V^u[(v_i)_{i\leq 0}]$, we see that  $\forall p\in V^u$,
$p\in V^u\subset\Psi_{x_0}[R_{p^u_0}(\un{0})]$, and by part 3 for every $k\geq 0$
$$f^{-k}(p)\in f^{-k}(V^u)\subseteq V^u[(v_{i-k})_{i\leq 0}]\subset\Psi_{x_{-k}}[R_{p^u_{-k}}(\un{0})]\subset \Psi_{x_{-k}}[R_{10Q_\e(x_{-k})}(\un{0})].$$
We have $\subseteq$.

\medskip
We prove $\supseteq$. Suppose $z\in \Psi_{x_0}[R_{p^u_0}(\un{0})]$ and $f^{-k}(z)\in\Psi_{x_{-k}}[R_{10Q_\e(x_{-k})}(\un{0})]$ for all $k\geq 0$. Write $z=\Psi_{x_0}(v_0,w_0)$. We show that $z\in V^u$ by proving that $v_0=F(w_0)$, where $F$ is the function which represents $V^u$.

Introduce for this purpose the point $\ov{z}=\Psi_{x_0}(\ov{v}_0,\ov{w}_0)$, where $\ov{w}_0=w_0$ and $\ov{v}_0=F(\ov{w}_0)$. For every $k\geq 0$, $f^{-k}(z),f^{-k}(\ov{z})\in \Psi_{x_{-k}}[R_{10 Q_\e(x_{-k})}(\un{0})]$, the first point by assumption, and the second point because  $f^{-k}(\ov{z})\in f^{-k}(V^u)\subset V^u[(v_{i-k})_{i\leq 0}]$. It is therefore possible to write
$$
f^{-k}(z)=\Psi_{x_{-k}}(v_{-k},w_{-k})\textrm{ and }f^{-k}(\ov{z})=\Psi_{x_{-k}}(\ov{v}_{-k},\ov{w}_{-k})\ \ (k\geq 0),
$$
where $|v_{-k}|, |w_{-k}|, |\ov{v}_{-k}|, |\ov{w}_{-k}|\leq 10 Q_\e(x_{-k})$ for all $k\geq 0$.

Proposition \ref{Prop_f_xy}, in its version for $f^{-1}$, says that for every $k\geq 0$, $f_{x_{-k-1} x_{-k}}^{-1}=\Psi_{x_{-k-1}}^{-1}\circ f^{-1}\circ \Psi_{x_{-k}}$ can be put in the form
$$
f_{x_{-k-1} x_{-k}}^{-1}(v,w)=(A_k^{-1}v+g_1^{(k)}(v,w), B_k^{-1}w+g_2^{(k)}(v,w)),
$$
where $|A_k|<e^{-\chi/2}$, $|B_k|>e^{\chi/2}$, and $\max_{R_{10 Q_\e(x_{-k})}}\|\nabla g_i^{(k)}\|<\e$ (provided $\e$ is small enough).

Let $\Delta v_{-k}:=v_{-k}-\ov{v}_{-k}$ and $\Delta w_{-k}:=w_{-k}-\ov{w}_{-k}$. Since  for every $k\leq 0$,
$(v_{-k-1}, w_{-k-1})=f_{x_{-k-1} x_{-k}}^{-1}(v_{-k},w_{-k})$ and
$(\ov{v}_{-k-1}, \ov{w}_{-k-1})=f_{x_{-k-1} x_{-k}}^{-1}(\ov{v}_{-k},\ov{w}_{-k})$,
\begin{align*}
|\Delta v_{-k-1}|&\geq |A_{k}^{-1}|\cdot |\Delta v_{-k}|-\max\|\nabla g_1^{(k)}\|\cdot\bigl(|\Delta v_{-k}|+|\Delta w_{-k}|\bigr)\\
&\geq (e^{\chi/2}-\e)|\Delta v_{-k}|-\e|\Delta w_{-k}|.\\
|\Delta w_{-k-1}|&\leq |B_k^{-1}|\cdot|\Delta w_{-k}|+\max\|\nabla g_2^{(k)}\|\cdot\bigl(|\Delta v_{-k}|+|\Delta w_{-k}|\bigr)\\
&\leq (e^{-\chi/2}+\e)|\Delta w_{-k}|+\e|\Delta v_{-k}|.
\end{align*}
Write for short $a_k:=|\Delta v_{-k}|$ and $b_k:=|\Delta w_{-k}|$. If we assume, as we may, that  $\e$ is so small that $e^{-\chi/2}+\e<e^{-\chi/3}$ and $e^{\chi/2}-\e\geq e^{\chi/3}$, then we obtain
\begin{align*}
a_{k+1}&\geq e^{\chi/3}a_k-\e b_k,\\
b_{k+1}&\leq e^{-\chi/3}b_k+\e a_k.
\end{align*}
By definition, $b_0=0$.

Suppose $\e$ is so small that $e^{-\chi/3}+\e<1$ and $e^{\chi/3}-\e>1$.
We claim that $a_k\leq a_{k+1}$ and $b_k\leq a_{k}$ for all $k$. For $k=0$, this is because $b_0=0$. Assume by induction that $a_k\leq a_{k+1}$ and $b_k\leq a_{k}$, then
\begin{align*}
b_{k+1}&\leq e^{-\chi/3}b_k+\e a_k\leq (e^{-\chi/3}+\e)a_k<a_k\leq a_{k+1}\\
a_{k+2}&\geq e^{\chi/3}a_{k+1}-\e b_{k+1}\geq (e^{\chi/3}-\e)a_{k+1}>a_{k+1}.
\end{align*}

We see that $a_{k+1}\geq (e^{\chi/3}-\e)a_k$ for all $k$, whence
$
a_k\geq (e^{\chi/3}-\e)^k a_0.
$
Either $a_0=0$ or $a_k\xrightarrow[k\to\infty]{}\infty$. But $a_k=|v_{-k}-\ov{v}_{-k}|\leq 20|Q_\e(x_{-k})|<20\e$, so   $a_0=0$.
 Since $a_0=0$, $v_0=\ov{v}_0$, and therefore $F(\ov{w}_0)=F(w_0)$. Thus $z=\Psi_x(F(w_0),w_0)\in V^u$.

 \medskip
 \noindent
 {\em Part 5.\/} H\"older continuity of $\un{u}\mapsto V^u[(u_i)_{i\in\Z}]$.

 \medskip
Suppose two chains $\un{v}=(v_i)_{i\in\Z}, \un{w}=(w_i)_{i\in\Z}$ satisfy $v_i=w_i$ for $i=-N, \ldots,N$. Given $n>N$, let $V^{u}_{-n}$ be a $u$--admissible manifold in $v_{-n}$, and let $W^{u}_{-n}$ be a $u$--admissible manifold in $w_{-n}$.

Let $\mathcal F_u^{\ell}(V^u_{-n})$ (resp.  $\mathcal F_u^{\ell}(W^u_{-n})$) denote the result of applying $\mathcal F_u$ $\ell$ times to $V^u_{-n}$ using the path $u_{-n}\to\cdots\to u_{-n+\ell}$ (resp. using $w_{-n}\to\cdots\to w_{-n+\ell}$).

$\mathcal F_u^{n-N}(V^u_{-n})$ and $\mathcal F_u^{n-N}(W^u_{-n})$ are $u$--admissible manifolds in $v_{-N}(=w_{-N})$. Let $F_{N}, G_{N}$ be their representing functions. Admissibility implies that
\begin{align*}
\|F_{N}-G_{N}\|_\infty &\leq \|F_{N}\|_\infty+\|G_{N}\|_\infty<2Q_\e<1\\
\|F_N'-G_N'\|_\infty &\leq \|F_{N}'\|_\infty+\|G_{N}'\|_\infty<2\e<1.
\end{align*}
Represent $\mathcal F_u^{n-k}[V^u_{-n}]$ and $\mathcal F_u^{n-k}[W^u_{-n}]$ by functions  $F_{k}$ and $G_{k}$. By(\ref{dandini}),
\begin{align}
\|F_{k-1}-G_{k-1}\|_\infty &\leq e^{-\chi/2}\|F_{k}-G_{k}\|_\infty\label{lopi}\\
\|F_{k-1}'-G_{k-1}'\|_\infty &\leq e^{-\chi/2}(\|F_{k}'-G_{k}'\|_\infty+2\|F_{k}-G_{k}\|_\infty^{\b/3}).\label{tal}
\end{align}
Iterating (\ref{lopi}) starting at $k=N$ and going down, we get
$
\|F_k-G_k\|_\infty\leq e^{-\frac{1}{2}\chi(N-k)}$, whence
$
\dist(\mathcal F_u^{n}[V^u_{-n}], \mathcal F_u^{n}[W^u_{-n}])\leq e^{-\frac{1}{2}\chi N}.
$
Passing to the limit $n\to\infty$, we get
$$
\dist(V^u[(v_i)_{i\leq 0}],V^u[(w_i)_{i\leq 0}])\leq e^{-\frac{1}{2}N\chi}.
$$

Now substitute $\|F_k-G_k\|_\infty\leq e^{-\frac{1}{2}\chi(N-k)}$ in (\ref{tal}), and set $c_k:=\|F_{k}'-G_{k}'\|_\infty$, $\theta_1:=e^{-\chi/2}$, and $\theta_2:=e^{-\frac{1}{6}\b\chi}$, then
$
c_{k-1}\leq \theta_1(c_k+2\theta_2^{N-k})
$. It is easy to see by induction that for every $0\leq k\leq N$,
$$
c_0\leq \theta_1^k c_k+2(\theta_1^k \theta_2^{N-k}+\theta_1^{k-1}\theta_2^{N-k+1}+\cdots+\theta_1\theta_2^{N-1}).
$$
We now take $k=N$, paying attention to the inequalities  $\theta_1<\theta_2$ and $c_N\leq 1$:
$
c_0\leq \theta_1^N+2N\theta_2^N<(2N+1)\theta_2^N.
$

It follows that
$
\dist_{C^1}(\mathcal F_u^{n}[V^u_{-n}], \mathcal F_u^{n}[W^u_{-n}])\leq 2(N+1) \theta_2^N.
$
In part 2, we saw that $\mathcal F_u^{n}[V^u_{-n}]$ and $\mathcal F_u^{n}[W^u_{-n}]$ converge to $V^u[(w_i)_{i\leq 0}]$ in $C^1$. Therefore if we
pass to the limit as $n\to\infty$, we get
$
\dist_{C^1}(V^u[(v_i)_{i\leq 0}],V^u[(w_i)_{i\leq 0}])\leq 2(N+1)\theta_2^N.
$
Now pick two constants $\theta\in (\theta_2,1)$ and $K>0$ s.t. $2(N+1)\theta_2^N\leq K\theta^N$ for all $N\geq 0$.
\end{proof}

\begin{thm}\label{Theorem_Markov_Extension} Given a chain of double charts $(v_i)_{i\in\Z}$, let  $\pi(\un{v})\!:=$unique intersection point of $V^u[(v_i)_{i\leq 0}]$ and $V^s[(v_i)_{i\geq 0}]$.
\begin{enumerate}
\item $\pi$ is a well--defined and $\pi\circ\s=f\circ \pi$;
\item $\pi:\Sigma\to M$ is  H\"older continuous map;
\item $\pi(\Sigma)\supset \pi(\Sigma^\#)\supset\NUH_\chi^\#(f)$, therefore $\pi(\Sigma)$ and $\pi(\Sigma^\#)$ have full probability w.r.t. any ergodic invariant probability measure with entropy larger than $\chi$.
\end{enumerate}
\end{thm}
\begin{proof} Proposition \ref{Prop_Intersection} guarantees that $\pi$ is well defined for every chain.

\medskip
\noindent
{\em Part 1.\/} $\pi\circ \s=f\circ\pi$.

\medskip
Suppose $\un{v}$ is a chain, and write $v_i=\Psi_{x_i}^{p^u_i, p^s_i}$ and $z=\pi(\un{v})$. We claim that
\begin{equation}\label{ShadowPi}
f^k(z)\in\Psi_{x_k}[R_{Q_\e(x_k)}(\un{0})]\ \ \ \ (k\in\Z).
\end{equation}
For $k=0$, this is because $z\in V^s[(v_i)_{i\geq 0}]$ and $V^s[(v_i)_{i\geq 0}]$ is $s$--admissible in $\Psi_{x_0}^{p^u_0,p^s_0}$. For $k>0$, we use Proposition \ref{Prop_V} part (3) to see that
$$
f^k(z)\in f^k(V^s[(v_i)_{i\geq 0}])\subset V^s[(v_{i+k})_{i\geq 0}].
$$
Since $V^s[(v_{i+k})_{i\geq 0}]$ is an $s$--admissible manifold in $\Psi_{x_k}^{p^u_k,p^s_k}$, $f^k(z)\in\Psi_{x_k}[R_{Q_\e(x_k)}(\un{0})]$. The case $k<0$ can be handled in the same way, using $V^u[(v_i)_{i\leq 0}]$.
Thus $z=\pi(\un{v})$ satisfies (\ref{ShadowPi}).

Any point which satisfies (\ref{ShadowPi}) must equal $z$, because by  Proposition \ref{Prop_V} part (4), it must lie on
$V^u[(v_i)_{i\leq 0}]\cap V^s[(v_i)_{i\geq 0}]$.
So (\ref{ShadowPi}) characterizes $\pi(\un{v})$.

It is now a simple matter to deduce that $\pi(\s(\un{v}))=f(\pi(\un{v}))$: $f^k[f(\pi(\un{v}))]=f^{k+1}[\pi(\un{v})]$ belongs to $\Psi_{x_{k+1}}[R_{Q_\e(x_{k+1})}(\un{0})]$ for all $k$, and this is the condition which characterizes $\pi(\s\un{v})$.

\medskip
\noindent
{\em Part 2.\/} $\pi$ is   H\"older continuous.

\medskip
We saw that $\un{u}\mapsto V^u[(u_i)_{i\leq 0}]$ and $\un{u}\mapsto  V^s[(u_i)_{i\geq 0}]$ are H\"older continuous (Proposition \ref{Prop_V}).
Since the the intersection point of an $s$--admissible manifold and a $u$ admissible manifold is a Lipschitz function of these manifolds  (Proposition \ref{Prop_Intersection} (3)), $\pi$ is also H\"older continuous.

\medskip
\noindent
{\em Part 3.\/} $\pi(\Sigma)$ has full probability with respect to any ergodic invariant probability measure with entropy larger than $\chi$.

\medskip
We prove that $\pi(\Sigma)\supset\NUH^\#_\chi(f)$.
Suppose $x\in\NUH^\#_\chi(f)$. By Proposition \ref{Prop_Chains_Exist}, there exist $\Psi_{x_k}^{p^u_k, p^s_k}\in\mathfs V$ s.t. $\Psi_{x_k}^{p^u_k, p^s_k}\to \Psi_{x_{k+1}}^{p^u_{k+1}, p^s_{k+1}}$ for all $k$, and s.t.
$\Psi_{x_k}^{p^u_k,p^s_k}$ $\e$--overlaps $\Psi_{f^k(x)}^{p^u_k\wedge p^s_k}$ for all $k\in\Z$.
By Proposition \ref{Prop_Overlap_Meaning}(1),  this implies that
$$
f^k(x)=\Psi_{f^k(x)}(\un{0})\in\Psi_{x_k}[R_{p^u_k\wedge p^s_k}(\un{0})]\subset\Psi_{x_k}[R_{Q_\e(x_k)}(\un{0})]\textrm{ for all $k\in\Z$}.
$$
Thus $x$ satisfies (\ref{ShadowPi}) with $\un{v}=(\Psi_{x_i}^{p^u_i,p^s_i})_{i\in\Z}$. It follows that $z=\pi(\un{v})$.

In fact this argument proves something stronger, that will be of use to us later. Looking closely into the proof of Proposition \ref{Prop_Chains_Exist}, we see that the chain we constructed above satisfies the property $p^u_i\wedge p^s_i\geq q_\e(f^i(x))$. By the definition of $\NUH^\#_\chi(f)$, there exist sequences $i_k,j_k\uparrow\infty$ for which $p^u_{i_k}\wedge p^s_{i_k}$ and $p^u_{-j_k}\wedge p^s_{-j_k}$ are bounded away from zero. By the discreteness property of $\mathfs A$ (Proposition \ref{Prop_A}), $\Psi_{x_i}^{p^u_i,p^s_i}$ must repeat some symbol infinitely often in the past, and (possibly a different symbol) in the future. Thus the above actually proves that
\begin{equation}\label{Sigma_Sharp_is_Large}
\pi(\Sigma^\#)\supset\NUH^\#_{\chi}(f),
\end{equation}
where  $\Sigma^\#:=\{\un{v}\in\Sigma: \exists v,w\in\mathfs V, \exists n_k,m_k\uparrow\infty\textrm{ s.t. }v_{n_k}=v\textrm{, and }v_{-m_k}=w\}$.
\end{proof}

\subsection{The relevant part of the extension}\label{SectionRelevant}
We cannot rule out the possibility that some of the vertices in $\mathfs V$ do not appear in the  coding of any point in $\NUH_\chi(f)$. Such vertices are called {\em irrelevant}. More precisely,

\begin{defi}
A double chart $v=\Psi_{x}^{p^u,p^s}$ is called {\em relevant} if there exists a chain $(v_i)_{i\in\mathbb Z}$ s.t. $v_0=v$ and $\pi(\un{v})\in\NUH_\chi(f)$. A double chart which is not relevant, is called {\em irrelevant}.
\end{defi}
\begin{defi}
 The {\em relevant part} of $\Sigma$ is
$
\Sigma_{rel}:=\{\un{v}\in\Sigma: v_i\textrm{ is relevant for all }i\}.
$
\end{defi}
\noindent
$\Sigma_{rel}$ is the topological Markov shift corresponding to the restriction of the graph $G(\mathfs V,\mathfs E)$ to the relevant vertices.

\begin{prop}
Theorem \ref{Theorem_Markov_Extension} holds with $\Sigma_{rel}$ replacing $\Sigma$.
\end{prop}
\begin{proof}
All the properties of $\pi:\Sigma_{rel}\to M$ are obvious, except for the statement that
$\pi(\Sigma_{rel}^\#)\supset\NUH_\chi^\#(f)$,  where
$
\Sigma^\#_{rel}:=\Sigma^\#\cap\Sigma_{rel}.
$

 Suppose $p\in\NUH^\#_\chi(f)$, then the proof of Theorem \ref{Theorem_Markov_Extension} shows that $\exists\un{v}\in\Sigma^\#$ s.t. $\pi(\un{v})=p$.
Since $\NUH_\chi^\#(f)$ is $f$--invariant and $f\circ\pi=\pi\circ\s$, $\pi(\s^i(\un{v}))=f^i(p)\in\NUH_\chi^\#(f)$, so $v_i$ is relevant for all $i\in\Z$. It follows that $\un{v}\in\Sigma^\#_{rel}$.
\end{proof}

{\em Henceforth we assume w.l.o.g. that all irrelevant vertices have been removed from $\mathfs V$, and we set $\Sigma:=\Sigma_{rel}$.}

\part{Regular chains which shadow the same orbit are close}

\section{The inverse problem for regular chains} \label{Section_Strategy}
\addtocounter{subsection}{1}

In the previous section we constructed a map $\pi$ from the space of chains to $M$, and showed that every $x\in\NUH^\#_\chi(f)$ takes the form $x=\pi(\un{v})$ for some  chain $\un{v}\in\Sigma^\#$.
In principle, there could be infinitely many chains $\un{v}$ s.t. $\pi(\un{v})=x$.  We ask what one can say about the solutions  $\un{v}$ to the equation  $\pi(\un{v})=x$.

Under the additional assumption that  one of the pre-images of $x$ is is regular (see below),  we shall see that the coordinates $v_i$ of $\un{v}$ are determined ``up to bounded error". Here is the precise statement:
\begin{defi}
A chain $(v_i)_{i\in\mathbb Z}$ is called {\em regular} if every $v_i$ is relevant (see \S\ref{SectionRelevant}), and if there are  $v,u$ s.t. for some $n_k,m_k\uparrow\infty$ $v_{-m_k}=u$, $v_{n_k}=v$ for all $k$.
\end{defi}
\noindent
Every element of $\Sigma^\#$ is regular, because of the convention stated in \S\ref{SectionRelevant} .

\begin{thm}\label{Theorem_Pi_Almost_1-1}
The following holds for all $\e$ small enough. Suppose  $(\Psi_{x_i}^{p^u_i,p^s_i})_{i\in\Z}$, $(\Psi_{y_i}^{q^u_i,q^s_i})_{i\in\Z}$ are regular chains s.t. $\pi[(\Psi_{x_i}^{p^u_i,p^s_i})_{i\in\Z}]=\pi[(\Psi_{y_i}^{q^u_i,q^s_i})_{i\in\Z}]$, then  for all $i$,
\begin{enumerate}
\item $d(x_i,y_i)<\e$;
\item $(\Psi_{y_i}^{-1}\circ \Psi_{x_i})(\un{u})=(-1)^{\s_i}\un{u}+\un{c}_i+\Delta_i(\un{u})$ for all $\un{u}\in R_\e(\un{0})$, where $\s_i\in \{0,1\}$, $\un{c}_i$ is a constant vector s.t. $\|\un{c}_i\|<10^{-1}(q^u_i\wedge q^s_i)$, and $\Delta_i$ is a vector field s.t. $\Delta_i(\un{0})=\un{0}$ and $\|(d\Delta_i)_{\un{v}}\|<\sqrt[3]{\e}$ on $R_\e(\un{0})$;
\item $p^u_i/q^u_i , p^s_i/q^s_i\in [e^{-\sqrt[3]{\e}},e^{\sqrt[3]{\e}}]$.
\end{enumerate}
\end{thm}

The proof of Theorem \ref{Theorem_Pi_Almost_1-1} is long, so we broke it into several sections (\S\ref{SectionCompR},\ref{SectionCompS/U},\ref{SectionCompP/Q}). Here is an overview.
Suppose $(\Psi_{x_i}^{p^u_i,p^s_i})_{i\in\Z}, (\Psi_{y_i}^{q^u_i,q^s_i})_{i\in\Z}$ are two chains in  $\Sigma^\#$ s.t.
\begin{equation}\label{Preimage_Equation}
\pi[(\Psi_{x_i}^{p^u_i,p^s_i})_{i\in\Z}]=\pi[(\Psi_{y_i}^{q^u_i,q^s_i})_{i\in\Z}]=x
\end{equation}
We want to show that    $\Psi_{x_i}$ is close to $\Psi_{y_i}$ for all $i$.

Equation (\ref{Preimage_Equation}) implies that  $f^i(x)$ is the intersection of a $u$--admissible and an $s$--admissible manifold in $\Psi_{x_i}^{p^u_i,p^s_i}$, therefore (Proposition \ref{Prop_Intersection}), $f^i(x)=\Psi_{x_i}(v_i,w_i)$ where $|v_i|,|w_i|\leq 10^{-2}(p^u_i\wedge p^s_i)$. By construction, Pesin charts are  $2$--Lipschitz, therefore $d(f^i(x),x_i)<50^{-1}(p^u_i\wedge p^s_i)$. Similarly $d(f^i(x),y_i)<50^{-1}(q^u_i\wedge q^u_i)$. It follows that $d(x_i,y_i)<25^{-1}\max\{p^u_i\wedge p^s_i,q^u_i\wedge q^s_i\}<\e$ for all $i\in\Z$.

Assume without loss of generality that $\e$ is smaller than the Lebesgue number of the cover $\mathfs D$ which we had constructed in \S\ref{SectionOC}, then   $x_i,y_i$ belong to the same element $D_i$ of $\mathfs D$.  This allows us to  write
\begin{align*}
\Psi_{x_i}&=\exp_{x_i}\circ\vartheta_{x_i}\circ C_{x_i}\\
\Psi_{y_i}&=\exp_{y_i}\circ\vartheta_{y_i}\circ C_{y_i}
\end{align*}
where $\vartheta_{z_i}:\R^2\to T_{z_i} M$ $(z_i=x_i,y_i)$  are the isometries we constructed in \S\ref{SectionOC}, and  $C_{x_i}, C_{y_i}\in\mathrm{GL}(2,\R)$ are given by  $C_\chi(x_i)=\vartheta_{x_i}\circ C_{x_i}$ and $C_\chi(y_i)=\vartheta_{y_i}\circ C_{y_i}$.

Let $z_i=x_i,y_i$, then   $C_\chi(z_i)$ is the unique linear operator which maps
$\un{e}^1={1\choose 0}$ to $s_\chi(z_i)^{-1}\un{e}^s(z_i)$, and $\un{e}^2={0\choose 1}$ to $u_\chi(z_i)^{-1}\un{e}^u(z_i)$. Writing as usual $\a(z_i):=\measuredangle(\un{e}^s(z_i),\un{e}^u(z_i))$, we see that
\begin{align}\label{C_z}
C_{z_i}&=R_{z_i}\left(\begin{array}{cc}
s_\chi(z_i)^{-1} & u_\chi(z_i)^{-1}\cos\a(z_i)\\
0 & u_\chi(z_i)^{-1}\sin\a(z_i)
\end{array}
\right),
\end{align}
where $R_{z_i}$ is the unique orientation preserving orthogonal matrix which rotates $\un{e}^1$ to the direction of $\vartheta_{z_i}^{-1}(\un{e}^s(z_i))$  ($z_i=x_i,y_i$).
Some terminology:
\begin{itemize}
\item $z_i$ are called {\em position parameters},
\item $R_{z_i}$ and $\a(z_i)$ are called {\em axes parameters},
\item $s_\chi(z_i), u_\chi(z_i)$ are called {\em scaling parameters},
\item $(p^u_i,p^s_i)$ are called {\em window parameters}.
\end{itemize}
The proof is done by comparing the parameters of $\Psi_{x_i}^{p^u_i,p^s_i}$ to those of $\Psi_{y_i}^{q^u_i,q^s_i}$.

The comparison of the position parameters had already been done above. We record the conclusion for future reference:
\begin{prop}\label{Prop_x_Parameter}
Let $(\Psi_{x_i}^{p^u_i,p^s_i})_{i\in\Z}, (\Psi_{y_i}^{q^u_i,q^s_i})_{i\in\Z}$ be two chains s.t. $\pi[(\Psi_{x_i}^{p^u_i,p^s_i})_{i\in\Z}]=\pi[(\Psi_{y_i}^{q^u_i,q^s_i})_{i\in\Z}]$, then $d(x_i,y_i)<25^{-1}\max\{p^u_i\wedge p^s_i, q^u_i\wedge q^s_i\}$ ($i\in\Z$).
\end{prop}
\noindent
Regularity is not needed here.  We shall make use of it when we analyze  the scaling parameters and the window parameters.

\section{Axes parameters}\label{SectionCompR}
Let $(\Psi_{x_i}^{p^u_i,p^s_i})_{i\in\Z}, (\Psi_{y_i}^{q^u_i,q^s_i})_{i\in\Z}$ be two  chains  s.t. $\pi[(\Psi_{x_i}^{p^u_i,p^s_i})_{i\in\Z}]=\pi[(\Psi_{y_i}^{q^u_i,q^s_i})_{i\in\Z}]$. We compare $R_{x_i}$ to $R_{y_i}$ and $\a(x_i)$ to $\a(y_i)$. The analysis relies on a special property of $V^{u}[(z_k)_{k\leq i}]$ and $V^s[(z_k)_{k\geq i}]$ ($z_k=x_k,y_k$), which we call {\em ``staying in windows"}. We begin by discussing this property.

\subsection{Staying in windows}
\begin{defi}
Suppose $V^u$ is a $u$--admissible manifold in $\Psi_{x}^{p^u,p^s}$. We say that $V^u$ {\em stays in windows} if there is a negative chain $(\Psi_{x_i}^{p^u_i,p^s_i})_{i\leq 0}$ with $\Psi_{x_0}^{p^u_0,p^s_0}=\Psi_{x}^{p^u,p^s}$ and $u$--admissible manifolds $W^u_i$ in $\Psi_{x_i}^{p^u_i,p^s_i}$ s.t. $f^{-|i|}(V^u_i)\subseteq W^u_i$ for all $i\leq 0$.
\end{defi}
\begin{defi}
Suppose $V^s$ is an $s$--admissible manifold in $\Psi_{x}^{p^u,p^s}$. We say that $V^s$ {\em stays in windows} if there is a positive chain $(\Psi_{x_i}^{p^u_i,p^s_i})_{i\geq 0}$ with $\Psi_{x_0}^{p^u_0,p^s_0}=\Psi_{x}^{p^u,p^s}$ and $s$--admissible manifolds $W^s_i$ in $\Psi_{x_i}^{p^u_i,p^s_i}$ s.t. $f^{i}(V^s_i)\subseteq W^s_i$ for all $i\geq 0$.
\end{defi}

If $\un{v}$ is a chain, then $V^u_i:=V^u[(v_k)_{k\leq i}]$ and  $V^s_i:=V^s[(v_k)_{k\geq i}]$ stay in windows, because $f^{-k}(V^u_i)\subset V^u_{i-k}$ and $f^k(V^s_i)\subset V^s_{i+k}$ for all $k\geq 0$ (Proposition \ref{Prop_V}).

The following proposition says that $s/u$--admissible manifolds which stay in windows are local stable/unstable manifolds in the sense of Pesin \cite{Pesin}:
\begin{prop}\label{Prop_Staying_In_Windows}
The following holds for all $\e$ small enough. Let  $V^{s}$ be an  admissible $s$--manifold in $\Psi_x^{p^u,p^s}$, and suppose $V^s$ stays in windows.
\begin{enumerate}
\item For every $y,z\in V^s$, $d(f^k(y),f^k(z))<e^{-\frac{1}{2}k\chi}$ for all $k\geq 0$.
\item For every $y\in V^s$, let $\un{e}^s(y)$ denote the positively oriented unit tangent vector to $V^s$ at $y$, then
$\|df^k_y\un{e}^s(y)\|_{f^k(y)}\leq 6\|C_\chi(x)^{-1}\|e^{-\frac{1}{2}k\chi}$ for all $k\geq 0$.
\item $\left|\log\|df^k_y \un{e}^s(y)\|_{f^k(y)}-\log\|df^k_z \un{e}^s(z)\|_{f^k(z)}\right|\!<\!Q_\e(x)^{\b/4}$  $(y,z\in V^s, k\geq 0)$.
\end{enumerate}
The symmetric statement holds for $u$--admissible manifolds which stay in windows: replace the $s$--tags by $u$--tags, and $f$ by $f^{-1}$.
\end{prop}
The proof is modeled on the proof of Pesin's Stable Manifold Theorem \cite[chapter 7]{BP}:  $f^n: V^s\to f^n(V^s)$ is given in coordinates by
$$\Psi_{x_n}^{-1}\circ f^n\circ\Psi_{x_0}=f_{x_{n-1} x_n}\circ\cdots\circ f_{x_0 x_1} .$$ Since $V^s$ stays in windows, the orbits of points in $V^s$ remain in the ``windows" where $f_{x_i x_{i+1}}$ is close to a linear hyperbolic map. One can then prove the proposition by direct calculations. See the appendix for details.

\begin{prop}\label{Prop_Uniqueness}
The following holds for all $\e$ small enough.
Let $V^s$ (resp. $U^s$) be an $s$--admissible manifold in $\Psi_x^{p^u,p^s}$ (resp. in $\Psi_y^{q^u,q^s}$). Suppose $V^s, U^s$ stay in windows. If $x=y$ then
either $V^s, U^s$ are disjoint, or one contains the other.

\medskip
\noindent
The same statement holds for $u$--admissible manifolds.
\end{prop}
\noindent
See the appendix for a proof.

\subsection{Comparison of $\a(x_i)$ to $\a(y_i)$}

\begin{prop}\label{Prop_Alpha_Comparison}
Let $(\Psi_{x_i}^{p^u_i,p^s_i})_{i\in\Z},(\Psi_{y_i}^{q^u_i,q^s_i})_{i\in\Z}$ be chains s.t.  $\pi[(\Psi_{x_i}^{p^u_i,p^s_i})_{i\in\Z}]=\pi[(\Psi_{y_i}^{q^u_i,q^s_i})_{i\in\Z}]$, then for all $i\in\Z$
\begin{enumerate}
\item $
e^{-\sqrt{\e}}\leq \frac{\sin\a(x_i)}{\sin\a(y_i)}\leq e^{\sqrt{\e}}
$
\item $|\cos\a(x_i)-\cos\a(y_i)|<\sqrt{\e}$
\end{enumerate}
\end{prop}
\begin{proof}
Write $v_i=\Psi_{x_i}^{p^u_i,p^s_i}$, $u_i=\Psi_{y_i}^{q^u_i,q^s_i}$,  $x:=\pi[(\Psi_{x_i}^{p^u_i,p^s_i})_{i\in\Z}]=\pi[(\Psi_{y_i}^{q^u_i,q^s_i})_{i\in\Z}]$, and
\begin{align*}
V^s_{x_k}:= V^s[(v_i)_{i\geq k}] && V^u_{x_k}:=V^u[(v_i)_{i\leq k}]&&E^{s/u}_{x_k}:=T_{f^k(x)} V^{s/u}_{x_k}\\
V^s_{y_k}:= V^s[(u_i)_{i\geq k}] && V^u_{y_k}:=V^u[(u_i)_{i\leq k}]&& E^{s/u}_{y_k}:=T_{f^k(x)} V^{s/u}_{y_k}.
\end{align*}
We claim that
\begin{enumerate}
\item[(i)] $\limsup\limits_{n\to\infty}\frac{1}{n}\log\|df^n_{f^k(x)}\un{w}\|<0\textrm{ on }E^{s}_{x_k}\setminus\{\un{0}\}\textrm{ and } E^s_{y_k}\setminus\{\un{0}\}$,
    \item[(ii)]
$\limsup\limits_{n\to\infty}\frac{1}{n}\log\|df^n_{f^k(x)}\un{w}\|>0\textrm{ on }E^{u}_{x_k}\setminus\{\un{0}\}\textrm{ and } E^u_{y_k}\setminus\{\un{0}\}$.
\end{enumerate}

 We give the details for $E^{s/u}_{x_k}$. The case of $E^{s/u}_{y_k}$ is identical.

Part (i) follows from Proposition \ref{Prop_Staying_In_Windows} (2), applied to $V^s_{x_k}$ and $V^s_{y_k}$.

The proof of (ii) is slightly more complicated. Suppose $\un{w}\in E^u_{x_k}\setminus\{\un{0}\}$, then $\un{w}$ is tangent to $V^u_{x_k}$ at $f^k(x)$. For every $n$, $f^{k+n}(x)=\pi[(v_{i+k+n})_{i\in\Z}]\in V^u_{k+n}$, so
$$
f^k(x)=f^{-n}(f^{k+n}(x))\in f^{-n}[V^u_{k+n}].
$$
It follows that $df^n_{f^k(x)}\un{w}\in T_{f^{k+n}(x)}[V^u_{k+n}]\setminus\{\un{0}\}$.

We apply  Proposition \ref{Prop_Staying_In_Windows} (2) in its version for $u$--admissible manifolds to the manifold $V^u_{x_{k+n}}$ and the vector  $df^n_{f^k(x)}\un{w}$. This gives the estimate
\begin{align*}
\|\un{w}\|&=\bigl\|df^{-n}_{f^{k+n}(x)}[df^n_{f^k(x)}\un{w}]\bigr\|\leq
6 e^{-\frac{1}{2}n\chi}\|C_\chi(x_{k+n})^{-1}\|\cdot \|df^n_{f^k(x)}\un{w}\|\\
&\leq 6 e^{-\frac{1}{2}n\chi} Q_\e(x_{k+n})^{-1}\|df^n_{f^k(x)}\un{w}\|\ \ (\textrm{definition of $Q_\e$})\\
&\leq 6 e^{-\frac{1}{2}n\chi} (p^u_{k+n}\wedge p^s_{k+n})^{-1}\|df^n_{f^k(x)}\un{w}\|\\
&\leq 6 e^{-\frac{1}{2}n\chi+n\e}(p^u_{k}\wedge p^s_{k})^{-1}\|df^n_{f^k(x)}\un{w}\|\ (\textrm{Lemma \ref{Lemma_Subordinated_Tempered}}).
\end{align*}
Thus $\|df^n_{f^k(x)}\un{w}\|\geq \frac{1}{6}e^{\frac{1}{2}n\chi+n\e}(p^u_{k}\wedge p^s_{k})\|\un{w}\|$. Part (ii) follows.

\medskip
By (i) and (ii),  $E^s_{x_k}, E^s_{y_k}=\{\un{w}\in T_{f^k(x)}M:\limsup\limits_{n\to\infty}\frac{1}{n}\log\|df^n_{f^k(x)}\un{w}\|<0\}.$
For reasons of symmetry,
$E^u_{x_k}, E^u_{y_k}=\{\un{w}\in T_{f^k(x)}M:\limsup\limits_{n\to\infty}\frac{1}{n}\log\|df^{-n}_{f^k(x)}\un{w}\|<0\}.
$
It follows that $E^s_{x_k}=E^s_{y_k}$ and $E^u_{x_k}=E^u_{y_k}$.

As a result,
$
\measuredangle(V^s_{x_k},V^u_{x_k})=\measuredangle(V^s_{y_k},V^u_{y_k}).
$
By  Proposition \ref{Prop_Intersection}
$
\sin\measuredangle(V^s_{x_k},V^u_{x_k})=e^{\pm(p^u_i\wedge p^s_i)^{\b/4}}\sin\a(x_k)$ and $\sin\measuredangle(V^s_{y_k},V^u_{y_k})=e^{\pm(q^u_i\wedge q^s_i)^{\b/4}}\sin\a(y_k)$.
Since $p^u_i\wedge p^s_i\leq Q_\e(x_i)<\e^{3/\b}$ and  $q^u_i\wedge q^s_i\leq Q_\e(y_i)<\e^{3/\b}$,
$
e^{-2{\e^{3/4}}}<\sin\a(x_k)/\sin\a(y_k)<e^{2{\e^{3/4}}}.
$
Similarly one sees that
$
|\cos\a(x_k)-\cos\a(y_k)|<4\e^{3/4},
$
and the proposition follows for all $\e$ so small that $4\e^{3/4}<\sqrt{\e}$.
\end{proof}
The proof actually gives the following stronger estimates, which we now record for future reference:
\begin{lem}\label{Lemma_Strong_Alpha_Comp}
Under the assumptions of the previous proposition,
\begin{enumerate}
\item $e^{-(p^u_i\wedge p^s_i)^{\b/4}-(q^u_i\wedge q^s_i)^{\b/4}}<\frac{\sin\a(x_i)}{\sin\a(y_i)}<e^{(p^u_i\wedge p^s_i)^{\b/4}+(q^u_i\wedge q^s_i)^{\b/4}}$;
\item $|\cos\a(x_i)-\cos\a(y_i)|<4[(p^u_i\wedge p^s_i)^{\b/4}+(q^u_i\wedge q^s_i)^{\b/4}]$.
\end{enumerate}
\end{lem}
\subsection{Comparison of $R_{x_i}$ to $R_{y_i}$}
\begin{prop}\label{Prop_R_Comparison}
The following holds for all $\e$ small enough. For any two chains  $(\Psi_{x_i}^{p^u_i,p^s_i})_{i\in\Z}$ and $(\Psi_{y_i}^{q^u_i,q^s_i})_{i\in\Z}$, if $\pi[(\Psi_{x_i}^{p^u_i,p^s_i})_{i\in\Z}]=\pi[(\Psi_{y_i}^{q^u_i,q^s_i})_{i\in\Z}]$, then $$R_{y_i}^{-1}R_{x_i}=(-1)^{\s_i} \id+\left(\begin{array}{cc}
\e_{11} & \e_{12}\\ \e_{21} & \e_{22}
\end{array}
\right),$$
where $\s_i\in\{0,1\}$ and  $|\e_{jk}|<[(p^u_i\wedge p^s_i)^{\b/5}+(q^u_i\wedge q^s_i)^{\b/5}]<\sqrt{\e}$.
\end{prop}
\begin{proof}
In order to keep the notation as light as possible, we only do  the case $i=0$, and write $\Psi_{x_0}^{p^u_0,p^s_0}=\Psi_{x}^{p^u,p^s}$ , $\Psi_{x_0}^{p^u_0,p^s_0}=\Psi_{y}^{q^u,q^s}$, $p:=p^u\wedge p^s$, and $q:=q^u\wedge q^s$. We also set as usual $v_i=\Psi_{x_i}^{p^u_i,p^s_i}$ and $u_i=\Psi_{y_i}^{q^u_i,q^s_i}$.

Let $z=\pi[\un{v}]=\pi[\un{u}]$. The manifold $V^{s}[(v_i)_{i\geq 0}]$ inherits an orientation from the chart $\Psi_{x}$.  Let $\un{e}^{s}_x(z)$ denote the positively oriented unit tangent vector to $V^{s}[(v_i)_{i\geq 0}]$ at $z$. The manifold $V^{s}[(u_i)_{i\geq 0}]$ inherits an orientation from the chart $\Psi_{y}$.  Let $\un{e}^{s}_y(z)$ denote the positively oriented unit tangent vector to $V^{s}[(u_i)_{i\geq 0}]$ at $z$.
Since  $T_z V^s[(v_i)_{i\in\Z}]=T_z V^s[(u_i)_{i\in\Z}]$ (see the proof of Proposition \ref{Prop_Alpha_Comparison}),
$
\un{e}^s_x(z)=\pm\un{e}^s_y(z).
$

We write $z$ and  $\un{e}^s_x(z),\un{e}^s_y(z)$ in coordinates in $\Psi_x$ and $\Psi_y$:
\begin{itemize}
\item $z=\Psi_x(\un{\zeta})$ and $\un{e}^s_x(z)=\frac{[(d\Psi_x)_{\un{\zeta}}]\un{a}}{\|[(d\Psi_x)_{\un{\zeta}}]\un{a}\|}$, where $\un{\zeta}\in R_{10^{-2}p}(\un{0})$, $\un{a}={1\choose a}$, and  $|a|\leq p^{\b/3}$ (see Proposition \ref{Prop_Intersection} and (\ref{deriv})).
\item $z=\Psi_y(\un{\eta})$ and $\un{e}^s_y(z)=\frac{[(d\Psi_y)_{\un{\eta}}]\un{b}}{\|[(d\Psi_y)_{\un{\eta}}]\un{b}\|}$, where $\un{\eta}\in R_{10^{-2}q}(\un{0})$, $\un{b}={1\choose b}$, and  $|b|\leq q^{\b/3}$ (see Proposition \ref{Prop_Intersection} and (\ref{deriv})).
\end{itemize}
Since $\un{e}^s_x(z)=\pm\un{e}^s_y(z)$, there is a non-zero (signed) scalar $\l$ such that
\begin{equation}\label{C_Compare_Eqn}
C_{x}\un{a}=\l [(d\exp_{x}\circ\vartheta_{x})_{C_x\un{\zeta}}]^{-1}[(d\exp_{y}\circ\vartheta_{y})_{C_y\un{\eta}}]C_{y}\un{b},
\end{equation}
where $C_{x}, C_{y}$ are given by (\ref{C_z}).

\medskip
\noindent
{\em Claim 1.\/} $C_x\un{a}\propto R_x{1\pm p^{\b/4}\choose 0\pm p^{\b/4}}$ and  $C_y\un{b}\propto R_y{1\pm q^{\b/4}\choose 0\pm q^{\b/4}}$. Here $\vec{a}\propto\vec{b}$ means that $\vec{a}=t\vec{b}$ for some $t\neq 0$, and $a\pm c$ means a quantity in $[a-c,a+c]$.

\medskip
\noindent
\begin{align*}
\textrm{\em Proof. } C_{x}\un{a}&=R_{x}{s_\chi(x)^{-1}+u_\chi(x)^{-1}\cos\a(x) a\choose u_\chi(x)^{-1}\sin\a(x) a}\\
&\propto R_x{1\pm \|C_\chi(x)^{-1}\|\cdot |a|\choose 0\pm\|C_\chi(x)^{-1}\|\cdot |a|},\textrm{ because $u_\chi>1$ and $s_\chi=\|C_\chi(x)^{-1}\un{e}^s(x)\|$}\\
&=R_x{1\pm p^{\b/4}\choose 0\pm p^{\b/4}}, \textrm{ because $|a|<p^{\b/3}\leq Q_\chi(x)^{\b/12}p^{\b/4}<\frac{p^{\b/4}}{\|C_\chi(x)^{-1}\|}$}.
\end{align*}
Similarly, $C_y\un{b}\propto R_y{1\pm q^{\b/4}\choose 0\pm q^{\b/4}}$.

\medskip
\noindent
{\em Claim 2.\/} There exists a constant $J>1$ (which only depends on $M$) s.t. for all $D\in\mathfs D$, $x,y\in D$, and $\|\un{w}_1\|,\|\un{w}_2\|<2$,
$$
\bigl\|[(d\exp_{x}\circ\vartheta_{x})_{\un{w}_1}]^{-1}[(d\exp_{y}\circ\vartheta_{y})_{\un{w}_2}]-\id\bigr\|<J(d(x,y)+\|\un{w}_1-\un{w}_2\|).
$$

\medskip
\noindent
{\em Proof.}
Let $J_1$ denote a common Lipschitz constant for the maps $$(w,\un{w})\mapsto (d\exp_w\circ\vartheta_w)_{\un{w}}$$  on $D\x B_2(\un{0})$ for all $D\in\mathfs D$. Let $J_2$ denote the maximum over $D\in\mathfs D$ of
$
\sup\{\|(d\exp_w\circ\vartheta_w)_{\un{w}}^{-1}\|:w\in D, \|\un{w}\|<2\}.
$
The claim holds with $J:=J_1 J_2+1$.

\medskip
\noindent
{\em Claim 3.\/} $R_x{1\choose 0}+\un{\e}_1\propto R_y{1\choose 0}+\un{\e}_2$ where  $\|\un{\e}_1\|$ and $\|\un{\e}_2\|$  are less than $3J(p^{\b/4}+q^{\b/4})$.

\medskip
\noindent
{\em Proof.\/}
  $C_\chi(\cdot)$ is a contraction, so
$
\|C_x\un{\zeta}-C_y\un{\eta}\|<\|\un{\zeta}\|+\|\un{\eta}\|<10^{-2}(p+q)
$.  Also, by Proposition \ref{Prop_x_Parameter}, $d(x,y)<25^{-1}(p+q)$.
Therefore, by Claim 2,
$$
[(d\exp_{x}\circ\vartheta_{x})_{C_{\e}(x)\un{\zeta}}]^{-1}[(d\exp_{y}\circ\vartheta_{y})_{C_{\e}(y)\un{\eta}}]=\id+E
$$
where $E$ is a matrix s.t. $\|E\|<J(p+q)$. The claim follows from (\ref{C_Compare_Eqn}) by direct calculation.

\medskip
We can now prove the proposition. $R_x$ and $R_y$ are rotation matrices, therefore $R_y^{-1}R_x$ is  a rotation matrix. The problem is to estimate the angle. Claim 3 allows us to write
\begin{equation}\label{RRminus1}
R_y^{-1} R_x{1\choose 0}=c\left[{1\choose 0}+R_y^{-1}\un{\e}_2-c^{-1}R_y^{-1}\un{\e}_1\right],
\end{equation}
where $c$ is a scalar s.t.
$
|c|=\frac{1\pm\|\un{\e}_1\|}{1\pm\|\un{\e}_2\|}
$. Since $\|\un{\e}_i\|<3J(p^{\b/4}+q^{\b/4})<6J\e^{3/4}$, $|c|\in [e^{-10 J\sqrt{\e}}, e^{10 J\sqrt{\e}}]$, at least provided $\e$ is small enough.

Since $R_x$ and $R_y$ are orthogonal matrices,
the vector on the right-hand side of (\ref{RRminus1}) is a unit vector. Put it in the form $(-1)^{\s_0}(\cos\theta,\sin\theta)$ where $\s_0\in\{0,1\}$ and $\theta\in(-\frac{\pi}{2},\frac{\pi}{2})$, then
\begin{align*}
|\theta|&\leq \tan^{-1}\left(\frac{\|\un{\e}_2\|+|c|^{-1}\cdot\|\un{\e}_1\|}{1-\|\un{\e}_2\|-|c|^{-1}\|\un{\e}_1\|}\right)
<\frac{\|\un{\e}_2\|+|c|^{-1}\cdot\|\un{\e}_1\|}{1-\|\un{\e}_2\|-|c|^{-1}\|\un{\e}_1\|}\\
&<\frac{3J(1+e^{10 J\sqrt{\e}})}{1-6J(1+e^{10J\sqrt{\e}})\e^{3/4}}(p^{\b/4}+q^{\b/4}).
\end{align*}
Since $p,q<\e^{3/\b}$, if $\e$ is small enough, then this is less than $p^{\b/5}+q^{\b/5}<2\e^{3/5}<\sqrt{\e}$. It follows that $(-1)^{\s_0}R_y^{-1} R_x$ is  a rotation by angle less than $p^{\b/5}+q^{\b/5}<\sqrt{\e}$.
\end{proof}

\section{Scaling parameters}\label{SectionCompS/U}
\subsection{The $s_\chi$ and $u_\chi$ parameters of admissible manifolds}
In \S\ref{Section_NUH} we defined  $s_\chi(\cdot)$ on $\NUH_\chi(f)$. We now extend this definition  to all points lying on  $s$--admissible manifolds $V^s$ which stay in windows.

 Suppose $y\in V^s$. If $y\in\NUH_\chi(f)$ define $\un{e}^s(y)$ as in \S\ref{Section_NUH}, and note that by proposition \ref{Prop_Staying_In_Windows}(2), $\un{e}^s(y)$ is tangent to $V^s$ at $y$. Motivated by this, we define $\un{e}^s(y)$ for $y\not\in\NUH_\chi(f)$ to be one of the two unit tangent vectors to $V^s$ at $y$ (it doesn't matter which), and then we let
 $$
 s_\chi(y):=\sqrt{2}\left(\sum_{k=0}^\infty e^{2k\chi}\|df^k_y\un{e}^s(y)\|^2_{f^k(y)}\right)^{\frac{1}{2}}\in (\sqrt{2},\infty].
 $$

Similarly, for any $u$--admissible manifold $V^u$ which stays in windows, and any $y\in V^u$ we define $\un{e}^u(y)$ as in \S\ref{Section_NUH} when $y\in\NUH_\chi(f)$, and we let $\un{e}^u(y)$  be one of the two unit tangent vectors to $V^u$ at $y$ when $y\not\in\NUH_\chi(f)$. Then we let
$$
 u_\chi(y):=\sqrt{2}\left(\sum_{k=0}^\infty e^{2k\chi}\|df^{-k}_y\un{e}^u(y)\|^2_{f^{-k}(y)}\right)^{\frac{1}{2}}\in(\sqrt{2},\infty].
$$

Although these numbers depend on $y$, they are not very sensitive to its value: by Proposition \ref{Prop_Staying_In_Windows} part 3, for any pair of points $y,z$ in the same
$s$--admissible manifold, if $s_\chi(y)$ is finite then $s_\chi(z)$ is finite, and
$$
e^{-\sqrt{\e}}< s_\chi(y)/s_\chi(z)< e^{\sqrt{\e}}.
$$
A similar statement holds for $u_\chi$--parameters on $u$--admissible manifolds.

\begin{defi}
Let $V^{s}$ be an $s$--admissible manifold in $\Psi_x^{p^u,p^s}$ with representing function $F^s$.
Let $V^{u}$ be a $u$--admissible manifold in $\Psi_x^{p^u,p^s}$ with representing function $F^u$.
If  $V^s$ and $V^u$ stay in windows, then
\begin{enumerate}
\item $s_\chi(V^s)$, the {\em $s_\chi$--parameter} of $V^s$, is $s_\chi(p)$ where $p:=\Psi_x(0,F^s(0))$,
\item $u_\chi(V^u)$, the {\em $u_\chi$--parameter} of $V^u$, is $u_\chi(q)$ where $q:=\Psi_x(F^u(0),0)$.
\end{enumerate}
\end{defi}

\begin{lem}\label{Main_Lemma_One}
The following holds for all $\e$ small enough. Suppose $\Psi_x^{p^u,p^s}\to\Psi_y^{q^u,q^s}$, and let  $V^s$ be an $s$--admissible manifold in $\Psi_y^{q^u,q^s}$ which stays in windows.
If $s_\chi(V^s)<\infty$ then $s_\chi(\mathcal F_s(V^s))<\infty$, and
 for every $\rho\geq \exp(\sqrt{\e})$,
\begin{equation}\label{miracle}
\frac{s_\chi(V^s)}{s_\chi(y)}\in[\rho^{-1},\rho]\Longrightarrow \frac{s_\chi(\mathcal F_s(V^s))}{s_\chi(x)}\in
\left[\rho^{-1} e^{Q_\e(x)^{\b/4}},\rho e^{-Q_\e(x)^{\b/4}}\right].
\end{equation}
A similar statement holds for $u$--admissible manifolds in $\Psi_x^{p^u,p^s}$ and $\mathcal F_u$.
\end{lem}
\noindent
Note that the ratio bound in (\ref{miracle}) improves.
\begin{proof}
Suppose $V^s$ is represented by the function $G$, and $U^s:=\mathcal F_s[V^s]$ is represented by the function $F$. Let $p:=\Psi_x(0,F(0))$ and $q:=\Psi_y(0,G(0))$.

Suppose $s_\chi(V^s)<\infty$, then $s_\chi(q)<\infty$.
By Proposition \ref{Prop_Graph_Transform}(4) (in its version for $s$--manifolds), $f^{-1}(q)\in U^s$. Since $U^s$ is one-dimensional,  $df_{f^{-1}(q)}\un{e}^s(f^{-1}(q))=\pm\|df_{f^{-1}(q)}\un{e}^s(f^{-1}(q))\|_q\cdot \un{e}^s(q)$, and so
\begin{align*}
s_\chi(f^{-1}(q))^2&\equiv 2\left(1+\sum_{k=1}^\infty e^{2k\chi}\|df^{k-1}_q df_{f^{-1}(q)}\un{e}^s(f^{-1}(q))\|^2_{f^{k-1}(q)}
\right)\\
&=2+e^{2\chi}\|df_{f^{-1}(q)}\un{e}^s(f^{-1}(q))\|_q^2\cdot s_\chi(q)^2<\infty.
\end{align*}
Since $f^{-1}(q)\in U^s$, $s_\chi(U^s)\leq e^{\sqrt{\e}} s_\chi(f^{-1}(q))<\infty$.

\medskip
Next assume that $s_\chi(V^s)$ is finite, and
$$
\frac{s_\chi(V^s)}{s_\chi(y)}\in[\rho^{-1},\rho].
$$
where $\rho\geq\exp(\sqrt{\e})$.
Since $s_\chi(U^s)=s_\chi(p)$,
\begin{equation}\label{Decomposition}
\frac{s_\chi(U^s)}{s_\chi(x)}=\frac{s_\chi(p)}{s_\chi(f^{-1}(q))}\cdot \frac{s_\chi(f^{-1}(q))}{s_\chi(f^{-1}(y))}\cdot\frac{s_\chi(f^{-1}(y))}{s_\chi(x)}.
\end{equation}
The three terms are well--defined and finite, because (proceeding from right to left):
\begin{itemize}
\item $s_\chi(x), s_\chi(f^{-1}(y))$ are well--defined and finite, because $x,y\in\NUH_\chi(f)$;
\item $s_\chi(f^{-1}(q))$ is finite by the argument at the beginning of the proof;
\item $s_\chi(p)<\infty$, because $s_\chi(p)=S_\chi(U_s)<\infty$ (see above).
\end{itemize}

The first factor in (\ref{Decomposition}) belongs to $[e^{-Q_\e(x)^{\b/4}},e^{Q_\e(x)^{\b/4}}]$ by Proposition \ref{Prop_Staying_In_Windows}(3). The third factor in (\ref{Decomposition}) takes values in $[e^{-Q_\e(x)^{\b/4}},e^{Q_\e(x)^{\b/4}}]$ because $\Psi_x^{p^u,p^s}\to\Psi_y^{q^u,q^s}$, see Lemma \ref{Lemma_QQ}. To prove the proposition, it is enough to show that
\begin{equation}\label{WhatWeNeed}
\frac{1}{\rho}\exp[3Q_\e(x)^{\b/4}]<\frac{s_\chi(f^{-1}(q))}{s_\chi(f^{-1}(y))}<\rho\exp[-3Q_\e(x)^{\b/4}].
\end{equation}

We begin with some identities. We omit the tags of the Riemannian norm, to avoid heavy notation.
Since $df_{f^{-1}(y)}\un{e}^s(f^{-1}(y))=\pm\|df_{f^{-1}(y)}\un{e}^s(f^{-1}(y))\|\cdot \un{e}^s(y)$,
\begin{align}
s_\chi(f^{-1}(y))^2&=2\left(1+\sum_{k=1}^\infty e^{2k\chi}\|df^{k-1}_{y}df_{f^{-1}(y)}\un{e}^s(f^{-1}(y))\|^2\right)\notag\\
&=2+e^{2\chi}s_\chi(y)^2\|df_{f^{-1}(y)}\un{e}^s(f^{-1}(y))\|^2.\label{fior}
\end{align}
Similarly, $df_{f^{-1}(q)}\un{e}^s(f^{-1}(q))=\pm\|df_{f^{-1}(q)}\un{e}^s(f^{-1}(q))\|\cdot\un{e}^s(q)$, so
\begin{align*}
s_\chi(f^{-1}(q))^2&=2+e^{2\chi}s_\chi(q)^2\|df_{f^{-1}(q)}\un{e}^s(f^{-1}(q))\|^2\\
&\leq 2+\rho^2 e^{2\chi}s_\chi(y)^2\|df_{f^{-1}(q)}\un{e}^s(f^{-1}(q))\|^2\ (\because \frac{s_\chi(q)}{s_\chi(y)}=\frac{s_\chi(V^s)}{s_\chi(y)}\leq \rho)\\
&\leq\biggl(2+\rho^2 e^{2\chi}s_\chi(y)^2\|df_{f^{-1}(y)}\un{e}^s(f^{-1}(y))\|^2\biggr)\x\\
&\hspace{1cm} \x \exp\biggl(2\left|\log\|df_{f^{-1}(q)}\un{e}^s(f^{-1}(q))\|
-\log\|df_{f^{-1}(y)}\un{e}^s(f^{-1}(y))\|
\right|\biggr).
\end{align*}
We obtain the estimate
\begin{equation}\label{MiniDecomp}
\hspace{-0.1cm}
\begin{aligned}
\frac{s_\chi(f^{-1}(q))^2}{s_\chi(f^{-1}(y))^2}&\leq \left(\frac{2+\rho^2 e^{2\chi}s_\chi(y)^2\|df_{f^{-1}(y)}\un{e}^s(f^{-1}(y))\|^2}{2+e^{2\chi}s_\chi(y)^2\|df_{f^{-1}(y)}\un{e}^s(f^{-1}(y))\|^2}\right)\x\\
\x & \exp\biggl(2\left|\log\|df_{f^{-1}(q)}\un{e}^s(f^{-1}(q))\|
-\log\|df_{f^{-1}(y)}\un{e}^s(f^{-1}(y))\|
\right|\biggr).
\end{aligned}
\end{equation}
Call the first factor $\mathrm I$ and the second factor $\mathrm{II}$.

\medskip
\noindent
{\em Analysis of $\mathrm I$.\/}
\begin{align*}
\mathrm{I}&=\rho^2-\frac{2(\rho^2-1)}{2+e^{2\chi}s_\chi(y)^2
\|df_{f^{-1}(y)}\un{e}^s(f^{-1}(y))\|^2}\\
&=\rho^2-\frac{2(\rho^2-1)}{s_\chi(f^{-1}(y))^2},\textrm{ by (\ref{fior})}\\
&\leq \rho^2-\frac{e^{-2\e^{6/\b}}\cdot 2(\rho^2-1)}{s_\chi(x)^2}, \textrm{ because $\frac{s_\chi(f^{-1}(y))}{s_\chi(x)}=\exp[\pm\e^{6/\b}]$ by  Lemma \ref{Lemma_QQ}}\\
&\leq \rho^2\left(1-\frac{2e^{-2\e^{6/\b}}(1-\rho^{-2})}{\|C_\chi(x)^{-1}\|^2}\right),\textrm{ since $s_\chi(x)=\|C_\chi(x)^{-1}\un{e}^s(x)\|\leq \|C_\chi(x)^{-1}\|$}\\
&\leq \rho^2\left(1-\frac{\e^{1/2}}{\|C_\chi(x)^{-1}\|^2}\right)\textrm{ for all $\e$ small enough, because $\rho\geq e^{\sqrt{\e}}$.}
\end{align*}
By the definition of  $Q_\e(x)$,
$$
\frac{\e^{1/2}}{\|C_\chi(x)^{-1}\|^2}>Q_\e(x)^{\b/6}=Q_\e(x)^{-{\b}/{12}}Q_\e(x)^{\b/4}>\e^{-1/4}Q_\e(x)^{\b/4}.
$$
In particular, for all $\e$ small enough,
$
\frac{\e^{1/2}}{\|C_\chi(x)^{-1}\|^2}>7Q_\e(x)^{\b/4},
$
and by the inequality $1-x<e^{-x}$ for $0<x<1$,
$
\mathrm I\leq \rho^2\exp[-7Q_\e(x)^{\b/4}]$.

\medskip
\noindent
{\em Analysis of $\mathrm{II}$.\/} Since $f$ is a  $C^{1+\b}$--diffeomorphism and $\|\un{e}^s(\cdot)\|=1$,  there exists a constant $K_0$, which only depends on $f$, so that
$$
\mathrm{II}\leq \exp \bigg[K_0
d_M(f^{-1}(q),f^{-1}(y))^\b+ K_0 d_{TM}\big(\un{e}^s(f^{-1}(q)),\un{e}^s(f^{-1}(y))\big)\bigg],
$$
where $d_M$ and $d_{TM}$ are the Riemannian distance functions on $M$ and its tangent bundle. Since $f$ is a  $C^{1+\b}$ diffeomorphism and $\un{e}^s(\cdot)$ are unit vectors, there is another constant  $H_1$ (which only depends on $f$), such that
$$
\mathrm{II}\leq \exp \bigg[H_1
d_M(q,y)^\b+ H_1 d_{TM}\big(\un{e}^s(q),\un{e}^s(y)\big)^\b\bigg].
$$

We estimate $d(q,y)$. By definition $q=\Psi_y(0,G(0))$ and $y=\Psi_y(0,0)$. Since  Pesin charts have Lipschitz constant smaller than or equal to $2$,
$$
d(q,y)<2|G(0)|\leq 2\cdot 10^{-3}(q^u\wedge q^s)\leq 2\cdot 10^{-3}\cdot e^\e (p^u\wedge p^s)
$$ (see Lemma \ref{Lemma_Subordinated_Tempered}). In particular, $d(q,y)<Q_\e(x)$.

We estimate $d_{TM}(\un{e}^s(q),\un{e}^s(y))$. By the definition of $\Psi_y$,  $\un{e}^s(y)$ is the normalization of
$
(d\Psi_y)_{\un{0}}{1\choose 0}=(d\exp_y)_{\un{0}}\left[C_\chi(y){1\choose 0}\right]
$
, and $\un{e}^s(q)$ is the normalization of
$$
(d\Psi_y)_{(0,G(0))}{1\choose G'(0)}=(d\exp_y)_{C_\chi(y){0\choose G(0)}}\left[C_\chi(y){1\choose G'(0)}\right].
$$
It is not difficult to see using the admissibility of $V^s$ and Lemma \ref{Lemma_Subordinated_Tempered} that $|G(0)|<Q_\e(x)$ and $|G'(0)|<Q_\e(x)^{\b/3}$. Since $C_\chi(y)$ is a contraction, $p\mapsto \exp_p$ is smooth, and $d(q,y)<Q_\e(x)$, there exists a constant $G_0$ (which only depends on the smoothness of the exponential function) such that $d_{TM}(\un{e}^s(q),\un{e}^s(y))<G_0 Q_\e(x)^{\b/3}$.

We see that $\mathrm{II}\leq \exp[(H_1+H_1 G_0)Q_\e(x)^{\b/3}]$. It follows that for all $\e$ sufficiently small,  $\mathrm{II}\leq \exp [Q_\e(x)^{\b/4}]$.

\medskip
\noindent
{\em Summary.\/} Combining the estimates of $\mathrm{I}$ and $\mathrm{II}$, we find that
$$
\frac{s_\chi(f^{-1}(q))}{s_\chi(f^{-1}(y))}\leq \rho\exp[-3 Q_\e(x)^{\b/4}].
$$
The other half of (\ref{WhatWeNeed}) is proved in a similar way. First, one proves that
\begin{equation*}
\hspace{-0.1cm}
\begin{aligned}
\frac{s_\chi(f^{-1}(q))^2}{s_\chi(f^{-1}(y))^2}&\geq \left(\frac{2+\rho^{-2} e^{2\chi}s_\chi(y)^2\|df_{f^{-1}(y)}\un{e}^s(f^{-1}(y))\|^2}{2+e^{2\chi}s_\chi(y)^2\|df_{f^{-1}(y)}\un{e}^s(f^{-1}(y))\|^2}\right)\x\\
&\x \exp\biggl(-2\left|\log\|df_{f^{-1}(q)}\un{e}^s(f^{-1}(q))\|
-\log\|df_{f^{-1}(y)}\un{e}^s(f^{-1}(y))\|
\right|\biggr),
\end{aligned}
\end{equation*}
and then one analyzes the two terms as before.
\end{proof}

\subsection{Comparison of $s_\chi(x_i),u_\chi(x_i)$ to  $s_\chi(y_i),u_\chi(y_i)$.}

\begin{prop}\label{Prop_Comparison_U/S}
The following holds for all $\e$ small enough.
 For any two regular chains $(\Psi_{x_i}^{p^u_i,p^s_i})_{i\in\Z}$, $(\Psi_{y_i}^{q^u_i,q^s_i})_{i\in\Z}$, if $\pi[(\Psi_{x_i}^{p^u_i,p^s_i})_{i\in\Z}]=\pi[(\Psi_{y_i}^{q^u_i,q^s_i})_{i\in\Z}]$, then
$$
e^{-4\sqrt{\e}}\leq \frac{s_\chi(x_i)}{s_\chi(y_i)}\leq e^{4\sqrt{\e}}\textrm{ and }e^{-4\sqrt{\e}}\leq \frac{u_\chi(x_i)}{u_\chi(y_i)}\leq e^{4\sqrt{\e}}\textrm{ for all $i\in\Z$}.
$$
\end{prop}
\begin{proof}
Write $\un{v}:=(\Psi_{x_i}^{p^u_i,p^s_i})_{i\in\Z}$,  $\un{u}=(\Psi_{y_i}^{q^u_i,q^s_i})_{i\in\Z}$, and $p:=\pi(\un{v})=\pi(\un{u})$.

Let
$
V^s_k:=V^s[(v_i)_{i\geq k}]$, $V^u_k:=V^u[(v_i)_{i\leq k}]$,
$U^s_k:=V^s[(u_i)_{i\geq k}]$, $U^u_k:=V^u[(u_i)_{i\leq k}]$.
We claim that it is enough to prove that
\begin{equation}\label{seven.five}
\frac{s_\chi(V^s_k)}{s_\chi(x_k)},\frac{u_\chi(V^s_k)}{u_\chi(x_k)},
\frac{s_\chi(U^s_k)}{s_\chi(y_k)},\frac{u_\chi(U^s_k)}{u_\chi(y_k)}\in[e^{-\sqrt{\e}},e^{\sqrt{\e}}].
\end{equation}
Here is the reason.
The manifolds $V^s_k$ stay in windows and contain $f^k(p)$, therefore by Proposition \ref{Prop_Staying_In_Windows}(3) $s_\chi(V^s_k)/s_\chi(f^k(p))\in [e^{-\sqrt{\e}},e^{\sqrt{\e}}]$. The same argument applies to $U^s_k, V^u_k, U^u_k$, so
$
\frac{s_\chi(V^s_k)}{s_\chi(f^k(p))},\frac{u_\chi(V^u_k)}{u_\chi(f^k(p))}, \frac{s_\chi(U^s_k)}{s_\chi(f^k(p))},\frac{u_\chi(U^u_k)}{u_\chi(f^k(p))}\in [e^{-\sqrt{\e}},e^{\sqrt{\e}}].
$
Decomposing
$
\frac{s_\chi(x_k)}{s_\chi(y_k)}=\frac{s_\chi(x_k)}{s_\chi(V^s_k)}\cdot\frac{s_\chi(V^s_k)}{s_\chi(f^k(p))}\cdot
\frac{s_\chi(f^k(p))}{s_\chi(U^s_k)}\cdot\frac{s_\chi(U^s_k)}{u_\chi(y_k)}$, we see that (\ref{seven.five}) implies that  $s_\chi(x_k)/s_\chi(y_k)\in [e^{-4\sqrt{\e}},e^{4\sqrt{\e}}].
$
Similarly,  ${u_\chi(x_k)}/{u_\chi(y_k)}\in [e^{-4\sqrt{\e}},e^{4\sqrt{\e}}]$.

\medskip
We  show  that $s_\chi(V^s_0)/s_\chi(x_0)\in [e^{-\sqrt{\e}},e^{\sqrt{\e}}]$. The other parts of (\ref{seven.five}) are proved in the same way, and are left to the reader.

\medskip
We are assuming that  $\un{v}$ is regular, therefore  there exists a relevant double chart $v$ and a sequence
 $n_k\uparrow\infty$ s.t. $v_{n_k}=v$ for all $k$. Write $v=\Psi_x^{p^u,p^s}$.

\medskip
\noindent
{\em Claim 1.\/} There exists some $\rho\geq \exp(\sqrt{\e})$  which only depends on $v$ such that  ${s_\chi(V^s_{n_k})}/{s_\chi(x_{n_k})}\in [\rho^{-1},\rho]$ for all $k$.

\medskip
\noindent
{\em Proof.\/}
By convention $v$ is relevant (see \S\ref{SectionRelevant}). Choose a chain $\un{w}$ s.t. $w_0=v$ and $w:=\pi(\un{w})\in\NUH_\chi(f)$. Let $W^s:=V^s[(w_i)_{i\geq 0}]$. This manifold has a finite $s_\chi$--parameter, because $s_\chi(W^s)\leq e^{\sqrt{\e}}s_\chi(w)$ and   $w\in\NUH_\chi(f)$ so $s_\chi(w)<\infty$. Let
$$
\rho_0:=\max\left\{\frac{s_\chi(W^s)}{s_\chi(x)}, \frac{s_\chi(x)}{s_\chi(W^s)},\exp(\sqrt{\e})
\right\}.
$$

$W^s$ is an admissible manifold in $v_{n_k}=v$.
By Proposition \ref{Prop_V}, if we take $W^s$ at $v_{n_{k+\ell}}$ and apply to it the graph transform $\mathcal F_s$ $n_{k+\ell}-n_k$ times using the path $(v_{n_k},\ldots,v_{n_{k+\ell}})$, then the resulting manifold
$$
W^s_\ell:=\mathcal F_s^{n_{k+l}-n_k}[W^s]
$$
is an $s$--admissible manifold in $v_{n_k}$, which  converges to $V^s_{n_k}$. By Lemma
\ref{Main_Lemma_One},
\begin{equation}\label{Yofi!}
\frac{s_\chi(W^s_\ell)}{s_\chi(x)}\in [\rho^{-1}_0,\rho_0].
\end{equation}

The convergence of $W^s_\ell$ to $V^s_{n_k}$ means that if $W^s_\ell$ is represented in $v_{n_k}=\Psi_x^{p^u,p^s}$ by the function $F_\ell$, and $V^s_{n_k}$ is represented in $\Psi_x^{p^u,p^s}$  by $F$, then
$
\|F_{\ell}-F\|_\infty\xrightarrow[\ell\to\infty]{}0.
$
In fact, since $\sup\|F_\ell'\|_{\b/3}<\infty$,  we have the stronger statement that
$$
\|F_\ell-F\|_\infty+\|F_\ell'-F'\|_\infty\xrightarrow[\ell\to\infty]{}0,
$$
see part 2 of the proof of Proposition \ref{Prop_V}.
Therefore, if $\xi:=\Psi_x(0,F(0))$ and $\xi_\ell=\Psi_x(0,F_\ell(0))$, then
$
\xi_\ell\xrightarrow[\ell\to\infty]{}\xi\textrm{ and }\un{e}^s(\xi_\ell)\xrightarrow[\ell\to\infty]{}\un{e}^s(\xi).
$

Fix some $N$ large and $\d>0$ small. Since  $df$ is continuous, there exists $\ell$ so large that
\begin{align*}
\sqrt{2}\left(\sum_{j=0}^N e^{2j\chi}\|df^j_{\xi}\un{e}^s(f^j(\xi))\|^2_{f^j(\xi)}\right)^{\frac{1}{2}}&\leq e^\d\cdot \sqrt{2}\left(\sum_{j=0}^N e^{2j\chi}\|df^j_{\xi_\ell}\un{e}^s(f^j(\xi_\ell))\|^2_{f^j(\xi_\ell)}\right)^{\frac{1}{2}}.
\end{align*}
The expression on the right is smaller than $e^\d s_\chi(W^s_\ell)$, and therefore by (\ref{Yofi!}), smaller than $e^\d\rho_0 s_\chi(x)$. Since this is true for all $N$ and $\d$,
$
s_\chi(V^s_{n_k})\leq \rho_0\cdot s_\chi(x)$.

Recalling that $x_{n_k}=x$ and that $s_\chi(V^s_{n_k})\geq \sqrt{2}$, we see that $s_\chi(V^s_{n_k})/s_\chi(x_{n_k})\in [\sqrt{2}/s_\chi(x),\rho_0]$. The claim follows with $\rho=\rho_0\cdot s_\chi(x)$.

\medskip
\noindent
{\em Claim 2.\/} $s_\chi(V^s_0)/s_\chi(x_0)\in [\exp(-\sqrt{\e}),\exp(\sqrt{\e})]$.

\medskip
\noindent
{\em Proof.\/} Fix $k$ large.
By claim 1,
$$
\frac{s_\chi(V^s_{n_k})}{s_\chi(x_{n_k})}\in [\rho^{-1},\rho].
$$
By Proposition \ref{Prop_V} (3), $\mathcal F_s(V^s_{n_k})=V^s_{n_k-1}$, and by  Lemma \ref{Main_Lemma_One}, the bounds for $
\frac{s_\chi(V^s_{n_k})}{s_\chi(x_{n_k})}$ improve. We ignore these improvements and write
$
\frac{s_\chi(V^s_{n_k-1})}{s_\chi(x_{n_k-1})}\in [\rho^{-1},\rho].
$
Another application of $\mathcal F_s$ gives $
\frac{s_\chi(V^s_{n_k-2})}{s_\chi(x_{n_k-2})}\in [\rho^{-1},\rho]
$.
Continuing this way, we eventually reach  the index $n_{k-1}+1$ and the bound
$$
\frac{s_\chi(V^s_{n_{k-1}+1})}{s_\chi(x_{n_{k-1}+1})}\in [\rho^{-1},\rho]
$$

Since $x_{n_k}=x$, the next application of $\mathcal F_s$ improves the ratio bound by at least $\exp[Q_\e(x)^{\b/4}]$:
$$
\frac{s_\chi(V^s_{n_{k-1}})}{s_\chi(x_{n_{k-1}})}\in [\rho^{-1}e^{Q_\e(x)^{\b/4}},\rho e^{-Q_\e(x)^{\b/4}}].
$$

We repeat the procedure by applying $\mathcal F_s$  $n_{k-1}-n_{k-2}+1$ times, whilst ignoring the potential improvements of the error bounds, and then applying $\mathcal F_s$ once more and arriving at
$$
\frac{s_\chi(V^s_{n_{k-2}})}{s_\chi(x_{n_{k-2}})}\in [\rho^{-1}e^{2Q_\e(x)^{\b/4}},\rho e^{-2Q_\e(x)^{\b/4}}].
$$

We are free to choose $k$ as large as we want. If we make it so large that  $\exp[k Q_\e(x)^{\b/4}]>\rho\exp(-\sqrt{\e})$ , then eventually we will reach a time  $n_{k_0}$ when the ratio bound is smaller than or equal to $\exp(\sqrt{\e})$:
$$
\frac{s_\chi(V^s_{n_{k_0}})}{s_\chi(x_{n_{k_0}})}\in [\exp(-\sqrt{\e}),\exp(\sqrt{\e})].
$$

This is the threshold the applicability of Lemma \ref{Main_Lemma_One}. Henceforth we cannot claim that the ratio bound improves. On the other hand it is guaranteed  that the ratio bound does not deteriorate. Therefore,  after additional  $n_{k_0}$ iterations, we obtain $
\frac{s_\chi(V^s_{0})}{s_\chi(x_{0})}\in [\exp(-\sqrt{\e}),\exp(\sqrt{\e})]$ as desired.
\end{proof}

\section{Window parameters}\label{SectionCompP/Q}
\subsection{$\e$--maximality}
Let $\un{v}=(\Psi_{x_i}^{p^u_i,p^s_i})_{i\in\Z}, \un{u}=(\Psi_{y_i}^{q^u_i,q^s_i})_{i\in\Z}$ be two regular chains such that $\pi[\un{v}]=\pi[\un{u}]$. We compare $p^u_i$ to $q^u_i$, and $p^s_i$ to $q^s_i$.
The idea is to use regularity to see that  the $q$--parameters of  $V^u[(v_i)_{i\leq 0}]$ and $V^s[(v_i)_{i\geq 0}]$ are ``almost maximal"  in a certain sense that we  describe below.

%

But first, some notation and terminology: (a) a positive or negative chain is called {\em regular}, if it can be completed to a regular chain (equiv. every coordinate is relevant, and some double chart appears infinitely many times); (b) if $v$ is a double chart, then  $p^{u}(v)$ and $p^s(v)$ means the $p^u$ and $p^s$ in $v=\Psi_x^{p^u,p^s}$.

\begin{defi}
A negative chain $(v_i)_{i\leq 0}$ is called {\em $\e$--maximal} if it is regular, and
$$
p^u(v_0)\geq e^{-\sqrt[3]{\e}}p^u(u_0)
$$
for every regular  chain $(u_i)_{i\in\Z}$ for which there is a positive regular chain $(v_i)_{i\geq 0}$ s.t.  $\pi[(v_i)_{i\in\Z}]=\pi[(u_i)_{i\in\Z}]$.
\end{defi}
\begin{defi}
A positive chain $(v_i)_{i\geq 0}$  is called {\em $\e$--maximal} if it is regular, and
$$
p^s(v_0)\geq e^{-\sqrt[3]{\e}}p^s(u_0)
$$
for every regular chain $(u_i)_{i\in \Z}$ for which there is a negative regular chain $(v_i)_{i\leq 0}$ s.t. $\pi[(v_i)_{i\in\Z}]=\pi[(u_i)_{i\in\Z}]$.
\end{defi}

\begin{prop}\label{Prop_Maximality}
The following holds for all $\e$ small enough: for every regular chain $(v_i)_{i\in\Z}$, $(v_i)_{i\leq 0}$ and $(v_i)_{i\geq 0}$ are $\e$--maximal.
\end{prop}
\begin{proof}
The proof is made of several steps.

\medskip
\noindent
{\em Step 1.\/} The following holds for all $\e$ small enough: Let $\un{u}$ and $\un{v}$  be two regular chains s.t. $\pi[\un{u}]=\pi[\un{v}]$. If $u_0=\Psi_x^{p^u,p^s}$ and $v_0=\Psi_y^{q^u,q^s}$, then $Q_\e(x)/Q_\e(y)\in [e^{-\sqrt[3]{\e}},e^{\sqrt[3]{\e}}]$.

\medskip
\noindent
{\em Proof.\/} Propositions \ref{Prop_Alpha_Comparison} and \ref{Prop_Comparison_U/S} say that
$
\frac{\sin\a(x)}{\sin\a(y)}\in [e^{-\sqrt{\e}},e^{\sqrt{\e}}]$, $\frac{s_\chi(x)}{s_\chi(y)}\in [e^{-4\sqrt{\e}},e^{4\sqrt{\e}}]$, and $\frac{u_\chi(x)}{u_\chi(y)}\in [e^{-4\sqrt{\e}},e^{4\sqrt{\e}}]$.
 By Lemma \ref{Lemma_C_norm}
$
\frac{\|C_\chi(x)^{-1}\|_{Fr}}{\|C_\chi(y)^{-1}\|_{Fr}}\in \bigl[\exp(-5\sqrt{\e}), \exp(5\sqrt{\e})\bigr],
$
whence  $Q_\e(x)/Q_\e(y)\in \bigl[
\exp(-\tfrac{60}{\b}\sqrt{\e}-\tfrac{1}{3}\e),\exp(\tfrac{60}{\b}\sqrt{\e}+\tfrac{1}{3}\e)
\bigr]$. If $\e$ is small enough,  then $Q_\e(x)/Q_\e(y)\in [\exp(-\sqrt[3]{\e}),\exp(\sqrt[3]{\e})]$.

\medskip
\noindent
{\em Step 2.\/}
The following holds for all $\e$ small enough: Every regular negative chain $(v_i)_{i\leq 0}$  s.t. $v_0=\Psi_x^{p^u,p^s}$ where $p^u=Q_\e(x)$ is $\e$--maximal, and every regular positive chain
$(v_i)_{i\geq 0}$  s.t. $v_0=\Psi_x^{p^u,p^s}$ where $p^s=Q_\e(x)$ is $\e$--maximal.

\medskip
\noindent
{\em Proof.\/}
Suppose $(v_i)_{i\leq 0}$ is regular, and $v_0=\Psi_x^{p^u,p^s}$ where $p^u=Q_\e(x)$. We show that $(v_i)_{i\leq 0}$ is $\e$--maximal.

Suppose $(v_i)_{i\in\Z}$ is a regular extension of $(v_i)_{i\leq 0}$ and let $(u_i)_{i\in\Z}$ be some regular chain s.t.   $\pi[(u_i)_{i\in\Z}]=\pi[(v_i)_{i\in\Z}]$. Write  $u_0=\Psi_y^{q^u,q^s}$. We have to show that  $p^u\geq e^{-\sqrt[3]{\e}}q^u$.
Indeed, by step 1, $p^u=Q_\e(x)\geq e^{-\sqrt[3]{\e}}Q_\e(y)\geq e^{-\sqrt[3]{\e}}q^u$.

The proof of the second half of step 2 is similar, and we therefore omit it.

\medskip
\noindent
{\em Step 3.\/} Let  $(v_i)_{i\leq 0}$ be a regular negative chain and suppose $v_0\to v_1$. If $(v_i)_{i\leq 0}$ is $\e$--maximal, then $(v_i)_{i\leq 1}$ is $\e$--maximal. Let $(v_i)_{i\geq 0}$ be a regular positive chain, and suppose $v_{-1}\to v_0$. If $(v_i)_{i\geq 0}$ is $\e$--maximal, then $(v_i)_{i\geq {-1}}$ is $\e$--maximal.

\medskip
\noindent
{\em Proof.\/}
 Let $(v_i)_{i\leq 0}$ be an $\e$--maximal regular positive chain, and suppose $v_0\to v_1$. We prove that $(v_i)_{i\leq 1}$ is $\e$--maximal.

Suppose $(u_i)_{i\in\Z}$, $(v_i)_{i\geq 1}$ are regular and  there is an extension of $(v_i)_{i\geq 1}$ to a regular chain $(v_i)_{i\in\Z}$ s.t.  $\pi[(v_{i+1})_{i\in\Z}]=\pi[(u_{i+1})_{i\in\Z}]$. We write $v_i=\Psi_{x_i}^{p^u_i,p^s_i}$ , $u_i=\Psi_{y_i}^{q^u_i,q^s_i}$, and  show that $p^u_1\geq e^{-\sqrt[3]{\e}}q^u_1$.

Since $\pi[(v_{i+1})_{i\in\Z}]=\pi[(u_{i+1})_{i\in\Z}]$ and $\pi\circ \s=f\circ\pi$,
$\pi[(v_{i})_{i\in\Z}]=\pi[(u_{i})_{i\in\Z}]$. Therefore, since   $(v_i)_{i\leq 0}$ is $\e$--maximal,
$
p^u_0\geq e^{-\sqrt[3]{\e}}q^u_0.
$
Also, by step 1, $Q_\e(x_1)\geq e^{-\sqrt[3]{\e}}Q_\e(y_1)$.
It follows that
\begin{align*}
p^u_1&=\min\{e^\e p^u_0,Q_\e(x_1)\} \ \ \ (\because v_0\to v_1)\\
&\geq \min\{e^{\e}\cdot e^{-\sqrt[3]{\e}}q_0^u, e^{-\sqrt[3]{\e}}Q_\e(y_1)\}\\
&=e^{-\sqrt[3]{\e}}\min\{e^{\e}q_0^u, Q_\e(y_1)\}=e^{-\sqrt[3]{\e}}q^u_1\ \ \ (\because u_0\to u_1).
\end{align*}
This proves the part of step 3  dealing with negative chains. The case of positive chains is similar, and we leave it  to the reader.

\medskip
\noindent
{\em Step 4.\/} Proof of the proposition.

\medskip
\noindent
Suppose $(v_i)_{i\in\Z}$ is a regular chain, and write $v_i=\Psi_{x_i}^{p^u_i,p^s_i}$. Since $(v_i)_{i\in\Z}$ is a chain, $\{(p^u_i,p^s_i)\}_{i\in\Z}$ is $\e$--subordinated to $\{Q_\e(x_i)\}_{i\in\Z}$.  Since $(v_i)_{i\in\Z}$ is regular, $\limsup\limits_{i\to\pm\infty} (p^u_i\wedge p^s_i)>0$, therefore by Lemma \ref{Lemma_Subordinated_Sharp}, $p^u_n=Q_\e(x_n)$ for some $n<0$ and $p^s_\ell=Q_\e(x_\ell)$ for some $\ell>0$.

By step 2, $(v_i)_{i\leq n}$ is an $\e$--maximal negative chain, and $(v_i)_{i\geq \ell}$ is an $\e$--maximal positive chain.

By step 3, $(v_i)_{i\leq 0}$ is an $\e$--maximal negative chain, and $(v_i)_{i\geq 0}$ is an $\e$--maximal positive chain.
\end{proof}

\subsection{Comparison of $p^{u/s}_i$ to $q^{u/s}_i$} We can now easily compare the window parameters of all regular chains with the same $\pi$ image.
\begin{prop}\label{Prop_Comparison_P}
 Let $(\Psi_{x_i}^{p^u_i,p^s_i})_{i\in\Z}$ and $(\Psi_{y_i}^{q^u_i,q^s_i})_{i\in\Z}$ be two regular chains such that  $\pi[(\Psi_{x_i}^{p^u_i,p^s_i})_{i\in\Z}]=\pi[(\Psi_{y_i}^{q^u_i,q^s_i})_{i\in\Z}]$, then
$
{p^u_i}/{q^u_i}, {p^s_i}/{q^s_i}\in [\exp(-\sqrt[3]{\e}),\exp(\sqrt[3]{\e})]$   for all $i\in\Z$.
\end{prop}
\begin{proof}
By Proposition \ref{Prop_Maximality} $(\Psi_{x_i}^{p^u_i,p^s_i})_{i\leq 0}$ is $\e$--maximal, so
$p^u_0\geq e^{-\sqrt[3]{\e}}q^u_0$. $(\Psi_{y_i}^{q^u_i,q^s_i})_{i\leq 0}$ is also $\e$--maximal, so
$q^u_0\geq e^{-\sqrt[3]{\e}}p^u_0$. It follows that $p^u_0/q^u_0\in [e^{-\sqrt[3]{\e}},e^{\sqrt[3]{\e}}]$.
Similarly, $p^s_0/q^s_0\in [e^{-\sqrt[3]{\e}},e^{\sqrt[3]{\e}}]$.

Working with the shifted sequences $(\Psi_{x_{i+k}}^{p^u_{i+k},p^s_{i+k}})_{i\in\Z}$ and $(\Psi_{y_{i+k}}^{q^u_{i+k},q^s_{i+k}})_{i\in\Z}$, we obtain $p^s_k/q^s_k, p^u_k/q^u_k\in [e^{-\sqrt[3]{\e}},e^{\sqrt[3]{\e}}]$.
\end{proof}

\section{Proof of Theorem \ref{Theorem_Pi_Almost_1-1}}\label{SectionProof}
Parts (1) and (3) of the theorem are handled by Propositions \ref{Prop_x_Parameter} and \ref{Prop_Comparison_P}, so we focus on part (2).

Suppose $\pi[(\Psi_{x_i}^{p^u_i,p^s_i})_{i\in\Z}]=\pi[(\Psi_{y_i}^{q^u_i,q^s_i})_{i\in\Z}]$ where $(\Psi_{x_i}^{p^u_i,p^s_i})_{i\in\Z})$ and $(\Psi_{y_i}^{q^u_i,q^s_i})_{i\in\Z}$ are regular chains.
 We compare $\Psi_{x_i}$ and $\Psi_{y_i}$.
Write, as in  \S\ref{Section_Strategy},  $\Psi_{x_i}=\exp_{x_i}\circ \vartheta_{x_i}\circ C_{x_i}$ and $\Psi_{y_i}=\exp_{y_i}\circ \vartheta_{y_i}\circ C_{y_i}$. We also let $p_i:=p^u_i\wedge p^s_i$ and $q_i:=q^u_i\wedge q^s_i$.

\medskip
\noindent
{\em Claim 1.\/} $C_{y_i}^{-1}C_{x_i}=(-1)^{\s_i}\id+E$ where $\s_i\in\{0,1\}$ and $E$ is a matrix all of whose entries have absolute value less than $7\sqrt{\e}$.

\medskip
\noindent
\noindent{\em Proof.\/}
By (\ref{C_z}) and Proposition \ref{Prop_R_Comparison},
\begin{multline*}
C_{y_i}^{-1}C_{x_i}=\left(\begin{array}{cc}
s_\chi(y_i) & -\frac{s_\chi(y_i)}{\tan\a(y_i)}\\
0 & \frac{u_\chi(y_i)}{\sin\a(y_i)}
\end{array}
\right)R_{y_i}^{-1} R_{x_i}\left(
\begin{array}{cc}
s_\chi(x_i)^{-1} & u_\chi(x_i)^{-1}\cos\a(x_i)\\
0 & u_\chi(x_i)^{-1}\sin\a(x_i)
\end{array}
\right)\\
=\left(\begin{array}{cc}
s_\chi(y_i) & -\frac{s_\chi(y_i)}{\tan\a(y_i)}\\
0 & \frac{u_\chi(y_i)}{\sin\a(y_i)}
\end{array}
\right)\left[(-1)^{\s_i}\id+E'\right]\left(
\begin{array}{cc}
s_\chi(x_i)^{-1} & u_\chi(x_i)^{-1}\cos\a(x_i)\\
0 & u_\chi(x_i)^{-1}\sin\a(x_i)
\end{array}
\right),
\end{multline*}
where $\s_i\in\{0,1\}$ and   $E'=(\e_{ij})_{2\x 2}$ and $|\e_{ij}|<p_i^{\b/5}+q_i^{\b/5}<\sqrt{\e}$.

We call the contribution of $(-1)^{\s_i}\id$ the  ``main term", and the  contribution of $E'$, the  ``error term".

\medskip
\noindent
{\em Main term\/}: This equals $(-1)^{\s_i}\left(\begin{array}{cc}
\frac{s_\chi(y_i)}{s_\chi(x_i)} & \frac{s_\chi(y_i) \sin[\a(y_i)-\a(x_i)]}{u_\chi(x_i)\sin\a(y_i)}\\
0 & \frac{u_\chi(y_i)}{u_\chi(x_i)}\frac{\sin\a(x_i)}{\sin\a(y_i)}
\end{array}\right)$.

Proposition \ref{Prop_Comparison_U/S} says that  $\frac{s_\chi(y_i)}{s_\chi(x_i)}$ and $\frac{u_\chi(y_i)}{u_\chi(x_i)}$ belong to $[\exp(-4\sqrt{\e}),\exp(4\sqrt{\e})]$, and Proposition \ref{Prop_Alpha_Comparison} says that $\frac{\sin\a(x_i)}{\sin\a(y_i)}\in [\exp(-\sqrt{\e}),\exp\sqrt{\e}]$. It follows that the $(1,1)$ and (2,2) terms of the main term are, up to a sign $(-1)^{\s_i}$, in $[\exp(-5\sqrt{\e}),\exp(5\sqrt{\e})]$.

We bound the $(1,2)$ term: Since  $u_\chi(y_i)\geq \sqrt{2}>1$ and $\frac{s_\chi(y_i)}{|\sin\a(y_i)|}<\|C_\chi(y_i)^{-1}\|_{Fr}$ (Lemma (\ref{Lemma_C_norm})),
\begin{align*}
\left|\frac{s_\chi(y_i) \sin[\a(y_i)-\a(x_i)]}{u_\chi(x_i)\sin\a(y_i)}\right|&\leq
\|C_\chi(y_i)^{-1}\|_{Fr}\cdot|\sin(\a(y_i)-\a(x_i))|\\
&\hspace{-3cm}\leq \|C_\chi(y_i)^{-1}\|_{Fr}\cdot\bigl(|\sin\a(y_i)-\sin\a(x_i)|+|\cos\a(y_i)-\cos\a(x_i)|\bigr).
\end{align*}
By Lemma \ref{Lemma_Strong_Alpha_Comp}, if $\e$ is small enough,
\begin{align*}
\left|\frac{s_\chi(y_i) \sin[\a(y_i)-\a(x_i)]}{u_\chi(x_i)\sin\a(y_i)}\right|&\leq \|C_\chi(y_i)^{-1}\|_{Fr}\cdot 6(p_i^{\b/4}+q_i^{\b/4}).
\end{align*}
By Proposition \ref{Prop_Comparison_P}, $p_i\leq e^{\sqrt[3]{\e}}q_i$, therefore
$$
p_i^{\b/4}+q_i^{\b/4}<(e^{\sqrt[3]{\e}\b/4}+1) q_i^{\b/4}<2q_i^{\b/4}<2Q_\e(y_i)^{\b/4}<2\e^{3/4}\|C_\chi(y_i)^{-1}\|^{-3}_{Fr}.
$$
Since $\|C_\chi(\cdot)^{-1}\|_{Fr}>1$,
$
\left|\frac{s_\chi(y_i) \sin[\a(y_i)-\a(x_i)]}{u_\chi(x_i)\sin\a(y_i)}\right|<\sqrt{\e}
$, for all $\e$ small enough.
We see that the main term equals $(-1)^{\s_i}\id+(m_{ij})_{2\x 2}$ where   $|m_{ij}|<6\sqrt{\e}$.

\medskip
\noindent
{\em Error term\/}: This is
$$
\left(\begin{array}{cc}
s_\chi(y_i) & -\frac{s_\chi(y_i)}{\tan\a(y_i)}\\
0 & \frac{u_\chi(y_i)}{\sin\a(y_i)}
\end{array}
\right)\left(
\begin{array}{cc}
\e_{11} & \e_{12}\\
\e_{21} & \e_{22}
\end{array}
\right)\left(
\begin{array}{cc}
s_\chi(x_i)^{-1} & u_\chi(x_i)^{-1}\cos\a(x_i)\\
0 & u_\chi(x_i)^{-1}\sin\a(x_i)
\end{array}
\right).
$$

Every entry of the product matrix is the sum of four products, each consisting of three terms, one  for each matrix.

The term  from the left matrix is bounded by $\|C_\chi(y_i)^{-1}\|_{Fr}$ (Lemma \ref{Lemma_C_norm}). The term  from the middle matrix is bounded by
$$
p_i^{\b/5}+q_i^{\b/5}<q_i^{\b/5}(1+e^{\sqrt[3]{\e}\b/5})<2Q_\e(y_i)^{\b/5}.
$$
The term from the right matrix is bounded by one.
The product of these terms is  bounded by
$
4 \|C_\chi(y_i)^{-1}\|_{Fr}\cdot 2Q_\e(y_i)^{\b/5}\cdot 1
$.
By the definition of $Q_\e(y_i)$, this is less than $8\e^{3/5}<\sqrt{\e}$.

\medskip
Combining the two estimates we see that every entry of $C_{y_i}^{-1}C_{x_i}-(-1)^{\s_i}\id$ is less than $7\sqrt{\e}$ in absolute value.

\medskip
\noindent
{\em Claim 2.\/} $\Psi_{y_i}^{-1}\circ \Psi_{x_i}$ is well defined on $R_\e(\un{0})$.

\medskip
\noindent
{\em Proof.\/} We use the constants $L_1,\ldots,L_4$ introduced in the proof of Proposition \ref{Prop_Overlap_Meaning}, and the ball notation of \S\ref{SectionPC}. We assume that $\e$ satisfies (\ref{irene}).

Suppose $\un{v}\in R_\e(\un{0})$. By Proposition \ref{Prop_x_Parameter}, $d(x_i,y_i)<25^{-1}(p_i+q_i)$, and  by Proposition \ref{Prop_Comparison_P}, $p_i\leq e^{\sqrt[3]{\e}}q_i$, so $d(x_i,y_i)<q_i$. By the definition of $L_1$ (page \pageref{L_Page}),
$$
d\bigl((\exp_{x_i}\circ\vartheta_{x_i})(C_{x_i}\un{v}), (\exp_{y_i}\circ\vartheta_{y_i})(C_{x_i}\un{v})\bigr)\leq
L_1d({x_i},{y_i})<L_1 q_i.
$$
Therefore, $\Psi_{x_i}(\un{v})\in B:=B_{L_1q_i}(\exp_{y_i}\circ\vartheta_{y_i}(C_{x_i}\un{v}))$.

As in the proof of Proposition \ref{Prop_Overlap_Meaning},
$\exp_{y_i}^{-1}$ is well defined on $B$, and has Lipschitz constant at most $L_3$ there, so
$$
\exp_{y_i}^{-1}(B)\subset B_{L_1 L_3 q_i}^{y_i}(\vartheta_{y_i}(C_{x_i}\un{v})).
$$
It follows that $\Psi_{x_i}(\un{v})\in \exp_{y_i}[\exp_{y_i}^{-1}(B)]\subset \exp_{y_i}[B_{L_1 L_3 q_i}^{y_i}(\vartheta_{y_i}(C_{x_i}\un{v}))]\equiv \Psi_{y_i}[E]$, where
$
E:=C_{\chi}(y_i)^{-1}[B_{L_1 L_3 q_i}^{y_i}(\vartheta_{y_i}(C_{x_i}\un{v}))]\subset B_{L_1 L_3 \|C_{y_i}^{-1}\| q_i}(C_{y_i}^{-1}C_{x_i}\un{v})$.

We now use the inequalities  $q_i\leq Q_\e(y_i)<\e^{3/\b}\|C_{\chi}(y_i)^{-1}\|^{-1}$ and  (claim 1)
$$
\|C_{y_i}^{-1}C_{x_i}-(-1)^{\s_i}\id\|\leq \|C_{y_i}^{-1}C_{x_i}-(-1)^{\s_i}\id\|_{Fr}<14\sqrt{\e}.
$$
These give $E\subset  B_{L_1 L_3 \e^{3/\b}+14\sqrt{\e}\|\un{v}\|}((-1)^{\s_i}\un{v})\subset B_{L_1 L_3 \e^{3/\b}+14\sqrt{\e}\|\un{v}\|+\|\un{v}\|}(\un{0})$.
Since $\un{v}\in R_\e(\un{0})$, for all $\e$ small enough
$$
L_1 L_3 \e^{3/\b}+14\sqrt{\e}\|\un{v}\|+\|\un{v}\|<(L_1 L_2 \e^2+14\sqrt{\e}+1)\sqrt{2}\e<2\e<r(M),
$$
 where $r(M)$ is given in (\ref{r_M}). It follows that $E\subset B_{r(M)}(\un{0})$.

 We just showed that for every $\un{v}\in R_\e(\un{0})$, $\Psi_{x_i}(\un{v})\in \Psi_{y_i}[B_{r(M)}(\un{0})]$.   In other words,    $\Psi_{x_i}[R_\e(\un{0})]\subset \Psi_{y_i}[B_{r(M)}(\un{0})]$. By the definition of $r(M)$,   $\Psi_{y_i}:B_{r(M)}(\un{0})\to M$ is a diffeomorphism onto its image. It follows that  $\Psi_{y_i}^{-1}\circ\Psi_{x_i}$ is well defined and smooth on $R_\e(\un{0})$.

 \medskip
 \noindent
 {\em Claim 3.\/} $\Psi_{y_i}^{-1}\circ\Psi_{x_i}(\un{v})=(-1)^{\s_i}\un{v}+\un{c}_i+\Delta_i(\un{v})$ where
 $\s_i\in\{0,1\}$, $\un{c}_i$ is a constant vector s.t. $\|\un{c}_i\|<10^{-1} q_i$, and $\Delta_i(\cdot)$ is a vector field s.t.  $\Delta_i(\un{0})=\un{0}$ and $\|(d\Delta_i)_{\un{v}}\|<\sqrt[3]{\e}$ on $R_\e(\un{0})$.

\medskip
\noindent
{\em Proof.\/} Choose $\s_i$ as in claim 1. One can always put $\Psi_{y_i}^{-1}\circ\Psi_{x_i}$ in the form $$\Psi_{y_i}^{-1}\circ\Psi_{x_i}(\un{v})=(-1)^{\s_i}\un{v}+\un{c}_i+\Delta_i(\un{v})$$ where
$\un{c}_i:=(\Psi_{y_i}^{-1}\circ\Psi_{x_i})(\un{0})$ and
$\Delta_i(\un{v}):=(\Psi_{y_i}^{-1}\circ\Psi_{x_i})(\un{v})-(\Psi_{y_i}^{-1}\circ\Psi_{x_i})(\un{0})-(-1)^{\s_i}\un{v}$.
\begin{align*}
\Delta_i(\un{v})&=[C_{y_i}^{-1}\vartheta_{y_i}^{-1}\exp_{y_i}^{-1}\exp_{x_i}\vartheta_{x_i}C_{x_i}](\un{v})
-\un{c}_i-(-1)^{\s_i}\un{v}\\
&=C_{y_i}^{-1}(\vartheta_{y_i}^{-1}\exp_{y_i}^{-1}\exp_{x_i}\vartheta_{x_i}-\id)C_{x_i}\un{v}+
(C_{y_i}^{-1}C_{x_i}-(-1)^{\s_i}\id)\un{v}
-\un{c}_i\\
&=C_{y_i}^{-1}(\vartheta_{y_i}^{-1}\exp_{y_i}^{-1}-\vartheta_{x_i}^{-1}\exp_{x_i}^{-1})(\Psi_{x_i}(\un{v}))+
(C_{y_i}^{-1}C_{x_i}-(-1)^{\s_i}\id)\un{v}
-\un{c}_i.
\end{align*}
It is clear that $\Delta_i(\un{0})=\un{0}$, and that for all $\un{v}\in R_\e(\un{0})$
\begin{align*}
\|(d\Delta_i)_{\un{v}}\|&\leq \|C_{y_i}^{-1}\|\cdot\|d(\vartheta_{y_i}^{-1}\exp_{y_i}^{-1})_{\Psi_{x_i}(\un{v})}
-d(\vartheta_{x_i}^{-1}\exp_{x_i}^{-1})_{\Psi_{x_i}(\un{v})}\|\|(d\Psi_{x_i})_{\un{v}}\|\\
&\hspace{7cm}+\|C_{y_i}^{-1}C_{x_i}-(-1)^{\s_i}\id\|\\
&\leq 2\|C_{y_i}^{-1}\|\cdot \|d(\vartheta_{y_i}^{-1}\exp_{y_i}^{-1})_{\Psi_{x_i}(\un{v})}
-d(\vartheta_{x_i}^{-1}\exp_{x_i}^{-1})_{\Psi_{x_i}(\un{v})}\|+14\sqrt{\e}\\
&\leq 2\|C_{y_i}^{-1}\|\cdot L_2 d(x_i,y_i)+14\sqrt{\e},
\end{align*}
where $L_2$ is a common Lipschitz constant for the maps $x\mapsto \vartheta_x^{-1}\exp_x^{-1}$ from $D$ to $C^2(D,\R^2)$ $(D\in\mathfs D)$. As we saw above,  $d(x_i,y_i)<q_i<\e^{3/\b}\|C_{y_i}^{-1}\|^{-1}$, whence
$$
\|(d\Delta_i)_{\un{v}}\|\leq 2L_2\e^{3/\b}+14\sqrt{\e}.
$$
This is smaller than $\sqrt[3]{\e}$ for all $\e$ small enough.

Finally we estimate $\un{c}_i$. Let $z:=f^i(\pi[(\Psi_{x_i}^{p^u_i,p^s_i})_{i\in\Z}])=f^i(\pi[(\Psi_{y_i}^{q^u_i,q^s_i})_{i\in\Z}])$. This is the intersection of a $u$--admissible manifold and an $s$--admissible manifold in $\Psi_{x_i}^{p^u_i,p^s_i}$, therefore by Proposition \ref{Prop_Intersection},
$
f^i(z)=\Psi_{x_i}^{p^u_i,p^s_i}(\un{\zeta}),\textrm{ for some }\un{\zeta}\in R_{10^{-2}p_i}(\un{0}).
$
Similarly,
$
z=\Psi_{y_i}^{q^u_i,q^s_i}(\un{\eta}),\textrm{ for some }\un{\eta}\in R_{10^{-2}q_i}(\un{0}).
$
It follows that
$$
\un{\eta}=(\Psi_{y_i}^{-1}\circ\Psi_{x_i})(\un{\zeta})=(-1)^{\s_i}\un{\zeta}+\un{c}_i+\Delta_i(\un{\zeta}),
$$
and consequently $\|\un{c}_i\|\leq \|\un{\eta}\|+\|\un{\zeta}\|+\|\Delta_i(\un{\zeta})\|$.

Now $\|\un{\zeta}\|<10^{-2}\sqrt{2}p_i<10^{-2}\sqrt{2} e^{\sqrt[3]{\e}} q_i$, $\un{\eta}<10^{-2}\sqrt{2}q_i$, and by the bound on $\|d\Delta_i\|$, $\|\Delta_i(\un{\zeta})\|\leq\sqrt[3]{\e}\|\un{\zeta}\|$. It follows that $\|\un{c}_i\|<10^{-1}q_i$.
\hfill$\Box$

\part{Markov partitions and symbolic dynamics}\label{Part_Finite_To_One}

\section{A locally finite countable Markov cover}
\subsection{The cover}
In  \S\ref{SectionGraph} we constructed a countable Markov shift $\Sigma$ with countable alphabet $\mathfs V$, and a H\"older continuous map $\pi:\Sigma\to M$ which commutes with the left shift $\s:\Sigma\to\Sigma$, so that $\pi(\Sigma)$ has full measure w.r.t. any ergodic invariant probability measure with entropy larger than $\chi$.
Moreover, if\footnote{this uses the convention from \S\ref{SectionRelevant} that every element of $\mathfs V$ is relevant.}
\begin{align*}
\Sigma^\#&=\{\un{u}\in\Sigma:\un{u} \textrm{ is  a regular chain}\}\\
&=\{\un{v}\in\Sigma: \exists v,w\in\mathfs V\ \exists n_k,m_k\uparrow\infty \textrm{ s.t. }v_{n_k}=v, v_{-m_k}=w\},
\end{align*}
then $\pi(\Sigma^\#)\supset\NUH^\#_\chi(f)$, therefore $\pi(\Sigma^\#)$ has full probability w.r.t. any ergodic invariant probability measure with entropy larger than $\chi$.

\medskip
In this section we study the following countable cover of $\NUH_\chi^\#(f)$:
\begin{defi}
$\mathfs Z:=\{Z(v):v\in\mathfs V\}$, where
$
Z(v):=\{\pi(\un{v}):\un{v}\in\Sigma^\#,\ \un{v}_0=v\}.
$
\end{defi}
\noindent
This is a cover of $\NUH_\chi^\#(f)$. The following property of $\mathfs Z$ is the hinge on which our entire approach turns (see \S\ref{SectionOverview}):
\begin{thm}\label{Theorem_LF}
For every $Z\in\mathfs Z$,  $|\{Z'\in\mathfs Z: Z'\cap Z\neq \emptyset\}|<\infty$.
\end{thm}
\begin{proof}
Fix some  $Z=Z(\Psi_x^{p^u,p^s})$. If $Z'=Z(\Psi_y^{q^u,q^s})$ intersects $Z$, then there must exist two chains $\un{v},\un{w}\in\Sigma^\#$ s.t. $v_0=\Psi_x^{p^u,p^s}$, $w_0=\Psi_y^{q^u,q^s}$, and
$\pi(\un{v})=\pi(\un{w})$.
Proposition \ref{Prop_Comparison_P} says that in this case
$$
q^u\geq e^{-\sqrt[3]{\e}}p^u\textrm{ and }q^s\geq e^{-\sqrt[3]{\e}}p^s.
$$
It follows that $Z'$ belongs to
$
\{Z(\Psi_y^{q^u,q^s}):\Psi_y^{q^u,q^s}\in\mathfs V, \ q^u\wedge q^s\geq e^{-\sqrt[3]{\e}}(p^u\wedge p^s)\}.
$
By the definition of $\mathfs V$, this set has   cardinality   less than or equal to
$$
|\{\Psi_y^\eta\in\mathfs A:\eta\geq e^{-\sqrt[3]{\e}}(p^u\wedge p^s)\}|\x
|\{(q^u,q^s)\in I_\e\x I_\e:q^u\wedge q^s\geq e^{-\sqrt[3]{\e}}(p^u\wedge p^s)
\}|.
$$
This is a finite number, because of the discreteness of $\mathfs A$ (Proposition \ref{Prop_A}).
\end{proof}

\subsection{Product structure}
Suppose $x\in Z(v)\in\mathfs Z$, then $\exists\un{v}\in\Sigma^\#$ s.t. $v_0=v$ and $\pi(\un{v})=x$.
Associated to $\un{v}$ are two admissible manifolds in $v$: $V^s[(v_i)_{i\leq 0}]$ and $V^u[(v_i)_{i\geq 0}]$ (Proposition \ref{Prop_V}).
These manifolds do not depend on the choice of $\un{v}$: if $\un{w}\in\Sigma^\#$ is another chain s.t. $w_0=v$ and $\pi(\un{w})=x$, then
$$
V^u[(w_i)_{i\leq 0}]=V^u[(v_i)_{i\leq 0}]\textrm{ and }V^s[(w_i)_{i\geq 0}]=V^s[(v_i)_{i\geq 0}],
$$
because of Proposition \ref{Prop_Uniqueness}.
We are therefore free to make the following definition:

\begin{defi}
Suppose $Z=Z(v)\in\mathfs Z$. For any $x\in Z$:
\begin{enumerate}
\item $V^s(x,Z):=V^s[(v_i)_{i\geq 0}]$ for some (every) $\un{v}\in\Sigma^\#$ s.t. $v_0=v$ and $\pi(\un{v})=x$. $W^s(x,Z):=V^s(x,Z)\cap Z$.
\item $V^u(x,Z):=V^u[(v_i)_{i\leq 0}]$ for some (every) $\un{v}\in\Sigma^\#$ s.t. $v_0=v$ and $\pi(\un{v})=x$. $W^u(x,Z):=V^u(x,Z)\cap Z$.
\end{enumerate}
\end{defi}

It is important to understand the difference between $V^{s/u}(x,Z)$ and $W^{s/u}(x,Z)$.
Whereas $V^{u/s}(x,Z)$ are smooth manifolds, $W^{u/s}(x,Z)$ could in principle be  totally disconnected. Whereas $V^{u/s}(x,Z)$ extend all the way across $\Psi_x[R_{p^{u/s}}(\un{0})]$ (assuming $v=\Psi_x^{p^u,p^s}$), $W^{u/s}(x,Z)$ are subsets of the much smaller set $\Psi_x[R_{10^{-2}(p^u\wedge p^s)}(\un{0})]$, because every point in $W^{u/s}(x,Z)$ is the intersection of an $s$--admissible manifold in $v$ and a $u$--admissible manifold in $v$ (Proposition \ref{Prop_Intersection}).

\begin{prop}\label{Prop_W(Z)_disjoint}
Suppose $Z\in\mathfs Z$. For every $x,y\in Z$, $V^u(x,Z)$ and $V^u(y,Z)$ are either equal or they are disjoint. Similarly for $V^s(x,Z)$ and $V^s(y,Z)$, for $W^u(x,Z)$ and $W^u(y,Z)$, and for $W^s(x,Z)$ and $W^s(y,Z)$.
\end{prop}
\begin{proof}
The statement holds for $V^{u/s}$ because of Proposition \ref{Prop_Uniqueness}. The statement for $W^{u/s}$ is an immediate corollary.
\end{proof}

\begin{prop}
Suppose  $Z\in \mathfs Z$ and  $x,y\in Z$, then $V^u(x,Z)$ and $V^s(y,Z)$ intersect at a unique point $z$, and $z\in Z$.  Thus $W^u(x,Z)\cap W^s(y,Z)=\{z\}$.
\end{prop}
\begin{proof}
Write $Z=Z(v)$ where $v\in\mathfs V$. $V^u(x,Z)$ is a $u$--admissible manifold in $v$, and $V^s(x,Z)$ is an $s$--admissible manifold in $v$. Consequently, $V^u(x,Z)$ and $V^s(x,Z)$ intersect at a unique point $z$ (Proposition \ref{Prop_Intersection}).

We claim that $z\in Z$. There are chains $\un{v}, \un{w}\in\Sigma^\#$ s.t. $v_0=w_0=v$ and so that
$V^u(x,Z)=V^u[(v_i)_{i\leq 0}]$ and $V^s(x,Z)=V^s[(w_i)_{i\geq 0}]$. Define $\un{u}=(u_i)_{i\in\Z}$ by
$$
u_i=\begin{cases}
v_i & i\leq 0\\
w_i & i\geq 0
\end{cases}.
$$
It is easy to see that $\un{u}\in\Sigma^\#$ and $u_0=v$, therefore $\pi(\un{u})\in Z$. By definition,
$$
\{\pi(\un{u})\}=V^u[(u_i)_{i\leq 0}]\cap V^s[(u_i)_{i\geq 0}]=V^u[(v_i)_{i\leq 0}]\cap V^s[(w_i)_{i\geq 0}]=V^u(x,Z)\cap V^s(y,Z).
$$
It follows that $z=\pi(\un{u})\in Z$.
\end{proof}

\begin{defi}
The {\em Smale bracket} of two points $x,y\in Z\in\mathfs Z$ is the unique point  $[x,y]_Z\in W^u(x,Z)\cap W^s(x,Z)$.
\end{defi}
\noindent
Compare with \cite{Smale} or \cite[chapter 3]{B}.

\begin{lem}\label{Lemma_SB}
Suppose $x,y\in Z(v_0)$ and $f(x),f(y)\in Z(v_1)$. If $v_0\to v_1$,  then $f([x,y]_{Z(v_0)})=[f(x),f(y)]_{Z(v_1)}$.
\end{lem}

\begin{proof}
Write $Y=Z(v_0)$, $Z=Z(v_1)$, and $w:=[x,y]_Y$. By definition
\begin{equation}\label{f(w)}
\{f(w)\}=f[W^u(x,Y)\cap W^s(y,Y)]\subset f[V^u(x,Y)]\cap f[V^s(y,Y)].
\end{equation}

\medskip
\noindent
{\em Claim\/}: $f[V^s(y,Y)]\subset V^s(f(y),Z)$ and $f[V^u(x,Y)]\supset V^u(f(x),Z)$.

\medskip
\noindent
{\em Proof.\/}
Since $f(y)\in Z(v_1)=Z$, $V^s:=V^s(f(y),Z)$ is an $s$--admissible manifold in $v_1$, and this manifold stays in windows. Applying the graph transform  (Proposition \ref{Prop_Graph_Transform}) we see that  $f^{-1}[V^s(f(y),Z)]$ contains an $s$--admissible manifold $\mathcal F_s[V^s]$ in $v_0$. Since $V^s$ stays in windows, $\mathcal F_s[V^s]$ stays in windows.

Since $\mathcal F_s[V^s]$ is $s$--admissible in $v_0$, it  intersects every $u$--admissible manifold in $v_0$.  The larger set $f^{-1}(V^s)$ intersects $V^u(y,Y)$ at a unique point (Proposition \ref{Prop_Graph_Transform} (2)). This point must be $y$, so
$
\mathcal F_s[V^s]\cap V^u(y,Y)=\{y\}
$, whence  $\mathcal F_s[V^s]\owns y$.

This means that $\mathcal F_s[V^s]$ intersects $V^s(y,Y)$. These manifolds are  $s$--admissible in $v_0$, and they stay in windows. Since they intersect,  they are equal. It follows that
$
f^{-1}(V^s)\supset \mathcal F_s[V^s]=V^s(y,Y)$,
whence $f[V^s(y,Y)]\subset V^s$, which is the first half of the claim.
The other half of the claim is proved in the same way.

\medskip
Returning to  (\ref{f(w)}) we  see that
$
f(w)\in f[V^u(x,Y)]\cap V^s(f(y),Z).
$
By the second half of the claim,
$$
f[V^u(x,Y)]\cap V^s(f(y),Z)\supseteq V^u(f(x),Z)\cap V^s(f(y),Z)\owns\{[f(x),f(y)]_Z\},
$$
thus $f[V^u(x,Y)]\cap  V^s(f(y),Z)\owns f(w),[f(x),f(y)]_Z$. But Proposition \ref{Prop_Graph_Transform} part (2) says that $f[V^u(x,Y)]$ intersects $V^s(f(y),Z)$ at a single point. It follows that $f(w)=[f(x),f(y)]_Z$.
\end{proof}

Occasionally we will need to form the Smale bracket of points belonging to different elements of $\mathfs Z$:

\begin{lem}\label{Lemma_Strong_bracket} The following holds for all $\e$ small enough:
Suppose $Z,Z'\in\mathfs Z$. If $Z\cap Z'\neq\emptyset$, then for any $x\in Z$ and $y\in Z'$, $V^u(x,Z)$ and $V^s(y,Z')$ intersect at a unique point.
\end{lem}
\noindent We do not claim that this point is in $Z$ or $Z'$.
\begin{proof}
Suppose $Z=Z(\Psi_{x_0}^{p^u_0,p^s_0}), Z'=Z(\Psi_{y_0}^{q^u_0,q^s_0})$ and $z\in Z\cap Z'$, then there are $\un{v},\un{w}\in\Sigma^\#$ s.t. $v_0=\Psi_{x_0}^{p^u_0,p^s_0}$, $w_0=\Psi_{y_0}^{q^u_0,q^s_0}$, and $z=\pi(\un{v})=\pi(\un{w})$.  Write $p:=p^u_0\wedge p^s_0$ and $q:=q^u_0\wedge q^s_0$. By Theorem \ref{Theorem_Pi_Almost_1-1}, $p^u_0/q^u_0, p^s_0/q^s_0,p/q\in [e^{-\sqrt[3]{\e}},e^{\sqrt[3]{\e}}]$ and
$$
\Psi_{y_0}^{-1}\circ\Psi_{x_0}=(-1)^{\s}\id+\un{c}+\Delta\textrm{ on }R_\e(\un{0}),
$$
where $\s\in\{0,1\}$,  $\un{c}$ is a constant vector s.t. $\|\un{c}\|<10^{-1}q$, and $\Delta:R_\e(\un{0})\to\R^2$ satisfies   $\Delta(\un{0})=\un{0}$, and $\|(d\Delta)_{\un{u}}\|<\sqrt[3]{\e}$ for all $\un{u}\in R_\e(\un{0})$. By the Mean Value Theorem, $\|\Delta(\un{u})\|\leq \sqrt[3]{\e}\|\un{u}\|$ for all $\un{u}\in R_\e(\un{0})$.

Now suppose $x\in Z$.  $V^u:=V^u(x,Z)$ is a $u$--admissible in $\Psi_{x_0}^{p^u_0,p^s_0}$, therefore it can be put in the form
$
V^u(x,Z)=\Psi_{x_0}\{(F(t),t):|t|\leq p^u_0\},
$
where $F:[-p^u_0,p^u_0]\to\R$ satisfies $|F(0)|\leq 10^{-3}p$, $\|F\|_\infty\leq 10^{-2} p^u_0$ and  $\Lip(F)<\e$.

We write $V^u(x,Z)$ in $\Psi_{y_0}$--coordinates. Let $\un{c}=(c_1,c_2)$, $\Delta=(\Delta_1,\Delta_2)$,  then
\begin{align*}
V^u(x,Z)&=[\Psi_{y_0}\circ(\Psi_{y_0}^{-1}\circ\Psi_{x_0})]\{(F(t),t):|t|\leq p^u_0\}\\
&=\Psi_{y_0}\{((-1)^{\s}F(t)+c_1+\Delta_1(F(t),t),(-1)^\s t+c_2+\Delta_2(F(t),t))\!:\!|t|\leq p^u_0\}\\
&=\Psi_{y_0}\{(\wt{F}(\theta)+c_1+\wt{\Delta}_1(\wt{F}(\theta),\theta), \underset{=:\tau(\theta)}{\underbrace{\theta+c_2+\wt{\Delta}_2(\wt{F}(\theta),\theta)}}):|\theta|\leq p^u_0\},
\end{align*}
where we have used the transformations $\theta:=(-1)^\s t$, $\wt{F}(s):=(-1)^\s F((-1)^\s s)$, and $\wt{\Delta}_i(u,v):=\Delta_i((-1)^\s u,(-1)^\s v)$. Notice that $|\wt{F}(0)|=|F(0)|\leq 10^{-3}p$, $\|\wt{F}\|_\infty=\|F\|_\infty\leq 10^{-2} p^u_0$ and $\Lip(\wt{F})=\Lip(F)<\e$. Also $\wt \Delta(\un{0})=\un{0}$ and $\|(d\wt{\Delta})_{\un{u}}\|=\|(d\Delta)_{\un{u}}\|<\sqrt[3]{\e}$ on $R_\e(\un{0})$.

Let $\tau(\theta):=\theta+c_2+\wt{\Delta}_2(\wt{F}(\theta),\theta)$. Assuming $\e$ is small enough, we have
\begin{itemize}
\item $\tau'\in [e^{-2\sqrt[3]{\e}},e^{2\sqrt[3]{\e}}]$;
\item $|\tau(0)|\leq |c_2|+|\wt{\Delta}_2(\wt{F}(0),0)|<10^{-1}q+\sqrt[3]{\e}\cdot 10^{-3} p<\frac{1}{6}p$ \ $(\because p\leq e^{\sqrt[3]{\e}}q)$.
\end{itemize}
It follows that $\tau$ is one-to-one, and
$
\tau[-p^u_0,p^u_0]=[\a,\b]
$
where $\a:=\tau(-p^u_0)$ and $\b:=\tau(p^u_0)$.
It is easy to see that
$
|\a+p^u_0|<\frac{1}{6}p^u_0$  and $|\b-p^u_0|<\frac{1}{6}p^u_0$: both quantities are less than $|c_2|+\sup_{R_{p_0^u}(\un{0})}|\wt{\Delta}_2|$, which is less than $\frac{1}{6}p^u_0$ provided $\e$ is small enough. It follows that $\tau[-p^u_0,p^u_0]=[\a,\b]\supset [-\frac{2}{3}q,\frac{2}{3}q]$.

Since $\tau:[-p^u_0,p^u_0]\to [\a,\b]$ is one-to-one and onto, it has a well defined inverse function $\theta:[\a,\b]\to [-p^u_0,p^u_0]$.
Let $G(s):=\wt{F}(\theta(s))+c_1+\wt{\Delta}_1(\wt{F}(\theta(s)),\theta(s))$, then
$$
V^u(x,Z)=\Psi_{y_0}\{(G(s),s):s\in[\a,\b]\}.
$$

Using the properties of $\tau$, it is not difficult   to check that
 $\theta'\in [e^{-2\sqrt[3]{\e}},e^{2\sqrt[3]{\e}}]$ and
  $|\theta(0)|=|\theta(0)-\theta(\tau(0))|\leq e^{2\sqrt[3]{\e}}|\tau(0)|<\frac{1}{6}e^{2\sqrt[3]{\e}} p$.
It follows that $|\wt{F}(\theta(0))|\leq |\wt{F}(0)|+\e|\theta(0)|<(10^{-3}+\frac{1}{6}e^{2\sqrt[3]{\e}}\e)p<10^{-2}p$, whence
\begin{align*}
|G(0)|&\leq 10^{-2}p+10^{-1}q+\sqrt[3]{\e}p<\min\{\tfrac{1}{6}p,\tfrac{1}{6}q\}\ \ \ (\because q/p\in [e^{-\sqrt[3]{\e}},e^{\sqrt[3]{\e}}])\\
|G'|&\leq \|\wt{F}'\|_\infty |\theta'|+\sqrt[3]{\e}\sqrt{1+|\wt{F}'|^2}\cdot|\theta'|<2\sqrt[3]{\e}.
\end{align*}
It follows that (for all $\e$ small enough) $G[-\frac{2}{3}p,\frac{2}{3}p]\subset [-\frac{2}{3}p,\frac{2}{3}p]$.

We can now show that $|V^u(x,Z)\cap V^s(y,Z')|\geq 1$ (compare with \cite[S.3.7]{KM}). Represent
$$
V^s(y,Z')=\Psi_{y_0}\{(t,H(t)):|t|\leq q^s_0\}.
$$
By admissibility, $|H(0)|<10^{-3}q$ and $\Lip(H)<\e$, so $H[-\frac{2}{3}p,\frac{2}{3}p]\subset [-\frac{2}{3}p,\frac{2}{3}p]$. It follows that $H\circ G$ is a contraction of $[-\frac{2}{3}p,\frac{2}{3}p]$ into itself. Such a map has a (unique) fixed point $(H\circ G)(s_0)=s_0$. It is easy to see that $\Psi_{y_0}(G(s_0),s_0)$ belongs to $V^u(x,Z)\cap V^s(y,Z')$.

Next we claim that $V^u(x,Z)\cap V^s(y,Z')$ contains at most one point. Extend $G$ and $H$ to $\e$--Lipschitz functions $\wt{G}, \wt{H}$ on $[-a,a]$ where $a:=\max\{|\a|,|\b|,q_0^s\}$. By construction, $|\wt{G}(0)|\leq\frac{1}{6}a$, so $\wt{G}[-a,a]\subset [-a,a]$. Also $|\wt{H}(0)|\leq 10^{-3}a$, so $\wt{H}[-a,a]\subset [-a,a]$. It follows that $\wt{H}\circ\wt{G}$ is a contraction of $[-a,a]$ into itself, and therefore it has a unique fixed point. Every point in $V^u(x,Z)\cap V^s(y,Z')$ takes the form  $\Psi_{y_0}(G(s),s)$ where $s\in [\a,\b]$ and  $s=(H\circ G)(s)\equiv (\wt{H}\circ\wt{G})(s)$. Since the equation $s=(\wt{H}\circ\wt{G})(s)$ has at most one solution in $[-a,a]$, it has at most one solution in  $[\a,\b]$. It follows that $|V^u(x,Z)\cap V^s(y,Z')|\leq 1$.
\end{proof}

\subsection{The symbolic Markov property}
\begin{prop}\label{Prop_SMP}
If $x=\pi[(v_i)_{i\in\Z}]$ where $\un{v}\in\Sigma^\#$, then
$
f[W^s(x,Z(v_0))]\subset W^s(f(x),Z(v_1))$ and $f^{-1}[W^u(f(x),Z(v_1)]\subset W^u(x,Z(v_0)).
$
\end{prop}
\begin{proof}
We prove the inclusion for the $s$--manifolds. The case of $u$--manifolds follows by symmetry.

\medskip
\noindent
{\em Step 1.\/} $
f[W^s(x,Z(v_0))]\subset V^s(f(x),Z(v_1))
$.

\medskip
\noindent
By definition, $W^s(x,Z(v_0))\subset V^s(x,Z(v_0))\equiv V^s[(v_i)_{i\geq 0}]$. By Proposition \ref{Prop_V}, $f(V^s[(v_i)_{i\geq 0}])\subseteq V^s[(v_{i+1})_{i\geq 0}]$. Since $f(x)=\pi[(v_{i+1})_{i\in\Z}]$, the last manifold is equal to $V^s(f(x),Z(v_1))$. Thus
$
f[W^s(x,Z(v_0))]\subset V^s(f(x),Z(v_1)).
$

\medskip
\noindent
{\em Step 2.\/} $f[W^s(x,Z(v_0))]\subset Z(v_1)$.

\medskip
\noindent
Suppose $y\in W^s(x,Z(v_0))$.
\begin{itemize}
\item  Since $y\in Z(v_0)$,
$
y\in\Psi_{x_0}[R_{10^{-2}(p^u_0\wedge p^s_0)}(\un{0})]
$ (it is the intersection of a $u$ and an $s$--admissible manifolds in $v_0$)
\item Since  $y\in V^s[(v_i)_{i\geq 0}]$,
$
f^k(y)\in V^s[(v_{i+k})_{i\geq 0}]\subset\Psi_{x_k}[R_{Q_\e(x_k)}(\un{0})]\textrm{ for all }k> 0,
$ where $v_k=\Psi_{x_k}^{p^u_k,p^s_k}$.
\item  Since $y\in Z(v_0)$, $\exists \un{w}\in\Sigma^\#$ s.t. $w_0=v_0$ and $y=\pi(\un{w})\in V^u[(w_i)_{i\leq 0}]$. It follows that
$
f^{-k}(y)\in V^u[(w_{i-k})_{i\leq 0}]\subset \Psi_{y_{-k}}[R_{Q_\e(y_{-k})}(\un{0})]\textrm{ for all }k\geq 0,
$
where $w_i=\Psi_{y_i}^{q^u_i,q^s_i}$.
\end{itemize}
 Writing
$
u_i=\begin{cases}
w_i & i\leq 0\\
v_i & i >0
\end{cases}$ and $u_i=\Psi_{z_i}^{r^u_i,r^s_i},
$
we see that $\un{u}\in\Sigma^\#$, $u_0=v_0$, $y\in \Psi_{z_0}[R_{p^u_0\wedge p^s_0}(\un{0})]$, and   $f^k(y)\in\Psi_{z_k}[R_{Q_\e(z_k)}(\un{0})]$ for all $k\in\Z$. By Proposition \ref{Prop_V} part (4), $y=\pi(\un{u})$. It follows that $f(y)=\pi[\s(\un{u})]\in Z(u_1)\equiv Z(v_1)$.
\end{proof}

\begin{lem}\label{Lemma_Overlapping_Z}
Suppose $Z,Z'\in\mathfs Z$ and $Z\cap Z'\neq\emptyset$.
\begin{enumerate}
\item If $Z=Z(\Psi_{x_0}^{p^u_0,p^s_0})$ and $Z'=Z(\Psi_{y_0}^{q^u_0,q^s_0})$, then $Z\subset \Psi_{y_0}[R_{q^u_0\wedge q^s_0}(\un{0})]$.
\item For any $x\in Z\cap Z'$, $W^u(x,Z)\subset V^u(x,Z')$ and $W^s(x,Z)\subset V^s(x,Z')$.
\end{enumerate}
\end{lem}
\begin{proof}
Fix some
$x\in Z\cap Z'$. Write $x=\pi(\un{v})$, $x=\pi(\un{w})$ where $\un{v},\un{w}\in\Sigma^\#$ satisfy $v_0=\Psi_{x_0}^{p^u_0,p^s_0}$ and $w_0=\Psi_{y_0}^{q^u_0,q^s_0}$. Write $p:=p^u_0\wedge p^s_0$ and $q:=q^u_0\wedge q^s_0$. Since $\pi(\un{v})=\pi(\un{w})$, we have by Theorem \ref{Theorem_Pi_Almost_1-1} that   $p/q\in [e^{-\sqrt[3]{\e}},e^{\sqrt[3]{\e}}]$ and
$$
\Psi_{y_0}^{-1}\circ\Psi_{x_0}=(-1)^\s\id+\un{c}+\Delta\textrm{ on }R_\e(\un{0}),
$$
where $\s\in\{0,1\}$,  $\un{c}$ is a constant vector s.t. $\|\un{c}\|<10^{-1}q$, and $\Delta:R_\e(\un{0})\to\R^2$ satisfies   $\Delta(\un{0})=\un{0}$, and $\|(d\Delta)_{\un{u}}\|<\sqrt[3]{\e}$ for all $\un{u}\in R_\e(\un{0})$. By the Mean Value Theorem, $\|\Delta(\un{u})\|\leq \sqrt[3]{\e}\|\un{u}\|$ for all $\un{u}\in R_\e(\un{0})$.

Every point in $Z$ is the intersection of a $u$--admissible and an $s$--admissible manifold in $\Psi_{x_0}^{p^u_0,p^s_0}$, therefore $Z$ is contained in $\Psi_{x_0}[R_{10^{-2}p}(\un{0})]$ (Proposition \ref{Prop_Intersection}). Thus
\begin{align*}
Z&\subseteq \Psi_{y_0}\bigl[(\Psi_{y_0}^{-1}\circ\Psi_{x_0})[R_{10^{-2}p}(\un{0})]\bigr]\subset \Psi_{y_0}\bigl[(\Psi_{y_0}^{-1}\circ\Psi_{x_0})[B_{\sqrt{2}\cdot 10^{-2}p}(\un{0})]\bigr]\\
&\subseteq \Psi_{y_0}\bigl[B_{(1+\sqrt[3]{\e})\sqrt{2}\cdot 10^{-2}p}(\un{c})\bigr]\subseteq \Psi_{y_0}\bigl[B_{(1+\sqrt[3]{\e})\sqrt{2}\cdot 10^{-2} e^{\sqrt[3]{\e}}q+10^{-1}q}(\un{0})\bigr]\\
&\subseteq \Psi_{y_0}\bigl[R_{(1+\sqrt[3]{\e})\sqrt{2}\cdot 10^{-2} e^{\sqrt[3]{\e}}q+10^{-1}q}(\un{0})\bigr]\subset \Psi_{y_0}[R_q(\un{0})]\ \ (\because 0<\e<1).
\end{align*}
This proves the first statement of the lemma.

\medskip
Next we show that $W^s(x,Z)\subset V^s(x,Z')$.
Write $v_i=\Psi_{x_i}^{p^u_i,p^s_i}$ and $w_i=\Psi_{y_i}^{q^u_i,q^s_i}$.
Since $x=\pi(\un{v})$ and $Z=Z(v_0)$,  we have by the symbolic Markov property that
$$
f^k[W^s(x,Z)]\subset W^s(f^k(x),Z(v_k))\ \ \ (k\geq 0).
$$
The sets $Z(v_k)$ and $Z(w_k)$ intersect, because they both contain $f^k(x)$. By the first part of the lemma,
$
Z(v_k)\subset \Psi_{y_k}[R_{q^u_k\wedge q^s_k}(\un{0})].
$ It follows that
$$
f^k[W^s(x,Z)]\subset \Psi_{y_k}[R_{q^u_k\wedge q^s_k}(\un{0})]\subset \Psi_{y_k}[R_{Q_\e(y_k)}(\un{0})]
$$
 for all $k\geq 0$. By Proposition \ref{Prop_V} part 4, $W^s(x,Z)\subset V^s[(w_i)_{i\geq 0}]\equiv V^s(x,Z')$. \end{proof}

\section{A countable Markov partition}
In the previous section we described a locally finite countable cover $\mathfs Z$ of $\NUH_\chi^\#(f)$ by sets equipped with a Smale bracket and satisfying the symbolic Markov property (Proposition \ref{Prop_SMP}). Here we produce a pairwise disjoint cover of $\NUH_\chi^\#(f)$ with similar properties.

Sinai and Bowen showed how to do this in the case of finite covers \cite{Sinai}, \cite{B}. Thanks to the finiteness property of $\mathfs Z$, their ideas  apply  to our case almost without change. The only difference is that in our case, the sets $Z\in\mathfs Z$ are not the closure of their interior, and therefore we cannot use  ``relative boundaries" and ``relative interiors" of $Z\in\mathfs Z$ as done in \cite{Sinai} and \cite{B}.

\subsection{The Bowen--Sinai refinement} Write $\mathfs Z=\{Z_1,Z_2,Z_3,\ldots\}$. Following \cite{B}, we define  for every  $Z_i,Z_j\in\mathfs Z$ s.t. $Z_i\cap Z_j\neq \emptyset$,
\begin{equation*}
\begin{aligned}
T^{us}_{ij}&:=\{x\in Z_i:W^u(x,Z_i)\cap Z_j\neq\emptyset\ , \ W^s(x,Z_i)\cap Z_j\neq \emptyset\},\\
T^{u\emptyset}_{ij}&:=\{x\in Z_i:W^u(x,Z_i)\cap Z_j\neq\emptyset\ , \ W^s(x,Z_i)\cap Z_j=\emptyset\},\\
T^{\emptyset s}_{ij}&:=\{x\in Z_i:W^u(x,Z_i)\cap Z_j=\emptyset\ , \ W^s(x,Z_i)\cap Z_j\neq \emptyset\},\\
T^{\emptyset\emptyset}_{ij}&:=\{x\in Z_i:W^u(x,Z_i)\cap Z_j=\emptyset\ , \ W^s(x,Z_i)\cap T_j=\emptyset\}.
\end{aligned}
\end{equation*}
Let $\mathfs T:=\{T^{\a\b}_{ij}:i,j\in\N, Z_i\cap Z_j\neq \emptyset, \ \a\in\{u,\emptyset\}, \b\in\{s,\emptyset\}\}$.

Notice that $T_{ii}^{us}=Z_i$, therefore $\mathfs T$ covers the same set as $\mathfs Z$, namely $\pi(\Sigma^\#)$.
Another useful identity is $T_{ij}^{us}=Z_i\cap Z_j$. The inclusion $\supseteq$ is trivial.
 To see $\subseteq$ suppose $x\in T_{ij}^{us}$.
Choose some $y\in W^u(x,Z_i)\cap Z_j$, then $y\in Z_i\cap Z_j$, so
$W^u(x,Z_i)=W^u(y,Z_i)\subset V^u(y,Z_j)$ (Lemma \ref{Lemma_Overlapping_Z}). Similarly, for every $z\in W^s(x,Z_i)\cap Z_j$,  $W^s(x,Z_i)\subset V^s(z,Z_j)$. It follows that
$$
\{x\}=W^u(x,Z_i)\cap W^s(x,Z_i)\subseteq V^u(y,Z_j)\cap V^s(z,Z_j)\subset Z_j,
$$
whence $x\in Z_i\cap Z_j$.

\begin{defi}
For every $x\in \pi(\Sigma^\#)$, let $R(x):=\bigcap\{T\in\mathfs T: T\owns x\}$, and set
$\mathfs R:=\{R(x):x\in\pi(\Sigma^\#)\}$.
\end{defi}

\begin{prop}
$\mathfs R$ is a countable pairwise disjoint cover of $\NUH_\chi^\#(f)$.
\end{prop}
\begin{proof}
We claim that each $R(x)$ is a finite intersection. By the finiteness property of  $\mathfs Z$  (Theorem \ref{Theorem_LF}), there are at most finitely many $Z_i\in\mathfs Z$ which contain $x$. Again by Theorem \ref{Theorem_LF}, for every $Z_i\in\mathfs Z$ which contains $x$, there are at most finitely many $Z_j\in\mathfs Z$ which intersect $Z_i$. As a result, there are at most finitely many $T\in\mathfs T$ which contain $x$. Thus $R(x)$ is a finite intersection.

Since there are countably many finite subsets of $\mathfs T$, there are countably many elements in  $\mathfs R$.

Since every $x\in T\in\mathfs T$ belongs to $R(x)\in\mathfs R$, $\bigcup\mathfs R=\bigcup\mathfs T$. We saw above that for every $Z_i\in\mathfs Z$, $T_{ii}^{us}=Z_i$. Consequently, $\bigcup\mathfs T=\bigcup\mathfs Z=\pi(\Sigma^\#)$. Since $\pi(\Sigma^\#)\supset \NUH_\chi^\#(f)$ (see the proof of Theorem \ref{Theorem_Markov_Extension}), $\mathfs R$ covers $\NUH_\chi^\#(f)$.

It remains to prove that $\mathfs R$ is pairwise disjoint. We do this by proving that $R(x)$ is the equivalence class of $x$ for the following equivalence relation on $\bigcup\mathfs R$:
\begin{equation}\label{Relation}
x\sim y\textrm{ iff  $\forall Z,Z'\in\mathfs Z$, } \left(
\begin{array}{rcl}
x\in Z&\Leftrightarrow& y\in Z\\
W^u(x,Z)\cap Z'\neq \emptyset&\Leftrightarrow &W^u(y,Z)\cap Z'\neq\emptyset\\
W^s(x,Z)\cap Z'\neq \emptyset&\Leftrightarrow &W^s(y,Z)\cap Z'\neq\emptyset
\end{array}
\right)
\end{equation}
So for every $x,y\in\bigcup\mathfs R$, either $R(x)=R(y)$, or $R(x)\cap R(y)=\emptyset$.

\medskip
\noindent
{\em Part 1.\/} If $x\sim y$, then $x\in R(y)$.

\medskip
If $x\sim y$, then $x$ and $y$  belong to exactly the same elements of $\mathfs T$. So   $R(x)=R(y)$.

\medskip
\noindent
{\em Part 2.\/} If $x\in R(y)$, then $x\sim y$.

\medskip
Fix some $Z_i\in\mathfs Z$. We claim that $x\in Z_i\Leftrightarrow y\in Z_i$.
Recall that $Z_i=T^{us}_{ii}$.

If  $y\in Z_i$, then $T_{ii}^{us}$ is one of the sets in the intersection which defines $R(y)$. Consequently, $x\in R(y)\subseteq T_{ii}^{us}=Z_i$, and $x\in Z_i$.

Next suppose $x\in Z_i$. Pick some $Z_k\in\mathfs Z$ which contains both $x$ and $y$ (any $k$ s.t. $T_{k\ell}^{\a\b}\owns y$ will do, because for such $k$ $Z_k\supset R(y)\owns x,y$). Since $y\in Z_k$ and $Z_k\cap Z_i\neq \emptyset$,  $y\in T^{\a\b}_{ki}$ for some $\a,\b$. By the definition of $R(y)$, $R(y)\subset T^{\a\b}_{ki}$, whence $x\in T^{\a\b}_{ki}$. But $x\in Z_k\cap Z_i\equiv T^{us}_{ki}$, so necessarily $(\a,\b)=(u,s)$. Thus $y\in T^{us}_{ki}=Z_k\cap Z_i\subset Z_i$. This completes the proof that $x\in Z_i\Leftrightarrow y\in Z_i$.

\medskip
Next we show that if $x\in R(y)$, then  $W^u(x,Z_i)\cap Z_j\neq\emptyset\Leftrightarrow W^u(y,Z_i)\cap Z_j\neq\emptyset$. If $W^u(x,Z_i)\cap Z_j\neq\emptyset$, then $x\in T_{ij}^{u\ast}$, where $\ast$ stands for $s$ or $\emptyset$. In particular $x\in Z_i$. By the previous paragraph, $y\in Z_i$, and as a result $y\in T_{ij}^{\a\b}$ for some $\a,\b$. Therefore $x\in R(y)\subset T_{ij}^{\a\b}$, and since $T_{ij}^{u\ast}\cap T_{ij}^{\emptyset\ast}=\emptyset$, $\a=u$. It follows that $y\in T_{ij}^{u\ast}$, whence
$W^u(y,Z_i)\cap Z_j\neq\emptyset$ as required. The other implication is trivial: If  $W^u(y,Z_i)\cap Z_j\neq\emptyset$, then $y\in T_{ij}^{u\ast}$, whence $x\in R(y)\subseteq T_{ij}^{u\ast}$, and so $W^u(x,Z_i)\cap Z_j\neq \emptyset$.

The proof that if $x\in R(y)$, then  $W^s(x,Z_i)\cap Z_j\neq\emptyset\Leftrightarrow W^s(y,Z_i)\cap Z_j\neq\emptyset$ is exactly the same.
\end{proof}

\begin{lem}\label{Lemma_R_In_Z}
$\mathfs R$ is a locally finite refinement of $\mathfs Z$:
\begin{enumerate}
\item for every $R\in\mathfs R$ and $Z\in\mathfs Z$, if $R\cap Z\neq\emptyset$ then $R\subset Z$;
\item for every $Z\in\mathfs Z$, $\bigl|\{R\in\mathfs R:Z\supset R\bigr|<\infty$.
\end{enumerate}
\end{lem}
\begin{proof}
Suppose $R\cap Z\neq\emptyset$ and let $x\in R\cap Z$.
If $Z=Z_i$, then $Z=T^{us}_{ii}$. Since $x\in Z$,  $T^{us}_{ii}$ appears in the intersection which defines $R(x)$, therefore $R(x)\subset T^{us}_{ii}$. Since $x\in R$, $R$ intersects $R(x)$, and therefore by the previous proposition $R=R(x)$. It follows that $R=R(x)\subset T^{us}_{ii}=Z$, which proves the first part of the proposition.

We turn to the second part. If $R\subset Z$, then $R$ is the intersection of a subset of   $\mathfs T(Z):=\{T^{\a\b}_{ij}\in\mathfs T: T^{\a\b}_{ij}\cap Z\neq\emptyset\}$. If $T^{\a\b}_{ij}\cap Z\neq \emptyset$, then $Z_i\cap Z\neq\emptyset$, $Z_j\cap Z_i\neq \emptyset$, and $\{\a,\b\}\subset\{u,s,\emptyset\}$. By Theorem \ref{Theorem_LF}, there are finitely many possibilities for $Z_i$, and therefore also finite many possibilities for $Z_j$. Thus $\mathfs T(Z)$ is finite.

Since $\mathfs T(Z)$ is finite, and any $R\subset Z$ is the intersection of a subset of $\mathfs T(Z)$, $|\{R\in\mathfs R:R\subset Z\}|\leq 2^{|\mathfs T(Z)|}<\infty$.
\end{proof}

\subsection{Product structure and hyperbolicity}

\begin{defi}
For any $R\in\mathfs R$ and $x\in R$, let
\begin{align*}
W^s(x,R)&:=\bigcap\{W^s(x,Z_i)\cap T_{ij}^{\a\b}: T_{ij}^{\a\b}\in\mathfs T\text{ contains }R \},\\
W^u(x,R)&:=\bigcap\{W^u(x,Z_i)\cap T_{ij}^{\a\b}: T_{ij}^{\a\b}\in\mathfs T\text{ contains }R \}.
\end{align*}
\end{defi}
\begin{prop}
Suppose   $R\in\mathfs R$ and $x,y\in R$.
\begin{enumerate}
\item $W^u(x,R), W^s(x,R)\subset R$ and $W^u(x,R)\cap W^s(x,R)=\{x\}$.
\item Either $W^u(x,R), W^u(y,R)$  are equal, or they are disjoint. Similarly for $W^s(x,R)$ and $W^s(y,R)$.
\item $W^u(x,R)$ and $W^s(y,R)$ intersect at a unique point $z$, and  $z\in R$.
\item If $\xi,\eta\in W^s(x,R)$, then $d(f^n(\xi),f^n(\eta))\xrightarrow[n\to\infty]{}0$. If $\xi,\eta\in W^u(x,R)$, then $d(f^{-n}(\xi),f^{-n}(\eta))\xrightarrow[n\to\infty]{}0$.
\end{enumerate}
\end{prop}
\begin{proof} Suppose $R\in\mathfs R$ and $x,y\in R$.

\medskip
\noindent
{\em Part (1).\/}
By definition, $W^{u/s}(x,R)\subset\bigcap\{T_{ij}^{\a\b}\in\mathfs T:T_{ij}^{\a\b}\supset R\}\equiv R$. It follows that  $W^{u/s}(x,R)\subset R$.

If $x\in R$, then for every $T_{ij}^{\a\b}\in\mathfs T$ which contains $R$, $x\in W^{s/u}(x,Z_i)\cap R\subset W^{s/u}(x,Z_i)\cap T_{ij}^{\a\b}$. Passing to the intersection, we see that $x\in W^{s/u}(x,R)$. Thus
$
x\in W^u(x,R)\cap W^s(x,R)
$. On the other hand for every $Z_i\supseteq R$, $W^s(x,R)\cap W^u(x,R)\subset W^u(x,Z_i)\cap W^s(x,Z_i)=\{x\}$, so $W^u(x,R)\cap W^s(x,R)=\{x\}$.

\medskip
\noindent
{\em Part (2).\/}  Suppose $W^u(x,R)\cap W^u(y,R)\neq \emptyset$, then $W^u(x,Z_i)\cap W^u(y,Z_i)\neq\emptyset$ for every $i$ s.t. there is some $T_{ij}^{\a\b}\in\mathfs T$ which contains $R$. By Proposition \ref{Prop_W(Z)_disjoint}, $W^u(x,Z_i)=W^u(y,Z_i)$, whence $W^u(x,Z_i)\cap T_{ij}^{\a\b}=W^u(y,Z_i)\cap T_{ij}^{\a\b}$. Passing to the intersection, we see that $W^u(x,R)=W^u(y,R)$. Similarly, one shows that if $W^s(x,R)\cap W^s(y,R)\neq\emptyset$, then $W^s(x,R)=W^s(y,R)$.

\medskip
\noindent
{\em Part (3).\/}
For every $T_{ij}^{\a\b}\in\mathfs T$ which covers $R$ and for every $z\in R$, let
$$
W^u(z,T_{ij}^{\a\b}):=W^u(z,Z_i)\cap T_{ij}^{\a\b}\textrm{ and }W^s(z,T_{ij}^{\a\b}):=W^s(z,Z_i)\cap T_{ij}^{\a\b}.
$$

Fix  $x,y\in R$. For every $T_{ij}^{\a\b}\in\mathfs T$ which contains $R$,  $W^u(x,Z_i)\cap W^s(y,Z_i)=\{z_i\}$ where $z_i:=[x,y]_{Z_i}$. By Proposition \ref{Prop_W(Z)_disjoint}, $W^{u}(z_i,Z_i)=W^u(x,Z_i)$ and $W^s(z_i,Z_i)=W^s(y,Z_i)$. It follows that  $z_i\in T_{ij}^{\a\b}$, whence
$$
W^u(x,T_{ij}^{\a\b})\cap W^s(y,T_{ij}^{\a\b})=\{z_i\}.
$$
Since $z_i=[x,y]_{Z_i}$, $z_i$ is independent of $j,\a,$ and $\b$. In fact $z_i$ is also independent of $i$: If $T_{k\ell}^{\g\d}\in\mathfs T$ also covers $R$, then $x,y\in Z_i\cap Z_k$ and so
\begin{align*}
\{z_i\}&=W^u(x,Z_i)\cap W^s(y,Z_i)\subset V^u(x,Z_i)\cap V^s(y,Z_i)\\
\{z_k\}&=W^u(x,Z_k)\cap W^s(y,Z_k)\subset V^u(x,Z_i)\cap V^s(y,Z_i)\ \ (\textrm{Lemma \ref{Lemma_Overlapping_Z}}).
\end{align*}
Since $V^u(x,Z_i)\cap V^s(y,Z_i)$ is a singleton, $z_i=z_k$.

Denote the common value of $z_i$ by $z$, then $
W^u(x,T_{ij}^{\a\b})\cap W^s(y,T_{ij}^{\a\b})=\{z\}.
$ for all $T_{ij}^{\a\b}\in\mathfs T$ which covers $R$. Passing to the intersection, we obtain that $W^u(x,R)\cap W^s(y,R)=\{z\}$. By part (1) of the lemma, $z\in R$.

\medskip
\noindent
{\em Part (4).\/} Fix some $Z\in\mathfs Z$ such that $R\subseteq Z$, then $x=\pi(\un{v})$ where $\un{v}$ is a regular chain such that  $Z:=Z(v_0)$.  By construction, $W^s(x,R)\subset V^s[(v_i)_{i\geq 0}]$ and $W^u(x,R)\subset V^u[(v_i)_{i\leq 0}]$. Part (4) follows from Proposition \ref{Prop_Staying_In_Windows}(1).
\end{proof}

Given $x,y\in R$, we let $[x,y]$ denote the unique element of $W^u(x,R)\cap W^s(x,R)$. As the proof of the previous proposition shows, $[x,y]$ is equal to the Smale bracket of $x$ and $y$ in any of the $Z\in\mathfs Z$ which contain $R$.

\subsection{The Markov property}  $\mathfs R$ satisfies Sinai's {\em Markov property} \cite{Sinai}:
\begin{prop}\label{Theorem_Markov_Property}
Let $R_0, R_1\in\mathfs R$. If $x\in R_0$ and $f(x)\in R_1$, then
$$
f[W^s(x,R_0)]\subset W^s(f(x),R_1)\textrm{ and }f^{-1}[W^u(f(x),R_1)]\subset W^u(x,R_0).
$$
\end{prop}
\begin{proof} The proof is the same as Bowen's \cite[pages 54,55]{B}, except that our ``rectangles" $R\in\mathfs R$ are different. We give all the details to convince the reader that everything works out as it should.

It is enough to show that $f[W^s(x,R_0)]\subset W^s(f(x),R_1)$: the statement for $W^u$ follows by symmetry.

Suppose $y\in W^s(x,R_0)$. We prove that $f(y)\in W^s(f(x), R_1)$ by checking that for every $T_{ij}^{\a\b}\in\mathfs T$ which covers $R_1$, $f(y)\in W^s(f(x),Z_i)\cap T_{ij}^{\a\b}$.

That $f(y)\in W^s(f(x),Z_i)$ can be shown as follows.  Since $T_{ij}^{\a\b}$  covers $R_1$, $T_{ij}^{\a\b}$ contains $f(x)$. Thus $f(x)\in  T_{ij}^{\a\b}\subset Z_i$. Write $Z_i=Z(v)$ and $f(x)=\pi(\s\un{v})$ where $\un{v}\in \Sigma^\#$ satisfies $v_1=v$. Since $f\circ\pi=\pi\circ\s$, $x=\pi(\un{v})\in Z(v_{0})$. It follows that $Z(v_{0})\supseteq R(x)=R_0$, whence
$
y\in W^s(x,R_0)\subset W^s(x,Z(v_{0})).
$
By the  symbolic Markov property (Proposition \ref{Prop_SMP}),
$$
f[W^s(x,Z(v_{0}))]\subset W^s[f(x),Z(v_1)],
$$
so $
f(y)\in f[W^s(x,R_0)]\subset f[W^s(x,Z(v_{0}))]\subset W^s(f(x),Z(v_1))\equiv W^s(f(x),Z_i)$.

It remains to prove that if $y\in W^s(x,R_0)$, then $f(x)\in T^{\a\b}_{ij}\Leftrightarrow f(y)\in T^{\a\b}_{ij}$. Since
$
y\in W^s(x,R_0)\Leftrightarrow W^s(x,R_0)=W^s(y,R_0)$,
this is equivalent to showing that  if $W^s(x,R_0)=W^s(y,R_0)$, then for every $Z_i,Z_j\in\mathfs Z$ s.t. $Z_i\cap Z_j\neq \emptyset$,
\begin{itemize}
\item $f(x)\in Z_i\Leftrightarrow f(y)\in Z_i$;
\item $W^s(f(x),Z_i)\cap Z_j\neq \emptyset\Leftrightarrow W^s(f(y),Z_i)\cap Z_j\neq \emptyset$;
\item $W^u(f(x),Z_i)\cap Z_j\neq \emptyset\Leftrightarrow W^u(f(y),Z_i)\cap Z_j\neq \emptyset$.
\end{itemize}
We only prove $\Rightarrow$. The other implication follows by symmetry.

\medskip
\noindent
{\em Step 1.\/} $f(x)\in Z_i\Rightarrow f(y)\in Z_i$.

\medskip
If $f(x)\in Z_i$, then $f(x)\in T^{us}_{ii}\equiv Z_i$. Thus $T^{us}_{ii}\supseteq R(f(x))=R_1$.  We  saw above that if $T_{ij}^{\a\b}$ covers $R_1$, then $f(y)\in W^s(f(x),Z_i)$. Applying this to $T^{us}_{ii}$, we see that  $f(y)\in W^s(f(x),Z_i)\subset Z_i$.

\medskip
\noindent
{\em Step 2.\/} $W^s(f(x),Z_i)\cap Z_j\neq\emptyset\Rightarrow W^s(f(y),Z_i)\cap Z_j\neq\emptyset$.

\medskip
Write $Z_i=Z(v)$.
Since $f(x)\in Z_i$, $f(x)=\pi[\s\un{v}]$ where $\un{v}\in\Sigma^\#$ and $v_1=v$. Since $f\circ\pi=\pi\circ\s$, $x=\pi(\un{v})$. By the symbolic Markov property, $f[W^s(x,Z(v_0))]\subset W^s(f(x), Z(v_1))=W^s(f(x),Z_i)$. Since $x=\pi(\un{v})$, $x\in Z(v_0)$, whence $R_0=R(x)\subset Z(v_0)$. Consequently,
\begin{align*}
f(y)&\in f[W^s(y,R_0)]=f[W^s(x,R_0)]\ \ \ \textrm{(by assumption)}\\
&\subset f[W^s(x,Z(v_0))]\subset W^s(f(x),Z(v_1))\equiv W^s(f(x),Z_i).
\end{align*}
Since $f(y)\in W^s(f(x),Z_i)$, $W^s(f(y), Z_i)=W^s(f(x),Z_i)$. It is now clear that $W^s(f(x),Z_i)\cap Z_j\neq\emptyset\Rightarrow W^s(f(y),Z_i)\cap Z_j\neq\emptyset$.

\medskip
\noindent
{\em Step 3.\/} $W^u(f(x),Z_i)\cap Z_j\neq\emptyset\Rightarrow W^u(f(y),Z_i)\cap Z_j\neq\emptyset$.

\medskip
In order to reduce the number of indices, we write $Z_i=Z$, $Z_j=Z^\ast$, and prove that
$
W^u(f(x),Z)\cap Z^\ast\neq\emptyset\Rightarrow W^u(f(y),Z)\cap Z^\ast\neq\emptyset.
$
We do this by picking some
 $f(z)\in W^u(f(x),Z)\cap Z^\ast$, and  showing that $W^u(f(y),Z)\cap Z^\ast\owns f(w)$ where $w:=[y,z]_Y$ for some suitable $Y\in\mathfs Z$ that we proceed to construct.

 Since $f(x)\in Z_i=Z$, there exists $\un{v}\in\Sigma^\#$ such that $\pi(\s\un{v})=f(x)$ and $Z=Z(v_1)$. Let
 $
 Y:=Z(v_0),
 $
then $x=\pi(\un{v})\in Y$. By assumption, $R(x)=R_0=R(y)$, therefore, $x\sim y$  in the sense of (\ref{Relation}). Since $x\in Y$ and $y\sim x$, $y\in Y$.

 Since $f(z)\in W^u(f(x),Z)\cap Z^\ast$, $f(z)\in Z^\ast$. This means that there exists $\un{v}^\ast\in\Sigma^\#$ such that $\pi(\s\un{v}^\ast)=f(z)$ and $Z^\ast=Z(v^\ast_1)$. Let
 $
 Y^\ast:=Z(v^\ast_0),
 $
 then $z=\pi(\un{v}^\ast)\in Y^\ast$.
 By the symbolic Markov property,
 $$
 z\in f^{-1}[W^u(f(x),Z)]\equiv f^{-1}[W^u(f(x),Z(v_1))]\subset W^u(x,Z(v_0))\equiv W^u(x,Y).
 $$
 Thus $z\in W^u(x,Y)\cap Y^\ast$. In particular, $z\in Y\cap Y^\ast$.

Since $y,z\in Y$, the Smale bracket  $w:=[y,z]_Y$ is well defined.
We show that $f(w)\in W^u(f(y),Z)\cap Z^\ast$.

 By construction, $w=[y,z]_{Y}$. Since $f(y)\in Z$ (by Step 1), $f(z)\in Z$ (by choice), and
$Y=Z(v_0), Z=Z(v_1)$ and $v_0\to v_1$ (by construction), we have by Lemma \ref{Lemma_SB} that
$
f(w)=f([y,z]_Y)=[f(y),f(z)]_Z\in W^u(f(y),Z)$.

Next recall that  $W^u(x,Y)\cap Y^\ast$ is non--empty (it contains $z$). Since $x\sim y$,
$W^u(y,Y)\cap Y^\ast$ is non-empty. Pick some $y'\in W^u(y,Y)\cap Y^\ast$. Since $y', z\in Y\cap Y^\ast$, we have by Lemma \ref{Lemma_Overlapping_Z} that
$$
\{w\}=W^u(y',Y)\cap W^s(z,Y)\subset V^u(y',Y^\ast)\cap V^s(z,Y^\ast)\equiv\{[y',z]_{Y^\ast}\}.
$$
Thus $w=[y',z]_{Y^\ast}\in W^s(z,Y^\ast)$. Now $Y^\ast=Z(v_0^\ast)$, $Z^\ast=Z(v_1^\ast)$ and $z=\pi(\un{v}^\ast)$, therefore by the symbolic Markov property,
$$
f(w)\in f[W^s(z,Y^\ast)]\subset W^s(f(z),Z^\ast)\subset Z^\ast.
$$
It follows that $f(w)\in Z^\ast$. This completes the proof of Step 3. The theorem follows from the discussion before Step 1.
\end{proof}

\section{Symbolic dynamics}
\subsection{A directed graph} In the previous section we constructed a Markov partition $\mathfs R$ for $f$. Here we use this partition to relate  $f$ to a topological Markov shift. The shift is $\Sigma(\wh{\mathfs G})$ where
 $\wh{\mathfs G}$ is the directed graph with vertices  $\wh{\mathfs V}:=\mathfs R$ and edges
\begin{align*}
\wh{\mathfs E}&:=\{(R_1,R_2)\in\mathfs R^2:R_1, R_2\in\wh{\mathfs V}\textrm{ s.t. }R_1\cap f^{-1}(R_2)\neq \emptyset\}.
\end{align*}
If $(R_1,R_2)\in\wh{\mathfs E}$, then we write $R_1\to R_2$.

For every finite path  $R_m\to R_{m+1}\to\cdots\to R_n$ in $\wh{\mathfs G}$, let
$
_\ell[R_m,\ldots,R_n]:=\bigcap\limits_{k=\ell}^{\ell+n-m} f^{-k}(R_{k+m-\ell}).
$
In particular, $$
_{m}[R_{m},\ldots,R_n]=\bigcap\limits_{k=m}^n f^{-k}(R_k).
$$

\begin{lem}\label{Lemma_Non_Empty_Cylinders}
Suppose $m\leq n$ and $R_m\to R_{m+1}\to\cdots\to R_n$ is a finite path on $\wh{\mathfs G}$, then $_m[R_m,\ldots,R_n]\neq \emptyset$.
\end{lem}
\begin{proof}
We use induction on $n$.

\medskip
If $n=m$,   then the statement is obvious.

\medskip
Suppose by induction the statement is true for $n-1$, and let $R_m\to\cdots \to R_{n-1}$ be a path on $\wh{\mathfs G}$. By the induction hypothesis, $_m[R_m,\ldots,R_{n-1}]\neq\emptyset$, therefore there exists a point
$
y\in \bigcap_{k=m}^{n-1} f^{-k}(R_k).
$
Since $R_{n-1}\to R_n$, there exists a point
$
z\in R_{n-1}\cap f^{-1}(R_n).
$
Let $x$ be the point such that $$
\{f^{n-1}(x)\}=W^u(f^{n-1}(y),R_{n-1})\cap W^s(z,R_{n-1}).
$$

We claim that $x\in {_m[R_m,\ldots,R_{n}]}$. This follows from the  Markov property (Theorem  \ref{Theorem_Markov_Property}):
\begin{itemize}
\item $f^n(x)\in R_n$, because $f^n(x)\in f[W^s(z,R_{n-1})]\subset W^s(f(z),R_n)\subset R_n$;
\item $f^{n-1}(x)\in R_{n-1}$ by construction;
\item $f^{n-2}(x)\in R_{n-2}$, because $f^{n-1}(x)\in W^u(f^{n-1}(y),R_{n-1})\subset R_{n-1}$ so
$$
f^{n-2}(x)\in f^{-1}[W^u(f^{n-1}(y),R_{n-1})]\subset W^u(f^{n-2}(y),R_{n-2})\subset R_{n-2}.
$$
\item $f^{n-3}(x)\in R_{n-3}$, because $f^{n-2}(x)\in W^u(f^{n-2}(y),R_{n-2})$ so
$$
f^{n-3}(x)\in f^{-1}[W^u(f^{n-2}(y),R_{n-2})]\subset W^u(f^{n-3}(y),R_{n-3})\subset R_{n-3}.
$$
\end{itemize}
Continuing this way, we see that $f^{n-k}(x)\in R_{n-k}$ for all $0\leq k\leq n-m$.
\end{proof}

We compare the paths on $\wh{\mathfs G}$ to the paths on $\mathfs G$ (the graph we introduced in \S\ref{SectionGraph}). Recall the map $\pi:\Sigma\to M$ from Theorem \ref{Theorem_Markov_Extension}, and define for any finite path $v_m\to\cdots \to v_n$ on $\mathfs G$,
$$
Z_m(v_m,\ldots,v_n):=\{\pi(\un{w}):\un{w}\in\Sigma^\#, w_i=v_i\textrm{ for all }i=m,\ldots,n\}.
$$

\begin{lem}\label{Lemma_G_G_hat}
For every infinite path $\cdots\to R_i\to R_{i+1}\to\cdots $ in $\wh{\mathfs G}$ there exists a chain $(v_i)_{i\in\Z}\in \Sigma$ such that for every $i$, $R_i\subset Z(v_i)$, and for every $n$,
$
_{-n}[R_{-n},\ldots,R_n]\subset Z_{-n}(v_{-n},\ldots,v_n).
$
\end{lem}
\begin{proof}
Fix, using Lemma \ref{Lemma_Non_Empty_Cylinders}, points $y_n\in {_{-n}[R_{-n},\ldots,R_n]}$.

Pick some $v_0\in\mathfs V$ s.t. $R_0\subset Z(v_0)$. Since $y_n\in R_0$, there is a chain $\un{v}^{(n)}=(v_i^{(n)})_{i\in\Z}\in\Sigma^\#$ such that $v_i^{(0)}=v_0$ and $y_n=\pi[\un{v}^{(n)}]$.

For every $|k|\leq n$,  $f^k(y_n)=\pi[\s^k(\un{v}^{(n)})]\in Z(v^{(n)}_k)$, therefore
$Z(v_k^{(n)})$ covers $R(f^k(y_n))$. Since, by construction, $f^k(y_n)\in R_k$, $R(f^k(y_n))=R_k$. It follows that
$$
R_k\subset Z(v_k^{(n)})\textrm{ for every }k=-n,\ldots,n.
$$

Every vertex in the graph $\mathfs G$ has finite degree (Lemma \ref{Lemma_Subordinated_Tempered}). Therefore, there are only finitely many paths of length $k$  on $\mathfs G$ which start at $v_0$. As a result, every set of the form
$\{v^{(n)}_k:n\in\N\}$ is finite.
Using the diagonal argument, choose a subsequence $n_i\uparrow\infty$ s.t. for every $k$ the sequence $\{v^{(n_i)}_{k}\}_{i\geq 1}$ is eventually constant. Call the constant  $v_k$.

The sequence $\un{v}:=(v_k)_{k\in\Z}$ is a chain, and $R_k\subset Z(v_k)$ for all $k\in\Z$. We claim that ${_n[R_{-n},\ldots,R_n]}\subset Z_{-n}(v_{-n},\ldots,v_n)$ for all $n$.

Suppose $y\in {_{-n}[R_{-n},\ldots,R_n]}$. Since $f^n(y)\in R_n$ and $R_n\subset Z(v_n)$, there exists a chain $\un{w}\in \Sigma^\#$ s.t.
$
f^n(y)=\pi[\s^n(\un{w})]\textrm{ and }w_n=v_n.
$
Since $f^{-n}(y)\in R_{-n}$ and $R_{-n}\subset Z(v_{-n})$, there exists a chain $\un{u}\in \Sigma^\#$ s.t.
$
f^{-n}(y)=\pi[\s^{-n}(\un{u})]\textrm{ and }u_{-n}=v_{-n}.
$
Let
$$
\un{a}=(a_i)_{i\in\Z}\textrm{ where }a_i=\begin{cases}
u_i & i\leq -n\\
v_i & -n\leq i\leq n\\
w_i & i\geq n.
\end{cases}
$$
For every $k$, $f^k(y)\in Z(a_k)$, because
\begin{itemize}
\item for all $k\leq -n$, $f^k(y)\in V^u[(u_i)_{i\leq k}]\subset Z(u_i)=Z(a_i)$,
\item for all $-n\leq k\leq n$, $f^k(y)\in R_k\subset Z(v_k)=Z(a_k)$,
\item for all $k\geq n$ $f^k(y)\in V^s[(w_i)_{i\geq k}]\subset Z(w_i)=Z(a_i)$.
\end{itemize}
Writing $a_i=\Psi_{x_i}^{p^u_i,p^s_i}$, we see that $y\in \Psi_{x_i}[R_{Q_\e(x_i)}(\un{0})]$ for all $i\in\Z$. By Proposition \ref{Prop_V} part 4, $y\in V^u[(a_i)_{i\leq 0}]\cap V^s[(a_i)_{i\geq 0}]$, so   $y=\pi(\un{a})\in Z_{-n}(v_{-n},\ldots,v_n)$.
\end{proof}

\begin{prop}
Every vertex of $\wh{\mathfs G}$ has finite degree.
\end{prop}
\begin{proof}
Fix $R_0\in\mathfs R$. We bound the number of paths $R_{-1}\to R_0\to R_1$.

Consider all the possible  paths $v_{-1}\to v_0\to v_1$ on $\mathfs G$ s.t. $_{-1}[R_{-1},R_0,R_1]\subset Z_{-1}(v_{-1},v_0,v_1)$. There are finitely many possibilities for $v_0$, because any two possible choices $v_0, v_0'$ satisfy
$
Z(v_0)\cap Z(v_0')\supset R_0\neq \emptyset$,
and $\mathfs Z$ has the finiteness property  (Theorem \ref{Theorem_LF}). Since every vertex of $\mathfs G$ has finite degree, there are also only finitely many possibilities for $v_{-1}$ and $v_1$.
By Lemma \ref{Lemma_R_In_Z}(1), $R_i\subset Z(v_i)$ $(|i|\leq 1$). By Lemma \ref{Lemma_R_In_Z}(2) the number of possible $R_{-1}$, $R_0$ or $R_1$ is finite.
\end{proof}

\subsection{The Markov extension} Let
$$
\wh{\Sigma}:=\Sigma(\wh{\mathfs G})=\{(R_i)_{i\in\Z}\in\mathfs R^\Z:R_i\to R_{i+1}\textrm{ for all }i\in\Z\}.
$$
Abusing notation, we denote the left shift map on $\wh{\Sigma}$ by $\s$, and the natural metric on $\wh{\Sigma}$ by $d(\cdot,\cdot)$: $d(\un{x},\un{y})=\exp[-\min\{|k|:x_k\neq y_k\}]$. Since every vertex of $\wh{\mathfs G}$ has finite degree, $\wh{\Sigma}$ is locally compact.
Define as before
$$
\wh{\Sigma}^\#:=\{(R_i)_{i\in\Z}:\exists R,S\in\mathfs R,\exists n_k,m_k\uparrow\infty \textrm{ s.t. }
R_{n_k}=R\textrm{ and }R_{-m_k}=S\}.
$$
Clearly $\wh{\Sigma}^\#$ contains every periodic point for $\s$. By Poincar\'e's Recurrence Theorem, every $\s$--invariant probability measure on $\wh{\Sigma}$ is supported on $\wh{\Sigma}^\#$.

Our aim is to construct a finite-to-one H\"older continuous map $\wh{\pi}:\wh{\Sigma}\to M$ which intertwines $\s$ and $f$, and such that $\wh{\pi}(\wh{\Sigma})$ (and even $\wh{\pi}(\Sigma^\#)$)  has full probability w.r.t any ergodic invariant probability measure with entropy larger than $\chi$.

We start with the following simple observation:
\begin{lem}\label{Lemma_Cylinder_Diameter}
There exist constants $C$ and $0<\theta<1$ s.t. for every $(R_i)_{i\in\Z}\in\wh{\Sigma}$,  $\diam({_{-n}[R_n,\ldots,R_n]})<C\theta^n$.
\end{lem}
\begin{proof}
Recall that $\pi:\Sigma\to M$ is H\"older continuous, therefore there are  $C$ and $0<\theta<1$ s.t. for every $\un{v},\un{u}\in\Sigma$, if $v_i=u_i$ for all $|i|\leq n$ then $d(\pi(\un{u}),\pi(\un{v}))<C\theta^n$.
By Lemma \ref{Lemma_G_G_hat} there exists a chain $(v_i)_{i\in\Z}\in\Sigma$ s.t.
$$
_{-n}[R_{-n},\ldots,R_n]\subset Z_{-n}(v_{-n},\ldots,v_n).
$$
The diameter of $Z_{-n}(v_{-n},\ldots,v_n)$ is less than or equal to $C\theta^n$. Therefore the diameter of $_{-n}[R_{-n},\ldots,R_n]$ is less than or equal to $C\theta^n$.
\end{proof}

Suppose $(R_i)_{i\in\Z}\in\wh{\Sigma}$, and let
 $F_n:=\ov{_{-n}[R_{-n},\ldots,R_n]}$ (closure in $M$). Lemmas \ref{Lemma_Non_Empty_Cylinders} and \ref{Lemma_Cylinder_Diameter} say that $\{F_n\}_{n\geq 1}$ is a decreasing sequence of non--empty compact subsets of $M$, whose diameters tend to zero. It follows that $\bigcap_{n\geq 1} F_n$ consists of a single point. We call this point $\wh{\pi}[(R_i)_{i\in\Z}]$:
 $$
 \bigl\{\wh{\pi}[(R_i)_{i\in\Z}]\bigr\}=\bigcap_{n=0}^\infty \ov{_{-n}[R_{-n},\ldots,R_n]}
 $$

\begin{thm}\label{Theorem_Pi_Hat}
$\wh{\pi}:\wh{\Sigma}\to M$ has the following properties:
\begin{enumerate}
\item $\wh{\pi}\circ\s=f\circ\wh{\pi}$;
\item $\wh{\pi}$ is H\"older continuous;
\item $\wh{\pi}(\wh{\Sigma})\supset\wh{\pi}(\wh{\Sigma}^\#)\supset\NUH_\chi^\#(f)$, therefore the image of $\wh{\pi}$ has full measure w.r.t every ergodic invariant probability measure with entropy larger than $\chi$;
\end{enumerate}
\end{thm}
\begin{proof}
The commutation relation   is because for every $\un{R}=(R_i)_{i\in\Z}$ in $\wh{\Sigma}$,
\begin{align*}
\{\pi[\s(\un{R})]\}&=\bigcap_{n=0}^\infty\ov{_{-n}[R_{-n+1},\ldots,R_{n+1}]}\supset \bigcap_{n=0}^\infty\ov{_{-n-2}[R_{-n-1},\ldots,R_{n+1}]}\\
&=\bigcap_{n=0}^\infty \ov{\bigcap_{k=-n-2}^n f^{-k}(R_{k+1})}=\bigcap_{N=0}^\infty\ov{f\left(_{-N}[R_{-N},\ldots,R_{N}]\right)}\\
&=\bigcap_{N=0}^\infty f\left(\ov{_{-N}[R_{-N},\ldots,R_{N}]}\right),\ \textrm{ because $f$ is a homeomorphism}\\
&=f\left(\bigcap_{N=0}^\infty \ov{_{-N}[R_{-N},\ldots,R_{N}]}\right),\ \textrm{ because $f$ is a bijection}\\
&\equiv f\left(\{\pi(\un{R})\}\right)=\{f[\pi(\un{R})]\}.
\end{align*}

The H\"older continuity of $\pi$ is because if $\un{R},\un{S}\in\wh{\Sigma}$ and $R_i=S_i$ for all $|i|\leq N$, then $\wh{\pi}(\un{R}), \wh{\pi}(\un{S})\in {_{-N}[R_{-N},\ldots,R_N]}$, whence by Lemma \ref{Lemma_Cylinder_Diameter}
$$
d(\wh{\pi}(\un{R}),\wh{\pi}(\un{S}))\leq \diam ({_{-N}[R_{-N},\ldots,R_N]})\leq C\theta^N.
$$

Finally  we claim that $\wh{\pi}(\wh{\Sigma})$ and $\wh{\pi}(\wh{\Sigma}^\#)$  contain $\NUH_\chi^\#(f)$.
Suppose $x\in\NUH_\chi^\#(f)$. By Theorem \ref{Theorem_Markov_Extension},   $\pi(\Sigma^\#)\supset\NUH_\chi^\#(f)$, therefore there exists  a chain $\un{v}\in\Sigma^\#$ s.t. $\pi(\un{v})=x$.  $\Sigma^\#$ is $\s$--invariant and $f\circ\pi=\pi\circ\s$, so
$
f^i(x)\in\pi(\Sigma^\#)
$
for all $i\in\Z$. The collection $\mathfs R$ covers $\pi(\Sigma^\#)$, therefore for every $i\in\Z$ there is some $R_i\in\mathfs R$ s.t. $f^i(x)\in R_i$. Obviously $R_i\to R_{i+1}$, so $\un{R}:=(R_i)_{i\in\Z}$ belongs to $\wh{\Sigma}$. Also,
$$
x\in \bigcap_{n=0}^\infty \ov{{_{-n}[R_{-n},\ldots,R_n]}}
$$
(even without the closure), so $x=\pi(\un{R})$. It follows that $\wh{\pi}(\wh{\Sigma})\supset\NUH_\chi^\#(f)$.

We claim that the sequence $\un{R}$ which was constructed above belongs to $\wh{\Sigma}^\#$, and deduce that  $\wh{\pi}(\wh{\Sigma}^\#)\supset\NUH_\chi^\#(f)$.

The sequence $\un{v}$ is in $\Sigma^\#$ by construction, therefore there exists $v$ and $u$ s.t. $v_i=u$ for infinitely many  negative $i$, and $v_i=v$ for infinitely many positive $i$.

The sets $R_i$ and $Z(v_i)$ intersect, because  they both contain $f^i(x)$. By Lemma \ref{Lemma_R_In_Z}, $R_i\subset Z(v_i)$ for all $i\in\Z$.
It follows that there are infinitely many negative $i$ s.t. $R_i\subset Z(u)$, and infinitely many positive $i$ s.t. $R_i\subset Z(v)$.

 The sets $\mathfs R(w):=\{R\in\mathfs R: R\subset Z(w)\}$ $(w=u,v)$ are finite (Lemma \ref{Lemma_R_In_Z}). Therefore $\exists n_k\uparrow\infty$ and $\exists R\in\mathfs R(v)$ s.t. $R_{n_k}=R$ for all $k$, and  $\exists m_k\uparrow\infty$ and $\exists S\in\mathfs R(u)$ s.t. $R_{-m_k}=S$ for all $k$. Thus $\un{R}\in\wh{\Sigma}^\#$ as required.
\end{proof}

The following result is not needed for the purposes of this paper, but we  anticipate some future applications.
\begin{proposition}
For every $x\in \wh{\pi}(\wh{\Sigma})$, $T_x M=E^s(x)\oplus E^u(x)$ where
\begin{itemize}
\item[(a)] $\limsup\limits_{n\to\infty}\frac{1}{n}\log\|df^n_x\un{v}\|_{f^n(x)}\leq -\frac{\chi}{2}$ on $E^s(x)\setminus\{\un{0}\}$;
\item[(b)] $\limsup\limits_{n\to\infty}\frac{1}{n}\log\|df^{-n}_x\un{v}\|_{f^{-n}(x)}\leq -\frac{\chi}{2}$ on $E^u(x)\setminus\{\un{0}\}$.
\end{itemize}
The maps $\un{R}\mapsto E^{u/s}(\wh{\pi}(\un{R}))$   are H\"older continuous as maps from $\wh{\Sigma}$ to $TM$.
\end{proposition}
\begin{proof}
Suppose $x=\wh{\pi}(\un{R})$ where $\un{R}\in\wh{\Sigma}$. By Lemma \ref{Lemma_G_G_hat}, there is a chain $(v_i)_{i\in\Z}$ s.t. $R_i\subset Z(v_i)$ for all $i$ and $_{-n}[R_{-n},\ldots,R_n]\subset Z_{-n}(v_{-n},\ldots,v_n)$ for every $n$. Then $f^n(x)\in \ov{Z(v_n)}$ for all $n$. Every element of $Z(v_n)$ is the intersection of $s/u$--admissible manifolds in $v_n$, so if $v_n=\Psi_{x_n}^{p^u_n,p^s_n}$, then $\ov{Z(v_n)}\subset \Psi_{x_n}[R_{p^s_n\wedge p^u_n}(\un{0})]$ (Proposition \ref{Prop_Intersection} (2)). By Proposition \ref{Prop_V} (4),  $x\in V^u[(v_i)_{i\leq 0}]\cap V^s[(v_i)_{i\geq 0}]$.

Let $E^s(x):=T_x V^s[(v_i)_{i\geq 0}]$ and $E^u(x):=T_x V^u[(v_i)_{i\leq 0}]$. These spaces satisfy (a) and (b), because they are tangent to admissible manifolds which stay in windows  (Proposition  \ref{Prop_Staying_In_Windows}). This definition of $E^s(x), E^u(x)$  is independent of the choice of $(v_i)_{i\in\Z}$, because there can be only one  decomposition of $T_x M$ into two spaces which satisfy (a) and (b).

Suppose $x=\wh{\pi}(\un{R})$ and $y=\wh{\pi}(\un{S})$ where $R_i=S_i$ for $i=-N,\ldots,N$, and let $\un{v}=(v_i)_{i\in\Z}$ be as before.
The argument in the first paragraph shows that $x=\pi(\un{v})$.
We claim that $y=\pi(\un{w})$ where $\un{w}$ is a chain s.t. $w_i=v_i$ for all $|i|\leq N$.

By assumption, $y\in \ov{_{-n}[S_{-N},\ldots,S_N]}=\ov{_{-n}[R_{-N},\ldots,R_N]}\subset\ov{Z_{-N}(v_{-N},\ldots,v_N)}$, so $y=\lim\pi(\un{w}^{(n)})$ where  $\un{w}^{(n)}\in\Sigma$ satisfy $w^{(n)}_i=v_i$ for all $|i|\leq N$.  Since every vertex of $\mathfs G$ has finite degree, each of the sets $\{w^{(n)}_i:n\in\N\}$ is finite. It follows that  there is a convergent subsequence $\un{w}^{(n_k)}\xrightarrow[k\to\infty]{}\un{w}$. The limit is a chain $\un{w}$ s.t. $y=\pi(\un{w})$ and $w_i=v_i$ for all $|i|\leq N$.

Write $v_0=\Psi_{x_0}^{p^u_0, p^s_0}$, and let $F_{u},F_s$ be the representing functions in $\Psi_{x_0}$ for  $V^{u}[(v_i)_{i\leq 0}]$, $V^s[(v_i)_{i\geq 0}]$.   Let $G_u, G_s$ be the representing functions  for  $V^{u}[(w_i)_{i\leq 0}]$, $V^s[(w_i)_{i\geq 0}]$. By Proposition \ref{Prop_V}(5),
\begin{equation}\label{FG}
\begin{aligned}
\|F_u-G_u\|_\infty+\|F_u'-G_u'\|_\infty&<K\theta^N\\
\|F_s-G_s\|_\infty+\|F_s'-G_s'\|_\infty&<K\theta^N
\end{aligned}
\end{equation}
for some global constants $K>0, \theta\in(0,1)$.

The intersection of the (vertical) graph of $F_u$ and the (horizontal) graph of $F_s$ is the point $\un{\xi}\in\R^2$ s.t. $\Psi_{x_0}(\un{\xi})=x$. The intersection of the vertical and horizontal graphs of $G_u$ and $G_s$ is the point $\un{\eta}\in\R^2$ s.t. $\Psi_{x_0}(\un{\eta})=y$. (\ref{FG}) and  Proposition \ref{Prop_Intersection} imply that  $\|\un{\xi}-\un{\eta}\|<3K\theta^N$.

By admissibility, $F_u, F_s, G_u, G_s$ have $\frac{\b}{3}$--H\"older exponent at most $\frac{1}{2}$. Together with (\ref{FG}), this implies
$
|F_s'(\xi_1)-G_s'(\eta_1)|,|F_u'(\xi_2)-G_u'(\eta_2)|<\frac{1}{2}(3K\theta^N)^{\b/3}+K\theta^N.
$ It follows that $\dist_{T\R^2}\bigl(T_{\un{\xi}}[\mathrm{graph}(F_s)],T_{\un{\eta}}[\mathrm{graph}(G_s)]\bigr)=O(\theta^{\frac{1}{3}\b N})$.

$E^s(x), E^s(y)$ are the  images of $T_{\un{\xi}}[\mathrm{graph}(F_s)]$ and $T_{\un{\eta}}[\mathrm{graph}(G_s)]$ under $d\Psi_{x_0}$. Since the norm of the differential of a Pesin chart is bounded above by two, $\dist_{TM}(E^s(x), E^s(y))=O(\theta^{\frac{1}{3}\b N})$. Similarly,  $\dist_{TM}(E^u(x), E^u(y))=O(\theta^{\frac{1}{3}\b N})$. All implied constants are uniform, so  $\un{R}\mapsto E^{s/u}(\wh{\pi}(\un{R}))$ are H\"older continuous.
\end{proof}

\subsection{The extension is finite-to-one} Say that $R,R'\in\mathfs R$  are {\em affiliated}, if there exist $Z,Z'\in\mathfs Z$ s.t. $R\subset Z$, $R'\subset Z'$, and $Z\cap Z'\neq\emptyset$.  For every $R\in\mathfs R$, let
$$
N(R):=|\{(R',Z')\in\mathfs R\x\mathfs Z: \textrm{$R'$ is affiliated to $R$ and }Z'\textrm{ contains } R'
\}|.
$$
\begin{lem}
$N(R)<\infty$.
\end{lem}
\begin{proof}
Suppose $R\in\mathfs R$.
The set $A(R):=\{Z\in\mathfs Z: Z\supset R\}$ is finite, because if
 $Y\in\mathfs Z$ contains $R$ then every $Z\in A(R)$ intersects $Y$, and the number of such $Z$ is finite (Theorem \ref{Theorem_LF}).

Since $A(R)$ is finite,  $B(R):=\{Z'\in\mathfs Z:\exists Z\in A(R)\textrm{ s.t. }Z'\cap Z\neq \emptyset\}$ is finite (Theorem \ref{Theorem_LF}). For every $Z'\in B$ there are at most finitely many $R'\in\mathfs R$ s.t. $R'\subset Z'$ (Lemma \ref{Lemma_R_In_Z}). Therefore, $C(R):=\{R'\in\mathfs R:\textrm{ $R, R'$ are affiliated}\}$ is finite. It follows that
$
N(R)= \sum_{R'\in C(R)}|A(R')|<\infty$.
\end{proof}

\begin{thm}\label{Theorem_Finite_To_One}
Every $x\in \wh{\pi}(\wh{\Sigma}^\#)$ has a finite number of $\wh{\pi}$--pre-images.
If $x=\wh{\pi}(\un{R})$ where $R_i=R$ for infinitely many $i<0$ and $R_i=S$ for infinitely many $i>0$, then $|\wh{\pi}^{-1}(x)|<\vf_\chi(R,S):=N(R)N(S)$.
\end{thm}
\begin{proof}
The  proof is based on an idea of Bowen's \cite[pp. 13--14]{Bsimple} (see also   \cite[page 229]{PP}), who used it in  the context of Axiom A diffeomorphisms. We show that the product structure described above is sufficient to implement his argument in our setting.

\medskip
Suppose $x\in\wh{\pi}(\wh{\Sigma}^\#)$, then $x$ has a $\wh{\pi}$--preimage $\un{R}\in\wh{\Sigma}$ s.t. $R_i=R$ for infinitely many negative $i$, and $R_i=S$ for infinitely many positive $i$. We show that the number of $\wh{\pi}$--pre-images of $x$ is less than or equal to $N:=N(R)N(S)$.

Suppose by way of contradiction that there are $N+1$ different points in $\wh{\Sigma}$ whose image under $\wh{\pi}$ is equal to $x$. Call these points $\un{R}^{(j)}=(R_i^{(j)})_{i\in\Z}$ $(j=0,\ldots,N)$. Assume w.l.o.g. that $\un{R}^{(0)}=\un{R}$.

By Lemma \ref{Lemma_G_G_hat} there are  chains $\un{v}^{(j)}=(v_i^{(j)})_{i\in\Z}\in\Sigma$ s.t. for every $n$
\begin{equation}\label{Schwester}
R^{(j)}_n\subset Z(v_n^{(j)})\textrm{ and }
_{-n}[R^{(j)}_{-n},\ldots,R^{(j)}_n]\subset Z_{-n}(v_{-n}^{(j)},\ldots,v_n^{(j)}).
\end{equation}

\medskip
\noindent
{\em Claim 1.\/} $\pi(\un{v}^{(j)})=x$ for every $0\leq j\leq N$.

\medskip
The following inclusions hold:
\begin{align}
\pi(\un{v}^{(j)})&\in  \bigcap_{n=0}^\infty{Z_{-n}(v_{-n}^{(j)},\ldots,v_n^{(j)})}\subset
\bigcap_{n=0}^\infty\ov{Z_{-n}(v_{-n}^{(j)},\ldots,v_n^{(j)})},\label{Brun}
\end{align}
\begin{align}
x=\wh{\pi}(\un{R}^{(j)})&\in \bigcap_{n=0}^\infty \ov{_{-n}[R^{(j)}_{-n},\ldots,R^{(j)}_n]}\subset
\bigcap_{n=0}^\infty\ov{Z_{-n}(v_{-n}^{(j)},\ldots,v_n^{(j)})}.\notag
\end{align}
Since $\pi$ is H\"older continuous,   $\diam\biggl[\ov{Z_{-n}(v_{-n}^{(j)},\ldots,v_n^{(j)})}\biggr]\xrightarrow[n\to\infty]{}0$, so $\pi(\un{v}^{(j)})=x$.

\medskip
\noindent
{\em Claim 2\/}: Suppose $i\in\Z$, then  $R_{i}^{(0)},\ldots,R_{i}^{(N)}$  are  affiliated.

\medskip
\noindent

\medskip
\noindent
{\em Proof.\/} By (\ref{Brun})  $x=\pi(\un{v}^{(j)})\in\bigcap_{n=0}^\infty{Z_{-n}(v_{-n}^{(j)},\ldots,v_n^{(j)})}$, so
 $f^i(x)\in Z(v_i^{(j)})$. Thus $Z(v_i^{(0)}), \ldots,Z(v_i^{(N)})$ have a common intersection. Since  $R^{(j)}_i\subset Z(v^{(j)}_i)$, $R^{(0)}_i,\ldots,R^{(N)}_i$ are affiliated.

\medskip
\noindent
{\em Claim 3\/}: There exist $k,\ell\geq 0$ and $0\leq j_1,j_2\leq N$ such that
\begin{itemize}
\item $(R^{(j_1)}_{-k},\cdots,R^{(j_1)}_\ell)\neq (R^{(j_2)}_{-k},\cdots,R^{(j_2)}_\ell)$;
\item $R^{(j_1)}_{-k}=R^{(j_2)}_{-k}$ and $R^{(j_1)}_\ell=R^{(j_2)}_\ell$;
\item $v^{(j_1)}_{-k}=v^{(j_2)}_{-k}$ and $v^{(j_1)}_\ell=v^{(j_2)}_\ell$.
\end{itemize}

\medskip
\noindent
{\em Proof.\/} We are assuming that   $\un{R}^{(j)}$ are different, therefore there exists some $m$ such that the words
$
(R_{-m}^{(j)},\ldots,R_m^{(j)})
$ $(0\leq j\leq N)$ are different.

We are assuming that $R^{(0)}_i$ equals $R$ for infinitely many negative $i$, and equals $S$ for infinitely many positive $i$. Choose $k,\ell\geq m$ s.t. $R_{-k}^{(0)}=R$ and $R_\ell^{(0)}=S$.
The words
$
(R_{-k}^{(j)},\ldots,R_\ell^{(j)})
$ $(0\leq j\leq N)$ are different.

By claims 1 and 2, $R_{-k}^{(j)}$ are all affiliated to $R_{-k}^{(0)}=R$, and by (\ref{Schwester}) $R_{-k}^{(j)}\subset Z(v_{-k}^{(j)})$, therefore
$
\bigl|\{(R_{-k}^{(j)},v_{-k}^{(j)}):j=0,\ldots,N\}
\bigr|\leq N(R).
$
In the same way, one can show that $
\bigl|\{(R_{\ell}^{(j)},v_{\ell}^{(j)}):j=0,\ldots,N\}
\bigr|\leq N(S)$. It follows that
$$
\bigl|
\{\bigl(R_{-k}^{(j)},v_{-k}^{(j)};R_{\ell}^{(j)},v_{\ell}^{(j)}\bigr):j=0,\ldots,N\}\bigr|\leq N(R)N(S)=N.
$$
By the pigeonhole principle, at least two quadruples coincide, proving the claim.

\medskip
To ease up the notation, we let $\un{A}:=\un{R}^{(j_1)}$, $\un{B}:=\un{R}^{(j_2)}$, $\un{a}:=\un{v}^{(j_1)}$ and $\un{b}:=\un{v}^{(j_2)}$, and we write
$
A_{-k}=B_{-k}=:B\ , \ A_\ell=B_\ell=:A\ , \ a_{-k}=b_{-k}=:b\ , a_\ell=b_\ell=:a.
$
By Lemma \ref{Lemma_Non_Empty_Cylinders}, there are  two points
$$
x_A\in {_{-k}[A_{-k},\ldots,A_\ell]}\textrm{ and }
x_B\in {_{-k}[B_{-k},\ldots,B_\ell]}.
$$
By definition, $f^{-k}(x_A),f^{-k}(x_B)\in B\subset Z(b)$ and $f^{\ell}(x_A),f^{\ell}(x_B)\in A\subset Z(a)$.
Define two points $z_A,z_B$ by the equations
\begin{align*}
f^{-k}(z_A)&\in W^u(f^{-k}(x_B),B)\cap W^s(f^{-k}(x_A),B);\\
f^{\ell}(z_B)&\in W^u(f^{\ell}(x_B),A)\cap W^s(f^{\ell}(x_A),A).
\end{align*}

\medskip
\noindent
{\em Claim 4.\/} $z_A\neq z_B$.

\medskip
\noindent
{\em Proof.\/} By construction, $f^{-k}(z_A)\in W^s(f^{-k}(x_A),A_{-k})$.   By the Markov property (Theorem \ref{Theorem_Markov_Property}),
\begin{align*}
f^{-k+1}(z_A)&\in f[W^s(f^{-k}(x_A),A_{-k})]\subset W^s(f^{-k+1}(x_A),A_{-k+1})\\
f^{-k+2}(z_A)&\in f[W^s(f^{-k+1}(x_A),A_{-k+1})]\subset W^s(f^{-k+2}(x_A),A_{-k+2})
\end{align*}
and so on. It follows that $f^{-k}(z_A)\in {_{-k}[A_{-k},\ldots,A_\ell]}$.
Similarly, if we start from $f^\ell(z_B)\in W^u(f^\ell(x_B),B_\ell)$ and apply $f^{-1}$ repeatedly, then the Markov property will give us that $f^{-k}(z_B)\in {_{-k}[B_{-k},\ldots,B_\ell]}$.

But
$
(A_{-k},\ldots,A_\ell)\equiv(R^{(j_1)}_{-k},\ldots,R^{(j_1)}_\ell)\neq (R^{(j_2)}_{-k},\ldots,R^{(j_2)}_\ell)\equiv (B_{-k},\ldots,B_\ell)
$, and the elements of $\mathfs R$ are pairwise disjoint,
so ${_{-k}[A_{-k},\ldots,A_\ell]}\cap {_{-k}[B_{-k},\ldots,B_\ell]}=\emptyset$ and $z_A\neq z_B$.

\medskip
\noindent
{\em Claim 5.\/} $z_A=z_B$ (a contradiction).

\medskip
\noindent
{\em Proof.\/} We saw above that  $f^{-k}(z_A)\in {_{-k}[A_{-k},\ldots,A_\ell]}$ , $f^{-k}(z_B)\in {_{-k}[B_{-k},\ldots,B_\ell]}$. In particular, $f^{-k}(z_B)\in B_{-k}=B\subset Z(b)$ and $f^\ell(z_A)\in A_\ell=A\subset Z(a)$.

Construct chains $\un{\a},\un{\b}\in\Sigma^\#$ such that
$
z_A=\pi(\un{\a})\textrm{ , }\a_{\ell}=a\textrm{ and }
z_B=\pi(\un{\b})\textrm{ , }\b_{-k}=b.
$
Define a sequence $\un{c}$ by
$$
c_i=\begin{cases}
\b_i & i\leq -k\\
a_i & -k+1\leq i\leq \ell-1\\
\a_i & i\geq \ell.
\end{cases}
$$
This is a chain because $\b_{-k}=b=a_{-k}$ and $\a_\ell=a=a_\ell$. This chain belongs to $\Sigma^\#$, because $\un{\a},\un{\b}\in\Sigma^\#$. We write $c_i:=\Psi_{x_i}^{p^u_i,p^s_i}$.

We claim that $f^{-k}(z_A), f^{-k}(z_B)\in V^u[(c_i)_{i\leq -k}]$. Note firstly that both points belong to $W^u(f^{-k}(x_B),B)$: $f^{-k}(z_A)$ by definition, and $f^{-k}(z_B)$ because of  the inclusion
$f^{\ell}(z_B)\in W^u(f^\ell(x_B),B_\ell)$ and the Markov property. Since $B\subset Z(b)$,
$$
W^u(f^{-k}(x_B),B)\subset V^u(f^{-k}(x_B),Z(b))= V^u[(\b_i)_{i\leq -k}]\equiv V^u[(c_i)_{i\leq -k}].
$$
It follows that $f^{-k}(z_A), f^{-k}(z_B)\in V^u[(c_i)_{i\leq -k}]$.

This together with the fact that  $f^{-k}(z_A), f^{-k}(z_B)\in Z(b)=Z(c_{-k})$ implies that
\begin{equation}\label{rach}
f^{i}(z_A),f^{i}(z_B)\in Z(c_i)\subset \Psi_{x_i}[R_{p^u_i\wedge p^s_i}(\un{0})]\textrm{ for all }i\leq -k.
\end{equation}
Similarly, one can show that $f^\ell(z_A),f^\ell(z_B)\in V^s[(c_i)_{i\geq \ell}]$, whence
\begin{equation}\label{ma}
f^{i}(z_A),f^{i}(z_B)\in Z(c_i)\subset \Psi_{x_i}[R_{p^u_i\wedge p^s_i}(\un{0})]\textrm{ for all }i\geq \ell.
\end{equation}
Using the inclusions $f^{-k}(z_A)\in {_{-k}[A_{-k},\ldots,A_\ell]}$, $f^{-k}(z_B)\in {_{-k}[B_{-k},\ldots,B_\ell]}$ (see the proof of claim 4), we see that if $-k<i<\ell$ then
$f^i(z_A),f^i(z_B)\in A_i\cup B_i$. Therefore $f^i(z_A),f^i(z_B)\in Z(a_i)\cup Z(b_i)$. The sets $Z(a_i)$, $Z(b_i)$ intersect, because by claim 1
$
f^i(x)=\pi[\s^i(\un{a})]=\pi[\s^i(\un{b})]\in Z(a_i)\cap Z(b_i).
$ Thus by Lemma \ref{Lemma_Overlapping_Z},
\begin{equation}\label{ninov}
f^{i}(z_A),f^{i}(z_B)\in Z(a_i)\cup Z(b_i)\subset \Psi_{x_i}[R_{Q_\e(x_i)}(\un{0})]\textrm{ for all }-k< i< \ell.
\end{equation}

In summary, $f^i(z_A),f^i(z_B)\in \Psi_{x_i}[R_{Q_\e(x_i)}(\un{0})]$, where $c_i=\Psi_{x_i}^{p^u_i,p^s_i}$ is a chain. By Proposition \ref{Prop_V}(4),
$
z_A,z_B\in V^u[(c_i)_{i\leq 0}]\cap V^s[(c_i)_{i\geq 0}].
$
So $z_A=\pi(\un{c})=z_B$, and the claim is proved.

\medskip
The contradiction between claims 4 and 5 shows that $x$ cannot have more than $N$ pre-images.
\end{proof}
\section{Invariant measures}\addtocounter{subsection}{1}
Let $\s:\wh{\Sigma}\to\wh{\Sigma}$ denote the finite-to-one Markov extension of $f$ which we constructed in part \ref{Part_Finite_To_One}.
We compare the invariant Borel measures of $\s:\wh{\Sigma}\to\wh{\Sigma}$ to the invariant Borel measures of $f:M\to M$. We restrict our attention to measures whose entropy is larger than $\chi$.

\begin{prop}
Suppose $\wh{\mu}$ is an ergodic  Borel probability measure on $\wh{\Sigma}$, then $\mu:=\wh{\mu}\circ\wh{\pi}^{-1}$ is an ergodic Borel probability measure on $M$, and $h_\mu(f)=h_{\wh{\mu}}(\s)$
\end{prop}
\begin{proof}
It is clear that $\mu$ is well-defined,  ergodic and invariant.

By Poincar\'e's Recurrence Theorem there exists a vertex $R\in\mathfs R$ s.t.
$$
\Upsilon:=\{\un{R}\in\wh{\Sigma}:\exists n_k,m_k\uparrow\infty\textrm{ s.t. }R_{n_k}, R_{-m_k}=R\}
$$ has full measure with respect to $\wh{\mu}$. The map $\wh{\pi}:\Upsilon\to M$ is bounded-to-one (the bound is $\vf_\chi(R,R)$).
Finite extensions preserve entropy, so $h_\mu(f)=h_{\wh{\mu}}(\s)$.
\end{proof}

The other direction, ``every invariant measure $\mu$ supported on $\wh{\pi}(\wh{\Sigma})$ lifts to an invariant measure on $\wh{\Sigma}$", is less clear.\footnote{$\mu\circ\wh{\pi}$ does not work: it is not even $\s$--additive.} Lifting measures to  Markov extensions is a difficult issue in general, and it  has  received considerable attention (see e.g.  \cite{Hof},\cite{Ke},\cite{Br},\cite{BT},\cite{PSZ},\cite{BuzziMulti},\cite{Z}). But our case is very   simple, because our Markov extension is finite-to-one.

Indeed, suppose $\mu$ is an ergodic $f$--invariant probability measure on $M$ s.t. $h_\mu(f)>\chi$. Define $\wt{\mu}$ by
\begin{equation}\label{Mu Tilde}
\wt{\mu}(E):=\int_M\biggl(\frac{1}{|\wh{\pi}^{-1}(x)|}\sum_{\wh{\pi}(\un{R})=x}1_E(\un{R})\biggr)d\mu(x).
\end{equation}

\begin{prop}
Suppose $\mu$ is an ergodic $f$--invariant Borel probability measure on $M$ s.t. $h_\mu(f)>\chi$.
\begin{enumerate}
\item  $\wt\mu$ is a well--defined $\s$--invariant Borel probability measure on $\wh{\Sigma}$.
\item Almost every ergodic component $\wh{\mu}$ of $\wt{\mu}$ is an ergodic $\s$--invariant probability measure such that   $\wh{\mu}\circ\wh{\pi}^{-1}=\mu$ and  $h_{\wh{\mu}}(\s)=h_\mu(f)$.
\end{enumerate}
\end{prop}
\begin{proof}
The first thing to do is to verify that the integrand in (\ref{Mu Tilde}) is  measurable.
We recall some basic facts from set theory (see e.g. \cite[\S 4.5, \S 4.12]{Sriv}): Let $X,Y$ be two complete separable metric spaces.
 \begin{enumerate}
\item[(I)] $F:X\to Y$ is Borel iff $\graph(F)$ is a Borel subset of $X\x Y$.
\item[(II)] Suppose $F:X\to Y$ is Borel and countable-to-one (i.e. $F^{-1}(y)$ is finite or countable for all $y\in Y$). If $E\subset X$ is Borel, then $F(E)\subset Y$ is Borel.
\item[(III)] Lusin's Theorem: Suppose $B\subset X\x Y$ is Borel. If  $B_x:=\{y:(x,y)\in B\}$ is finite or countable for every $x\in X$, then $B$ is a countable disjoint union of Borel graphs of partially defined Borel functions.
 \end{enumerate}

Since $h_\mu(f)>\chi$, $\mu$ is carried by $\wh{\pi}(\Sigma^\#)$. Since $\wh{\pi}:\Sigma^\#\to M$ is finite-to-one, $\wh{\pi}(\Sigma^\#)$ is Borel. Henceforth we work inside  $\wh{\pi}(\Sigma^\#)$.

\medskip
\noindent
{\em Step 1.\/} $x\mapsto |\wh{\pi}^{-1}(x)|$ is constant on a Borel set $\Omega$ s.t.  $\mu(\Omega)=1$.

\medskip
\noindent
{\em Proof.\/} Since $\wh{\pi}\circ\s=f\circ\wh{\pi}$ and $f$ is a bijection, $x\mapsto |\wh{\pi}^{-1}(x)|$ is $f$--invariant.

We show that  the restriction of $x\mapsto |\wh{\pi}^{-1}(x)|$ to $\wh{\pi}(\Sigma^\#)$  is  Borel measurable. The claim will then follow from the ergodicity of $\mu$.

 Graphs of Borel functions are Borel, therefore
$
B:=\{(\wh{\pi}(\un{R}),\un{R}):\un{R}\in\wh{\Sigma}^\#\}
$
is a Borel subset of $M\x \wh{\Sigma}$.

By Lusin's theorem, $\exists$ partially defined Borel functions $\vf_n: M_n\to \wh{\Sigma}^\#$ s.t. $M_n$ are pairwise disjoint Borel subsets of $M$ and
$B=\{(x,\vf_n(x)):x\in M_n,\ n\in\N\}$. In particular,
$
\wh{\pi}^{-1}(x)=\{\vf_i(x):i\in\N \textrm{ s.t. } M_i\owns x\}.
$
The graphs of $\vf_n$ are pairwise disjoint, so $i\neq j\Rightarrow \vf_i(x)\neq \vf_j(x)$. Consequently,
 $$
 |\wh{\pi}^{-1}(x)|=\sum_{i=1}^\infty 1_{M_i}(x)\textrm{ on }\wh{\pi}(\wh{\Sigma}^\#).
 $$
Since $M_i$ are Borel, $x\mapsto |\wh{\pi}^{-1}(x)|$ is Borel on $\wh{\pi}(\wh{\Sigma}^\#)$.

\medskip
\noindent
{\em Step 2.\/} Let $\Upsilon:=\wh{\pi}^{-1}(\Omega)$ and let $N$ denote the number of pre-images of points $x\in\Omega$. There exists a Borel partition $\Upsilon=\biguplus_{i=1}^N\Upsilon_i$ such that
$\wh{\pi}:\Upsilon_i\to \Omega$ is one-to-one and onto for every $i$.

\medskip
\noindent
{\em Proof.\/} This is a consequence of Lusin's Theorem.

Let $B_1:=\{(\wh{\pi}(\un{y}),\un{y}):\un{y}\in\wh{\pi}^{-1}(\Omega)\}$. Each $x$--fibre of $B_1$ has $N$ elements.  By Lusin's Theorem $B_1=\biguplus_{n\geq 1}\graph(\vf_n)$ where $\vf_n:M_n\to\wh{\Sigma}$ are Borel. $\Omega=\biguplus_{n\geq 1}M_n$.

Define $\psi_1:\Omega\to \wh{\Sigma}$ by $\psi_1=\vf_i$ on $M_i\setminus\bigcup_{j<i}M_j$ $(i\in\N)$, then $\psi_1$ is Borel and $\psi_1(x)\in\wh{\pi}^{-1}(x)$ for all $x$. Since $\wh{\pi}\circ \psi_1=\id$, $\psi_1$ is
one-to-one. It follows that $\Upsilon_1:=\psi_1(\Omega)$ is Borel, and $\wh{\pi}:\Upsilon_1\to \Omega$ is one-to-one and onto.

Now take $B_2:=B_1\setminus\graph\psi_1$. Each $x$--fibre of $B_2$ has $N-1$ elements, and $B_2$ is disjoint from $\graph(\psi_1)$. Apply  the previous process to $B_2$ to obtain  $\Upsilon_2$. After $N$ steps, we are done.

\medskip
\noindent
{\em Step 3.\/} The restriction of the integrand in (\ref{Mu Tilde}) to $\Omega$ is Borel measurable.

\medskip
\noindent
{\em Proof.\/}
Every $x\in\Omega$ has exactly $N$ pre-images, one in every $\Upsilon_i$. It follows that for every Borel set $E\subset \wh{\Sigma}$,
$$
\frac{1}{|\wh{\pi}^{-1}(x)|}\sum\limits_{\wh{\pi}(\un{y})=x} 1_E(\un{y})=\frac{1}{N}\sum_{i=1}^N 1_{\wh{\pi}(E\cap\Upsilon_i)}(x)\textrm{ on }\Omega.
$$
Since $\wh{\pi}$ is one-to-one on $\Upsilon_i$, $\wh{\pi}(E\cap\Upsilon_i)$ is a Borel set. It follows that the right-hand-side is Borel measurable.

\medskip
\noindent
{\em Step 4.\/} $\wt{\mu}$ is an invariant Borel probability measure such that $\wt{\mu}\circ\wh{\pi}^{-1}=\mu$ and $h_{\wt{\mu}}(\s)=h_\mu(f)$.

\medskip
\noindent
{\em Proof.\/} We saw that $\wh{\mu}(E)$ is well--defined for all Borel sets $E\subset\wh{\Sigma}$. This set function is obviously $\s$--additive, and it is clear that $\wt{\mu}(\wh{\Sigma})=1$. Thus $\wt{\mu}$ is a Borel probability measure.

This measure is $\s$--invariant, because
\begin{align*}
\wt{\mu}(\s^{-1}E)&=\int_M\biggl(\frac{1}{|\wh{\pi}^{-1}(x)|}
\sum_{\wh{\pi}(\un{R})=x}1_E(\s(\un{R}))\biggr)d\mu(x)\\
&=\int_M\biggl(\frac{1}{|\wh{\pi}^{-1}(f(x))|}
\sum_{\wh{\pi}(\s\un{R})=f(x)}1_E(\s(\un{R}))\biggr)d\mu(x)\ \ (\because \wh{\pi}\circ\s=f\circ\wh{\pi})
\end{align*}
\begin{align*}
&=\int_M\biggl(\frac{1}{|\wh{\pi}^{-1}(f(x))|}
\sum_{\wh{\pi}(\un{S})=f(x)}1_E(\un{S})\biggr)d\mu(x)\\
&=\wt{\mu}(E) \ \ (\because \mu\circ f^{-1}=\mu).
\end{align*}
It is a lift of $\mu$ because
\begin{align*}
\wt{\mu}(\wh{\pi}^{-1}E)&=\int_M\biggl(\frac{1}{|\wh{\pi}^{-1}(x)|}
\sum_{\wh{\pi}(\un{R})=x}1_E(\wh{\pi}(\un{R}))\biggr)d\mu(x)=\int_M 1_E(x) d\mu(x)=\mu(E).
\end{align*}
Finally $\wt{\mu}$ and $\mu$ have the same entropy, because $\wh{\pi}$ is $N$--to--one on a set of full measure, and finite extensions preserve entropy.

\medskip
\noindent
{\em Step 5.\/} Almost every ergodic component of $\wt{\mu}$ satisfies $\wh{\mu}\circ\wh{\pi}^{-1}=\mu$ and $h_{\wh{\mu}}(\s)=h_\mu(\mu)$.

\medskip
\noindent
Let $\wt{\mu}=\int \wh{\mu}_y d\nu(y)$ be the ergodic decomposition of $\wt{\mu}$, then
$\mu=\wt{\mu}\circ\wh{\pi}^{-1}=\int \wh{\mu}_y\circ\wh{\pi}^{-1} d\nu_y$. Each of the measures
$\wh{\mu}_y\circ\wh{\pi}^{-1}$ is $f$--invariant. Since $\mu$ is ergodic,  $\wh{\mu}_y\circ\wh{\pi}^{-1}=\mu$ for a.e. $y$.

The equality of the entropies follows as before from the fact that finite extensions preserve entropy.
\end{proof}

\appendix\part{Appendix: Proofs of standard results in Pesin Theory}
\setcounter{section}{1}
\addtocounter{equation}{-1}
\medskip
\noindent
{\bf Proof of Theorem \ref{TheoremOPR}} (Compare with  Theorem {3.5.5} in \cite{BP}.) The idea is to evaluate  $A_\chi(x):=C_\chi(f(x))^{-1}\circ df_x\circ C_\chi(x)$ on the standard basis of $\R^2$.

We start from the identity $df_x E^s(x)=E^s(f(x))$. Both sides of the equation are one--dimensional, therefore
$
df_x \un{e}^s(x)=\pm\|df_x\un{e}^s(x)\|_{f(x)}\un{e}^s(f(x))$.
It follows that
\begin{align*}
A_\chi(x)\un{e}_1&=s_\chi(x)^{-1}[C_\chi(f(x))^{-1}\circ df_x]\un{e}^s(x)\\
&=\pm s_\chi(x)^{-1}\|df_x \un{e}^s(x)\|_{f(x)} C_\chi(f(x))^{-1}\un{e}^s(f(x))\\
&=\pm\frac{s_\chi(f(x))}{s_\chi(x)}\|df_x \un{e}^s(x)\|_{f(x)}\un{e}_1.
\end{align*}
We see that $\un{e}_1$ is an eigenvector of $A_\chi(x)$ with eigenvalue
\begin{equation}\label{lambda_e}
\l_\chi(x):=\pm\frac{s_\chi(f(x))}{s_\chi(x)}\|df_x \un{e}^s(x)\|_{f(x)}.
\end{equation}
Similarly, $\un{e}_2$ is an eigenvector of $A_\chi(x)$ with eigenvalue
\begin{equation}\label{mu_e}
u_\chi(x):=\pm\frac{u_\chi(f(x))}{u_\chi(x)}\|df_x \un{e}^u(x)\|_{f(x)}.
\end{equation}

We estimate the eigenvalues:
\begin{align*}
s_\chi(x)^2&\equiv 2\sum_{k=0}^\infty e^{2k\chi}\|(df^k)_x \un{e}^s(x)\|^2_{f^k(x)}>
2\sum_{k=1}^\infty e^{2k\chi}\|(df^k)_x \un{e}^s(x)\|^2_{f^k(x)}\\
&=2\sum_{k=0}^\infty e^{2(k+1)\chi}\|(df^k)_{f(x)}df_x \un{e}^s(x)\|^2_{f^{k+1}(x)}\\
&=2\|df_x \un{e}^s(x)\|_{f(x)}^2\sum_{k=0}^\infty e^{2(k+1)\chi}\|(df^k)_{f(x)} \un{e}^s(f(x))\|^2_{f^{k+1}(x)}\\
&=e^{2\chi}\|df_x \un{e}^s(x)\|_{f(x)}^2 s_\chi(f(x))^2.
\end{align*}
Rearranging terms, we find that
$
e^{-2\chi}>\frac{s_\e(f(x))^2}{s_\e(x)^2}\|df_x \un{e}^s(x)\|^2_{f(x)}=\l_\chi(x)^2.
$
It follows that $|\l_\chi(x)|<e^{-\chi}$. Similarly, one shows that $|\mu_\chi(x)|>e^\chi$.

\medskip

Since $f$ is a diffeomorphism, the number  $M_f:=\max\{\|df_x\|, \|df_x^{-1}\|:x\in M\}$ is well defined and finite. It is easy to see that $M_f\geq 1$. By \cite[Cor. 3.2.10]{KH}, $h_{top}(f)\leq 2\log M_f$.

By definition of $s_\chi(x)$, and the identity
$
df_x\un{e}^s(x)=\pm\|df_x\un{e}^s(x)\|\un{e}^s(f(x))$,
\begin{align*}
s_\chi(x)^2&=2\left(1+\sum_{k=1}^\infty e^{2k\chi}\|df^{k-1}_{f(x)}\un{e}^s(f(x))\|^2_{f^k(x)}\|df_x\un{e}^s(x)\|_x^2\right)\\
&\leq 2\left(1+e^{2\chi}M_f^2\sum_{k=0}^\infty e^{2k\chi}\|df^{k}_{f(x)}\un{e}^s(f(x))\|^2_{f^{k+1}(x)}\right)\\
&\leq 2+e^{2\chi}M_f^2s_\chi(f(x))^2\\
&\leq (M_f^6+1)s_\chi(f(x))^2\ \ (\because s_\chi>\sqrt{2}\textrm{ and }\chi<h_{top}(f)\leq 2\log M_f).
\end{align*}
Therefore by  (\ref{lambda_e})
\begin{align}
|\l_\chi(x)|&>(1+M_f^6)^{-1/2}\|df_x\un{e}^s(x)\|_{f(x)}\geq M_f^{-1}(1+M_f^6)^{-1/2}.\label{clarinet}
\end{align}
Similarly, one can bound $|\mu_\chi(x)|$ from above by a function of $M_f$.
 \hfill$\Box$

\medskip
\noindent
{\bf Proof of Lemma \ref{Lemma_C_norm}}
We put  the standard basis $\un{e}_1={1\choose 0}, \un{e}_2={0\choose 1}$ on $\R^2$, and the basis $\un{e}^s(x),\un{e}^s(x)^{\perp}$ on $T_x M$, where $\un{v}^\perp$ denotes the unique vector s.t. the signed angle from $\un{v}$ to $\un{v}^\perp$ is $\pi/2$.
The linear map $C_\chi(x):\R^2\to T_x$ is represented in these bases by the matrix
$$
\left(
\begin{array}{cc}
s_\chi(x)^{-1} & u_\chi(x)^{-1}\cos\a(x)\\
0 & u_\chi(x)^{-1}\sin\a(x)
\end{array}
\right).
$$
Inverting, we find that $C_\chi(x)^{-1}:T_x M\to\R^2$ is represented by
$$
\left(
\begin{array}{cc}
s_\chi(x) & -s_\chi(x)/\tan\a(x)\\
0 & u_\chi(x)/\sin\a(x)
\end{array}
\right).
$$
The lemma follows by direct calculation, using the  fact that  the Frobenius norm of a linear map represented by a matrix $(a_{ij})$ is equal to $(\sum a_{ij}^2)^{1/2}$.\hfill $\Box$

\medskip
\noindent
{\bf Proof of Lemma \ref{Lemma_C_contracts}}
Define an inner product $\<\cdot,\cdot\>_x^\ast$ on $T_x M$ by the conditions (a)  $\|\un{e}^s(x)\|_x^\ast=s_\chi(x)$, (b) $\|\un{e}^u(x)\|_x^\ast=u_\chi(x)$, and (c) $\<\un{e}^u(x), \un{e}^s(x)\>_x^\ast=0$ (compare with \cite[\S 3.5.1]{BP}).
The inner product $\|\cdot\|_x^\ast$ satisfies  $\|\cdot\|_x^\ast\geq \|\cdot\|_x$, because  for every ${\xi},{\eta}\in\R$
\begin{align*}
\|\xi \un{e}^s(x)+\eta \un{e}^u(x)\|_x^\ast&=\sqrt{\xi^2 s_\chi(x)^2+\eta^2 u_\chi(x)^2}>\sqrt{2(\xi^2+\eta^2)}\ \ (\because s_\chi, u_\chi>\sqrt{2})\\
&\geq |\xi|+|\eta|=\|\xi\un{e}^s(x)\|_x+\|\eta\un{e}^u(x)\|_x\geq \|\xi \un{e}^s(x)+\eta \un{e}^u(x)\|_x.
\end{align*}
$\therefore\|C_\chi(x){\xi\choose\eta}\|_x\leq \|C_\chi(x){\xi\choose\eta}\|_x^\ast=\|\xi s_\chi(x)^{-1}\un{e}^s(x)+\eta u_\chi(x)^{-1}\un{e}^u(x)\|_x^\ast=\sqrt{\xi^2+\eta^2}$. The lemma follows.
\hfill$\Box$

\medskip
\noindent
{\bf Proof of Lemma \ref{Lemma_C_tempered}}
Let $A_\chi(x):=C_\chi(f(x))^{-1}\circ df_x\circ C_\chi(x)$. Extend $A_\chi$  to  a cocycle $A_\chi^{(n)}$ using the identities $A_\chi^{(0)}:=A_\chi$ and  $A_\chi^{(m+n)}(x)=A_\chi^{(m)}(f^n(x)) A_\chi^{(n)}(x)$. The extension is unique, and is given by
$
A_\chi^{(n)}(x)=C_\chi(f^n(x))^{-1}df_x^n C_\chi(x).
$

Theorem \ref{TheoremOPR} says that $A_\chi(x)$ is a diagonal matrix with  entries in $[C_f^{-1},C_f]$ for every $x\in\NUH_\chi(f)$. In particular,  $\log\|A^{(0)}_\chi\|$ and $\log\|(A^{(0)}_\chi)^{-1}\|$ are uniformly bounded on $\NUH_\chi(f)$, whence absolutely integrable w.r.t any ergodic invariant probability measure with entropy larger than $\chi$. This allows us to apply the Multiplicative Ergodic Theorem to $A_\chi^{(n)}$ w.r.t. every ergodic invariant probability measure with entropy larger than $\chi$.

Let $\NUH^\dagger_\chi(f)$ denote the set points $x\in\NUH_\chi(f)$ for which there is a decomposition $T_x \R^2=E^s_\chi(x)\oplus E^u_\chi(x)$ so that
\begin{enumerate}
\item $E^s_\chi(x)=\Span\{\un{e}_\chi^s(x)\}$, $\|\un{e}_\chi^s(x)\|=1$,  $\lim\limits_{n\to\pm\infty}\frac{1}{n}\log\|A_\chi^{(n)}(x) \un{e}^s_\chi(x)\|<0$;
 \item $E^u_\chi(x)=\Span\{\un{e}^u_\chi(x)\}$, $\|\un{e}^u_\chi(x)\|=1$,  $\lim\limits_{n\to\pm\infty}\frac{1}{n}\log\|A_\chi^{(n)}(x) \un{e}^u_\chi(x)\|>0$;
      \item $\lim\limits_{n\to\infty}\frac{1}{n}\log|\sin\a_\chi(f^n(x))|=0$, where $\a_\chi(x):=\measuredangle(\un{e}^s_\chi(x),\un{e}^u_\chi(x))$;
 \item $A_\chi(x)[E^s_\chi(x)]=E^s_\chi(f(x))$ and $A_\chi(x)[E^u_\chi(x)]=E^u_\chi(f(x))$.
\end{enumerate}
By the discussion above, $\NUH^\dagger_\chi(f)$ has full measure w.r.t. to any ergodic invariant probability measure with entropy larger than $\chi$.

Let $\NUH^\ast_\chi(f)$ denote the subset of $\NUH^\dagger_\chi(f)$ which consists of all points $x$ for which there exist a sequence $n_k\uparrow\infty$ s.t. $C_\chi(f^{n_k}(x))\xrightarrow[k\to\infty]{}C_\chi(x)$ and a sequence $m_k\downarrow -\infty$ s.t. $C_\chi(f^{m_k}(x))\xrightarrow[k\to\infty]{}C_\chi(x)$. By the Poincar\'e Recurrence Theorem, every invariant probability  measure which is carried by $\NUH^\dagger_\chi(f)$ is carried by $\NUH^\ast_\chi(f)$, so  $\NUH^\ast_\chi(f)$ has full measure w.r.t. to every ergodic invariant measure with entropy greater than $\chi$.

On the set $\NUH^\ast_\chi(f)$, the Multiplicative Ergodic Theorem holds for both $df_x$ and $A_\chi^{(n)}(x)$, so the following two limits exist:
\begin{equation}\label{limits}
\lim\limits_{n\to\pm\infty}\frac{1}{n}\log\|df^n_x C_\chi(x)\un{e}_i\|_{f^n(x)}\textrm{ , }
\lim\limits_{n\to\pm\infty}\frac{1}{n}\log\|C_\chi(f^n(x))^{-1}df^n_x C_\chi(x)\un{e}_i\|.
\end{equation}

Let $n_k\uparrow\infty$ be a  subsequence for which $C_\chi(f^{n_k}(x))\xrightarrow[k\to\infty]{}C_\chi(x)$. The norms of $C_\chi(f^{n_k}(x))$ and $C_\chi(f^{n_k}(x))^{-1}$ are bounded along this sequence,   so $$\|C_\chi(f^{n_k}(x))^{-1}df^{n_k}_x C_\chi(x)\un{e}_i\|\asymp \|df^{n_k}_x C_\chi(x)\un{e}_i\|.$$ We see that the limits in (\ref{limits}) agree.
As a result
$E^s_\chi(x)=\R\x\{\un{0}\}$, $E^u_\chi(x)=\{\un{0}\}\x\R$, and $x$ has Lyapunov exponents $\log\l(x)$ and $\log\mu(x)$ w.r.t. $A_\chi^{(n)}$.

Let $\L_\chi(x):=\left(\begin{array}{cc}
\l(x) & 0\\
0 & \mu(x)
\end{array}
\right),$
then the limits (\ref{limits}) mean that
$$
\|(A_\chi^{(n)}(x)\L_\chi(x)^{-n})^{\pm 1}\|^{1/n}\xrightarrow[n\to\pm\infty]{}1.
$$
Similarly, if $\L(x)$ is the linear operator s.t. $\L(x)\un{e}^s(x)=\l(x) \un{e}^s(x)$ and $\L(x)\un{e}^u(x)=\mu(x)\un{e}^u(x)$, then
$$
\|(df_x^n\L(x)^{-n})^{\pm 1}\|^{1/n}\xrightarrow[n\to\pm\infty]{}1.
$$
Since  $
\L_\chi(x)=C_\chi(x)^{-1}\L(x) C_\chi(x)
$ and $A_\chi^{(n)}(x)=C_\chi(f^n(x))^{-1}\circ df_x^n\circ C_\chi(x)$,
\begin{align*}
\|(C_\chi\circ f^n)^{-1}\|^{1/n}&=\|A_\chi^{(n)}  C_\chi^{-1}(df_x^n)^{-1}\|^{1/n}\\
&=\|A_\chi^{(n)}  C_\chi^{-1}\L^{-n} C_\chi\cdot C_\chi^{-1}\cdot\L^n(df_x^n)^{-1}\|^{1/n}\\
&\leq \|A_\chi^{(n)} \L_\chi^{-n}\|^{1/n}\|C_\chi^{-1}\|^{1/n}\|(df_x^n\L^{-n})^{-1}\|^{1/n}\xrightarrow[n\to\pm\infty]{}1.
\end{align*}
Thus $\limsup\frac{1}{n}\log\|(C_\chi\circ f^n)^{-1}\|\leq 0$. On the other hand $C_\chi$ is a contraction (Lemma \ref{Lemma_C_contracts}), so  $\|(C_\chi\circ f^n)^{-1}\|^{1/n}\geq 1$, whence  $\liminf\frac{1}{n}\log\|C_\chi(f^n(x))^{-1}\|\geq 0$. The first part of the Lemma is proved.

\medskip
We prove  the second part of the Lemma: $\frac{1}{n}\log\|C_\chi(f^n(x))\un{e}_i\|_{f^n(x)}\xrightarrow[n\to\pm\infty]{}0$. We do this for $i=1$, and leave the case $i=2$ to the reader. Since the $A_\chi^{(n)}(\cdot)$ is diagonal, $A_\chi^{(n)}(x)\un{e}_1$ is proportional to $\un{e}_1$. The multiplicative ergodic theorem for $A_\chi^{(n)}(x)$ says that
$A_\chi^{(n)}(x)\un{e}_1=\pm \l(x)^n \exp[o(n)]\un{e}_1$, therefore
\begin{align*}
\lim\limits_{n\to\pm\infty}\|C_\chi(f^n(x))\un{e}_1\|^{1/n}_{f^n(x)}&=
\l(x)^{-1}\lim\limits_{n\to\pm\infty}\|C_\chi(f^n(x))A_\chi^{(n)}(x)
\un{e}_1\|^{1/n}_{f^n(x)}\\
&=\l(x)^{-1}\lim\limits_{n\to\pm\infty}\|(df_x^n)C_\chi(x)\un{e}_1\|^{1/n}_{f^n(x)}\\
&=\l(x)^{-1}\lim\limits_{n\to\pm\infty}\|(df_x^n)\un{e}^s(x)\|^{1/n}_{f^n(x)}=1,
\end{align*}
proving that $\frac{1}{n}\log\|C_\chi(f^n(x))\un{e}_1\|_{f^n(x)}\xrightarrow[n\to\pm\infty]{}0$.

\medskip
Finally, we prove that $\frac{1}{n}\log|\det C_\chi(f^n(x))|\xrightarrow[n\to\pm\infty]{}0$. We begin with some general comments on determinants.

Suppose  $L:V\to W$ is a linear operator between two two dimensional vector spaces with inner product. The determinant of  $\det L$ can be defined as $\det (L\Theta)$ for some (every) isometry $\Theta:W\to V$. The following fact holds:\footnote{Proof:
Let $\omega_V$, $\omega_W$ denote the volume $2$--forms on $V,W$, then $\omega_V(\un{u},\un{v})=\|\un{u}\|\|\un{v}\|\sin\measuredangle(\un{u},\un{v})$ and $\omega_W(\un{u},\un{v})=\|\un{u}\|\|\un{v}\|\sin\measuredangle(\un{u},\un{v})$. Since  $\omega_W(L\un{u},L\un{v})$ is also a $2$--form on $V$, and any two $2$--forms on $V$ are proportional, $\exists c$ s.t.  $\omega_W(L\un{u},L\un{v})=c\omega_V(\un{u},\un{v})$. Evaluating on an orthonormal basis of $V$, we find that  $c=\det L$. Consequently,
$\|L\un{u}\|\|L\un{v}\|\sin\measuredangle(L\un{u},L\un{v})=\det L \|\un{u}\|\|\un{v}\|\sin\measuredangle(\un{u},\un{v})$.
}
 If  $\un{u},\un{v}$ span $V$, then
\begin{equation}\label{Angle_Formula}
\frac{\sin \measuredangle(L\un{u}, L\un{v})}{\sin \measuredangle(\un{u}, \un{v})}=\frac{\|\un{u}\|\|\un{v}\|\det L}{\|L\un{u}\|\|L\un{v}\|}.
\end{equation}
It follows that
$$
|\det L|=\frac{\|L\un{u}\|\|L\un{v}\||\sin\measuredangle(L\un{u}, L\un{v})|}{\|\un{u}\|\|\un{v}\||\sin\measuredangle(\un{u},\un{v})|}\ \ \ \ \ (\un{u},\un{v}\textrm{ independent}).
$$

Applying this to $L=A_\chi^{(n)}$ with $\un{u}=\un{e}_1$, $\un{v}=\un{e}_2$, and to $L=df_x^n$ with $\un{u}=\un{e}^s(x)$, $\un{v}=\un{e}^u(x)$,  we find that
$$
\lim\limits_{n\to\pm\infty}\frac{1}{n}\log|\det A_\chi^{(n)}(x)|=\log\l(x)+\log\mu(x)=
\lim\limits_{n\to\pm\infty}\frac{1}{n}\log|\det df_x^n|.
$$
But $|\det A_\chi^{(n)}(x)|=|\det C_\chi(f^n(x))|^{-1}|\det df_x^n||\det C_\chi(x)|$. It follows that
$$
\frac{1}{n}\log|\det C_\chi(f^n(x))|\xrightarrow[n\to\infty]{}0
$$
as required. \hfill$\Box$

\medskip
\noindent
{\bf Proof of Lemma \ref{Lemma_Q_tempered}}
Parts (1) and (3) are obvious, and part (4) is a consequence of Lemma \ref{Lemma_C_tempered} and the  estimate
$
Q_\e(f^n(x))\asymp\|C_\chi(f^n(x))^{-1}\|^{-12/\b}.
$
For part (6), define $q_\e(x)$ on $\NUH^\ast(f)$ by the formula
$$
\frac{1}{q_\e(x)}=\frac{1}{\e}\sum_{k=-\infty}^\infty e^{-\frac{1}{3}|k|\e}\frac{1}{Q_\e(f^k(x))}.
$$
The sum converges because $\frac{1}{k}\log Q_\e(f^k(x))\xrightarrow[k\to\pm\infty]{}0$, and it is easy to check that $q_\e(x)$ behaves as required. (Compare with \cite[Lemma 3.5.7]{BP}.)

\medskip
It remains to prove parts (2) and (5). First we prove the following claim.

\medskip
\noindent
{\em Claim.\/} There exists a constant $C$, which only depends on $M,f$ and $\chi$, such that
$
C^{-1}\leq \|C_\chi(f(x))^{-1}\|/\|C_\chi(x)^{-1}\|\leq C\textrm{ on }\NUH_\chi(f).
$

\medskip
\noindent
{\em Proof.\/} By Lemma \ref{Lemma_C_norm}  it is enough to show that
$$
\frac{s_\chi\circ f}{s_\chi}\ , \ \frac{u_\chi\circ f}{u_\chi}\ ,\ \frac{|\sin\a\circ f|}{|\sin\a|}
$$
are uniformly bounded away from zero and infinity on $\NUH_\chi(f)$.

The following quantity is well defined and finite, because
 $f$ is a diffeomorphism and $M$ is compact:
$$
F_0:=\max\{\|df_x\|, \|df_x^{-1}\|, |\det(df_x)|, |\det(df_x^{-1})|:x\in M\}.
$$
Notice  that $F_0>1$.

Equation (\ref{lambda_e}) makes it clear that $\frac{s_\chi(f(x))}{s_\chi(x)}=F_0^{\pm 1}|\l_\e(x)|\in [(C_f F_0)^{-1},C_f F_0]$ on $\NUH_\chi(f)$.
Similarly, $\frac{u_\chi(f(x))}{u_\chi(x)}$ takes values in $[(C_f F_0)^{-1},C_f F_0]$ on $\NUH_\chi(f)$.
Finally, by (\ref{Angle_Formula}) and the fact that   $\un{e}^{s/u}(f(x))$ have the same direction as $df_x \un{e}^{s/u}(x)$ up to a sign,
$$
\frac{|\sin\a(f(x))|}{|\sin\a(x)|}=
\frac{|\sin\measuredangle(\un{e}^s(f(x)),\un{e}^u(f(x))|}{|\sin\measuredangle(\un{e}^s(x),\un{e}^u(x))|}=\frac{|\det df_x|}{\|df_x \un{e}^s(x)\|\|df_x \un{e}^u(x)\|}.
$$
The last quantity takes values in $[F_0^{-3}, F_0^3]$. The claim follows.

\medskip
Part (5) follows directly from the claim. For part (2), we start by noting that $Q_\e(x)<\e^{3/\b}\|C_\chi(x)^{-1}\|_{Fr}^{-12/\b}<\e^{3/\b}\|C_\chi(x)^{-1}\|^{-12}$, therefore also
$Q_\e(x)<(\e^{3/\b} C^{12/\b})\cdot \|C_\chi(f^{\pm 1}(x))^{-1}\|^{-12}$. If $\e$ is small enough then $\e^{1/\b} C^{12/\b}<1$, and the proof of part (2) is complete.\hfill$\Box$

\medskip
\noindent
{\bf Proof of Theorem \ref{Theorem_OP_charts}} What follows is based on   \cite[Theorem 5.6.1]{BP}.

Recall the following basic fact from differential geometry \cite[chapter 9]{Spivak}: Every $p\in M$ has an open neighborhood $W_p$ and a positive number $r>0$ s.t.
\begin{enumerate}
\item any $q,q'\in W_p$ are connected by a unique geodesic of length less than $r$;
\item for each $q\in W_p$, $\exp_q$ maps $B_r^q(\un{0})\subset T_q M$ diffeomorphically onto an open set $U_q\supseteq W_p$ in a $2$--bi-Lipschitz way, and $d(\exp_q)_{\un{0}}=\id$;
\item for every $q,q'\in W_p$, there is a unique vector $\un{v}(q,q')\in T_q M$ s.t. $\|\un{v}(q,q')\|_q<r$ and $\exp_q[\un{v}(q,q')]=q'$;
\item
$(q,q')\mapsto \un{v}(q,q')$ is a well--defined $C^\infty$ map from $W_p\x W_p$ to $M$.
\end{enumerate}

Since $M$ is compact, there exist positive constants $r(M), \rho(M)$ s.t. for every $p\in M$, $\exp_p$ maps $B_{r(M)}^p(\un{0})\subseteq T_p M$ diffeomorphically onto a neighborhood of $B_{\rho(M)}(p)\subset M$, in a $2$--bi-Lipschitz way.
Let
\begin{equation}\label{r_0}
r_0:=\frac{\min\{1,r(M),\rho(M)\}}{10[\Lip(f)+\Lip(f^{-1})]}.
\end{equation}
Note that $r_0<1$.

Suppose $\e<r_0/5$.
By the definition of $Q_\e(x)$, $Q_\e(x)<\e^3$, so $10Q_\e(x)<r_0/\sqrt{2}$. By Lemma \ref{Lemma_C_contracts}, $C_\chi(x)$ maps $R_{10 Q_\e(x)}(\un{0})$ contractively into $B_{r_0}(\un{0})$. Therefore $\Psi_x=\exp_x\circ C_\chi(x)$ maps $R_{10 Q_\e(x)}(\un{0})$ diffeomorphically in a $2$--Lipschitz way into $M$.
The first part of the theorem is proved.

Next we show that $f_x:=\Psi_{f(x)}^{-1}\circ f\circ \Psi_x$ is well defined on $R_{10 Q_\e(x)}(\un{0})$ and establish its properties.

Since  $\exp_x$ is $2$--Lipschitz, $C_\chi(x)$ is a contraction, and $10 Q_\e(x)<r_0/\sqrt{2}$,
$$
\Psi_x\textrm{ maps }R_{10Q_\e(x)}(\un{0})\textrm{ diffeomorphically into }B_{2r_0}(x).
$$
It follows that $f\circ \Psi_x$ maps $R_{10 Q_\e(x)}(\un{0})$ diffeomorphically into $B_{2\Lip(f) r_0}(f(x))$, which by the definition of $r_0$ is a subset of $B_{\rho(M)}(f(x))$, whence a subset of $\exp_{f(x)}[B_{r(M)}^x(\un{0})]$. It follows that $f_x:=\Psi_{f(x)}^{-1}\circ f\circ \Psi_x$ is well defined, smooth and injective on $R_{10 Q_\e(x)}(\un{0})$.

For every $p\in M$,  $\exp_p(\un{0})=p$ and $d(\exp_p)_{\un{0}}=\id$. It easily follows that  $f_x(\un{0})=\un{0}$, and $(df_x)_{\un{0}}=C_\chi(f(x))^{-1}\circ (df)_x\circ C_\chi(x)$. By Theorem \ref{TheoremOPR}, this is a diagonal matrix with diagonal elements $A(x)=\l_\e(x), B(x)=\mu_\e(x)$, and  $C_f^{-1}<|A(x)|<e^{-\chi}$, $e^\chi<|B(x)|<C_f$.

We compare $f_x$ to its linearization at $\un{0}$ by analyzing
$$
r_x(\un{u}):=f_x(\un{u})-(df_x)_{\un{0}}(\un{u}).
$$
By assumption $f$ is $C^{1+\b}$, so there is a constant $L$ s.t. for all $\un{u},\un{v}\in R_{r_0}(\un{0})$,
$\|d(\exp_{f(x)}^{-1}\circ f\circ\exp_x)_{\un{u}}-d(\exp_{f(x)}^{-1}\circ f\circ\exp_x)_{\un{v}}\|\leq L\|\un{u}-\un{v}\|^{\b}$. For every $\un{u},\un{v}\in R_{r_0}(\un{0})$,
\begin{align*}
\|(dr_x)_{\un{u}}-(dr_x)_{\un{v}}\|&=\|C_\chi(f(x))^{-1}d(\exp_{f(x)}^{-1}\circ f\circ \exp_x)_{C_\chi(x)\un{u}}C_\chi(x)\\
&\hspace{1cm}-C_\chi(f(x))^{-1}d(\exp_{f(x)}^{-1}\circ f\circ \exp_x)_{C_\chi(x)\un{v}}C_\chi(x)\|\\
&=\|C_\chi(f(x))^{-1}[d(\exp_{f(x)}^{-1}\circ f\circ \exp_x)_{C_\chi(x)\un{u}}\\
&\hspace{3cm}-d(\exp_{f(x)}^{-1}\circ f\circ \exp_x)_{C_\chi(x)\un{v}}]C_\chi(x)\|\\
&\leq \|C_\chi(f(x))^{-1}\|\cdot L\|C_\chi(x)\|^\b\|\un{u}-\un{v}\|^{\b}\cdot\|C_\chi(x)\|\\
&\leq (\|C_\chi(f(x))^{-1}\|\cdot L\|\un{u}-\un{v}\|^{\b/2})\cdot \|\un{u}-\un{v}\|^{\b/2}\ \ (\because \|C_\chi(x)\|<1).
\end{align*}
If $\un{u},\un{v}\in R_{10 Q_\e(x)}(\un{0})$, then the term in the brackets is smaller than
\begin{align*}
\|C_\chi(f(x))^{-1}\|\cdot L (20\sqrt{2} Q_\e(x))^{\b/2}.
\end{align*}
Plugging in the definition of $Q_\e(x)$ from (\ref{Qdef}), and recalling that $\|C_\chi(\cdot)^{-1}\|>1$ (because $C_\chi(\cdot)$ is a contraction),  we see that the term in the brackets is smaller than
$30^{\b/2}L \e^{3/2}$. Thus, if  $\e<\frac{1}{3}\cdot 30^{-\b/2}L^{-1}$, then
$$
\|(dr_x)_{\un{u}}-(dr_x)_{\un{v}}\|\leq \tfrac{1}{3}\e\|\un{u}-\un{v}\|^{\b/2}\ \ \ \ (\un{u},\un{v}\in R_{10 Q_\e(x)}(\un{0})).
$$

Since $(dr_x)_{\un{0}}=0$, we have that   $\|(dr_x)_{\un{u}}\|\leq \frac{1}{3}\e\|\un{u}\|^{\b/2}$ on $R_{10 Q_\e(x)}(\un{0})$. Now $Q_\e(x)<\e^{3/\b}$, so  $\|\un{u}\|\leq (10\sqrt{2})Q_\e(x)<15\e^{3/\b}$.
If $\e<15^{-\b/3}$, then $\|\un{u}\|<1$, so
$$
\|(dr_x)_{\un{u}}\|\leq \tfrac{1}{3}\e\textrm{ on }R_{10 Q_\e(x)}(\un{0}).
$$
Since $r_x(\un{0})=\un{0}$, we have by the mean value theorem that
$$
\|r_x(\un{u})\|\leq \tfrac{1}{3}\e \|\un{u}\|<\tfrac{1}{3}\e\text{ on }R_{10 Q_\e(x)}(\un{0}).
$$

In summary, if $\e$ is small enough,  then  the $C^{1+\b/2}$--distance between $r_x$ and $0$ on $R_{10 Q_\e(x)}(\un{0})$ is less than $\e$.
This shows that the $C^{1+\b/2}$--distance between $f_x$ and $(df_x)_{\un{0}}$ on this set is less than $\e$.

The treatment of $f_x^{-1}$ is similar, and is left to the reader.\hfill$\Box$

\medskip
\noindent
{\bf Proof of Proposition \ref{Prop_Intersection}} The proof of parts (1),(2) and (3) of the proposition is taken from \cite{KM}. Part (4) is new, but routine.
Assume that $0<\e<\frac{1}{2}$.

Write $V^u=\Psi_x\{(F(w),w):|w|\leq p^u\}$ and $V^s=\Psi_x\{(v,G(v)):|v|\leq p^s\}$, and let $\eta:=p^u\wedge p^s$. Note that $\eta<\e$, and that
$
|F(0)|, |G(0)|\leq 10^{-3}\eta$  and $\mathrm{Lip}(F), \mathrm{Lip}(G)\leq \e,
$
see (\ref{Lip(F)}).

The maps $H=F,G$ are contractions (with Lipschitz constant less than $\e$), and they  map the  interval $[-10^{-2}\eta,10^{-2}\eta]$ into itself, because for every $|t|<10^{-2}\eta$,
$$
|H(t)|\leq |H(0)|+\textrm{Lip}(H)|t|<10^{-3}\eta+\e\cdot 10^{-2}\eta=(10^{-1}+\e)10^{-2}\eta<10^{-2}\eta.
$$
It follows that  $G\circ F$ is a $\e^2$--contraction of $[-10^{-2}\eta,10^{-2}\eta]$ into itself. By the Banach Fixed Point Theorem, $G\circ F$ has a unique fixed point: $(G\circ F)(w)=w$.

Let $v:=F(w)$. We claim that $V^u, V^s$ intersect at $P:=\Psi_x(v,w)$.
\begin{itemize}
\item $P\in V^u$, because $v=F(w)$ and $|w|\leq 10^{-2}\eta<p^u$;
\item $P\in V^s$, because $w=(G\circ F)(w)=G(v)$, and $|v|<|F(0)|+\mathrm{Lip}(F)|w|\leq 10^{-3}\eta+\e\cdot 10^{-2}\eta<10^{-2}\eta<p^s$.
\end{itemize}
We also see that  $|v|,|w|\leq 10^{-2}\eta$.

We claim that $P$ is the unique intersection point of $V^u$ and $V^s$. Let $\xi:=p^u\vee p^s$ and extend $F,G$ (arbitrarily) to $\e$--Lipschitz continuous functions  $\wt{F}, \wt{G}:[-\xi,\xi]\to [-Q_\e(x),Q_\e(x)]$. Let $\wt{V}^u$ and $\wt{V}^s$ denote the $u/s$--sets represented by $\wt{F}, \wt{G}$. Any intersection point of $V^u, V^s$ is an intersection point of $\wt{V}^u, \wt{V}^s$. Such points take the form $\wt{P}=\Psi_x(\wt{v},\wt{w})$ where $\wt{v}=\wt{F}(\wt{w})$ and $\wt{w}=\wt{G}(\wt{v})$. Notice that  $\wt{w}$ is a fixed point of $\wt{G}\circ\wt{F}$.
The same calculations as before show that $\wt{G}\circ\wt{F}$ contracts $[-\xi,\xi]$ into itself. Such a map has a unique fixed point, therefore $\wt{w}=w$, whence $\wt{P}=P$.

Next we show that  $P$ is a Lipschitz function of $V^u, V^s$.  Suppose $V^u_i, V^s_i$ $(i=1,2)$ are represented by $F_i$ and $G_i$ $(i=1,2)$ respectively. Let $P_i$ denote the intersection points of $V^u_i\cap V^s_i$. We saw above that $P_i=\Psi_x(v_i,w_i)$ where $w_i$ is a fixed point of $G_i\circ F_i:[-10^{-2}\eta,10^{-2}\eta]\to [-10^{-2}\eta,10^{-2}\eta]$.  The maps $f_i:=G_i\circ F_i$ are  $\e^2$--contractions of $[-10^{-2}\eta,10^{-2}\eta]$ into itself, therefore
\begin{align*}
|w_1-w_2|&=|f_1^n(w_1)-f^n_2(w_2)|
\leq |f_1(f_1^{n-1}(w_1))-f_2(f_1^{n-1}(w_1))|\\
&\hspace{6cm}+|f_2(f_1^{n-1}(w_1))-f_2(f^{n-1}_2(w_2))|\\
&\leq \|f_1-f_2\|_\infty+\e^2|f_1^{n-1}(w_1)-f^{n-1}_2(w_2)|\\
&\leq \cdots \leq \|f_1-f_2\|_\infty(1+\e^2+\cdots +\e^{2n})\\
&\leq \frac{1}{1-\e^2}\|f_1-f_2\|_\infty.
\end{align*}
Similarly, $v_i$ is a fixed point of $F_i\circ G_i:[-10^{-2}\eta,10^{-2}\eta]\to [-10^{-2}\eta,10^{-2}\eta]$, and the same argument gives that $|v_1-v_2|\leq (1-\e^2)^{-1}\|g_1-g_2\|_\infty$ where $g_i=F_i\circ G_i$.
Since $\Psi_x$ is $2$--Lipschitz, this means that
$$
d(P_1,P_2)<\frac{2}{1-\e^2}\left(\|G_1\circ F_1-G_2\circ F_2\|_\infty+\|F_1\circ G_1-F_2\circ G_2\|_\infty\right).
$$
Now
\begin{align*}
\|F_1\circ G_1-F_2\circ G_2\|_\infty&\leq \|F_1\circ G_1-F_1\circ G_2\|_\infty+\|F_1\circ G_2-F_2\circ G_2\|_\infty\\
&\leq \textrm{Lip}(F_1)\|G_1-G_2\|_\infty+\|F_1-F_2\|_\infty\\
\|G_1\circ F_1-G_2\circ F_2\|_\infty&\leq \textrm{Lip}(G_1)\|F_1-F_2\|_\infty+\|G_1-G_2\|_\infty
\end{align*}
Since $\Lip(F_i), \Lip(G_i)\leq \e^2$,
$
d(P_1,P_2)<\frac{2(1+\e)}{1-\e^2}[\dist(V^u_1, V^u_2)+\dist(V^s_1,V^s_2)].
$
The coefficient is less than $3$ for all  $\e$ small enough. For such $\e$, $P$ is a $3$--Lipschitz function of $V^u, V^s$.

\medskip
Finally, we analyze the angle of intersection at $P$. We assume throughout that $\e$ is so small that
$
0<t\leq \e\Longrightarrow e^{-2t}<1-t<1+t<e^{2t}
$. In what follows we drop the subscript $x$ in $\|\cdot\|_x$.

Let $\un{v}=(v,w)$ be the $\Psi_x$--coordinates of $P$ (i.e. $P=\Psi_x(\un{v})$), and write $E^s=E^s(x)$, $E^u=E^u(x)$. The following identities hold:
\begin{align*}
\measuredangle(E^s, E^u)&=\measuredangle\bigl((d\Psi_x)_{\un{0}}\un{e}^1, (d\Psi_x)_{\un{0}}\un{e}^2 \bigr),\textrm{ where }\un{e}_1={1\choose 0},\textrm{ and } \un{e}_2={0\choose 1}\\
\measuredangle(V^s, V^u)&=\measuredangle\bigl((d\Psi_x)_{\un{v}}\un{v}^s, (d\Psi_x)_{\un{v}}\un{v}^u \bigr),\textrm{ where }  \un{v}^s={1\choose F'(v)}\text{ and  } \un{v}^u={F'(w)\choose 1}.
\end{align*}
It is not difficult to see that the admissibility of $V^s, V^u$ and the inequalities  $|v|,|w|<10^{-2}\eta$ imply that
$
|F'(w)|, |G'(v)|< \eta^{\b/3}.
$

We begin with the estimate of   $\frac{\sin \measuredangle(V^s, V^u)}{\sin\measuredangle(E^s, E^u)}=\frac{\sin\measuredangle((d\Psi_x)_{\un{v}}\un{v}^s, (d\Psi_x)_{\un{v}}\un{v}^u )}{\sin\measuredangle((d\Psi_x)_{\un{0}}\un{e}^1, (d\Psi_x)_{\un{0}}\un{e}^2)}$. By  (\ref{Angle_Formula}),
\begin{align*}
\frac{\sin \measuredangle(V^s, V^u)}{\sin\measuredangle(E^s, E^u)}&=\frac{\sin\measuredangle(\un{v}^s, \un{v}^u)}{\sin\measuredangle(\un{e}^1,\un{e}^2)}
\cdot \frac{\|\un{v}^s\|\|\un{v}^u\|}{\|\un{e}^1\|\|\un{e}^2\|} \cdot\frac{\det (d\Psi_x)_{\un{v}}}{\det (d\Psi_x)_{\un{0}}}\cdot\frac{\|(d\Psi_x)_{\un{0}}\un{e}^1\|\|(d\Psi_x)_{\un{0}}\un{e}^2\|}
 {\|(d\Psi_x)_{\un{v}}\un{v}^s\|\|(d\Psi_x)_{\un{v}}\un{v}^u\|}.
\end{align*}

\medskip
\noindent
{\bf First factor:} The first factor equals $\sin\measuredangle(\un{v}^s,\un{v}^u)$. Using the  formula for the sine of the difference of two angles it is not difficult to see that
$$
\sin\measuredangle(\un{v}^s,\un{v}^u)=\frac{1}{\|\un{v}^s\|\|\un{v}^u\|}
\det\left(\begin{array}{cc}
1 &F'(w) \\
F'(v) & 1
\end{array}
\right).
$$
Since    $|F'(v)|, |F'(w)|<\eta^{\b/3}$,  the first factor is $e^{\pm 2\eta^{2\b/3}}$.

\medskip
\noindent
{\bf Second factor:}  Since $|F'(v)|, |F'(w)|<\eta^{\b/3}$, the numerator is $e^{\pm \eta^{2\b/3}}$. Since the denominator is equal to one, the second factor is $e^{\pm \eta^{2\b/3}}$.

\medskip
\noindent
{\bf Third factor:} $\det(d\Psi_x)_{\un{v}}=\det (d\exp_x)_{C_\chi(x)\un{v}}\cdot \det C_\chi(x)$, and $\det (d\Psi_x)_{\un{0}}=\det C_\chi(x)$, therefore the third factor is equal to
$
\det (d\exp_x)_{C_\chi(x)\un{v}}.
$

The exponential map on $M$ is smooth, and $\det(d\exp_x)_{\un{0}}=1$,  therefore there exists a constant $K_1$ which only depends on $M$ s.t.
$$
\left|\det[(d\exp_x)_{\un{u}}]-1\right|<K_1\|\un{u}\|\textrm{ for all $x\in M$ and  $\|\un{u}\|<1$.}
$$
Since $C_\chi(x)$ is a contraction (Lemma \ref{Lemma_C_contracts}) and $\|\un{v}\|<2\eta$,
$
\det (d\exp_x)_{C_\chi(x)\un{v}}=1\pm 2K_1\eta.
$
Since $0<\eta<\e$, $2K_2\eta\ll\sqrt{\eta}$ for all $\e$ small enough. For such $\e$,  the third factor is $e^{\pm \sqrt{\eta}}$ (provided $\e$ is small enough).

\medskip
\noindent
{\bf Fourth factor:} Find a global constant $K_2$ s.t. $\|(\Theta_D d\exp_x)_{\un{u}}-\id\|<K_2\|\un{u}\|$ for all $x\in D\in\mathfs D$ and $\|\un{u}\|<1$ (cf. \S\ref{SectionOC}).

Write $\un{u}=C_\chi(x)\un{v}$, and choose some $D\in\mathfs D$ which contains $\Psi_x[R_{Q_\e(x)}(\un{0})]$, then
\begin{equation}\label{Bim}
\begin{aligned}
\|\Theta_D (d\Psi_x)_{\un{v}}\un{v}^s-\Theta_D(d\Psi_x)_{\un{0}}\un{e}^1\|
&\leq \|\Theta_D(d\Psi_x)_{\un{v}}-\Theta_D(d\Psi_x)_{\un{0}}\|\|\un{v}^s\|\\
&\hspace{2cm}+
\|\Theta_D(d\Psi_x)_{\un{0}}\|\|\un{v}^s-\un{e}^1\|\\
&\leq \|\Theta_D(d\exp_x)_{\un{u}}-\id\|\|C_\chi(x)\|\|\un{v}^s\|\\
&\hspace{2cm}+2\|C_\chi(x)\|\|\un{v}^s-\un{e}^1\|\\
&< 3K_2\eta+2\eta^{\b/3},
\end{aligned}
\end{equation}
because $C_\chi(x)$ is a contraction, $\|\un{v}\|<2\eta$,  and $\un{v}^s={1\choose 0\pm \eta^{\b/3}}$. Consequently,
$\left|\|(d\Psi_x)_{\un{v}}\un{v}^s\|-\|(d\Psi_x)_{\un{0}}\un{e}^1\|\right|<(3K_2+2)\eta^{\b/3}$.
Since also
\begin{equation}\label{Bam}
\|(d\Psi_x)_{\un{0}}\un{e}^1\|=\|C_\chi(x)\un{e}^1\|\geq \|C_\chi(x)^{-1}\|^{-1},
\end{equation}
$\left|\frac{\|(d\Psi_x)_{\un{v}}\un{v}^s\|}{\|(d\Psi_x)_{\un{0}}\un{e}^1\|}-1\right|
<(3K_2+2)\|C_\chi(x)^{-1}\|\eta^{\b/3}$.

Since  $\eta\leq Q_\e(x)$ and $Q_\e(x)<\e^{3/\b}\|C_\chi(x)^{-1}\|^{-12/\b}$,
\begin{equation}\label{Bom}
\|C_\chi(x)^{-1}\|\eta^{\b/3}\leq \|C_\chi(x)^{-1}\|\eta^{\b/12}\cdot\eta^{\b/4}<\e^{1/4}\eta^{\b/4}.
\end{equation}
It follows that for all $\e$ small enough,
$
\frac{\|(d\Psi_x)_{\un{v}}\un{v}^s\|}{\|(d\Psi_x)_{\un{0}}\un{e}^1\|}=\exp\left[\pm\left(\frac{1}{3}\eta^{\b/4}\right)\right].
$
How small depends only on $K_2$, and therefore only on the surface $M$.

Similarly, one can show that  $\frac{\|(d\Psi_x)_{\un{u}}\un{v}^u\|}{\|(d\Psi_x)_{\un{0}}\un{e}^2\|}=\exp[\pm\frac{1}{3}\eta^{\b/4}]$, with the result that the fourth factor is $\exp[\pm \frac{2}{3}\eta^{\b/4}]$.

\medskip
Putting all these estimates together, we see that
$$
\frac{\sin \measuredangle(V^u, V^s)}{\sin\measuredangle(E^u, E^s)}=\exp\left[\pm(2\eta^{2\b/3}+\eta^{2\b/3}+\sqrt{\eta}+\frac{2}{3}\eta^{\b/4})
\right].
$$
Since $0<\eta<\e$, for all $\e$ small enough, this is $e^{\pm\eta^{\b/4}}$.  How small just depends  on $K_1$, $K_2$, and $\b$.

\medskip
Next we estimate $|\cos\measuredangle(V^s,V^u)-\cos\measuredangle(E^s,E^u)|$. This is equal to
$$
\hspace{-4.5cm}\left|
\frac{\left<
(d\Psi_x)_{\un{v}}\un{v}^s, (d\Psi_x)_{\un{v}}\un{v}^u
\right>}{\|(d\Psi_x)_{\un{v}}\un{v}^s\| \| (d\Psi_x)_{\un{v}}\un{v}^u\|}
- \frac{\left<(d\Psi_x)_{\un{0}}\un{e}^1, (d\Psi_x)_{\un{0}}\un{e}^2
\right>}{\|(d\Psi_x)_{\un{0}}\un{e}^1\| \| (d\Psi_x)_{\un{0}}\un{e}^2\|}
\right|\leq
$$
\begin{align*}
&\leq \frac{|\left<(d\Psi_x)_{\un{v}}\un{v}^s, (d\Psi_x)_{\un{v}}\un{v}^u
\right>|}{\|(d\Psi_x)_{\un{0}}\un{e}^1\| \| (d\Psi_x)_{\un{0}}\un{e}^2\|}\x
\left|\frac{\|(d\Psi_x)_{\un{0}}\un{e}^1\| \| (d\Psi_x)_{\un{0}}\un{e}^2\|}{\|(d\Psi_x)_{\un{v}}\un{v}^s\| \| (d\Psi_x)_{\un{v}}\un{v}^u\|}-1
\right|+\\
&\hspace{0.7cm}+\frac{1}{\|(d\Psi_x)_{\un{0}}\un{e}^1\| \| (d\Psi_x)_{\un{0}}\un{e}^2\|}\x\left|
\left<(d\Psi_x)_{\un{v}}\un{v}^s, (d\Psi_x)_{\un{v}}\un{v}^u
\right>-\left<(d\Psi_x)_{\un{0}}\un{e}^1, (d\Psi_x)_{\un{0}}\un{e}^2
\right>
\right|\\
&\leq \frac{\|(d\Psi_x)_{\un{v}}\un{v}^s\| \|(d\Psi_x)_{\un{v}}\un{v}^u
\|}{\|(d\Psi_x)_{\un{0}}\un{e}^1\| \| (d\Psi_x)_{\un{0}}\un{e}^2\|}\x
\left|\frac{\|(d\Psi_x)_{\un{0}}\un{e}^1\| \| (d\Psi_x)_{\un{0}}\un{e}^2\|}{\|(d\Psi_x)_{\un{v}}\un{v}^s\| \| (d\Psi_x)_{\un{v}}\un{v}^u\|}-1
\right|+\\
&\hspace{0.7cm}+\frac{1}{\|(d\Psi_x)_{\un{0}}\un{e}^1\| \| (d\Psi_x)_{\un{0}}\un{e}^2\|}\x\left|
\left<(d\Psi_x)_{\un{v}}\un{v}^s, (d\Psi_x)_{\un{v}}\un{v}^u
\right>-\left<(d\Psi_x)_{\un{0}}\un{e}^1, (d\Psi_x)_{\un{0}}\un{e}^2
\right>
\right|.
\end{align*}
By (\ref{Bam}) and the estimate of the ``fourth factor" above, this is smaller than
\begin{equation}\label{stimmt}
e^{\frac{2}{3}\eta^{\b/4}}\cdot \eta^{\b/4}+\|C_\chi(x)^{-1}\|^2\left|
\left<(d\Psi_x)_{\un{v}}\un{v}^s, (d\Psi_x)_{\un{v}}\un{v}^u
\right>-\left<(d\Psi_x)_{\un{0}}\un{e}^1, (d\Psi_x)_{\un{0}}\un{e}^2
\right>
\right|.
\end{equation}

Since $\Theta_D$ is an isometry,  the difference of the inner products is equal to
$$
\hspace{-3cm}\left|
\left<\Theta_D(d\Psi_x)_{\un{v}}\un{v}^s, \Theta_D(d\Psi_x)_{\un{v}}\un{v}^u
\right>-\left<\Theta_D(d\Psi_x)_{\un{0}}\un{e}^1, \Theta_D(d\Psi_x)_{\un{0}}\un{e}^2
\right>
\right|
$$
\begin{align*}
&\leq \|\Theta_D(d\Psi_x)_{\un{v}}\un{v}^s-\Theta_D(d\Psi_x)_{\un{0}}\un{e}^1\|\cdot\|(d\Psi_x)_{\un{v}}\un{v}^u\|\\
&\hspace{1cm}+\|\Theta_D (d\Psi_x)_{\un{0}}\un{e}^1\|\cdot\|\Theta_D(d\Psi_x)_{\un{v}}\un{v}^u-\Theta_D(d\Psi_x)_{\un{0}}\un{e}^2\|\\
&\leq 3\!\left(\|\Theta_D(d\Psi_x)_{\un{v}}\un{v}^s-\Theta_D(d\Psi_x)_{\un{0}}\un{e}^1\|
+\|\Theta_D(d\Psi_x)_{\un{v}}\un{v}^u-\Theta_D(d\Psi_x)_{\un{0}}\un{e}^2\|
\right)\\
&\leq 3\!\left(
\|\Theta_D(d\Psi_x)_{\un{v}}\|\|\un{v}^s-\un{e}^1\|+2\|\Theta_D(d\Psi_x)_{\un{v}}-\Theta_D(d\Psi_x)_{\un{0}}\|\right.\\
&\hspace{6cm}+\left.
\|\Theta_D(d\Psi_x)_{\un{v}}\|\|\un{v}^u-\un{e}^2\|
\right)\\
&\leq 3[2\eta^{\b/3}+2\cdot 2K_2\eta+2\eta^{\b/3}],
\end{align*}
because  $\Theta_D$ is an isometry, $\|d\Psi_x\|\leq 2$ on $R_{Q_\e(x)}(\un{0})$, and $\|\un{v}^{s/u}-\un{e}^{1/2}\|<\eta^{\b/3}$. Thus
$
\left|
\left<(d\Psi_x)_{\un{v}}\un{v}^s, (d\Psi_x)_{\un{v}}\un{v}^u
\right>-\left<(d\Psi_x)_{\un{0}}\un{e}^1,(d\Psi_x)_{\un{0}}\un{e}^2
\right>
\right|<K_3\eta^{\b/3},
$
where $K_3$ only depends on $M$.
It now follows from  (\ref{stimmt}) and the inequality $\eta<\e$ that
\begin{align*}
|\cos\measuredangle(V^s,V^u)-\cos\measuredangle(E^s,E^u)|&\leq e^{\frac{2}{3}\e^{3/4}}\eta^{\b/4}+
\|C_\chi(x)^{-1}\|^2\cdot K_3\eta^{\b/3}.
\end{align*}
We now argue as in (\ref{Bom}) and deduce that
\begin{align*}
|\cos\measuredangle(V^s,V^u)-\cos\measuredangle(E^s,E^u)|&\leq (e^{\frac{2}{3}\e^{3/4}}+K_3\e^{1/4})\eta^{\b/4}.
\end{align*}
This is smaller than $2\eta^{\b/4}$, for all $\e$ small enough.
\hfill$\Box$

\medskip
\noindent
{\bf Proof of Proposition \ref{Prop_Graph_Transform} (Graph Transform)} The proof is a straightforward adaptation of the arguments in \cite{KM} and \cite[chapter 7]{BP} (see also \cite{Pesin}).

Let  $V^u=\Psi_x\{(F(t),t):|t|\leq p^u\}$ be a $u$--admissible manifold in $\Psi_x^{p^u,p^s}$. We denote the parameters of $V^u$ by $\s,\g,\vf$, and $q$, and let $\eta:=p^u\wedge p^s$. $V^u$ is admissible, so
\begin{equation}\label{AdCon}
\s\leq\frac{1}{2}, \g\leq \frac{1}{2}\eta^{\b/3}, \vf\leq 10^{-3}\eta,q=p^u,\textrm{ and }\Lip(F)<\e,
\end{equation}
see Definition \ref{Def_Admissible} and Equation (\ref{Lip(F)}).

We analyze $\Gamma_y^u:=\Psi_y^{-1}[f(V^u)]\subset\R^2$, looking for parameterizations of large $u$--sub-manifolds. Notice that  $$\Gamma_y^u=f_{xy}[\mathrm{graph}(F)],$$ where $f_{xy}=\Psi_y^{-1}\circ f\circ\Psi_x$ and $\mathrm{graph}(F):=\{(F(t),t):|t|\leq q\}$.

Since $V^u$ is admissible, $\mathrm{graph}(F)\subset R_{Q_\e(x)}(\un{0})$. On this domain, $f_{xy}$ can be expanded as follows (Proposition \ref{Prop_f_xy}):
\begin{equation}\label{f_xy}
f_{xy}(u,v)=\bigl(Au+h_1(u,v), Bv+h_2(u,v)\bigr)
\end{equation}
where $C_f^{-1}<|A|<e^{-\chi}$, $e^\chi<|B|<C_f$;  and  $h_i$ are $C^{1+\frac{\b}{3}}$--functions s.t. $|h_i(0)|<\e\eta$, $\|\nabla h_i(\un{0})\|<\e\eta^{\b/3}$, and $\|\nabla h_i(\un{u})-\nabla h_i(\un{v})\|\leq \e\|\un{u}-\un{v}\|^{\b/3}$.
Necessarily,
 $\|\nabla h_i\|<\e\eta^{\b/3}+\e [\sqrt{2}Q_\e(x)]^{\b/3}<3\e Q_\e(x)^{\b/3}$  and $|h_i|<\e\eta+3\e Q_\e(x)^{\b/3}\cdot Q_\e(x)$.
 Since $\eta\leq Q_\e(x)$, and $Q_\e(x)<\e^{3/\b}$, the following holds for provided $\e$ is small enough:
\begin{equation}\label{nabla_h}
\|\nabla h_i\|<3\e^2\textrm{ and }|h_i|<\e^2\textrm{ on }\graph(F).
\end{equation}

Using (\ref{f_xy}), we can put $\Gamma^u_y$ in the following form:
\begin{equation}\label{Gamma_y}
\Gamma_y^u=\{(AF(t)+h_1(F(t),t), Bt+h_2(F(t),t)):|t|\leq q\}.
\end{equation}
The idea is to call the second coordinate $\tau$, solve $t=t(\tau)$, and substitute the result in the first coordinate.

\medskip
\noindent
{\em Claim 1.\/} The following holds for all $\e$ small enough:  $Bt+h_2(F(t),t)=\tau$ has a unique solution $t=t(\tau)$ for all
$
\tau\in
[-e^{\chi-\sqrt{\e}} q, e^{\chi-\sqrt{\e}}q]
$, and
\begin{enumerate}
\item[(a)] $\Lip(t)<e^{-\chi+\e}$;
\item[(b)] $|t(0)|<2\e\eta$;
\item[(c)] the $C^{\b/3}$--norm of $t'$ is  smaller than $|B|^{-1}e^{3\e}$.
\end{enumerate}

\medskip
\noindent
{\em Proof.\/} Let $\tau(t):=Bt+h_2(F(t),t)$. For every $|t|\leq q$,
\begin{align*}
|\tau'(t)|&\geq |B|-\max\|\nabla h_2\|\cdot\|(F'(t),1)\|>|B|-3\e^2\sqrt{1+\e^2}\ \ (\because (\ref{nabla_h}), (\ref{AdCon}))\\
&>|B|(1-3\e^2\sqrt{1+\e^2})\ \ (\because |B|>e^\chi>1 )\\
&> e^{-\e} |B|>1 \textrm{ provided $\e$ is small enough.}
\end{align*}
It follows that $\tau$ is $e^{-\e}|B|$--expanding, whence  one-to-one.

Since $\tau$ is one-to-one, $\tau^{-1}$ is well--defined on $\tau[-q,q]$. We estimate this set.
Since $\tau$ is continuous and $e^{-\e}B$--expanding,
$
\tau[-q,q]\supset (\tau(0)-e^{-\e}|B|q, \tau(0)+e^{\e}|B|q).
$
The center of the interval can be estimated as follows:
\begin{align*}
|\tau(0)|&=|h_2(F(0),0)|\leq |h_2(\un{0})|+\max\|\nabla h_2\|\cdot |F(0)|\\
&\leq \e\eta+3\e^2\cdot 10^{-3}\eta<2\e\eta\ \ (\textrm{admissibility and }(\ref{nabla_h})).
\end{align*}
Recall that $\eta\equiv p^u\wedge p^s\leq p^u\equiv q$, therefore $|\tau(0)|<2\e q$. Since $|\tau'|>e^{-\e}|B|$,
\begin{align*}
\tau[-q,q]&\supseteq [2\e q-e^{-\e}|B|q, -2\e q+e^{-\e}|B|q]
\supseteq [-(|B| e^{-\e}-2\e)q, (|B|e^{-\e}-2\e)q]\\
&\supseteq [-|B|(e^{-\e}-2\e)q, |B|(e^{-\e}-2\e)q].
\end{align*}
Since $|B|(e^{-\e}-2\e)> e^\chi (e^{-2\e}-2\e)>e^{\chi-\sqrt{\e}}$ for all $\e$ small enough,  $\tau^{-1}$ is well defined on
$
[-e^{\chi-\sqrt{\e}}q, e^{\chi-\sqrt{\e}} q].
$

\medskip
Since $t(\cdot)$ is the inverse of a $|B|e^{-\e}$--expanding map, $\Lip(t)\leq e^\e|B|^{-1}<e^{-\chi+\e}$, proving (a).

We saw above that $|\tau(0)|<2\e \eta$. For all $\e$ small enough, this is (much) smaller than $e^{\chi-\sqrt{\e}} q$, therefore $\tau(0)$ belongs to the domain of $t$. It follows that
$$
|t(0)|=|t(0)-t(\tau(0))|<\Lip(t)|\tau(0)|<e^{-\chi+\e}\cdot 2\e \eta.
$$
For all $\e$ small enough, this is less than $2\e\eta$, proving (b).

\medskip
Next we calculate the $C^{\b/3}$--norm of $t'(\cdot)$.

\medskip
We remind the reader that the $C^{\a}$--norm of  $\vf:[-q,q]^{d_1}\to\R^{d_2}$ $(0<\a<1)$ is defined by $\|\vf\|_{\a}:=\|\vf\|_\infty+\textrm{H\"ol}_\a(\vf)$, where
$$
\textrm{H\"ol}_\a(\vf):=\sup\left\{\frac{\|\vf(\un{u})-\vf(\un{v})\|}{\|\un{u}-\un{v}\|^\a}:\un{u},\un{v}\in[-q,q]^{d_1}\textrm{ different}\right\}.
$$
The following inequalities are easy to verify:
\begin{enumerate}
\item[(H1)] $\|\vf \cdot \psi\|_\a\leq \|\vf\|_\a \|\psi\|_\a$ for all $\vf,\psi\in C^\a[-q,q]$;
\item[(H2)] $\|\vf\circ g\|_\a\leq \|\vf\|_\infty+\text{H\"ol}_\a(\vf)\Lip(g)^\a$ for all $\vf$ $\a$--H\"older and $g$ Lipschitz;
\item[(H3)] In case $d_2=1$ and $\|\vf\|_\a<1$, $\|1/(1+\vf)\|_\a\leq (1-\|\vf\|_\a)^{-1}$.
\end{enumerate}

Differentiating the identity
$s=\tau(t(s))=Bt(s)+h_2(F(t(s)),t(s))$ w.r.t $s$, we obtain after some manipulations
$$
t'(s)=B^{-1}\left(1+B^{-1}\frac{\del h_2}{\del x}\bigl(F(t(s)),t(s)\bigr)F'(t(s))+B^{-1}\frac{\del h_2}{\del y}\bigl(F(t(s)),t(s)\bigr)\right)^{-1}.
$$
We write this in the form $t'(s)=B^{-1}(1+T(s))^{-1}$, where
$$
T(s):=B^{-1}\frac{\del h_2}{\del x}\bigl(F(t(s)),t(s)\bigr)F'(t(s))+B^{-1}\frac{\del h_2}{\del y}\bigl(F(t(s)),t(s)\bigr).
$$
By (H3), it is enough to find $\|T\|_{\b/3}$. Here is the estimation:
\begin{align*}
\left\|\frac{\del h_2}{\del x}\bigl(F(t(s)),t(s)\bigr)\right\|_{\b/3}&\leq \left\|\frac{\del h_2}{\del x}\right\|_\infty+\text{H\"ol}_{\b/3}(\nabla h_2)[\Lip(F\circ t,t)]^{\b/3}\ \ \because\text{(H2)}\\
&< 3\e^2+\e\cdot \left[\Lip(F)^2(\Lip(t))^2+(\Lip(t))^2\right]^{\b/6}\\
&<3\e^2 +\e [\sqrt{\e^2+1}(e^{\e}|B|^{-1})]^{\b/3}\ \ \because (\ref{AdCon}), (\ref{nabla_h})\\
&<\e,\textrm{ provided $\e$ is small enough.}\\
\left\|\frac{\del h_2}{\del y}\bigl(F(t(s)),t(s)\bigr)\right\|_{\b/3}&<\e\textrm{ (same proof).}\\
\|F'(t(s))\|_{\b/3}&\leq \|F'\|_\infty+\|F'\|_{\b/3}\Lip(t)^{\b/3}\ \textrm{ (see (H2) above)}\\
&\leq \s+\s\cdot (e^{-\chi+\e})^{\b/3}<1\textrm{ provided $\e$ is small enough.}
\end{align*}
Putting these estimates together, we see that
$
\|T\|_{\b/3}<2\e.
$
It now follows from (H3) that $\|t'\|_{\b/3}<|B|^{-1}(1-2\e)^{-1}$. This is smaller than $e^{3\e}|B|^{-1}$  for all  $\e$ small enough. This proves (c), and completes the proof of the claim.

\medskip
We now return to (\ref{Gamma_y}). Substituting $t=t(\tau)$,  we find that
$$
\Gamma^u_y\supset\{(G(\tau),\tau): |\tau|<e^{\chi-\sqrt{\e}} q\},
$$
where $G(\tau):=AF(t(\tau))+h_1(F(t(\tau)),t(\tau))$. Claim 1 guarantees that $G(\tau)$ is well-defined and $C^{1+\b/3}$ on $[-e^{\chi-\sqrt{\e}} q,e^{\chi-\sqrt{\e}} q]$. We find the parameters of $G$.
\medskip

\noindent
{\em Claim 2.\/} For all $\e$ small enough,  $|G(0)|<e^{-\chi+\sqrt{\e}} [\vf+\sqrt{\e}(q^u\wedge q^s)]$, and
$|G(0)|<10^{-3}(q^u\wedge q^s)$.

\medskip
\noindent
{\em Proof.\/}
Claim 1 says that $|t(0)|<2\e\eta$. Since $\Lip(F)<\e$, $|F(0)|<\vf$ and $\vf\leq 10^{-3}\eta$,  $|F(t(0))|<\vf+2\e^2\eta<\eta$ provided $\e$ is small enough. Thus
\begin{align*}
|G(0)|&\leq |A|\cdot |F(t(0))|+|h_1(F(t(0)),t(0))|\\
&\leq |A|(\vf+2\e^2\eta)+\left[|h_1(\un{0})|+\max\|\nabla h_1\|\cdot\|(F(t(0)),t(0))\|\right]\\
&\leq |A|(\vf+2\e^2\eta)+\left[\e\eta+3\e^2\cdot \sqrt{\eta^2+(2\e\eta)^2}\right]\ \ (\because|F(t(0))|<\eta)\\
&\leq |A|\left[\vf+\eta\bigl(2\e^2+\e+3\e^2\sqrt{1+4\e^2}\,\bigr)\right].
\end{align*}
Recalling that $|A|<e^{-\chi}$ and  $\eta\equiv(p^u\wedge p^s)\leq e^\e(q^u\wedge q^s)$ (Lemma \ref{Lemma_Subordinated_Tempered}), we see that
$
|G(0)|<e^{-\chi+\e}[\vf+2\e(q^u\wedge q^s)]
$
for all $\e$ small enough.

Since  $\vf\leq 10^{-3}(p^u\wedge p^s)\leq 10^{-3} e^\e(q^u\wedge q^s)$,
$
|G(0)|<e^{-\chi+\e}[10^{-3}+2\e](q^u\wedge q^s).
$
This is less than $10^{-3}(q^u\wedge q^s)$ for all $\e$ sufficiently small. The claim follows.

\medskip
\noindent
{\em Claim 3.\/} For all $\e$ small enough, $|G'(0)|<e^{-2\chi+\sqrt{\e}}[\g+\e^{\b/3}(q^u\wedge q^s)^{\b/3}]$, and $|G'(0)|<\frac{1}{2}(q^u\wedge q^s)^{\b/3}$.

\medskip
\noindent
{\em Proof.\/} $|G'(0)|\leq |t'(0)|\bigl[|A| \cdot|F'(t(0))|+\|\nabla h_1(F(t(0)),t(0))\|\cdot \|(F'(t(0)),1)\|\bigr]$, and
\begin{itemize}
\item $|t'(0)|\leq \Lip(t)<e^{-\chi+\e}$ (Claim 1).
\item $|F'(t(0))|<\g+\frac{2}{3}\e^{\b/3}\eta^{\b/3}$, because by Claim 1(b)
$$
\hspace{1.5cm}|F'(t(0))|<|F'(0)|+\Hol_{\b/3}(F')|t(0)|^{\b/3}<\g+\s\cdot (2\e\eta)^{\b/3}<\g+\tfrac{2}{3}\e^{\b/3}\eta^{\b/3}.
$$
\item $\|\nabla h_1(F(t(0)),t(0))\|\leq 3\e\eta^{\b/3}$, because $|F(t(0))|<\eta$ (proof of Claim 2), and  $|t(0)|<2\e\eta$ (Claim 1), so by the H\"older regularity of $\nabla h_i$,
\begin{align*}
\|\nabla h_1(F(t(0)),t(0))\|&\leq \|\nabla h_1(\un{0})\|+\e \left(\sqrt{|F(t(0))|^2+|t(0)|^2}\right)^{\b/3}\\
&\leq \e\eta^{\b/3}+\e(\sqrt{\eta^2+(2\e\eta)^2})^{\b/3}<3\e\eta^{\b/3}.
\end{align*}
\item $\|(F'(t(0)),1)\|<\sqrt{1+\e^2}<2$.
\end{itemize}

Putting these estimates together, we see that
\begin{align*}
|G'(0)|&<e^{-\chi+\e}|A|\left[\g+\frac{2}{3}\e^{\b/3}\eta^{\b/3}+|A|^{-1}\cdot 3\e\eta^{\b/3}\cdot 2\right]\\
&<e^{-2\chi+\e}\left[\g+\left(\frac{2}{3}\e^{\b/3}+6C_f\e\right)\eta^{\b/3}\right],\ \because C_f^{-1}<|A|<e^{-\chi}\\
&\leq e^{-2\chi+\e}\left[\g+\left(\frac{2}{3}\e^{\b/3}+6C_f\e\right)e^{\e\b/3}(q^u\wedge q^s)^{\b/3}\right]\ \ \because p^u\wedge p^s\leq e^\e(q^u\wedge q^s).
\end{align*}

This implies that
 for all $\e$ small enough,  $|G'(0)|<e^{-2\chi+\e}\left[\g+\e^{\b/3}(q^u\wedge q^s)^{\b/3}\right]$, which is stronger than the estimate in the claim.

 Since $\g\leq \frac{1}{2}(p^u\wedge p^s)^{\b/3}$ and $(p^u\wedge p^s)\leq e^{\e}(q^u\wedge q^s)$,  we also get that  for all $\e$ small enough, $|G'(0)|<\frac{1}{2}(q^u\wedge q^s)^{\b/3}$, as required.

\medskip
\noindent
{\em Claim 4.\/} For all $\e$ small enough,  $\|G'\|_{\b/3}<e^{-2\chi+\sqrt{\e}}[\s+\sqrt{\e}]$, and $\|G'\|_{\b/3}<\frac{1}{2}$.

\medskip
\noindent
{\em Proof.\/} Differentiating, we see that $G'=t'\cdot[AF'\circ t+\frac{\del h_1}{\del x}(F\circ t, t) F'\circ t+\frac{\del h_1}{\del y}(F\circ t,t)]$. By Claim 1 and its proof
\begin{itemize}
\item $\|t'\|_{\b/3}\leq |B|^{-1}e^{3\e}$ ,
\item $\|F'\circ t\|_{\b/3}\leq \s$, because $\|F'\|_{\b/3}\leq \s$  and $t$ is a contraction,
\item $\|\frac{\del h_1}{\del x}(F\circ t, t)\|_{\b/3}<\e$, and  $\|\frac{\del h_1}{\del y}(F\circ t, t)\|_{\b/3}< \e$.
\end{itemize}
Thus by (H1), $\|G'\|_{\b/3}\leq |B|^{-1}e^{3\e}\left[|A|\s+\e\s+\e\right]$.
Since $\s\leq \frac{1}{2}$, $e^{\chi}<|B|<C_f$, and   $C_f^{-1}<|A|<e^{-\chi}$,
$
\|G'\|_{\b/3}\leq e^{-2\chi+3\e}\left[\s+\tfrac{3}{2}C_f\e\right].
$
If $\e$ is small enough, then  $\|G'\|_{\b/3}<e^{-2\chi+\sqrt{\e}}[\s+\sqrt{\e}]$, and  $\|G'\|_{\b/3}<\frac{1}{2}$.

\medskip
\noindent
{\em Claim 5.\/} For all $\e$ small enough, $\wh{V}^u:=\Psi_y\{(G(\tau), \tau):|\tau|\leq \min\{e^{\chi-\sqrt{\e}} q,Q_\e(y)\}\}$ is a $u$--manifold in $\Psi_y$, the parameters of $\wh{V}^u$ satisfy (\ref{GraphTransform}), and $\wh{V}^u$  contains a $u$--admissible manifold  in $\Psi_y^{q^u, q^s}$.

\medskip
\noindent
{\em Proof.\/}
To see that $\wh{V}^u$ is a $u$--manifold in $\Psi_y$ we have to check that $G$ is $C^{1+\b/3}$ and $\|G\|_\infty\leq Q_\e(y)$.

Claim 1 shows that  $G$ is $C^{1+\b/3}$. To see that $\|G\|_\infty\leq Q_\e(y)$, we first observe that
for all $\e$ small enough, $\Lip(G)<\sqrt{\e}$, because
\begin{align*}
|G'|&\leq |G'(0)|+\Hol_{\b/3}(G)Q_\e(y)^{\b/3}\leq \e +\frac{1}{2}  \e<\sqrt{\e},\textrm{ provided $\e$ is small enough.}
\end{align*}
It follows that $\|G\|_\infty\leq |G(0)|+\sqrt{\e}Q_\e(y)<(10^{-3}+\sqrt{\e})Q_\e(y)<Q_\e(y)$.

Next we claim that $\wh{V}^u$ contains a $u$--admissible manifold in $\Psi_y^{q^u,q^s}$. Since $\Psi_x^{p^u,p^s}\to\Psi_y^{q^u,q^s}$, $q^u=\min\{e^\e p^u, Q_\e(y)\}$. Consequently,
for every  $\e$ small enough,
\begin{equation}\label{pergo}
e^{\chi-\sqrt{\e}} q\equiv e^{\chi-\sqrt{\e}}p^u>e^\e p^u\geq q^u,
\end{equation}
so
$
\wh{V}^u
$
restricts to a $u$-manifold with $q$--parameter equal to $q^u$.
Claims 2--4 guarantee that  this manifold is $u$--admissible in $\Psi_y^{q^u,q^s}$, and that (\ref{GraphTransform}) holds.

\medskip
\noindent
{\em Claim 6.\/} $f(V^u)$ contains exactly one $u$--admissible manifold in $\Psi_y^{q^u,q^s}$. This manifold contains $f(p)$ where  $p=\Psi_x(F(0),0)$.

\medskip
\noindent
{\em Proof.\/} The previous claim shows existence. We prove uniqueness. By formula (\ref{Gamma_y}), any $u$--admissible manifold in $\Psi_y^{q^u,q^s}$ which is contained in $f(V^u)$ must be a subset of
$$
\Psi_y\{(AF(t)+h_1(F(t),t), Bt+h_2(F(t),t)):|t|\leq q, |Bt+h_2(F(t),t))|\leq q^u\}.
$$
We saw in (\ref{pergo}) that for all $\e$ small enough, $q^u<e^{\chi-\sqrt{\e}} q$. By claim 1, the equation
$$
\tau=Bt+h_2(F(t),t)
$$
has a unique solution $t=t(\tau)\in [-q,q]$ for all $|\tau|\leq q^u$. Our manifold must therefore  equal
$
\Psi_y\{(AF(t(\tau))+h_1(F(t(\tau)),t(\tau)), \tau):|\tau|\leq q^u\}.
$
This is exactly the $u$--admissible manifold that we constructed above.

\medskip
Let $\mathcal F_u[V^u]$ denote the unique $u$--admissible manifold in $\Psi_y^{q^u,q^s}$ contained in $f(V^u)$. We claim that $\mathcal F_u[V^u]\owns f(p)$ where $p=\Psi_x(F(0),0)$.
 By the previous paragraph, it is enough to check that the second coordinate of $\Psi_y^{-1}[f(p)]$ has absolute value less than $q^u$. Call this second coordinate $\tau$, then
\begin{align*}
|\tau|&=\textrm{second coordinate of }f_{xy}(F(0),0)=|h_2(F(0),0)|\\
&\leq |h_2(\un{0})|+\max\|\nabla h_2\|\cdot |F(0)|<\e\eta+3\e^2\cdot 10^{-3}\eta<e^{-\e}\eta<(q^u\wedge q^s)\leq q^u.
\end{align*}

\medskip
\noindent
{\em Claim 7.\/} $f({V}^u)$ intersects any $s$--admissible manifold in $\Psi_y^{q^u,q^s}$ at a unique point.

\medskip
\noindent
{\em Proof.\/} Let $W^s$ be an $s$--admissible manifold in $\Psi_y^{q^u,q^s}$. We saw in the previous claim that $f({V}^u)$ contains a $u$--admissible manifold $W^u$ in $\Psi_y^{q^u,q^s}$. By  Proposition \ref{Prop_Intersection}, $W^u$ and  $W^s$ intersect. Therefore  $f({V}^u)$ and $W^s$ intersect at least at one point.

\medskip
We claim that the intersection point it unique. Recall that one can put $f(V^u)$ in the form
$$
f(V^u)=\Psi_y\{(A F(t)+h_1(F(t),t),Bt+h_2(F(t),t)):|t|\leq q\}.
$$
We saw in the proof of claim 1 that
the second coordinate, $\tau(t):=Bt+h_2(F(t),t)$, is a one-to-one continuous map  whose image is an interval $[\a,\b]$ with endpoints
$
\a<-e^{\chi-\sqrt{\a}}q<-q^u\ , \ \b>e^{\chi-\sqrt{\e}}q>q^u$. We also saw that $|\tau'|>e^{-\e}|B|\geq e^{\chi-\e}$.
Consequently, the inverse function $t:[\a,\b]\to [-q,q]$ satisfies $|t'(\tau)|<1$, and so
$$
f(V^u)=\Psi_y\{(G(\tau),\tau):\tau\in [\a,\b]\},\textrm{ where }\Lip(G)\leq \e.
$$

Let $H:[-q^u,q^u]\to\R$ denote the function which represents $W^s$ in $\Psi_y$, then $\Lip(H)\leq \e$. Extend it to an $\e$--Lipschitz function on $[\a,\b]$. The extension represents a Lipschitz manifold $\wt{W}^s\supset W^s$.
The same argument we used  to prove Proposition \ref{Prop_Intersection} shows that $f(V^u)$ and $\wt{W}^u$ intersect at a unique point. We see that $f({V}^u)$ and $W^s$ intersect at most at one point.

\medskip
This completes the proof of the proposition, in the case of $u$-manifolds. The case of $s$--manifolds follows from the symmetry between $s$ and $u$--manifolds:
\begin{enumerate}
\item $V$ is a $u$--admissible manifold w.r.t. $f$ iff $V$ is a an $s$--admissible manifold w.r.t. $f^{-1}$, and the parameters are the same.
\item $\Psi_x^{p^u,p^s}\to\Psi_y^{q^u,q^s}$ w.r.t. $f$ iff $\Psi_y^{q^u,q^s}\to \Psi_x^{p^u,p^s}$ w.r.t. $f^{-1}$.\hfill$\Box$
\end{enumerate}

\medskip
\noindent
{\bf Proof of Proposition \ref{Prop_Graph_Contracts}.} We prove the proposition for $\mathcal F_u$, and leave the case of $\mathcal F_s$ to the reader.

Suppose $\Psi_x^{p^u,p^s}\to\Psi_y^{q^u,q^s}$, and
 let $V_i^u$ be two $u$--admissible manifolds in $\Psi_{x}^{p^u,p^s}$. We take $\e$ to be small enough for the arguments of the previous section to work.

 We saw in the previous section that if $V_i=\Psi_x\{(F_i(t),t):|t|\leq p^u\}$, then $\mathcal F_u[V_i]=\Psi_y\{(G_i(\tau),\tau):|\tau|\leq q^u\}$, where
 \begin{itemize}
 \item $G_i(\tau)=A F_i (t_i(\tau))+h_1(F_i(t_i(\tau)),t_i(\tau))$;
 \item $t_i(\tau)$ is defined implicitly by $B t_i(\tau)+h_2(F_i(t_i(\tau)), t_i(\tau))=\tau$, and $|t_i'|<1$;
 \item $C_f^{-1}<|A|<e^{-\chi}$, $e^{\chi}<|B|<C_f$;
 \item $|h_i(\un{0})|<\e(p^u\wedge p^s)$, $\Hol_{\b/3}(\nabla h_u)\leq \e$, and $\max\|\nabla h_i\|<3\e^2$.
 \end{itemize}
In order to prove the proposition, we need to  estimate  $\|G_1-G_2\|_\infty$ and $\|G_1'-G_2'\|_\infty$ in terms of $\|F_1-F_2\|_\infty$ and $\|F_1'-F_2'\|_\infty$.

\medskip
\noindent
{\em Part 1.\/} For all $\e$ small enough,  $\|t_1-t_2\|_\infty\leq \e\|F_1-F_2\|_\infty$.

\medskip
By definition, $B t_i(\tau)+h_2(F_i(t_i(\tau)), t_i(\tau))=\tau$. Taking differences, we see that
\begin{align*}
|B|\cdot |t_1-t_2|&\leq |h_2(F_1(t_1),t_1)-h_2(F_2(t_2),t_2)|\\
&\leq \left\|\frac{\del h_2}{\del x}\right\|_\infty |F_1(t_1)-F_2(t_2)|+\left\|\frac{\del h_2}{\del x}\right\|_\infty |t_1-t_2|\\
&\leq 3\e^2\bigl(|F_1(t_1)-F_2(t_1)|+|F_2(t_1)-F_2(t_2)|+|t_1-t_2|\bigr)\\
&\leq 3\e^2\bigl(\|F_1-F_2\|_\infty+(\Lip(F_2)+1)|t_1-t_2|\bigr)\\
&\leq 3\e^2\|F_1-F_2\|_\infty+3\e^2(1+\e)|t_1-t_2|,\textrm{ see (\ref{Lip(F)}).}
\end{align*}
Rearranging terms, and recalling that $|B|>e^{\chi-\e}$,  we see that
$$
\|t_1-t_2\|_\infty<\frac{3\e^2\|F_1-F_2\|_\infty}{e^{\chi-\e}-3\e^2(1+\e)}.
$$
The claim follows.

\medskip
\noindent
{\em Part 2.\/} For all $\e$ small enough, $\|G_1-G_2\|_\infty<e^{-\chi/2}\|F_1-F_2\|_\infty$, whence $(\ref{magnifico})$.

\medskip
Subtracting the defining equations for $G_i$, we find that
\begin{align*}
|G_1-G_2|&\leq |A|\cdot |F_1(t_1)-F_2(t_2)|+|h_1(F_1(t_1),t_1)-h_1(F_2(t_2),t_2)|\\
&\leq |A|\cdot |F_1(t_1)-F_2(t_2)|+\|\nabla h_1\|\sqrt{|F_1(t_1)-F_2(t_2)|^2+|t_1-t_2|^2}\\
&\leq (|A|+3\e^2)|F_1(t_1)-F_2(t_2)|+3\e^2|t_1-t_2|\\
&\leq (|A|+3\e^2)(|F_1(t_1)-F_2(t_1)|+|F_2(t_1)-F_2(t_2)|)+3\e^2|t_1-t_2|\\
&\leq (|A|+3\e^2)(\|F_1-F_2\|_\infty+\Lip(F_2)|t_1-t_2|)+3\e^2|t_1-t_2|\\
&\leq (|A|+3\e^2)(1+\e\cdot\e+3\e^2\cdot\e)\|F_1-F_2\|_\infty,\textrm{ see part 1}\\
&\leq |A|(1+3C_f\e^2)(1+\e^2+3\e^3)\|F_1-F_2\|_\infty\\
&\leq e^{-\chi}(1+3C_f\e^2)(1+\e^2+3\e^3)\|F_1-F_2\|_\infty.
\end{align*}
It follows that for every $\e$ small enough, $\|G_1-G_2\|_\infty<e^{-\chi/2}\|F_1-F_2\|_\infty$.

\medskip
\noindent
{\em Part 3.\/} For all $\e$ small enough, $\|t_1'-t_2'\|_\infty<\sqrt{\e}(\|F_1'-F_2'\|_\infty+\|F_1-F_2\|_\infty^{\b/3})$.

\medskip
\noindent
Differentiating both sides of the defining equation of $t_i$ gives
$$
t_i'\left[B+\frac{\del h_2}{\del x}(F_i\circ t_i,t_i)F_i'\circ t_i+\frac{\del h_2}{\del y}(F_i\circ t_i,t_i)\right]=1.
$$
Taking differences, we obtain after some re-arrangement
$$
\hspace{-3cm}(t_1'-t_2')\left[B+\frac{\del h_2}{\del x}(F_1\circ t_1,t_1)F_1'\circ t_1+\frac{\del h_2}{\del y}(F_1\circ t_1,t_1)\right]=
$$
\begin{align*}
\hspace{2cm}&-t_2'\left[\frac{\del h_2}{\del x}(F_1\circ t_1,t_1)-
\frac{\del h_2}{\del x}(F_2\circ t_2,t_2)\right]F_1'\circ t_1 & =:\mathrm{I}\\
&-t_2'\frac{\del h_2}{\del x}(F_2\circ t_2,t_2)
\left[(F_1'\circ t_1-F_2'\circ t_1)+(F_2'\circ t_1-F_2'\circ t_2)\right] & =:\mathrm{II}\\
&-t_2'\left[\frac{\del h_2}{\del y}(F_1\circ t_1,t_1)-\frac{\del h_2}{\del y}(F_2\circ t_2,t_2)\right] &=:\mathrm{III}
\end{align*}
Since $|B|>e^{\chi}$, $|F_1'|<1$ and $\|\nabla h_2\|<3\e^2$,
$$
\|t_1'-t_2'\|_\infty\leq \frac{1}{e^\chi-6\e^2}\left\|\mathrm{I}+\mathrm{II}+\mathrm{III}\right\|_\infty.
$$

Since $\mathrm{I, II}$ and $\mathrm{III}$ involve partial derivatives of $h_2$ evaluated at $(F_i\circ t_i,t_i)$, we begin by analyzing $\nabla h_2(F_i\circ t_i,t_i)$.  Since
$\Hol_{\b/3}(\nabla h_i)\leq \e$,
\begin{itemize}
\item $\|\nabla h_2(F_1\circ t_1,t_1)-\nabla h_2(F_2\circ t_1,t_1)\|\leq \e\|F_1-F_2\|_\infty^{\b/3}$;
\item $\|\nabla h_2(F_2\circ t_1,t_1)-\nabla h_2(F_2\circ t_2,t_1)\|\leq \e\|t_1-t_2\|^{\b/3}$ (because $\Lip(F_2)<1$);
\item $\|\nabla h_2(F_2\circ t_2,t_1)-\nabla h_2(F_2\circ t_2,t_2)\|
\leq \e\|t_1-t_2\|_\infty^{\b/3}$.
\end{itemize}
By part 1, $\|t_1-t_2\|_\infty\leq \e\|F_1-F_2\|_\infty$. It follows that
$$
\|\nabla h_2(F_1\circ t_1,t_1)-\nabla h_2(F_2\circ t_2,t_2)\|<3\e\|F_1-F_2\|_\infty^{\b/3}.
$$
Using the facts that $|t_1'|<1$,
$|F_1'|<1$, $\Lip(F_2)<1$, and $\Hol_{\b/3}(F_2')<1$ (see the definition of admissible manifolds and the proof of Proposition \ref{Prop_Graph_Transform}), we get that
\begin{align*}
|\mathrm I|&\leq 3\e\|F_1-F_2\|_\infty^{\b/3};\\
|\mathrm{II}| &\leq 3\e^2\bigl(\|F_1'-F_2'\|_\infty+\|t_1-t_2\|_\infty^{\b/3}\bigr)\leq 3\e^2\|F_1'-F_2'\|_\infty+3\e^2\|F_1-F_2\|_\infty^{\b/3};\\
|\mathrm{III}|&\leq 3\e\|F_1-F_2\|_\infty^{\b/3}.
\end{align*}
So for all $\e$ sufficiently small,  $\|t_1'-t_2'\|_\infty<\sqrt{\e}\bigl(\|F_1'-F_2'\|_\infty+\|F_1-F_2\|_\infty^{\b/3}\bigr)$.

\medskip
\noindent
{\em Part 4.\/} $\|G_1'-G_2'\|_\infty<e^{-\chi/2}(\|F_1'-F_2'\|_\infty+\|F_1-F_2\|_\infty^{\b/3})$.

\medskip
By the definition of $G_i$, $G_i'=t_i'[A F_i'\circ t_i+\frac{\del h_1}{\del x}(F_i\circ t_i,t_i)F_i'\circ t_i+\frac{\del h_1}{\del y}(F_i\circ t_i,t_i)]$.
Taking differences, we see that
\begin{align*}
|G_1'-G_2'|&\leq |t_1'-t_2'|\cdot \left|A F_1'\circ t_1+\frac{\del h_1}{\del x}(F_1\circ t_1,t_1)F_1'\circ t_1+\frac{\del h_1}{\del y}(F_1\circ t_1,t_1)\right|& =:\mathrm{I'}\\
&\hspace{1cm}+ |t_2'|\cdot |A|\cdot\bigl(
\left|F_1'\circ t_1-F_2'\circ t_1\right|+\left|F_2'\circ t_1-F_2'\circ t_2\right|\bigr)& =:\mathrm{II'}\\
&\hspace{1cm}+|t_2'|\left|\frac{\del h_1}{\del x}(F_1\circ t_1,t_1)-\frac{\del h_1}{\del x}(F_2\circ t_2,t_2)\right| |F_1'\circ t_1|& =:\mathrm{III'}\\
&\hspace{1cm}+|t_2'|\left|\frac{\del h_1}{\del x}(F_2\circ t_2,t_2)
\right| |F_1'\circ t_1-F_2'\circ t_2|& =:\mathrm{IV'}\\
&\hspace{1cm}+|t_2'|\left|
\frac{\del h_1}{\del y}(F_1\circ t_1,t_1)-\frac{\del h_1}{\del y}(F_2\circ t_2,t_2)
\right|& =:\mathrm{V'}
\end{align*}
Using the same arguments that we used in part 3, one can show that
\begin{align*}
\mathrm{I'}&\leq \|t_1'-t_2'\|_\infty(e^{-\chi}+6\e^2)< \sqrt{\e}(\|F_1'-F_2'\|_\infty+\|F_1-F_2\|_\infty^{\b/3})\\
\mathrm{II'}&\leq e^{-\chi}(\|F_1'-F_2'\|_\infty+\|t_1-t_2\|_\infty^{\b/3})\leq e^{-\chi}(\|F_1'-F_2'\|_\infty+\|F_1-F_2\|_\infty^{\b/3})\ (\textrm{part 1})\\
\mathrm{III'}&\leq 3\e\|F_1-F_2\|_\infty^{\b/3}\ (\textrm{see the estimate of I in part 3})\\
\mathrm{IV'}&\leq  3\e^2\|F_1'-F_2'\|_\infty+3\e^3\|F_1-F_2\|_\infty^{\b/3}\ (\textrm{see the estimate of II in part 3})\\
\mathrm{V'}&\leq 3\e\|F_1-F_2\|_\infty^{\b/3} \ (\textrm{see the estimate of III in part 3}).
\end{align*}
It follows that
$\|G_1'-G_2'\|_\infty<(e^{-\chi}+10\e+\sqrt{\e})(\|F_1'-F_2'\|_\infty+\|F_1-F_2\|_\infty^{\b/3})$. If $\e$ is small enough, then
$\|G_1'-G_2'\|_\infty<e^{-\chi/2}(\|F_1'-F_2'\|_\infty+\|F_1-F_2\|_\infty^{\b/3})$.
\hfill$\Box$

\medskip
\noindent
{\bf Proof of Proposition \ref{Prop_Staying_In_Windows}}\label{AppendixSW} The following proof is based on   \cite[Chapter 7]{BP}.

Suppose $V^s$ is an $s$--admissible manifold in $\Psi_x^{p^u,p^s}$ which stays in windows, then there is a positive chain $(\Psi_{x_i}^{p^u_i,p^s_i})_{i\geq 0}$ s.t. $\Psi_{x_0}^{p^u_0,p^s_0}=\Psi_{x}^{p^u,p^s}$, and there are $s$--admissible manifolds $W^s_i$ in $\Psi_{x_i}^{p^u_i,p^s_i}$ s.t. $f^i(V^s)\subset W^s_i$ for all $i\geq 0$. We write
\begin{itemize}
\item $V^s=\Psi_x\{(t,F_0(t)):|t|\leq p^s\}$,
\item $W^s_i=\Psi_{x_i}\{(t,F_i(t)):|t|\leq p^s_i\}$,
\item $\eta_i:=p^u_i\wedge p^s_i$.
\end{itemize}
Admissibility means that
$\|F'_i\|_{\b/3}\leq \frac{1}{2}$, $|F'_i(0)|\leq \frac{1}{2}\eta_i^{\b/3}$ and  $|F_i(0)|\leq 10^{-3}\eta_i$. By Lemma \ref{Lemma_Subordinated_Tempered}, $e^{-\e}\leq \eta_i/\eta_{i+1}\leq e^\e$. By (\ref{Lip(F)}), $\Lip(F_i)< \e$.

\medskip
\noindent
{\em Part 1.\/} If $\e$ is so small that $e^{-\chi}+4\e^2<e^{-\chi/2}$, then  for every $y,z\in V^s$, $d(f^k(y),f^k(z))\leq 6 p_0^s e^{-\frac{1}{2}k\chi}$ for all $k\geq 0$.

\medskip
\noindent
{\em Proof.\/}
Since $V^s$ stays in windows, $f^k(V^s)\subset\Psi_{x_k}[R_{Q_\e(x_k)}(\un{0})]$ for all $k\geq 0$. Therefore, for any $y,z\in V^s$, one can write $f^k(y)=\Psi_{x_k}(\un{y}_k)$ and $f^k(z)=\Psi_{x_k}(\un{z}_k)$, where
$\un{y}_k=(y_k,F_k(y_k)), \un{z}_k=(z_k,F_k(z_k))$ belong to $R_{Q_\e(x_k)}(\un{0})$.

For every $k$, $\un{y}_{k+1}=f_{x_k x_{k+1}}(\un{y}_k)$ and
$\un{z}_{k+1}=f_{x_k x_{k+1}}(\un{z}_k)$, where $f_{x_k x_{k+1}}:=\Psi_{x_{k+1}}^{-1}\circ f\circ\Psi_{x_k}$. By (\ref{offenbach}),
$$
f_{x_k x_{k+1}}(v,w)=(A_k v+h_1(v,w), B_k w+h_2(v,w))\textrm{ on  $R_{Q_\e(x_k)}(\un{0})$},
$$
where $C_f^{-1}<|A_k|<e^{-\chi}$, $e^{\chi}<|B_k|<C_f$, and $\max\|\nabla h_i\|<3\e^2$. Thus
\begin{align*}
|y_{k+1}-z_{k+1}|&\leq |A_k|\cdot|y_k-z_k|+3\e^2\bigl(|y_k-z_k|+\Lip(F_k)|y_k-z_k|\bigr)\\
&\leq (e^{-\chi}+4\e^2)|y_k-z_k|<e^{-\frac{1}{2}\chi}|y_k-z_k|\leq \cdots\leq e^{-\frac{1}{2}(k+1)\chi}|y_0-z_0|.
\end{align*}
Since $\un{y}_0,\un{z}_0$ are on the graph of an $s$--admissible manifold in $\Psi_{x_0}^{p^u_0,p^s_0}$, their $x$--coordinates are in $[-p^s_0,p^s_0]$, so $|y_0-z_0|\leq 2p^s_0$. Thus $|y_k-z_k|\leq 2 e^{-\frac{1}{2}k\chi} p^s_0$.
Since $\un{y}_k=(y_k,F_k(y_k))$, $\un{z}_k=(z_k,F_k(z_k))$, and $\Lip(F_k)<\e$,
$\|\un{y}_k-\un{z}_k\|<3 p^s_0 e^{-\frac{1}{2}k\chi}$.

Pesin charts have Lipschitz constant less than two, so  $d(f^k(y),f^k(z))<6 p^s_0 e^{-\frac{1}{2}k\chi}$.

\medskip
\noindent
{\em Part 2.\/} Suppose $\e$ is so small that $e^{-\chi}+3\e^2+3\e^3<e^{-\frac{2}{3}\chi}$ and $C_f\e+3\e^2<1$. For every $y\in V^s$, let $\un{e}^s(y)$ denote the positively oriented unit tangent vector to $V^s$ at $y$. If $y\in V^s$,  then $\|df_y^k\un{e}^s(y)\|\leq 6e^{-\frac{2}{3}k\chi}\|C_\chi(x_0)^{-1}\|$  for all $k\geq 0$.

\medskip
\noindent
{\em Proof.\/}
If $y\in V^s$, then $f^k(y)\in W^s_k\subset\Psi_{x_k}[R_{Q_\e(x_k)}(\un{0})]$. So $df_y^k\un{e}^s(y)=(d\Psi_{x_k})_{\un{y}_k}{a_k\choose b_k}$ where ${a_k\choose b_k}$ is tangent to the graph of $F_k$. Since $\Lip(F_k)<\e$,
$
|b_k|\leq \e|a_k|
$
for all $k$.
The identity ${a_{k+1}\choose b_{k+1}}=(df_{x_k x_{k+1}})_{\un{y}_k}{a_k\choose b_k}$ holds. Since $\|\nabla h_i\|\leq 3\e^2$,
$$
{a_{k+1}\choose b_{k+1}}=\left(
\begin{array}{cc}
A_k+\frac{\del h_1}{\del x}(\un{y}_k) & \frac{\del h_1}{\del y}(\un{y}_k)\\
\frac{\del h_2}{\del x}(\un{y}_k) & B_k+\frac{\del h_2}{\del y}(\un{y}_k)
\end{array}
\right){{a_k}\choose{b_k}}={{(A_k\pm 3\e^2) a_k\pm 3\e^2|b_k|}\choose{(B_k\pm 3\e^2) b_k\pm 3\e^2 |a_k|}}.
$$
It follows that $|a_{k+1}|\leq (|A_k|+3\e^2+3\e^3)|a_k|$. By the bounds on $A_k$ and $B_k$ and the assumption on $\e$,
$$
|a_k|\leq e^{-\frac{2}{3}k\chi} |a_0|\textrm{ and }|b_k|\leq \e|a_k|\leq e^{-\frac{2}{3}k\chi}|a_0|.
$$
Returning to the defining relation $df_y^k\un{e}^s(y)=(d\Psi_{x_k})_{\un{y}_k}{a_k\choose b_k}$, and recalling that $\|d\Psi_{x_k}\|\leq 2$ (Theorem \ref{Theorem_OP_charts}),  we see that $\|df_y^k\un{e}^s(y)\|\leq 2\sqrt{2} e^{-\frac{2}{3}k\chi}|a_0|$.

Since  ${a_0\choose b_0}=(d\Psi_{x_0})_{\un{y}_0}^{-1}\un{e}^s(y)$,
$|a_0|\leq \|d\Psi_{x_0}^{-1}\|$, so $\|df_y^k\un{e}^s(y)\|\leq 2\sqrt{2} e^{-\frac{2}{3}k\chi}\|d\Psi_{x_0}^{-1}\|$.

 For every $x$, $\|d\Psi_x^{-1}\|\leq 2\|C_\chi(x)^{-1}\|$ because $C_\chi(x)^{-1}$ maps $B_{2Q_\e(x)}^{x}(\un{0})$ into
$B_{2\e^{3/\b}}(\un{0})\subset B_{2\e}(\un{0})\subset B_{\rho(M)}(\un{0})$, provided  $\e<\frac{1}{2}\rho(M)$, and by the definition  $\rho(M)$ is so small that $\|(d\exp_x^{-1})_y\|\leq 2$ for all $x\in M$ and $y\in B_{\rho(M)}(\un{0})$.

 It follows that $\|df_y^k\un{e}^s(y)\|\leq 6\|C_\chi(x_0)^{-1}\| e^{-\frac{2}{3}k\chi}$.

%
%

\medskip
\noindent
{\em Part 3.\/} The following holds for all $\e$ small enough: for all $y,z\in V^s$ and $n\geq 0$,
$\bigl|\log\|df^n_y \un{e}^s(y)\|-\log\|df^n_z \un{e}^s(z)\|\bigr|\leq Q_\e(x_0)^{\b/4}$.

\medskip
\noindent
{\em Proof.\/}
Call the quantity to be estimated $A$. For every $p\in V^s$,
\begin{align*}
df^n_p[\un{e}^s(p)]&=df^{n-1}_{f(p)}[df_p\un{e}^s(p)]=\pm\|df_p\un{e}^s(p)\|\cdot df_{f(p)}^{n-1}[\un{e}^s(f(p))]\\
&=\cdots =\pm\prod_{k=0}^{n-1}\|df_{f^k(p)}\un{e}^s(f^{k}(p))\|\cdot \un{e}^s(f^n(p)).
\end{align*}
Thus
$
A:=\left|\log\frac{\|df_y^n\un{e}^s(y)\|}{\|df_z^n\un{e}^s(z)\|}\right|\leq \sum\limits_{k=0}^{n-1}\left|\log\|df_{f^k(y)}\un{e}^s(f^k(y))\|-\log\|df_{f^k(z)}\un{e}^s(f^k(z))\|\right|.
$
We shall estimate the sum term-by-term, using the H\"older continuity of $df$.

In section \ref{SectionOC}
 we  covered $M$ by a finite collection $\mathfs D$ of open sets $D$, equipped with a smooth map $\Theta_D:TD\to \R^2$ s.t. $\Theta_D|_{T_x M}: T_x M\to\R^2$ is an isometry, and $\vartheta_x:=\Theta_D^{-1}|_{\R^2}:\R^2\to TD$ has the property that
$
(x,\un{v})\mapsto \vartheta_x(\un{v})
$
is Lipschitz on $D\x B_{1}(\un{0})$.
Since $f$ is a $C^{1+\b}$--diffeomorphism and $M$ is compact, $df_p[\un{v}]$  depends in a $\b$--H\"older way on $p$, and  in a Lipschitz way on $\un{v}$. It follows that   there exists a constant $H_0>1$ s.t. for every $D\in\mathfs D$, for every $y,z\in D$, and for every $\un{u},\un{v}\in\R^2$ of length one,
$
\biggl|\log\|df_y(\vartheta_y(\un{u}))\|-\log\|df_z(\vartheta_z(\un{v}))\|\biggr|<H_0\bigl(d(y,z)^\b+\|\un{u}-\un{v}\|\bigr)
$.

Choose $D_k\in\mathfs D$ s.t. $D_k\owns f^k(y),f^k(z)$. Such sets exist provided $\e$ is much smaller than the Lebesgue number of $\mathfs D$, because by part 1 $d(f^k(y),f^k(z))<6\e$.
Writing $\id=\Theta_{D_k}\circ\vartheta_{f^k(y)}$ and  $\id=\Theta_{D_k}\circ\vartheta_{f^k(z)}$, we see that
\begin{align}
A&\leq  \sum\limits_{k=0}^{n-1}\left|\log\|df_{f^k(y)}\vartheta_{f^k(y)}\Theta_{D_k}\un{e}^s(f^k(y))\|-
\log\|df_{f^k(z)}\vartheta_{f^k(z)}\Theta_{D_k}\un{e}^s(f^k(z))\|\right|\notag\\
&\leq \sum_{k=0}^{n-1}H_0\left(d(f^k(y),f^k(z))^\b+
\|\Theta_{D_k}\un{e}^s(f^k(y))-\Theta_{D_k}\un{e}^s(f^k(z))\|\right)\notag
\\
&\leq
\frac{H_0 (6p_0^s)^\b}{1-e^{-\frac{1}{2}\b\chi}}
+H_0\sum_{k=0}^{n-1}
\|\Theta_{D_k}\un{e}^s(f^k(y))-\Theta_{D_k}\un{e}^s(f^k(z))\|, \textrm{ by part 1.}\label{Plug_N_Here}
\end{align}

We estimate $N_k:=\|\Theta_{D_k}\un{e}^s(f^k(y))-\Theta_{D_k}\un{e}^s(f^k(z))\|$.
By definition, $\un{e}^s(f^k(y))$ and $\un{e}^s(f^k(z))$  are the positively oriented unit tangent vectors to $f^k(V^s)\subset W^s_k$, at $f^k(y)$ and $f^k(z)$. Defining $\un{y}_k$ and $\un{z}_k$ as before,  we obtain
 $$
 \un{e}^s(f^k(y))=\frac{(d\Psi_{x_k})_{\un{y}_k}{{1}\choose F_k'(y_k)}}{\|(d\Psi_{x_k})_{\un{y}_k}{{1}\choose F_k'(y_k)}\|}\textrm{ , }
  \un{e}^s(f^k(z))=\frac{(d\Psi_{x_k})_{\un{z}_k}{{1}\choose F_k'(z_k)}}{\|(d\Psi_{x_k})_{\un{z}_k}{{1}\choose F_k'(z_k)}\|}.
 $$
 We saw in part 1 that
$\|(d\Psi_{x_k})^{-1}_{\un{y}_k}\|$ and $\|(d\Psi_{x_k})^{-1}_{\un{z}_k}\|$ are bounded by $2\|C_{\chi}(x_k)^{-1}\|$, so the
 denominators are bounded below by $\frac{1}{2}\|C_\chi(x_k)^{-1}\|^{-1}$. Since for any two non-zero vectors $\un{v}, \un{u}$,
 $\bigl\|\un{v}/\|\un{v}\|-\un{u}/\|\un{u}\|\bigr\|<2\|\un{v}-\un{u}\|/\|\un{v}\|,
 $
 \begin{align*}
 N_k&\leq 2\|C_\chi(x_k)^{-1}\|\cdot
 \left\|
 \Theta_{D_k}(d\Psi_{x_k})_{\un{y}_k}{{1}\choose F_k'(y_k)}-\Theta_{D_k}(d\Psi_{x_k})_{\un{z}_k}{{1}\choose F_k'(z_k)}
 \right\|.\notag\\
  \end{align*}

On $D_k$ we can write  $\Psi_{x_k}=\exp_{x_k}\circ\vartheta_{x_k}\circ C_{x_k}$, where $\vartheta_{x_k}\circ C_{x_k}=C_\chi(x_k)$. Let
 $$
 \un{u}_k:=C_{\chi}(x_k)\un{y}_k, \un{u}_k':=C_{\chi}(x_k)\un{z}_k,\textrm{ and }\un{v}_k:=C_{x_k}{1\choose F_k'(y_k)},
 \un{v}_k':=C_{x_k}{1\choose F_k'(z_k)},
 $$
 then $
 N_k\leq 2\|C_\chi(x_k)^{-1}\|\cdot
 \left\|
 \Theta_{D_k}(d\exp_{x_k})_{\un{u}_k}[\vartheta_{x_k}(\un{v}_k)]-
 \Theta_{D_k}(d\exp_{x_k})_{\un{u}_k'}[\vartheta_{x_k}(\un{v}_k')] \right\|.
$
Since $\Theta_D, \vartheta_{x_k}$ are isometries,  $C_{x_k}$ are contractions, $\|(d\exp_{x_k})_{\un{u}_k}\|\leq 2$, and
$|F_k'(y_k)-F_k'(z_k)|\leq \frac{1}{2}|y_k-z_k|^{\b/3}$,
\begin{align*}
N_k&\leq 2\|C_\chi(x_k)^{-1}\|\cdot\left\|
\Theta_{D_k}(d\exp_{x_k})_{\un{u}_k}[\vartheta_{x_k}(\un{v}_k)]-
 \Theta_{D_k}(d\exp_{x_k})_{\un{u}_k}[\vartheta_{x_k}(\un{v}_k')]
\right\|+\\
&\hspace{1cm}+
2\|C_\chi(x_k)^{-1}\|\cdot\left\|
\Theta_{D_k}(d\exp_{x_k})_{\un{u}_k}[\vartheta_{x_k}(\un{v}_k')]-
 \Theta_{D_k}(d\exp_{x_k})_{\un{u}_k'}[\vartheta_{x_k}(\un{v}_k')]
\right\|\\
&\leq 2\|C_\chi(x_k)^{-1}\|\cdot |y_k-z_k|^{\b/3}+\\
&\hspace{1cm}+2\|C_\chi(x_k)^{-1}\|\cdot
\left\|
\Theta_{D_k}(d\exp_{x_k})_{\un{u}_k}[\vartheta_{x_k}(\un{v}_k')]-
 \Theta_{D_k}(d\exp_{x_k})_{\un{u}_k'}[\vartheta_{x_k}(\un{v}_k')]
\right\|.
\end{align*}

We study this expression. In what follows we identify the differential of a linear map with the map itself.

By construction, the map
$(x,\un{u},\un{v})\mapsto \left[\Theta_D\circ (d\exp_x)_{\un{u}}\right][\vartheta_x(\un{v})]$ is smooth on $D\x B_2(\un{0})\x B_2(\un{0})$ for every $D\in\mathfs D$. Therefore there exists a constant $E_0>1$   s.t. for every $(x,\un{u}_i,\un{v}_i)\in D\x B_2(\un{0})\x B_2(\un{0})$ and every $D\in \mathfs D$,
$$
\|\Theta_D(d\exp_x)_{\un{u}_1}[\vartheta_x(\un{v}_1)]-\Theta_D(d\exp_x)_{\un{u}_2}[\vartheta_x(\un{v}_2)]\|\leq E_0\bigl(\|\un{u}_1-\un{u}_2\|+
\|\un{v}_1-\un{v}_2\|\bigr).
$$
It follows that
\begin{align*}
N_k&\leq 2\|C_\chi(x_k)^{-1}\|\cdot \left( |y_k-z_k|^{\b/3}+E_0\left(\|\un{u}_k-\un{u}_k'\|+\|\un{v}_k-\un{v}_k'\|\right)\right)\\
&\leq 2\|C_\chi(x_k)^{-1}\|\cdot\left(|y_k-z_k|^{\b/3}+E_0\left(\|\un{y}_k-\un{z}_k\|+|y_k-z_k|^{\b/3}\right)\right)\\
&\leq 6E_0\|C_\chi(x_k)^{-1}\|\|\un{y}_k-\un{z}_k\|^{\b/3}\ \ (\because E_0>1)\\
&\leq 6E_0\|C_\chi(x_k)^{-1}\|(3p^s_0)^{\b/3} e^{-\frac{1}{6}\b\chi k}\ \ \textrm{because $\|\un{y}_k-\un{z}_k\|<3p_0^s e^{-\frac{1}{2}k\chi}$ (part 1)}\\
&\leq 9 E_0 \|C_\chi(x_k)^{-1}\|(p^s_0)^{\b/3} e^{-\frac{1}{6}\b\chi k}.
\end{align*}

By the definition of $Q_\e(\cdot)$,  $\|C_\chi(x_k)^{-1}\|\leq
\e^{1/4}Q_\e(x_k)^{-\b/12}\leq \e^{1/4}(p^s_k)^{-\b/12}$, and therefore
$
N_k\leq 9\e^{1/4} E_0 (p^s_k)^{-\b/12}(p^s_0)^{\b/3}e^{-\frac{1}{6}\b\chi k}.
$
Since $(\Psi_{x_i}^{p^u_i,p^s_i})_{i\in\Z}$ is a chain, $p^s_i=\min\{e^\e p^s_{i+1},Q_\e(x_i)\}\leq e^\e p^s_{i+1}$ for all $i$, whence $p^s_0\leq e^{k\e} p^s_k$.  It follows  that for all  $\e$ small enough,
\begin{equation}\label{N}
N_k\leq 9\e^{1/4}E_0 (p_0^s)^{\b/4}\exp[-\tfrac{1}{7}\b\chi k].
\end{equation}
Plugging this  in (\ref{Plug_N_Here}), we obtain
\begin{align*}
\left|\log\frac{\|df_y^n\un{e}^s(y)\|}{\|df_z^n\un{e}^s(z)\|}\right|&\leq \left(\frac{6^\b H_0(p_0^s)^{3\b/4}}{1-e^{-\frac{1}{2}\b\chi}}+\frac{9e^{1/4}E_0 H_0}{1-e^{-\frac{1}{7}\b\chi }}\right)(p_0^s)^{\b/4}\\
&< \left(\frac{9\e^{3\b/4}E_0 H_0}{1-e^{-\frac{1}{7}\b\chi}}\right)Q_\e(x_0)^{\b/4}.
\end{align*}
The term in the brackets is less than one for every $\e$ small enough. How small depends only on $M$ (through $E_0$), $f$ (through $H_0$ and $\b$), and $\chi$.
\hfill$\Box$

\medskip
\noindent
{\bf Proof of Proposition \ref{Prop_Uniqueness}}
We continue to use the notation of the previous proof.

Assume that $V^s\cap U^s\neq \emptyset$. We show that  $V^s\subseteq U^s$ or $U^s\subseteq V^s$.

Since $V^s$ stays in windows, there is a positive chain $(\Psi_{x_i}^{p^u_i,p^u_i})_{i\geq 0}$ such that
$\Psi_{x_i}^{p^u_i,p^s_i}=\Psi_x^{p^u,p^s}$ and such that for all $i\geq 0$, $f^i(V^s)\subset W^s_i$ where $W^s_i$ is an $s$--admissible manifold in $\Psi_{x_i}^{p^u_i,p^s_i}$.

\medskip
\noindent
{\em Claim 1.\/}
The following holds for all $\e$ small enough.
 $f^n(V^s)\subseteq \Psi_{x_n}[R_{\frac{1}{2}Q_\e(x_n)}(\un{0})]$ for all $n$ large enough.

\medskip
\noindent
{\em Proof.\/} Suppose $y\in V^s$, and write as in part 1 of the previous proof,
$f^n(y)=\Psi_{x_n}(\un{y}_n)$ where $\un{y}_n=(y_n,F_n(y_n))$ and $F_n$ is the function which represents $W^s_n$ in $\Psi_{x_n}$. We have $\un{y}_{n+1}=f_{x_n x_{n+1}}(\un{y}_n)$, which implies in the notation of the previous proof that if $\e$ is small enough, then
\begin{align*}
|y_{n+1}|&\leq |A_n|\cdot|y_n|+|h_1(\un{y}_n)|\leq |A_n|\cdot|y_n|+|h_1(\un{0})|+\|\nabla h_1\|(|y_n|+|F_n(y_n)|)\\
&< e^{-\chi}|y_n|+\e \eta_n+3\e^2(|y_n|+p^s_n)<(e^{-\chi}+3\e^2)|y_n|+2\e p^s_n\\
&<(e^{-\chi}+3\e^2)|y_n|+2\e \min\{ e^\e p^s_{n+1},Q_\e(x_n)\}\\
&<(e^{-\chi}+3\e^2)|y_n|+2e^\e \e p^s_{n+1}<e^{-\chi/2}|y_n|+4\e p^s_{n+1}.
\end{align*}

We see that $|y_n|\leq a_n$ where $a_n$ is defined by induction by
$$
a_0:=Q_\e(x_0)\textrm{ and }a_{n+1}=e^{-\chi/2}a_n+4\e p^s_{n+1}.
$$

We claim that if $\e$ is small enough, then $a_n<\frac{1}{4}p^s_n$ for some $n$.
Otherwise, $p^s_n\leq 4 a_n$ for all $n$, whence
$a_{n+1}\leq (e^{-\chi/2}+16\e) a_n$ for all $n$, which implies that
$
a_n<(e^{-\frac{1}{2}\chi}+16\e)^n a_0.
$
But by assumption,  $a_n\geq \frac{1}{4}p^s_n\geq \frac{1}{4}(p^u_n\wedge p^s_n)\geq \frac{1}{4}e^{-\e n}(p^u_0\wedge p^s_0)$ (Lemma \ref{Lemma_Subordinated_Tempered}), so necessarily
$
e^{-\e}\leq e^{-\chi/2}+16\e.
$
If $\e$ is small enough, this is false and we obtain a contradiction. It follows that $\exists n$ s.t.  $a_n<\frac{1}{4}p^s_n$.

It is clear from the definition of $a_n$, that  if $\e$ is small enough then
$a_n<\frac{1}{4}p^s_n\Longrightarrow a_{n+1}<\frac{1}{4}p^s_{n+1}$. Thus $a_n<\frac{1}{4}p^s_n$ for all $n$ large enough.

In particular, $|y_n|<\frac{1}{4}Q_\e(x_n)$ for all $n$ large enough. Since $\un{y}_n=(y_n,F_n(y_n))$ and
$
|F_n(y_n)|\leq |F_n(0)|+\Lip(F_n)|y_n|<(10^{-3}+\e)Q_\e(x_n),
$
$\|\un{y}_n\|<\frac{1}{2}Q_\e(x_n)$ for all $n$ large enough.

\medskip
\noindent
{\em Claim 2.\/}
The following holds for all $\e$ small enough:
 $f^n(U^s)\subseteq \Psi_{x_n}[R_{Q_\e(x_n)}(\un{0})]$ for all $n$ large enough.

\medskip
\noindent
{\em Proof.\/} $U^s$ stays in windows, so there exists a positive chain $\{\Psi_{y_i}^{q^u_i,q^s_i}\}_{i\geq 0}$ such that $\Psi_{y_0}^{q^u_0,q^s_0}=\Psi_y^{q^u,q^s}$ and such that for all $i\geq 0$, $f^i(U^s)$ is a subset of an $s$--admissible manifold in $\Psi_{y_i}^{q^u_i,q^s_i}$.

Let $z$ be a point in $U^s\cap V^s$. By Part 1 of Theorem \ref{Prop_Staying_In_Windows},  for any $w\in U^s$ $d(f^n(z),f^n(w))\leq 6 q^s_0 e^{-\frac{1}{2}n\chi}$. Therefore
$f^n(z), f^n(w)\in B_{Q_\e(x_n)+6q^s_0}(x_n)\subset B_{7\e}(x_n)$.
If $\e<\frac{1}{7}\rho(M)$ (cf. \S\ref{SectionPC}), then
$
\|\exp_{x_n}^{-1}[f^n(z)]-\exp_{x_n}^{-1}[f^n(w)]\|<12e^{-\frac{1}{2}n\chi} q^s_0,
$
so
$$
\left\|\Psi_{x_n}^{-1}[f^n(z)]-\Psi_{x_n}^{-1}[f^n(w)]\right\|<\|C_\chi(x_n)^{-1}\|\cdot 12 e^{-\frac{1}{2}n\chi}q^s_0.
$$
Since $p^s_n\leq Q_\e(x_n)\ll \|C_\chi(x_n)^{-1}\|^{-1}$,
$
\left\|\Psi_{x_n}^{-1}[f^n(z)]-\Psi_{x_n}^{-1}[f^n(w)]\right\|
\leq 12(p_n^s)^{-1}q^s_0e^{-\frac{1}{2}n\chi}.
$

Since $\{\Psi_{x_i}^{p^u_i,p^s_i}\}_{i\in\Z}$ is a chain, $p^s_i=\max\{e^\e p^s_{i+1},Q_\e(x_i)\}\leq e^\e p^s_{i+1}$ for all $i$. It follows that  $p^s_0\leq e^{n\e}p^s_n$, whence
\begin{align*}
\left\|\Psi_{x_n}^{-1}[f^n(z)]-\Psi_{x_n}^{-1}[f^n(w)]\right\|<12\left(\frac{q^s_0}{p^s_0}\right) e^{-\frac{1}{2}n\chi+n\e}\xrightarrow[n\to\infty]{}0\textrm{ exponentially fast}.
\end{align*}
Since  $Q_\e(x_n)\geq (p^u_n\wedge p^s_n)\geq e^{-\e n}(p^u_0\wedge p^s_0)$, for all $n$ large enough
\begin{align*}
\left\|\Psi_{x_n}^{-1}[f^n(z)]-\Psi_{x_n}^{-1}[f^n(w)]\right\|<\frac{1}{2}Q_\e(x_n).
\end{align*}
How large depends only on $(p^s_0,p^u_0)$ and $q^s_0$.

Since, by claim 1, $\|\Psi_{x_n}^{-1}(f^n(z))\|<\frac{1}{2}Q_\e(x_n)$ for all $n$ large enough, we have that $\|\Psi_{x_n}^{-1}(f^n(w))\|<Q_\e(x_n)$ for all $n$ large enough.
All the estimates are uniform in $w\in U^s$, so the claim is proved.

\medskip
\noindent
{\em Claim 3.\/} Recall that $V^s$ is $s$--admissible in $\Psi_x^{p^u,p^s}$ and $U^s$ is $s$--admissible in $\Psi_y^{q^u,q^s}$. If $p^s\leq q^s$ then $V^s\subseteq U^s$, and if $q^s\leq p^s$ then $U^s\subseteq V^s$.

\medskip
\noindent
{\em Proof.\/}
W.l.o.g.  $p^s\leq q^s$.
Pick $n_0$ s.t. $f^n(U^s), f^n(V^s)\subset\Psi_{x_n}[R_{Q_\e(x_n)}(\un{0})]$ for all $n\geq n_0$, then
$
f^{n_0}(V^s), f^{n_0}(U^s)\subset W^s:=V^s[(\Psi_{x_i}^{p^u_i,p^s_i})_{i\geq n_0}]
$
(Proposition \ref{Prop_V} (4)).

Let $G$ denote the function which represents $W^s$ in $\Psi_{x_{n_0}}$, then
 $\Psi_{x_n}^{-1}[f^n(U^s)]$ and  $\Psi_{x_n}^{-1}[f^n(V^s)]$ are two connected subsets of $\graph(G)$. Write
 \begin{align*}
 f^n(V^s)&=\Psi_{x_n}\{(t,G(t)):t\in [\a,\b]\},\\
 f^n(U^s)&=\Psi_{x_n}\{(t,G(t)):t\in [\a',\b']\}.
 \end{align*}
The manifold $f^n(V^s)$ has endpoints $A:=\Psi_{x_n}(\a,G(\a))$, $B:=\Psi_{x_n}(\b,G(\b))$, and the manifold $f^n(U^s)$ has endpoints $A':=\Psi_{x_n}(\a',G(\a'))$, $B':=\Psi_{x_n}(\b',G(\b'))$.

Since $V^s$ and $U^s$ intersect, $f^n(V^s)$ and $f^n(U^s)$ intersect. Consequently,   $[\a,\b]$ and $[\a',\b']$ overlap. We use the assumption that $p^s\leq q^s$ to show that $[\a,\b]\subseteq [\a',\b']$.

Otherwise $\a<\a'$ or $\b>\b'$. Assume by contradiction that $\a<\a'$. Then  $A'$ is in the relative interior
of $f^n(V^s)$. Since $f$ is a homeomorphism, $f^{-n}(A')$  is in the relative interior of $V^s$. Since $f^{-n}(A')$ is an endpoint of $U^s$, we obtain that $U^s$ has an endpoint at the relative interior of $V^s$.

We now use the assumption that $x=y$, and view $V^s$ and $U^s$ as sub-manifolds of the chart $\Psi_x$.
The endpoints of $U^s$ have $s$--coordinates equal in absolute value to $q^s$, and the points on $V^s$ have $s$--coordinates in $[-p^s,p^s]$. It follows that $q^s<p^s$, in contradiction to our assumption. The contradiction shows that $\a\geq \a'$.  Similarly one shows that $\b\leq \b'$, with the conclusion that $[\a,\b]\subset [\a',\b']$. It follows that  $f^n(V^s)\subseteq f^n(U^s)$, whence $V^s\subseteq U^s$.\hfill$\Box$

\section*{Acknowledgements}
The author would like to thank J. Buzzi, A. Katok, F. Ledrappier, and  M. Pollicott for useful discussions.

\end{document}